\documentclass[a4paper,11pt,reqno]{amsart}

\usepackage[utf8]{inputenc}
\RequirePackage[l2tabu, orthodox]{nag}
\usepackage[T1]{fontenc}
\usepackage{lmodern}
\usepackage{amsthm,amssymb,amsmath}
\usepackage{latexsym}
\usepackage{graphicx}
\usepackage{subfigure}
\usepackage{tikz}
\usetikzlibrary{arrows}
\usetikzlibrary{calc}
\usepackage[linktocpage=true]{hyperref}
\usepackage[left=3.8cm, right=3.8cm, top=3cm, bottom=3cm]{geometry}
\usepackage{color}
\usepackage{here}
\usepackage{mathrsfs}
\usepackage{pgf,tikz}
\usetikzlibrary{decorations.pathreplacing, shapes.multipart, arrows, matrix, shapes}
\usetikzlibrary{patterns}
\usepackage{caption}
\usepackage[frenchb,english]{babel}
\usepackage[all] {xy}
\usepackage{enumitem}

\makeatletter
\newcommand\footnoteref[1]{\protected@xdef\@thefnmark{\ref{#1}}\@footnotemark}
\makeatother

\hypersetup{ colorlinks   = true, %Colours links instead of ugly boxes
urlcolor  = blue, %Colour for external hyperlinks
linkcolor    = blue, %Colour of internal links
citecolor   = blue %Colour of citations
}

\def\beq{\begin{equation}}
\def\eeq{\end{equation}}

\def\det{\mathrm{det}\ }

\def\restriction#1#2{\mathchoice
              {\setbox1\hbox{${\displaystyle #1}_{\scriptstyle #2}$}
              \end{enumerate}
\restrictionaux{#1}{#2}}
              {\setbox1\hbox{${\textstyle #1}_{\scriptstyle #2}$}
              \restrictionaux{#1}{#2}}
              {\setbox1\hbox{${\scriptstyle #1}_{\scriptscriptstyle #2}$}
              \restrictionaux{#1}{#2}}
              {\setbox1\hbox{${\scriptscriptstyle #1}_{\scriptscriptstyle #2}$}
              \restrictionaux{#1}{#2}}}
\def\restrictionaux#1#2{{#1\,\smash{\vrule height .8\ht1 depth .85\dp1}}_{\,#2}}

\newcommand{\dif}{{\operatorname{Diff}^2(\mathbb{T}^3)}}
\newcommand{\Z}{{\mathbb Z}}
\newcommand{\R}{{\mathbb R}}

\newcommand{\T}{{\mathbb T}}
\newcommand{\E}{{\mathbb E}}
\newcommand{\N}{{\mathbb N}}

\newcommand{\TT}{{\mathbb{T}^3}}
\newcommand{\loc}{{\operatorname{loc}}}

\newcommand{\cA}{{\mathcal A}}
\newcommand{\cC}{{\mathcal C}}
\newcommand{\cD}{{\mathcal D}}

\newcommand{\cF}{{\mathcal F}}

\newcommand{\cB}{{\mathcal B}}
\newcommand{\cH}{{\mathcal H}}

\newcommand{\cL}{{\mathcal L}}
\newcommand{\cM}{{\mathcal M}}
\newcommand{\cN}{{\mathcal N}}
\newcommand{\cP}{{\mathcal P}}

\newcommand{\cV}{{\mathcal V}}
\newcommand{\cW}{{\mathcal W}}
\newcommand{\cWcu}{{\mathcal W}^{cu}}
\newcommand{\cWcs}{{\mathcal W}^{cs}}
\newcommand{\cWc}{{\mathcal W}^{c}}
\newcommand{\cWu}{{\mathcal W}^{u}}
\newcommand{\cWs}{{\mathcal W}^{s}}
\newcommand{\cPH}{{\mathcal{PH}}}

\newcommand{\cU}{{\mathcal U}}

\newcommand{\eps}{\varepsilon}

\newcommand{\length}{\operatorname{length}}
\newcommand{\eqdef}{\stackrel{\scriptscriptstyle\rm def.}{=}}
\def\dans{\mathop{\subset}}
\newcommand{\Leb}{\mathrm{Leb}}

\newtheorem{theorem}{Theorem}[section]
\newtheorem{remark}[theorem]{Remark}

\newtheorem{lemma}[theorem]{Lemma}
\newtheorem{defi}[theorem]{Definition}
\newtheorem{prop}[theorem]{Proposition}
\newtheorem{corollary}[theorem]{Corollary}

\newtheorem{theoalph}{Theorem}

\newtheorem*{conjecture}{Conjecture}

\begin{document}
	
\title[]{Rigidity of $\mathbf{\textit{U}}$-Gibbs measures near conservative Anosov diffeomorphisms on $\T^3$}

\author[Sébastien Alvarez, Martin Leguil, Davi Obata, Bruno Santiago]{Sébastien Alvarez, Martin Leguil, Davi Obata, Bruno Santiago}
\address{}
\email{}

\date{\today}

\begin{abstract} 
We show that within a $C^1$-neighbourhood $\cU$ of the set of volume preserving Anosov diffeomorphisms on the three-torus $\TT$ which are strongly partially hyperbolic with expanding center, any $f\in\cU\cap \dif$ satisfies the dichotomy: either the strong stable and unstable bundles $E^s$, $E^u$ of $f$ are jointly integrable, or any fully supported $u$-Gibbs measure of $f$ is SRB.
\end{abstract}

\maketitle

\tableofcontents

\section{Introduction}\label{intro}
%\subsection{Color code}
%
%\begin{itemize}
%\item {\color{orange} Asaf - Report X}
%\item {\color{blue} Gogolev? - Report Y}
%\item {\color{red} Mauricio - Report Z}
%\item {\color{brown} Changes in normal forms and leafwise quotient measures}
%\end{itemize}
%

\subsection{Context}

Invariant foliations play a key role in partially and uniformly hyperbolic dynamics. For example,  they can be used to obtain ergodicity, topological transitivity and mixing for certain systems. In the path of trying to understand these foliations, one can investigate their topological and ergodic properties. For topological properties, one may ask about minimality, or robust minimality, of the invariant foliations.  In this paper we are going to focus on understanding ergodic properties of the invariant foliations for a certain type of dynamical system.

We refer the reader to Section \ref{section definitions notations} for the definition of the dynamical objects that appear in this introduction. 
Let us denote by $\TT\eqdef\R^3/\Z^3$ the three-dimensional torus. We let $\cA^2(\TT)\subset\dif$ be the set of Anosov diffeomorphisms which are strongly partially hyperbolic with uniformly expanding center, that is, a diffeomorphism $f$ belongs to $\cA^2(\TT)$ if $f$ is Anosov and admits a splitting 
\[
T\TT = E^s \oplus E^c \oplus E^u,
\]
where $E^c$ expands uniformly under the action of $Df$. A diffeomorphism in $\cA^2(\TT)$ can be seen as an Anosov and as a partially hyperbolic diffeomorphism.  In this context the bundles $E^s,E^u,E^u\oplus E^c,E^s\oplus E^c$ and $E^c$ integrate to invariant foliations respectively denoted by $\cW^s,\cW^u,\cW^{cu},\cW^{cs}$ and $\cW^c$, and called the stable, unstable, center-unstable, center-stable and center foliations (see \cite{BrinBuragoIvanov,potrie}). If we see $f$ as an Anosov diffeomorphism, $\cW^{cu}$ is a (two dimensional) unstable foliation. If we see it as a partially hyperbolic diffeomorphism, $\cW^u$ is the (one dimensional) strong-unstable foliation.

There are two types of elements of $\cA^2(\TT)$: the \emph{conservative elements}, which preserve some volume (that must be ergodic by Hopf's argument) and form a set denoted by $\cA_m^2(\TT)$, and the \emph{dissipative ones}, which don't. In both cases there exists a unique invariant measure which is the ``most compatible'' with the volume and that is called the \emph{SRB measure} (for Sinai-Ruelle-Bowen): \emph{measures that are absolutely continuous with respect to the Lebesgue measure along center unstable leaves (see \S \ref{subsec.uandsrb}). In particular, they capture the ``statistical'' behavior of Lebesgue-almost every point (see \cite{LSYoung})}. SRB measures are very important in the theory of smooth dynamics.  Palis conjectured that for a typical dynamical system there are finitely many attractors,  each attractor supporting a unique SRB measure and these measures capture the behavior of Lebesgue almost every point \cite{Palis}.  This conjecture remains open.  

\subsection{Dynamics of (center)-unstable foliation}

Let $f\in \cA^2(\TT)$ with a splitting $T\TT = E^s \oplus E^c \oplus E^u$. The dynamics of the center-unstable foliation $\cW^{cu}$ is very well understood.  It is minimal (i.e., every leaf is dense in $\TT$) and there is a unique SRB measure.

On the other hand, recall that properties of the strong-unstable foliation $\cW^u$ are especially interesting for dissipative dynamics: the study of topological and ergodic properties of attractors or quasi-attractors (which are $\cW^u$-saturated, i.e., they contain its $\cW^u$-leaves) is closely related to the problem of understanding properties of these foliations (or laminations). Apart from some finiteness results (see \cite{CrovisierPotrieSamba,NancyChichi,HammPotrie}), the dynamical properties of strong-unstable foliation are not well understood even in the uniformly hyperbolic setting. 

For instance, it was only recently announced by Avila-Crovisier-Eskin-Potrie-Wilkinson-Zhang that $\cW^u$ is minimal for any $C^{1+\alpha}$ Anosov diffeomorphism of $\TT$.  In higher dimensions, Avila-Crovisier-Wilkinson recently announced that $C^1$-openly and $C^r$-densely among the transitive Anosov diffeomorphisms admitting a decomposition $E^s \oplus E^c \oplus E^u$, where $E^c$ is one dimensional and uniformly expanding,  the strong unstable manifold is minimal.

A type of invariant measures that is associated with $\cW^u$ are the so-called $u$-Gibbs measures. A measure is $u$-Gibbs if it admits conditional measures along $\cW^u$ leaves that are absolutely continuous with respect to the Lebesgue measure of these leaves. In particular the support of such a measure is $\cW^u$-saturated. Let us make a few remarks about $u$-Gibbs and SRB measures in our setting:
\begin{enumerate}
\item SRB measures are absolutely continuous along two dimensional objects ($\cW^{cu}$ leaves), while $u$-Gibbs are absolutely continuous along one dimensional objects ($\cW^u$ leaves).
\item SRB measures are also $u$-Gibbs measures. 
\item\label{point trois} In general, we don't know when a $u$-Gibbs measure is an SRB measure. 
\end{enumerate}

\subsection{Main result} The goal of this paper is related to item \eqref{point trois} above. We are interested in knowing when the $u$-Gibbs property implies SRB. In other words, given a measure that is absolutely continuous along $\cW^u$, when can we show that this measure is absolutely continuous along $\cW^{cu}$?

For $\cA^2(\TT)$, we say that $E^s$ and $E^u$ are \emph{jointly integrable} if there exists a two dimensional foliation $\cW^{su}$ tangent to $E^s \oplus E^u$. It is known that $C^1$-openly and $C^2$-densely in $\cA^2(\TT)$ the directions $E^s$ and $E^u$ are not jointly integrable (see \cite{HertzHertzUres}). 

Our main result is the following:

\begin{theoalph}
	\label{mainthm.dicotomia}
There exists an open neighbourhood $\cU$ of $\cA^2_{m}(\TT)$ within   $\dif$ so that for every $f\in\cU$, either
	\begin{enumerate}
		\item $E^s$ and $E^u$ are jointly integrable, or
		\item\label{case deux}  any fully supported ergodic $u$-Gibbs measure $\mu$ is SRB.
	\end{enumerate} 
\end{theoalph}
\begin{remark}
The hypothesis that the system is near a volume preserving one is used to have $C^1$-stable holonomies. The conclusion of Theorem \ref{mainthm.dicotomia} also holds for $f\in \cA^2(\TT)$ with $C^1$-stable holonomies, see Theorem \ref{mainthm.technique}.
\end{remark}

\begin{remark}
As we mentioned before, Avila et al.  announced that for any $C^{1+\alpha}$ Anosov diffeomorphism in $\TT$ the strong unstable foliation is minimal. Since the support of any $u$-Gibbs measure is saturated by $\cW^u$-leaves, their result would imply that the case \eqref{case deux} of Theorem \ref{mainthm.dicotomia} could be improved to every $u$-Gibbs measure is SRB.  We could also apply the result announced by Avila-Crovisier-Wilkinson, that we mentioned, to obtain that open and densely in $\mathcal{U}$ every $u$-Gibbs is SRB, where $\mathcal{U}$ is the open set from Theorem \ref{mainthm.dicotomia}.  
\end{remark}

Let us mention one application of our result.  In \cite{GKM},  Gogolev, Kolmogorov and Maimon consider the linear Anosov diffeomorphism on $\TT$ induced by the matrix
\[
A = 
\begin{bmatrix}
2 & 1 & 0\\
1 & 2 &1\\
0 & 1 & 1
\end{bmatrix}.
\]
The eigenvalues of $A$ are real and approximately $0.2$, $1.55$ and $3.25$.  In particular, the diffeomorphism induced by this matrix belong to $\cA^2_m(\TT)$.  In \cite{GKM}, the authors did a numerical study for two explicit families of perturbations of $A$, one conservative and one dissipative. Their numerical study indicates that for these families of perturbations of $A$, there is a unique $u$-Gibbs measure and this measure coincides with  the SRB measure.  They make the following conjecture.

\begin{conjecture}[\cite{GKM}, Conjecture $1.3$] 
For all analytic diffeomorphisms $f$ in a sufficiently small neighbourhood of $A$ there exists a unique $u$-Gibbs measure. 
\end{conjecture}

Our result gives that for any $C^2$-diffeomorphism in a neighbourhood of $A$, either $E^s$ and $E^u$ are jointly integrable or any $u$-Gibbs fully supported is SRB. If we assume Avila et al.'s result, then we would obtain that either we have joint integrability or there is only one $u$-Gibbs measure. In Section \ref{s.final_remarks} we introduce Gogolev-Kolmogorov-Maimon's conservative and dissipative families and prove that for both examples, $E^s$ and $E^u$ are not jointly integrable, applying a very convenient criterion of Gan-Shi \cite{GanShi}. That gives a theoretical explanation for their numerical study. Let us remark that in \cite{GKM} the authors also make conjectures about transitivity and minimality of $\cW^u$. In \cite{HertzUres}, Hertz-Ures gave a positive answer to their transitivity conjecture.

\subsection{Related works  and further results}

One can think of a $u$-Gibbs measure as a measure that is ``homogeneous'' along strong unstable manifolds and an SRB measure as being ``homogeneous'' along entire unstable manifolds. 

In a series of celebrated works, in the homogeneous setting, Ratner classified measures that are invariant by the action of a unipotent group  \cite{Rat1, Rat2, Rat3, Rat4}.  She proved that such measures are homogeneous, i.e.,  they are the Haar measure of some subgroup. Observe that unipotent flows parameterize unstable manifolds of the geodesic flow  on surfaces with constant negative curvature (the horocycle flow).  Hence,  a consequence of Ratner's measure rigidity result is measure rigidity of the $u$-Gibbs measures of the geodesic flow on surfaces with constant negative curvature. A key idea in Ratner's approach is the so-called polynomial drift, which allowed her to obtain extra invariance of the measure from invariance along orbits of the unipotent flow. 

In \cite{BenoistQuintI}, Benoist-Quint introduced the  idea of exponential drift to prove a measure rigidity result for stationary measures of  a Zariski dense random walk on homogeneous spaces.  

Outside the homogeneous setting, Eskin-Mirzakhani gave a non trivial modification of the exponential drift strategy, which is called the factorization method, to prove measure rigidity results for the action of $\mathrm{SL}(2,\mathbb{R})$ on moduli spaces \cite{EskinMirzakhani}.  Since then, these ideas were pushed to some different settings. In \cite{BRH}, Brown-Hertz classified the hyperbolic stationary measures of random products of surface diffeomorphisms. Cantat-Dujardin applied Brown-Hertz's result to classify random products of automorphisms of real and complex projective spaces \cite{CantatDujardin}. 

In the partially hyperbolic setting, the third author adapted Brown-Hertz's result to obtain a rigidity result for $u$-Gibbs measures for partially hyperbolic skew products with two dimensional center \cite{Obata}.  

In \cite{Katz}, Katz adapts the Eskin-Mirzakhani strategy for the smooth setting.  He proved that for any $C^{\infty}$ Anosov diffeomorphism $f$ having a splitting $TM = E^s \oplus E^c \oplus E^u$, where $E^c$ is one dimensional and  expanding, any $u$-Gibbs measure that verifies a technical condition called \emph{QNI} (quantified non-integrability) is SRB.  Eskin-Potrie-Zhang, in an ongoing work \cite{EskinPotrieZha}, obtained equivalent notions to QNI that are easier to work with. Their result will be used by Avila et al., in another ongoing project, to prove that for any Anosov diffeomorphism in $\cA^{\infty}(\TT)$ either $E^s$ and $E^u$ are jointly integrable, or every $u$-Gibbs measure is SRB.  All of these works are stated for $C^{\infty}$ regularity, but they can be obtained for $C^r$ regularity for $r\ggg1$. Part of the goal of this paper is to obtain this type of measure rigidity result for $u$-Gibbs measures, but in lower regularity (in our case $C^2$).  

\subsection{Ingredients of the proof}

The first ingredient concerns the transversality condition. We replace Katz's QNI condition by a zero-one law for angles inspired by Brown-Hertz \cite[Lemma 7.1]{BRH_Little} (see also \cite{BenoistQuintI}). We use the fact that
% for $x,y$ in the same stable manifold, 
stable holonomies 
$
(H_{x,y}^s)_{x,y}
$ are $C^1$ for any diffeomorphism $f \in \mathcal{A}^2(\TT)$ close to a conservative one (see Lemma \ref{c one holon}); thus, for $x,y$ in the same stable manifold, we can define an (unoriented) angle (see Figure \ref{hol angles}) 
\begin{equation*}
\alpha^s(x,y)\eqdef\angle (DH_{x,y}^s E^u(x),E^u(y)).
\end{equation*}

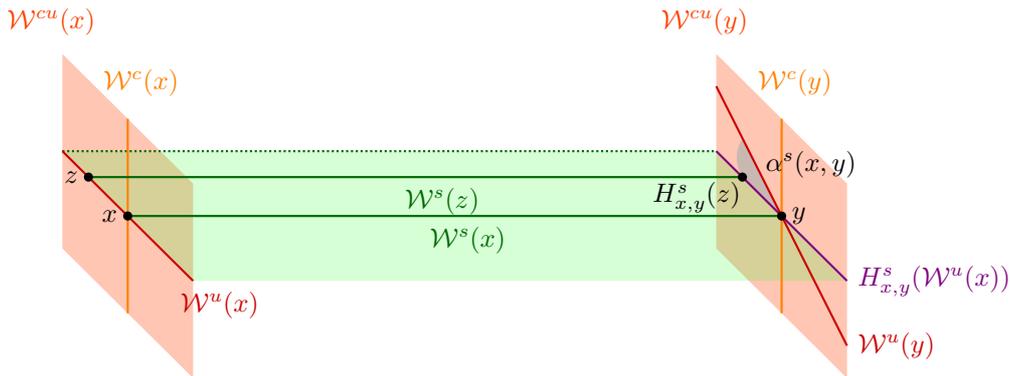
\begin{figure}[h]
	\centering
	\begin{tikzpicture}[scale=.86]
	\fill[red!50!orange, opacity=.3] (-1,2.5)--(1,0.5)--(1,-2.5)--(-1,-0.5)--(-1,2.5);
	\fill[red!50!orange, opacity=.3] (9,2.5)--(11,0.5)--(11,-2.5)--(9,-0.5)--(9,2.5);
	\fill[green!80!white, opacity=.2] (-1,1)--(1,-1)--(11,-1)--(9,1)--(-1,1);
	\draw[densely dotted, green!40!black, thick] (-1,1)--(9,1); 
	%\draw[densely dotted, green!40!black, thick] (1,-1)--(11,-1)node[midway,below]{\tiny $\color{red!50!green}\quad\mathcal{W}^{su}(x)$}; 
	\draw[green!40!black, thick] (0,0)--(10,0) node[midway,below]{\small $\quad\mathcal{W}^{s}(x)$}; 
	\draw[orange, thick] (0,-1.5)--(0,1.5); %node[above]{\tiny $\quad\mathcal{W}^{c}(x)$}; 
	\fill[fill=black] (0,1.7) node[above]{\color{orange}\small $\quad\mathcal{W}^{c}(x)$};
	%node[above]{\tiny $\quad\mathcal{W}^{c}(y)$}; 
	\fill[fill=black] (10,1.7) node[above]{\color{orange}\small $\quad\mathcal{W}^{c}(y)$};
	\fill[black!30!white, opacity=.7] (10,0) -- (9.4,0.6) to[bend left] (9.4,1.2) -- (10,0);
	\draw[orange, thick] (10,-1.5)--(10,1.5);
	\fill[fill=black] (11.3,0.8) node[left]{\small  $\alpha^s(x,y)$}; 
	\draw[red!80!black,thick] (-1,1)--(1,-1) node[right, below]{\small $\qquad \mathcal{W}^u(x)$};
	\draw[green!40!black,thick] (-0.6,0.6)--(9.4,0.6) node[midway, below]{\small $\qquad \mathcal{W}^s(z)$};
	%\draw[red!80!black,thick] (2,1)--(4,-1) node[right, above]{};
	%\draw[red!80!black,thick] (6,1)--(8,-1) node[right, above]{};
	\draw[violet,thick] (9,1)--(11,-1) node[right]{\small $H_{x,y}^s(\mathcal{W}^u(x))$};
	\draw[red!80!black,thick] (9,2)--(11,-2) node[right]{\small $\mathcal{W}^u(y)$};
	\fill[fill=black] (-2,3) node[right]{\small  $\color{red!50!orange} \mathcal{W}^{cu}(x)$};
	\fill[fill=black] (8,3) node[right]{\small  $\color{red!50!orange}  \mathcal{W}^{cu}(y)$}; 
	\fill[fill=black] (0,0) node[left]{\small  $x$} circle (2pt);
	\fill[fill=black] (-0.6,0.6) node[left]{\small $z$} circle (2pt);
	\fill[fill=black] (10,0) node[right]{\small $y$} circle (2pt);  
	\fill[fill=black] (9.4,0.6) circle (2pt);
	\fill[fill=black] (8.7,-0.1) node[left, above]{\small $H^s_{x,y}(z)$};
	\end{tikzpicture} 
	\caption{\label{hol angles} Stable holonomies and the angle function.}
\end{figure}

Note that conditional measures $\mu_x^s$ on stable manifolds  are not well defined. But full- and zero-measure sets for $\mu_x^s$ are well defined, see \S \ref{sub.subord partitions disint}. We can now state our zero-one law (see Theorem \ref{th_0-1} for a slightly more general statement). 
\begin{theoalph}[A zero-one law for angles]\label{thm.zeroonelawintro}
	There exists an open neighbourhood $\mathcal{U}$ of $\mathcal{A}_m^2(\TT)$ within $\mathrm{Diff}^2(\TT)$ such that for any $f\in \mathcal{U}$, and for any ergodic $f$-invariant measure $\mu$, the following dichotomy holds:
	 \begin{enumerate}
	 	\item for $\mu$-a.e. $x \in \TT$ and $\mu_x^s$-a.e. $y \in \mathcal{W}^s(x)$, $\alpha^s(x,y)=0$;
	 	\item for $\mu$-a.e. $x \in \TT$ and $\mu_x^s$-a.e. $y \in \mathcal{W}^s(x)$, $\alpha^s(x,y)>0$. 
	 \end{enumerate}
\end{theoalph}
This theorem is stated here in terms of the angle function $\alpha^s$; yet, it seems possible to generalize it to a broader context (see Remark \ref{remark general zero un}). 

%Let us say a few words on the reason why we can prove our result in lower regularity.  
The other important ingredient in our proof is the existence of \emph{normal forms} for the dynamics along two-dimensional unstable foliations. The dynamics along unstable manifolds are simplified when they are looked at using normal forms. The theory of non-stationary normal forms has been studied quite extensively since the pioneering work of Guysinsky-Katok \cite{GuyKat}, see for instance \cite{KalSadI,KalSadII,GogolevKalSad}.  Katz uses the result from \cite{KalSadI} where higher regularity is needed depending on the Lyapunov spectrum. 

Yet, these results do not apply directly here, due namely to the fact that $f$ is only assumed to be $C^2$. Using an \textit{ad hoc} construction, based on one-dimensional normal forms along the center/unstable directions, we show:
\begin{theoalph}[Normal forms]
	\label{thm.normalformsintro}
	%Let $f\colon\T^3 \to \T^3$ be a $C^2$ Anosov diffeomorphism with a partially hyperbolic decomposition $T \T^3 = E^s \oplus E^c \oplus E^u$. Let $\cWcu$ be the two-dimensional unstable foliation integrating $E^{cu}\eqdef E^c\oplus E^u$. Suppose that the center $E^c$ is expanding and that it integrates to a center foliation $\mathcal{W}^c$. Then, t
	Let $f\in \mathcal{A}^2(\TT)$. Then, there exists a family $\{\Phi_x\}_{x \in \TT}$ of $C^1$ diffeomorphisms $\Phi_x \colon \R^2 \to \mathcal{W}^{cu}(x)$ such that
	\begin{enumerate}
		\item $f \circ \Phi_x = \Phi_{f(x)} \circ N_x$, with $N_x\eqdef \begin{bmatrix}
		\lambda^u_x & 0\\
		0 & \lambda^c_x
		\end{bmatrix}
		$, letting $\lambda^*_x\eqdef\|Df(x)|_{E^*}\|$; %(\lambda^c_x s, \lambda^u_x t)$, for all $(s,t)\in \R^2$, where $\lambda_x^*\eqdef\|Df(x)|_{E^*}\|$, for $*=c,u$;
		\item $\Phi_x(0) = x$ and $D\Phi_x(0)(1,0) = v^u(x)$, $D\Phi_x(0)(0,1)= v^c(x)$, $v^u(x)$, resp. $v^c(x)$ being a unit vector in $E^u(x)$, resp. $E^c(x)$;
		\item $\Phi_x(\cdot)$ depends continuously with the choice of $x$ in the local $C^1$-topology\footnote{Uniform convergence of the function and its first derivative on compact sets.}; %;\marginpar{Affine structure?}
		\item $\Phi_x$ is a foliated chart for $\mathcal{W}^u$, i.e., for all $s\in \R$, $\Phi_x(\R\times \{s\})=\mathcal{W}^u(\Phi_x(0,s))$, and $\Phi_x(\{0\}\times \R)=\mathcal{W}^c(x)$. 
	\end{enumerate}
\end{theoalph}

\begin{remark}
	The normal forms $\{\Phi_x\}_{x \in \TT}$ do not define an affine structure on the foliation $\mathcal{W}^{cu}$; yet, we will see in Section \ref{section formes normales} certain invariance properties of these normal forms under changes of charts. 
\end{remark}

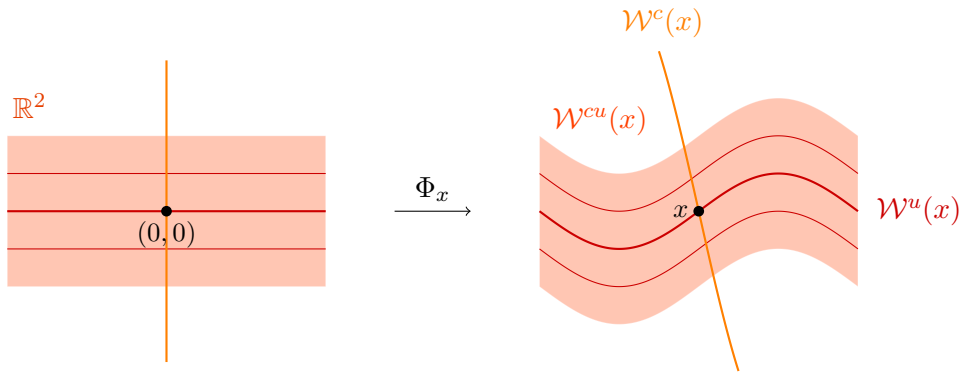
\begin{figure}[h]
	\centering
	\begin{tikzpicture} 
	%\shorthandoff{:} ...   
	\filldraw[draw=white,fill=red!50!orange,opacity=.3]
	plot[domain=-pi:pi,samples=200] (2*\x/3-7,{1})
	-- plot[domain=pi:-pi, samples=200] (2*\x/3-7,{-1})
	-- cycle;
	\draw [domain=-pi:pi, draw=red!80!black, samples=200] plot (2*\x/3-7,{-1/2});  
	\draw [domain=-pi:pi, draw=red!80!black, thick, samples=200] plot (2*\x/3-7,{0}); 
	\draw [domain=-pi:pi, draw=red!80!black, samples=200] plot (2*\x/3-7,{1/2});  
	\draw[draw=orange, thick] (-7,-2)--(-7,2) node[right]{};
	\fill[fill=black] (-7,0) node[left, below]{\small $(0,0)$} circle (2pt);
	\draw[->,draw=black] (-4,0)--(-3,0) node[midway,above]{$\Phi_x$};
	\filldraw[draw=white,fill=red!50!orange,opacity=.3]
	plot[domain=-pi:pi,samples=200] (2*\x/3,{(sin(\x r)-2)/2})
	-- plot[domain=pi:-pi, samples=200] (2*\x/3,{(sin(\x r)+2)/2})
	-- cycle;
	\draw [domain=-pi:pi, draw=red!80!black, samples=200] plot (2*\x/3,{(sin(\x r)-1)/2});  
	\draw [domain=-pi:pi, draw=red!80!black, thick, samples=200] plot (2*\x/3,{(sin(\x r))/2}); 
	\draw [domain=-pi:pi, draw=red!80!black, samples=200] plot (2*\x/3,{(sin(\x r)+1)/2});  
	\draw [domain=-pi/4:pi/4, draw=orange, thick, samples=200] plot (2*\x/3,{(6*sin(-\x r))/2});
	\fill[fill=black] (0.2,2.5) node[left]{\color{orange}$\mathcal{W}^c(x)$};
	\fill[fill=black] (-1.3,1.2) node[]{\color{red!50!orange}$\mathcal{W}^{cu}(x)$};
	\fill[fill=black] (2.9,0) node[]{\color{red!80!black}$\mathcal{W}^{u}(x)$};
	\fill[fill=black] (-8.8,1.4) node[]{\color{red!80!black!50!orange}$\mathbb{R}^2$};
	\fill[fill=black] (0,0) node[left]{\small $x$} circle (2pt);
	\end{tikzpicture}
	\caption{\label{2D NF} Theorem~\ref{thm.normalformsintro} provides non-stationary $C^1$ linearisations of the dynamics along center unstable leaves. They send horizontal lines onto unstable manifolds.}
\end{figure}

Note that Katz also uses the $C^{\infty}$ regularity many other times in his adaptation of Eskin-Mirzakhani's factorization method for Anosov systems. For example, at some moments, he has to approximate stable/unstable manifolds by Taylor polynomials with very high degree.  In our setting,  Theorems \ref{thm.zeroonelawintro} and \ref{thm.normalformsintro} are the two main reasons why we are able to adapt the Eskin-Mirzakhani's strategy in lower regularity.

We stress that, differently from previous works, we implement $Y$-configurations and a version of the factorization technique of \cite{EskinMirzakhani} without using suspensions nor any reparametrization. We make all estimations directly with the diffeomorphism. This is possible because we can obtain uniform estimates for the drift of leaf-wise (quotient) measures along the center as well as synchronization estimates for stopping times (see \S\ref{sss.quadri_couple}), using only basic distortion estimates (see \S\ref{sss.distortionbasic}).
%In the setting of Theorem \ref{mainthm.dicotomia}, since we are close to volume preserving, the stable foliation is $C^1$ (usually this foliation is only H\"older). This, together with Theorem \ref{thm.normalformsintro}, will allow us to prove our theorem using only regularity $C^2$. 

\subsection{Organization of the paper}

This paper is organized as follows: in Section~\ref{section definitions notations} we introduce the basic definitions and results we need. In Section~\ref{heuristics} we give an outline of the proof of Theorem~\ref{mainthm.dicotomia}. In Section~\ref{s.zeroouum} we establish a zero-one law for transversality between the bundles $E^s$ and $E^u$. In Section~\ref{s.main_thm} we reduce the proof of Theorem~\ref{mainthm.dicotomia} to a more technical result, see Theorem~\ref{mainthm.technique}. In Section~\ref{section formes normales} we construct a non-stationary family of $C^1$ linearisations of the dynamics restricted to center unstable manifolds and use them to construct a family of measures $\{\hat{\nu}^c_x\}_{x\in\TT}$ on the real line. We reduce the proof of our main technical result to proving that these measures are Lebesgue almost surely. In Section~\ref{s.Main_texh_lemma}, we explain an argument from \cite{BRH, KalininKatok} that shows that if the measures $\hat{\nu}_x^c$ are ``invariant'' by certain affine maps for many point $x$, then the measures $\hat{\nu}_x^c$ are actually the Lebesgue measure.  We then explain how this is achieved by reducing the proof in proving Proposition \ref{p.drift} In Section~\ref{s.Yconfigurations} we introduce $Y$-configurations and other objects crucial for our argument. In Section~\ref{s.coupled} we introduce matched $Y$-configurations and in Section~\ref{s.end_proof} we complete the proof.

\subsection*{Acknowledgements}
M.L. is most grateful to Federico Rodriguez Hertz for suggesting the problem during a visit at PennState University and many useful discussions, in particular on his joint work \cite{BRH} with Aaron Brown. 
We thank Alex Eskin for suggesting the matching argument from \cite{EskinMirzakhani} to us.  We thank Jonathan DeWitt for pointing out a problem on an earlier version of our work. We also thank Amie Wilkinson, Sylvain Crovisier, Aaron Brown and Rafael Potrie for useful conversations, as well as Marisa Cantarino for valuable comments on an early draft of the manuscript. Last but not least we wish to thank the anonymous referees for their many valuable comments that improved the presentation of our work.

{\footnotesize S.A. was supported by ANII via the Fondo Clemente Estable (project FCE\_3\_2018\_1\_148740), by CSIC, via the project I+D 389 and the Grupo I+D 159 ``Geometr\'ia y acciones de Grupos''. He was also supported by the ANR AAPG 2021 PRC CoSyDy: Conformally symplectic dynamics, beyond symplectic dynamics (ANR-CE40-0014).

M.L. was supported by the ERC project 692925 NUHGD of Sylvain Crovisier, by the ANR AAPG 2021 PRC CoSyDy: Conformally symplectic dynamics, beyond symplectic dynamics (ANR-CE40-0014), and by the ANR JCJC PADAWAN: Parabolic dynamics, bifurcations and wandering domains (ANR-21-CE40-0012).

B.S. was supported by \emph{Conselho Nacional de Desenvolcimento Científico e Tecnológico - CNPq} and FAPERJ and partially supported by \textit{Coordenação de Aperfeiçoamento de Pessoal de Nível Superior - Brasil
(CAPES)} - Finance Code 001. 
}

\section{Partially hyperbolic diffeomorphisms with expanding center}\label{section definitions notations}

In this section we introduce the class of dynamical systems we work with as well as the main objects, taking the opportunity to fix notations and recall important basic facts.

\subsection{Definitions}
For any integer $r \geq 1$, we let $\cPH^r(\TT)$ be the set of all $C^r$ \textit{(strongly) partially hyperbolic diffeomorphisms} of $\TT$ with one-dimensional stable/center/unstable bundles, i.e., the diffeomorphisms $f\colon \TT \to \TT$ such that there exist a continuous splitting of $T\TT$ into $Df$-invariant line bundles,
$$
T\TT = E^{s} \oplus E^{c} \oplus E^u,
$$ 
as well as a Riemannian metric $\|.\|$ \emph{adapted to this splitting} such that the functions  
\begin{equation*} 
x\mapsto \lambda_{x}^*\eqdef\|Df(x)|_{E^*}\|,\quad *\in \{s,c,u\},
\end{equation*}
are continuous and satisfy $\lambda_{x}^s<1<\lambda_{x}^u$ and  $\lambda_{x}^s<\lambda_{x}^c< \lambda_{x}^u$, for all $x \in \TT$. We also let $E^{cs}\eqdef E^c\oplus E^s$, resp. $E^{cu}\eqdef E^c\oplus E^u$ be the center-stable, resp. center-unstable subbundle, and set $E^{su}\eqdef E^s \oplus E^u$. We refer to Katok-Hasselblatt's book \cite{Katok_Hasselblatt} for more details.

\subsubsection{Anosov diffeomorphisms with uniformly expanding center} We denote by $\mathcal{A}^r(\TT)\subset \cPH^r(\TT)$ the subset consisting of partially hyperbolic diffeomorphisms $f\in \cPH^r(\TT)$ with uniformly expanding center, i.e., such that $\lambda_x^c>1$, for all $x \in \TT$; in particular, any such diffeomorphism $f$ is  \textit{Anosov}, for the hyperbolic splitting $E^{s}\oplus E^{cu}$. We also denote by $\mathcal{A}_m^r(\TT)\subset \mathcal{A}^r(\TT)$ the subset made of \textit{conservative} Anosov diffeomorphisms (i.e., that preserve some volume). 

To simplify the exposition, we assume that the bundles $E^*$ are orientable and that $f$ preserves their orientation (this can always be achieved by taking an orientable cover and considering powers of $f$). In particular, there are unitary vector fields $x\in \TT \mapsto v^*(x)\in E^*(x)$ such that 
\[
Df(x)v^*(x)\eqdef\lambda^*_xv^*(f(x)).
\] 

\subsubsection{Notations for orbits and derivatives}\label{sss.notaderiva}
To give a more friendly aspect of some long calculations we shall make, we introduce the following notation. For a point $x\in\TT$ we denote
\begin{equation}
\label{e.orbita}
x_n=f^n(x)\:\:\:\:\textrm{for}\:\:\:n\in\Z.
\end{equation}
Also, for $n\in\Z$ and for $*=s,c,u$ we denote the derivative of $f$ in restriction to the bundle $E^*$ by
\begin{equation}
\label{e.derivada}
\lambda^*_x(n)\eqdef\|Df^n(x)|_{E^*}\|.
\end{equation} 
The following \emph{cocycle property} follows from the chain rule and the fact that the bundles are one-dimensional:
\begin{equation}
\label{e.cocycle}
\lambda^*_x(n+m)=\lambda^*_{f^n(x)}(m)\lambda^*_x(n),\:\:\:*\in \{s,c,u\}. 
\end{equation}

The following quantities associated to $f$ will be useful for crude estimations
\[
\|Df\|\eqdef\max\{\|Df(x)v\|:x\in\TT, v\in T_x\TT,\|v\|=1\}
\]
and 
\[
m(Df)\eqdef\min\{\|Df(x)v\|:x\in\TT, v\in T_x\TT,\|v\|=1\}.
\]

\subsubsection{Adapted metric and hyperbolic estimates}\label{sss.hypestimates}
The following quantity will play a key role later when we introduce stopping times and $Y$-configurations:
\begin{equation}
\label{e.domonationalongy}
d^\ell_x\eqdef\frac{\lambda^c_{x_{-\ell}}(\ell)}{\lambda^u_{x_{-\ell}}(\ell)}.
\end{equation}
Notice that $d^\ell_x$ measures the amount of projective hyperbolicity we have for the dominated splitting $E^c\oplus E^u$. 

We fix a Riemannian metric on $\TT$ and constants $\chi^{*}_j\in\R$, for $*\in\{c,s,u,d\}$ and $j=1,2$ such that the following holds:
\begin{enumerate}
	\item $\chi_1^d<\chi_2^d<0$ and $e^{\chi_1^d\ell}<d^{\ell}_x<e^{\chi^d_2\ell}$, for every $x\in\TT$ and every $\ell\in\Z$.
	\item $\chi_1^s<\chi_2^s<0$ and $e^{\chi_1^s\ell}<\lambda^s_x(\ell)<e^{\chi^s_2\ell}$, for every $x\in\TT$ and every $\ell\in\Z$.
	\item $\chi_1^c>\chi_2^c>0$ and $e^{\chi_1^c\ell}>\lambda^c_x(\ell)>e^{\chi^c_2\ell}$, for every $x\in\TT$ and every $\ell\in\Z$.
	\item $\chi_1^u>\chi_2^u>0$ and $e^{\chi_1^u\ell}>\lambda^u_x(\ell)>e^{\chi^u_2\ell}$, for every $x\in\TT$ and every $\ell\in\Z$.
\end{enumerate}  

\subsubsection{Invariant manifolds} Let $r \geq 1$, and let $f \in \mathcal{A}^r(\TT)$. 
It is well-known (see \cite{HirschPughShub}) that the strong bundles $E^s$ and $E^u$ are uniquely integrable to $f$-invariant continuous foliations with $C^r$-leaves $\cWs_f=\cWs$ and $\cWu_f=\cWu$ respectively, called the \textit{strong stable} and \textit{strong unstable} foliations. Since the splitting $E^s\oplus E^{cu}$ is Anosov, the center-unstable bundle $E^{cu}$ also integrates uniquely to an $f$-invariant  continuous foliation $\cWcu_f=\cWcu$, called the \textit{center-unstable} foliation. For any $x \in\TT$ and $* = s,u,cu$, we denote by $\cW^*(x)$ the leaf of $\cW^*$ through $x$; it is an immersed $C^{r}$ manifold. 

We now define the concept of \emph{joint integrability} which appear in the statement of Theorem~\ref{mainthm.dicotomia}
\begin{defi}[Joint integrability]\label{def.joint}
We say that $f\in\cA^r(\TT)$ is (or that the bundles $E^s$ and $E^u$ are) \emph{jointly integrable} if the bundle $E^s\oplus E^u$ integrates to a continuous foliation with $C^1$ leaves. 
\end{defi} 

Let $r \geq 1$, and fix a $C^r$ Anosov diffeomorphism $f\in \mathcal{A}^r(\TT)$. %with a partially hyperbolic splitting $E_{f_0}^s\oplus E_{f_0}^c\oplus E_{f_0}^u$ such that $Df_0|_{E_{f_0}^c}$ is uniformly expanding. 
By Corollary 1.3 in \cite{casanova} the non-wandering set $\Omega(f)$ of $f$ is equal to $\Omega(f)=\TT$, and $f$ is topologically conjugated to a \emph{hyperbolic toral automorphism}. As a consequence one has that  
%, and $f$ is conjugated to $f_0$ by some homeomorphism $h \colon \TT \to \TT$, as $f_0$ is structurally stable. Since $f_0$ is a transitive Anosov diffeomorphism, its stable foliation is minimal, and we immediately deduce: 
%, for a partially hyperbolic splitting $E^s\oplus E^c\oplus E^u$, and $f$ is Anosov, for the hyperbolic splitting $E^s \oplus E^{cu}$, where $E^{cu}\eqdef E^c\oplus E^u$. 
\begin{lemma}\label{claim min}
The stable foliation of $f$ is minimal, i.e., each leaf of $\mathcal{W}^s$ is dense in $\TT$. 
\end{lemma}

\subsubsection{Dynamical coherence}

As remarked before, any diffeomorphism $f\in\cA^r(\TT)$ can be seen either as a strongly partially hyperbolic diffeomorphism with respect to the splitting $E^s\oplus E^c\oplus E^u$ and an Anosov diffeomorphism with respect to the splitting $E^s\oplus E^{cu}$. By Lemma~\ref{claim min} we know that $\Omega(f)=\TT$. 

By the results of Potrie \cite{potrie} and Brin-Burago-Ivanov \cite{BrinBuragoIvanov} on partially hyperbolic diffeomorphisms of $\TT$, we have that $f$ is \textit{dynamically coherent}. In particular, $E^{cs}$ is also integrable to an $f$-invariant continuous foliation $\cW^{cs}_f=\cW^{cs}$, called the \textit{center-stable foliation}. Moreover, $\cW^{s}$ subfoliates $\cWcs$, while $\cWu$ subfoliates $\cWcu$, and the collection of all leaves $\cWc(x)\eqdef \cWcs(x) \cap \cWcu(x)$, $x \in \TT$, forms a foliation $\cW^{c}_f=\cWc$, called the \textit{center foliation}, which integrates $E^{c}$, and subfoliates both $\cWcs$ and $\cWcu$. For $*\in \{u,c,s,cu,cs\}$, let $d_*$ be the leaf-wise distance, and for $x\in \TT$, $\sigma>0$, set $\cW_\sigma^*(x)\eqdef \{y \in \cW^*(x) \ \vert\ d_*(x,y) < \sigma \}$. 

In our Anosov case we can show rather easily that the foliations $\cWc$ and $\cWu$ are globally transverse inside each $\cWcu$ leaf, as the two lemmas below demonstrate. 

\begin{lemma}
	\label{l.coerenciaum}
For any $x \in \TT$, it holds $\cW^{cu}(x)=\cup_{y\in\cW^c(x)}\cW^u(y)$. 
\end{lemma}
\begin{proof}
	Let $z\in\cW^{cu}(x)$ be arbitrary. We need to show that $\cW^{c}(x)\cap\cW^u(z)\neq\emptyset$. Since the bundles $E^c$ and $E^u$ are integrable and the local leaves have uniform size due to hyperbolicity, since the splitting $E^c\oplus E^u$ is dominated and  backwards iteration under $f$ contracts distances uniformly along $\cW^{cu}$ we must have some $n\in\N$ such that 
	\[
	\cW^c(f^{-n}(x))\cap\cW^u(f^{-n}(z))\neq\emptyset,
	\] 
	with transverse intersection. By forward iteration and using that integral manifolds are invariant by the dynamics we obtain the conclusion of the lemma.
\end{proof}

\begin{lemma}
	\label{l.coorenciadois}
	For every $y\in\cW^c(x)$, it holds $\cW^c(x)\cap\cW^u(y)=\{y\}$.	
\end{lemma}
\begin{figure}[h!]
	\begin{tikzpicture}
	\draw[orange, thick] (0,0)--(0,3)node[right]{$\cW^c(x)$};
	\draw[red!80!black, thick] (-1,2.1)--(0,2)--(1,1.9).. controls (1.5,1.8) and (1.5,1.2).. (1,1.1)--(0,1)--(-1,0.9)node[left]{$\cW^u(y)$};
	\draw (0,1) node{$\bullet$};
	\draw (0,2) node{$\bullet$};
	\draw (-.15,1) node[below]{$y$};
	\draw (-.15,2) node[below]{$y'$};
	\begin{scope}[xshift=4cm, yshift=1cm, scale=.4]
	\draw[orange, thick] (0,0)--(0,3);
	\draw[red!80!black, thick] (-1,2.1)--(0,2)--(1,1.9).. controls (1.5,1.8) and (1.5,1.2).. (1,1.1)--(0,1)--(-1,0.9);
	\draw (0,1) node{$\bullet$};
	\draw (0,2) node{$\bullet$};
	\draw (-1.2,1.1) node[below]{\tiny$f^{-n}(y)$};
	\draw (-1.4,2.3) node[below]{\tiny$f^{-n}(y')$};
	\end{scope}
	\draw[->] (1.5,1.4) to[bend left] node[midway, above]{$f^{-n}$} (2.7,1.5);
	\end{tikzpicture}
	\caption{\label{Fig.coerencia} Proof of Lemma~\ref{l.coorenciadois}: iterating the picture of the left we arrive at a small scale where the picture violates the uniformly positive angle between $E^c$ and $E^u$.}
\end{figure}
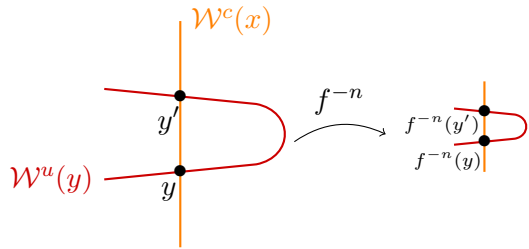

\begin{proof}
	Assume by contradiction the existence of a point $y'\neq y$ in $\cW^c(x)\cap\cW^u(y)$. Consider the piece $\gamma^u$ of unstable manifold joining $y$ to $y'$. Since the length of $f^{-n}(\gamma^u)$ decreases exponentially there exists some $n$ such that the curve $f^{-n}(\gamma^u)$ is entirely contained in a coordinate chart for which the line field $E^c$ is almost vertical. Since this curve joins the points $f^{-n}(y)$ and $f^{-n}(y')$ which belong to the same local integral curve of $E^c$, this proves that the tangent space of $f^{-n}(\gamma^u)$ is almost vertical somewhere. This contradicts the dominated splitting $E^c\oplus E^u$ and completes the proof.
\end{proof}

\subsubsection{Holonomies}\label{sss.holo}

Let $x_1,x_2\in \TT$ be two \textit{stably connected} points, i.e., such that $x_2\in \mathcal{W}^s(x_1)$. Set $r\eqdef d_s(x_1,x_2)$. By transversality, for any sufficiently small $\varepsilon>0$, there is $\sigma_1>0$ such that for any point $y_1\in \mathcal{W}^{cu}_{\sigma_1}(x_1)$, there exists a unique point $y_2 \in \mathcal{W}^{s}(y_1)\cap\mathcal{W}^{cu}_1(x_2)$ with $d_s(y_1,y_2)\in (r-\varepsilon,r+\varepsilon)$.  
We denote $H_{x_1,x_2}^s(y_1)\eqdef y_2$. 
% By transversality, there is a subset $\cC_1' \subset \cC_1$ such that for any $x \in  \cC_1'$, the local stable leaf through $x$ intersects $\cC_2$ at a unique point, denoted by $H_{\cC_1,\cC_2}^{s}(x)\in \cW^s(x)\cap \cC_2$. 
Thus, we get a well defined local homeomorphism
\begin{equation*}
H_{x_1,x_2}^{s}\colon
\cC_1\to \cC_2,
\end{equation*}
from a neighbourhood $\cC_1$ of $x_1$ within $\mathcal{W}^{cu}(x_1)$ to a neighbourhood $\cC_2$ of $x_2$ within $\mathcal{W}^{cu}(x_2)$, 
called the (local) \textit{stable holonomy map} between $\cC_1$ and $\cC_2$. 
Holonomies $H^u,H^{cs},H^{cu}$ along $\mathcal{W}^u$, $\mathcal{W}^{cs}$, $\mathcal{W}^{cu}$ are defined in a similar way.  

\subsection{Regularity of extreme bundles and holonomy maps}

It is crucial to our proof that certain holonomy maps are of class $C^1$. This is the case when some  \textit{bunching} inequalities are satisfied between the rates of contraction/expansion of the system; they actually hold in a neighbourhood of volume preserving Anosov diffeomorphisms of $\TT$, which is the main motivation behind this assumption in our result.   

 In our setting, we will say that $f$ is \emph{bunched} if there exists some $n\in \mathbb{N}$ such that for any $x \in \TT$,
\begin{equation}
\label{eq.centerbunching}
\lambda^s_{f,x}(n) < \frac{\lambda^c_{f,x}(n)}{\lambda^u_{f,x}(n)}.
\end{equation}
In the equation above, we explicited the dependence of the contraction and expansion rates on $f$. We remark that this is clearly a $C^1$-open condition. This definition appears in \cite{PSW} (see also \cite[\S 4.7]{CrovisierPotrie}).

%Loosely speaking, a partially hyperbolic diffeomorphism is \textit{center bunched} if the lack of conformality along the center bundle is dominated by the hyperbolicity along the strong unstable/stable bundles. %Let $f \in \cPH^2(\TT)$ be a small perturbation of the automorphism $A$. Since

\subsubsection{Regularity of the unstable bundle}

Let $f\in\cA^2(\TT)$. Then, $f$ is a strongly partially hyperbolic diffeomorphism with one-dimensional center bundle, hence $f$ is automatically center bunched.  By \cite[Theorem B]{PSW}, the unstable bundle $E^u$ is $C^1$ when restricted to any center unstable leaf $\cWcu(x)$. In particular, the vector field $v^u|_{\cWcu}$ is a $C^1$ vector field over the immersed submanifold $\cWcu(x)$ and its $C^1$ norm depends continuously with respect to $x\in\TT$.

As a corollary, we obtain that for any two small and nearby center curves $\gamma_1,\gamma_2\subset\cWcu(x)$, in the same center unstable leaf, the unstable holonomy map $H^u\colon\gamma_1\to\gamma_2$ is $C^1$.  This regularity will play a role in our argument in Section~\ref{ss.mainreduction}. Moreover, for some estimations in our proof it is important to quantify the Lipschitz constant of these unstable holonomy maps, as in the following result, which is simply a more precise statement of \cite[Theorem B]{PSW} in our case. 

\begin{lemma}
	\label{l.lemadorho0}
Let $f\in\cA^2(\TT)$. There exists $\rho_0>0,C^u>0$ such that for every $x,y\in\TT$, if $x\in\cW^u_2(y)$ then the unstable holonomy $H^u_{x,y}$ between local center manifolds is defined over $\cW^c_{\rho_0}(x)$ and for every $z,z'\in\cW^c_{\rho_0}(x)$ we have
	$$d(H^u_{x,y}(z),H^u_{x,y}(z'))\leq C^ud(z,z').$$
\end{lemma}

\subsubsection{Regularity of the stable bundle}

For the bundle $E^s$ a stronger statement can be made when $f$ is close to a volume preserving map. Indeed, fix arbitrarily $f_0\in\cA_m^1(\TT)$. 

\begin{lemma}\label{c one holon}
	There exists a neighbourhood $\mathcal{U}(f_0)$ of $f_0$ within $\mathrm{Diff}^2(\TT)$ such that for any diffeomorphism $f\in \mathcal{U}(f_0)$,  we have $f \in \mathcal{A}^2(\TT)$ and the stable bundle $E^s$ of $f$ is of class $C^1$. 
\end{lemma}
\begin{proof}
	%It is a consequence of the results of \cite{PSW} (see also \cite[Theorem 4.21]{CrovisierPotrie}). Indeed, d
	Since $f_0$ is uniformly hyperbolic, there exists a continuous function $\TT \ni x \mapsto C(x) \in \mathrm{GL}(3,\mathbb{R})$ such that for every point $x\in \TT$,  
	\[
	Df_0(x) = C(f_0(x))^{-1} \begin{bmatrix}
	\lambda^s_{f_0,x} & 0 & 0\\ 0 & \lambda^c_{f_0,x} & 0\\ 0 & 0 & \lambda^u_{f_0,x}
	\end{bmatrix} C(x).
	\]
	The matrix $C(x)$ is a matrix with positive determinant that takes a basis formed by unit vectors in $E^*(x)$,  and sends it to the orthogonal basis $(1,0,0)$, $(0,1,0)$ and $(0,0,1)$.  In particular, there exists a uniform constant $C\geq 1$ such that
	\[
	\max\{|\det C(x)|, |\det C(x)^{-1}|\} \leq C.
	\]
	Consequently,
	\[
	|\det Df_0(x)| = |\det C(f_0(x))^{-1}|\cdot|\det C(x)|\cdot\lambda^s_{f_0,x} \lambda^c_{f_0,x} \lambda^u_{f_0,x}.
	\] 
	 Observing that for every $n\in \N$, 
	 \[
	 |\det Df_0(x)| =|\det C(f^n_0(x))^{-1}|\cdot|\det C(x)|\cdot\lambda^s_{f_0,x}(n)\lambda^c_{f_0,x}(n) \lambda^u_{f_0,x}(n).
	 \]
	The diffeomorphism $f_0$ is conservative so its Jacobian is a coboundary (by Liv$\check{\mathrm{s}}$ic's Theorem \cite{Livsic} and \cite[Theorem 4.14]{BowenR}). Thus, there exists a continuous function $\phi\colon\TT\to(0,\infty)$ bounded away from $0$ and $\infty$ such that for any $x \in \TT$, we have 
	$$
	 |\det C(f_0(x))^{-1}|\cdot|\det C(x)|\cdot\lambda^s_{f_0,x} \lambda^c_{f_0,x} \lambda^u_{f_0,x}=\frac{\phi\circ f(x)}{\phi(x)}.$$ 
	Hence,   
	\begin{align*} 
	\lambda_{f_0,x}^s(n)&=\frac{\lambda_{f_0,x}^s(n)\lambda_{f_0,x}^c(n) \lambda_{f_0,x}^u(n)}{\lambda_{f_0,x}^c(n) \lambda_{f_0,x}^u(n)}\\
	&=|\det C(f_0(x))^{-1}|\cdot|\det C(x)|\cdot \frac{\phi\circ f^n(x)}{\phi(x)}\cdot\frac{1}{(\lambda_{f_0,x}^c(n))^2}\frac{\lambda_{f_0,x}^c(n)}{\lambda_{f_0,x}^u(n)}\\
	&<\frac{\lambda_{f_0,x}^c(n)}{\lambda_{f_0,x}^u(n)},
	\end{align*} 
	as long as $n$ is large enough so that
	$$
	|\det C(f_0(x))^{-1}|\cdot|\det C(x)| \frac{\phi\circ f^n(x)}{\phi(x)}< (\lambda_{f_0,x}^c(n))^2.
	$$
	Such an integer $n$ may be chosen independently of $x$ because $\phi$ is uniformly bounded away from $0$ and $\infty$, $\max\{|\det C(x)^{-1}|, |\det C(x)|\} \leq C$  and $E^c$ is uniformly expanding. In particular, if $f\in\mathrm{Diff}^2(\TT)$ is sufficiently $C^1$-close to $f_0$, then $f \in \mathcal{A}^2(\TT)$, and for the same choice of $n\in\N$ and  every $x \in \TT$, it holds
	$
	\lambda_{f,x}^s(n) < \frac{\lambda_{f,x}^c(n)}{\lambda_{f,x}^u(n)}
	$ and $f$ verifies \eqref{eq.centerbunching}.
	In other words, the hyperbolic splitting $E^s \oplus E^{cu}$ of $f$ is bunched (see \eqref{eq.centerbunching}). Then, according to the results of \cite{PSW}  (see also \cite[Theorem 4.21]{CrovisierPotrie}), the stable bundle $E_f^s$ is $C^1$, as well as the stable holonomy maps. 
\end{proof}

\begin{remark}\label{remarque bunching}
	As noted in the above proof, the assumption that $f$ is $C^1$-close to a conservative diffeomorphism ensures that the following \emph{bunching condition} is automatically satisfied (hence that stable holonomies $H^s$ are $C^1$, by \cite{PSW}):
\begin{equation}\label{condicondi bunching}
\lambda_{f,x}^s(n) < \frac{\lambda_{f,x}^c(n)}{\lambda_{f,x}^u(n)},\quad \forall\,x \in \TT.
\end{equation}
In other words, \eqref{condicondi bunching} means that the lack of conformality of $Df$ along $E^{cu}$ is dominated by the contraction along $E^s$. In particular, the conclusion of Theorem \ref{mainthm.dicotomia} holds true for any Anosov diffeomorphism $f\in \mathcal{A}^2(\TT)$ satisfying \eqref{condicondi bunching}. 
\end{remark}

\subsubsection{Hölder regularity of $cs$-holonomies}
Despite the above regularity results, a substantial source of technical difficulties for our strategy comes from the absence of smoothness for \emph{center stable holonomies}. The quest for overcoming this is the main reason behind our \emph{matching argument for $Y$-configurations} in Sections~\ref{s.coupled} and \ref{ss.matching}. In our setting, the best that can be said about center stable holonomies comes from Theorem A in \cite{PSW}, which we quote below in a convenient way for our purposes.

\begin{lemma}\label{thm_regularity}
Let $f\in\cA^2(\TT)$. Then there exist $\rho_0>0$ (which we can assume is the same from Lemma~\ref{l.lemadorho0}), $C^{cs}>0,\theta^{cs}>0$ such that for every $x,y\in\TT$ if $x\in\cW^{cs}_{\rho_0}(y)$ then the center-stable holonomy $H^{cs}_{x,y}$ is defined on $\cW^{u}_2(x)$ and for every $z,z'\in\cW^{u}_2(x)$ we have
$$d(H^{cs}_{x,y}(z),H^{cs}_{x,y}(z'))\leq C^{cs}d(z,z')^{\theta^{cs}}.$$
\end{lemma}
We remark that the roles of $\rho_0$ and $2$ are ``exchanged'' when compared with the role of these constants in Lemma \ref{l.lemadorho0}. Here, the transversal we are looking at for the center-stable holonomy has size $2$, while in Lemma \ref{l.lemadorho0} the transversal has size $\rho_0$.

\subsection{Basic distortion estimates}\label{sss.distortionbasic}

The goal of this subsection is to collect some classical distortion estimates and fix once and for all some constants which are going to be very important all along our constructions. 

\begin{lemma}[Basic distortion lemma]
	\label{l.distortionbasic}
	Let $\varphi\colon\TT\to\R$ be a Hölder continuous function. Then, there exists a constant $C=C(\varphi)>0$ such that 
	\begin{itemize}
		\item[(a)] If $y\in \cW^s_1(x)$ and $n>0$ or
		\item[(b)]  if $f^n(y)\in \cW^{cu}_1(f^n(x))$, for $n\in \N$,  then 
	\end{itemize} 
	\[
	\left|\sum_{\ell=0}^{n-1}\varphi(f^\ell(x))-\varphi(f^\ell(y))\right|\leq C.
	\]
	
\end{lemma}

We omit the proof as it is quite classical. Applying the lemma to the functions $\varphi=\log\|Df(.)|_{E^*}\|$, $*=c,u$ we obtain a constant $C_0=C_0(f)\geq 1$ such that if $y\in\cW^s_1(x)$ and $n>0$ or if $f^\ell(y)\in\cW^{cu}_1(f^\ell(x))$ for every $\ell=0,\dots,n$, then
\begin{equation}
\label{e.distortionbasic}
C_0^{-1}\leq\frac{\|Df^n(x)|_{E^*}\|}{\|Df^n(y)|_{E^*}\|}\leq C_0.
\end{equation}
Moreover, up to enlarging $C_0$, we also have that if $y\in\cWcu_1(x)$ then
\begin{equation}
\label{e.distortionbasicpasse}
C_0^{-1}\leq\frac{\|Df^{-n}(x)|_{E^*}\|}{\|Df^{-n}(y)|_{E^*}\|}\leq C_0.
\end{equation}
for every $n\in\N$.
Another important application is obtained by considering the function $\psi(x)=\log\frac{\|Df(x)|_{E^c}\|}{\|Df(x)|_{E^u}\|}$. We can assume that the constant $C_0$ also satisfies the following
\begin{corollary}
	\label{c.distortionbasic}
	Given $\ell\in\N$ recall from \eqref{e.domonationalongy} that $d^\ell_x=\frac{\|Df^\ell(f^{-\ell}(x))|_{E^c}\|}{\|Df^\ell(f^{-\ell}(x))|_{E^u}\|}$. If $f^{-\ell}(y)\in \cW^s_1(f^{-\ell}(x))$, then, 
	\[
	C_0^{-1}\leq\frac{d^\ell_x}{d^\ell_y}\leq C_0.
	\]
\end{corollary}   

\subsubsection{Distortion for quadrilaterals}

One of the key dynamical configuration for our strategy are the \emph{quadrilaterals}. A \textit{quadrilateral} is a quadruple $(x,x^u,y,y^u)\in(\TT)^4$ such that $x^u \in \cWu_1(x)$, $y\in \cWs_1(x)$, and $y^u\in \cWu_1(y)\cap\cWcs(x^u)$. For such a quadrilateral, we define the point $z^u\eqdef H^{s}_{x,y}(x^u)$, so that $z^u\in \mathcal{W}^s(x^u)\cap \cWc(y^u)$.

\begin{figure}[h!]
	\centering
	\begin{tikzpicture}[scale=.83]
	\fill[green!80!white, opacity=.2] (-2,2) .. controls (-1.5,1.8) and (-0.35,0.5) .. (0,0) .. controls (0.35,-0.5) and (1.5,-1.8) .. (2,-2) -- (5,-1.5).. controls (4.5,-1.3) and (3.35,0) .. (3,0.5).. controls (2.65,1) and (1.5,2.3) .. (1,2.5)--(-2,2);
	\draw[red!80!black, thick] (-2,2)  .. controls (-1.5,1.8) and (-0.35,0.5) .. (0,0) .. controls (0.35,-0.5) and (1.5,-1.8) .. (2,-2);
	\draw[thick, orange] (1,2.5) -- (1,2);
	\draw (-2,2) node{$\bullet$};
	\draw (-2,2) node[left]{$x^u$};
	\draw[green!40!black, thick] (0,0)--(3,0.5) node[midway, below]{$\cW^s_{1}(x)$};
	\draw (0,0) node{$\bullet$};
	\draw (0,0) node[left]{$x$};
	\draw[red!80!black] (1.8,-2) node[left]{$\cW^u_{1}(x)$};
	\draw[violet, thick] (5,-1.5) .. controls (4.5,-1.3) and (3.35,0) .. (3,0.5).. controls (2.65,1) and (1.5,2.3) .. (1,2.5) node[above]{\color{black} $z^u$};
	\draw (1,2.5) node{$\bullet$};
	\draw[red!80!black, thick] (1,2) .. controls (1.5,1.75) and (2,0.75) .. (3,0.5) .. controls (3.5,0.35) and (4.5,-.55) .. (5,-1) node[right]{$\cW^u_{1}(y)$}; 
	\draw[dotted] (5,-1) -- (5,-1.5);
	\draw (3,0.5) node{$\bullet$};
	\draw (3,.75) node[right]{$y$};
	\draw (1,2) node{$\bullet$};
	\draw (1,2) node[left]{$y^u$};
	\end{tikzpicture}
	\caption{\label{fig.quadrilatero}A quadrilateral.}
\end{figure}
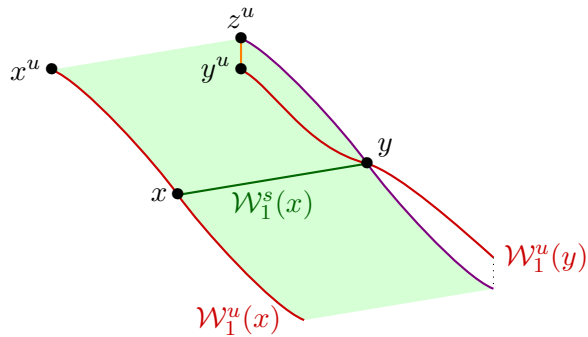

From our previous distortion results, we can take a larger constant $C_0=C_0(f)$ in order to have the following.

\begin{corollary}
	\label{c.distortionbasic2}
	Now, assume that $y\in \cW^s_1(x)$ and that $x^u\in \cW^u_1(x)$ and $y^u\in \cW^u_1(y)$ are such that $y^u=H^{cs}_{x,y}(x^u)$. Let $z^u=H_{x,y}^s(x^u)$. If $n\in\N$ satisfies $d(f^n(z^u),f^n(y^u))\leq 1$ then 
	\[
	C_0^{-1}\leq\frac{\|Df^\ell(x^u)|_{E^c}\|}{\|Df^\ell(y^u)|_{E^c}\|}\leq C_0,
	\]
	for every $\ell=0,\dots,n$.
\end{corollary}
\begin{proof}
	The result follows from \eqref{e.distortionbasic} since 
	\[
	\frac{\|Df^\ell(x^u)|_{E^c}\|}{\|Df^\ell(y^u)|_{E^c}\|}=\frac{\|Df^\ell(x^u)|_{E^c}\|}{\|Df^\ell(z^u)|_{E^c}\|}\times\frac{\|Df^\ell(z^u)|_{E^c}\|}{\|Df^\ell(y^u)|_{E^c}\|}.\qedhere
	\] 
\end{proof}

\subsection{Subordinate partitions and disintegrations}\label{sub.subord partitions disint}

We move now to the \emph{ergodic theory} of Anosov diffeomorphisms of $\TT$ with expanding center. Before giving the main definitions, we give some preliminaries on measurable partitions, disintegration and invariant measures.

\subsubsection{Measurable partitions}
Let $\mu$ be a probability measure of some standard Borel set $X$. Let $\xi_1,\xi_2$ be two partitions (mod $0$) of $X$ into measurable subsets. Say that $\xi_1$ \emph{is finer than} (or \emph{refines}) $\xi_2$,  denote $\xi_2\prec\xi_1$, if for $\mu$-a.e. $x\in X$ we have $\xi_1(x)\subset\xi_2(x)$ mod $0$. 

The \emph{join} of $\xi_1$ and $\xi_2$ is the partition defined as $\xi_1\vee\xi_2=\{\xi_1(x)\cap\xi_2(x):\,x\in X\}$.

A partition $\xi$ of $X$ is \emph{measurable} whenever there exists a sequence $(\xi_n)_{n\in\N}$ of finite partitions of $X$ by Borel subsets such that
$$\xi=\bigvee_{n=0}^\infty \xi_n.$$

Rokhlin proved the following fundamental result in \cite{Rok_meas-th}. The measure $\mu$ can be \emph{disintegrated} into atoms of any measurable partition $\xi$. It means that there exists a family of probability measures $\{\mu^\xi_x\}_{x}$, called a \emph{family of conditional measures of $\mu$} relative to $\xi$ defined for $\mu$-a.e. $x\in X$ and satisfying for $\mu$-a.e. $x\in X$:
\begin{enumerate}
	\item $\mu^\xi_x$ is a probability measure on $X$ satisfying $\mu^\xi_x(\xi(x))=1$;
	\item if $y\in\xi(x)$ then $\mu^\xi_y=\mu^\xi_x$.
\end{enumerate}
Moreover for every Borel subset $A\dans X$,
\begin{enumerate}
	\setcounter{enumi}{2}
	\item $x\mapsto \mu^\xi_x(A)$ is measurable;
	\item $\mu(A)=\int \mu^\xi_x(A)\,d\mu(x).$
\end{enumerate}

Moreover such a family is unique modulo a null set of $\mu$.

\subsubsection{Disintegration of invariant measures} Assume now that $\mu$ is invariant by some invertible measurable transformation $f\colon X\to X$. Let $\xi$ be a measurable partition and, for $n\in\Z$, let $\xi_n\eqdef f^n\xi$. Let $\{\mu_x\}_x\eqdef \{\mu_x^\xi\}_x$ and $\{\mu_{n,x}\}_x\eqdef\{\mu_x^{\xi_n}\}_x$ be families of conditional measures of $\mu$ relative to $\xi$ and $\xi_n$ respectively. The following lemma will be useful for our purposes.

\begin{lemma}
	\label{l_inv_con_meas}
	For every $n\in\N$ and $\mu$-almost every $x\in X$, we have
	$$\mu_{n,x}=f^n_\ast \mu_{f^{-n}(x)}.$$
\end{lemma}

\begin{proof}
	Let $\varphi\in L^1(X,\mu)$. We first disintegrate $\mu$ along  $\xi$ and use twice the $f$-invariance:
	\begin{align*}
	\int \varphi(x)\, d\mu(x)&=\int \varphi\circ f^n(x)\, d\mu(x)
	=\int\left(\int_{\xi(x)} \varphi\circ f^n(y)\, d\mu_{x}(y)\right) d\mu(x)\\
	&=\int\left(\int_{\xi(f^{-n}(x))} \varphi\circ f^n(y)\, d\mu_{f^{-n}(x)}(y)\right) d\mu(x)\\
	&=\int\left(\int_{\xi_n(x)} \varphi(y)\, d\big(f^n_{\ast}\mu_{f^{-n}(x)}\big)(y)\right)d\mu(x). 
	\end{align*}
	
	We deduce that $\{f^n_\ast\mu_{f^{-n}(x)}\}_x$ is a system of conditional measures of $\mu$ with respect to $\xi_n$. The lemma follows by uniqueness $\mu$-a.e. of conditional measures. 
\end{proof}

\subsubsection{Partitions subordinate to expanded and contracted foliations}
\label{s.subordinate}

Let $M$ be a closed manifold, $f\colon M\to M$ be a diffeomorphism of $M$, and $\mu$ be an ergodic invariant measure of $f$. Assume that $\cW^+$ is a foliation invariant by $f$. Assume furthermore that it is uniformly expanded, that is $\|Df^{-1}|_{T\cW^+}\|\leq \lambda$ for some constant and $0<\lambda<1$.

A measurable partition $\xi$ is \emph{subordinate to} $\cW^+$ if the following properties hold for $\mu$-a.e. $x\in M$:

\begin{enumerate}
	\item $\xi(x)$ is a subset of $\cW^+(x)$ of diameter less than $1$;
	\item $\xi(x)$ contains an open (in the internal topology) neighbourhood of $x$ in $\cW^+(x)$;
	\item $\xi\prec f^{-1} \xi$ (we say that $\xi$ is \emph{increasing});
	\item $\bigvee_{n=0}^\infty f^{-n}\xi$ is the partition into points.
\end{enumerate}
The existence of subordinate partitions was proven in \cite{LedStr} in a more general context (see also \cite{Yang_exp} and \cite[Appendix D]{Brown_ensaios}). 

\begin{remark}\label{rem_subfoliation_subordinate}
	Assume that $\cW'\subset\cW^+$ is an $f$-invariant subfoliation with expansion constant $0<\lambda'\leq\lambda$ and that $\xi$ is subordinate to $\cW^+$. Then it follows from the proofs given in \cite{Brown_ensaios,LedStr,Yang_exp} that the partition $\xi'=\xi\vee\cW'$ is subordinate to $\cW'$.
\end{remark}

Similarly, assume $\cW^-$ is invariant and \emph{uniformly contracted} by $f$. This means that $\|Df|_{T\cW^-}\|<\lambda$ for some $0<\lambda<1$. We say that $\xi$ is \emph{subordinate to $\cW^-$}  if properties $(1)$ through $(4)$ above holds replacing $f$ by $f^{-1}$.  We then say that $\xi$ is \emph{decreasing} for $f$.

%if the following properties hold for $\mu$-a.e. $x\in M$:
%
%\begin{enumerate}
%	\item $\xi(x)$ is subset of  $\cW^-(x)$ of diameter less than $1$;
%	\item $\xi(x)$ contains an open neighbourhood of $x$ in $\cW^-(x)$;
%	\item $\xi\prec f\xi$ (we say that $\xi$ is \emph{decreasing});
%	\item $\bigvee_{n=0}^\infty f^n\xi$ is the partition into points.
%\end{enumerate}

\begin{remark}\label{rem_subord_interval}
		If $\cW^+$ (resp. $\cW^-$) has dimension $1$ then atoms of the subordinate partition constructed in \cite{Brown_ensaios,LedStr,Yang_exp} are intervals. 
\end{remark}

\subsubsection{Superposition property of subordinate partitions}\label{subsec.superpo}

Let $M,f,\mathcal{W}^+$ be as in the above \S~\ref{s.subordinate}. 
Let $\xi$ be a partition subordinate to $\mathcal{W}^+$. Then, for every $n \geq 0$, $\xi_{-n}\eqdef f^{-n}\xi$ is also a partition subordinate to $\mathcal{W}^+$. Since $\xi_{-n}(x)$ contains an open neighbourhood of $x$ for $\mu$-a.e. $x \in M$, an atom of $\xi$ contains at most countably many atoms of $\xi_{-n}$. 
In the terminology of \cite[Definition 5.15]{EinsLind}, we say that $\xi$ and $\xi_{-n}$ are \textit{countably equivalent}. 

%Let $f\in\cA^2(\TT)$ an Anosov diffeomorphism with expanding center. Fix $\xi^s$ a measurable partition subordinate to the stable foliation $\cWs$. Consider the sequence $\xi^s_n\eqdef f^n(\xi^s)$. This defines a decreasing sequence of subordinate partitions whose corresponding $\sigma$-algebras are increasing. By Remark~\ref{rem_subord_interval}, we have that, for every $n\in\N$, each atom of $\xi^s$ can be covered by at most countably many atoms of $\xi^s_n$. 
From \cite{EinsLind}, Proposition 5.17 we obtain the following \emph{superposition property}.

\begin{lemma}[Superposition property]
	\label{l.superposition}
	For $\mu$ a.e. $x\in M$ it holds that %\annotation{ Why are we calling this superposition property?
	%Seb:	I don't remember: I thought that Einsiedler and Lindenstrauss were calling this property like this. }
	\[
	\mu_x(\xi_{-n}(x))>0.
	\]
	Moreover, if $\mu_{-n,x}$ is the disintegration of $\mu$ with respect to the partition $\xi_{-n}$, that is, for any measurable set $A$,
	\[
	\mu_{-n,x}(A)=\frac{\mu_{x}(A \cap \xi_{-n}(x))}{\mu_x(\xi_{-n}(x))}.
	\]
\end{lemma} 

By the definition of a subordinate partition, for $\mu$-a.e. $x\in M$, $\{\xi_{-n}(x)\}_{n \geq 0}$ contains a basis of open neighbourhoods of $x$. 
\begin{corollary}\label{coro.superposition}
	For $\mu$-a.e. $x \in M$, $\mu_x$ charges every open neighbourhood of $x$. 
\end{corollary}

\begin{remark}\label{remark.superposition}
	If $\mathcal{W}^-$ is contracting, and $\xi$ is subordinate to $\mathcal{W}^-$, then the analogous statements to Lemma \ref{l.superposition} and Corollary \ref{coro.superposition} hold for $\mathcal{W}^-$, replacing  $\xi_{-n}$ with $\xi_{n}$, for $n\geq 0$. 
\end{remark}

\subsubsection{Uniform growth property}

Another useful property of the sequence of partitions $\xi_n$ subordinate to a uniformly expanded foliation $\cW^+$ is the following \emph{uniform growth property}.

\begin{lemma}
	\label{lem.af1}
Given a real number $R>0$, for $\mu$ almost every $x\in M$ there exists $n_0=n_0(x,R)>0$ such that if $n>n_0$ then 

$$\cW^+_{R}(x)\subset\xi_n(x)$$
\end{lemma}
\begin{proof}
Define $\Lambda(\eps)=\{x\in\TT;\cW^+_{\eps}(x)\subset\xi(x)\}$, and observe that $\mu(\Lambda(\eps))\to 1$ as $\eps\to 0$, due to property (2) in the definition of subordinate partitions. The conclusion holds whenever $f^{-n_k}(x)\in\Lambda(1/k)$ where $n_k>\frac{\log(kR)}{-\log\lambda}$ (recall that $\|Df^{-1}|_{T\cW^+}\|<\lambda<1$). So the conclusion of the lemma holds in the set $\bigcup_k f^{n_k}(\Lambda(1/k))$ which has measure $1$ by $f$-invariance of $\mu$.
\end{proof}

\begin{remark}
	\label{rem.af1}
The above lemma has an analogous version for measurable partitions subordinate to uniformly contracted foliations: the integer $n_0$ is then negative and the conclusion holds for all $n<n_0$.
\end{remark}

\subsection{SRB and $u$-Gibbs measures}\label{subsec.uandsrb}

We return now to our partially hyperbolic setting where $f\in\cA^2(\TT)$. Let $\xi^{cu}$ be a measurable partition of $\TT$ subordinate to the center unstable foliation $\cWcu$. Let $\mu$ be an ergodic invariant measure for $f$. We say that $\mu$ is an \emph{SRB measure} if its disintegration $\{\mu^{cu}_x\}_{x\in\TT}$ with respect to the partition $\xi^{cu}$ satisfies that $\mu^{cu}_x$ is absolutely continuous with respect to the inner Lebesgue measure $\operatorname{Leb}^{cu}_x$ of $\cWcu(x)$ for $\mu$ a.e. $x\in\TT$.

Consider now $\xi^u$ a measurable partition subordinate to the unstable foliation $\cWu$ and let $\{\mu^u_x\}_{x\in\TT}$ denote the disintegration. We say that $\mu$ is a \emph{$u$-Gibbs measure} provided that $\mu^u_x$ is absolutely continuous with respect to the inner Lebesgue (length) measure $\operatorname{Leb}^u_x$ of $\cWu(x)$ for $\mu$ a.e. $x\in\TT$.

\begin{remark}
	By Ledrappier-Young \cite{LYI} none of the above definitions depend on the respective particular choice of subordinate partition.
\end{remark}

\subsubsection{Conditional measures along unstable leaves of an $u$-Gibbs measure}\label{sub.sub.conditional}

We can (and we shall) consider the particular case in which the atoms of $\xi^u$ are $\xi^u(x)=\cW^u(x)\cap\xi^{cu}(x)$. By Remark~\ref{rem_subfoliation_subordinate} we know that this defines a partition subordinate to the unstable foliation $\cW^u$.  This is specially important in our setting, in which we want to prove that a given $u$-Gibbs measure is SRB. In a similar fashion, we consider the partition $\xi^c$ whose atoms are $\xi^c(x)=\cW^c(x)\cap\xi^{cu}(x)$. 

Denote by $\xi^{*}_n\eqdef f^n(\xi^{*})$, for $*=u,cu,c$ and $n\geq 0$,  with $\xi^*_0=\xi^{*}$. Consider $\mu$ a $u$-Gibbs measure and let $\mu^{cu}_{n,x}$ denote the conditional measures along the partition $\xi^{cu}_n$.  Notice that $\{\xi^u_n(y)\}_{y\in\xi^{cu}_n(x)}$ is a measurable partition of the probability space $(\xi^{cu}_n(x),\mu^{cu}_{n,x},\cB|_{\xi^{cu}_n(x)})$, where $\cB|_{\xi^{cu}_n(x)}$ denotes the Borel sigma algebra of $\TT$ restricted to the atom $\xi^{cu}_n(x)$. Rokhlin's disintegration theorem in this case give us a probability measure $\mu^c_{n,x}$ defined on $\xi^c_n(x)$ and a family of probability measures $\{\mu^u_{n,y}\}_{y\in\xi^c_n(x)}$ such that for every $A\subset\xi^{cu}_n(x)$ Borel measurable set we have
\[
\mu^{cu}_{n,x}(A)=\int_{A\cap\cW^c(x)}\mu^u_{n,y}(A\cap\cW^u(y))d\mu^c_{n,x}(y).
\]  

\begin{remark}
	\label{rem.absolutevalue}
	For simplicity of the exposition in many situations, where no confusion may arise, given a set $A\subset\cW^*(x)$, we shall denote $|A|\eqdef\operatorname{Leb}^*_x(A)$ for $*\in\{s,c,u\}$.
\end{remark}

An easy consequence of the $u$-Gibbs property is the following.
\begin{lemma}\label{l.density_u-gibbs}
	There exists $\beta>1$ depending only on the diffeomorphism $f$ such that for $\mu$-almost every $x\in\TT$, 
	$$\frac{1}{\beta\,\Leb^u_x\left[\xi^u(x)\right]}<\frac{d\mu^u_x}{d\Leb^u_x}<\frac{\beta}{\Leb^u_x\left[\xi^u(x)\right]}$$
	inside $\xi^u(x)$.
\end{lemma}
\begin{proof}
	This follows from the fact that $\mu^u_x$ is a probability measure that has  a uniformly log-Hölder continuous density with respect to $\Leb^u_x$.
\end{proof}

\section{Heuristics of the proof}\label{heuristics}
This section can be used to get an overview of the proof of Theorem \ref{mainthm.dicotomia} and also as a guide to read our paper.

From now on let $f \in \cA^2(\TT)$ be a diffeomorphism that is close to a conservative one. In this case, stable holonomies $H_{x,y}^s$ are $C^1$, and we can define angles $\alpha^s(x,y)$ for two points $x,y$ in the same stable manifold (see Definition~\ref{def_angle-function}). The vanishing of the angle $\alpha^s$ can be seen as a kind of infinitesimal joint integrability. 

Let $\mu$ be an ergodic $u$-Gibbs measure for $f$. 

\subsection*{Starting point: a zero-one law}%\marginpar{S. J'ai viré la dépendance en $\mu$ dans le bad set parce que dans la section plus loin je l'ai pas mise et qu'ils ont l'air d'être relou avec la cohérence des notations. J'ai pas voulu la mettre plus tard parce que ça faisait bizarre de juste ne pas mettre la dépendance par rapport à la partition... J'ai cru bien faire.}

We first use Theorem \ref{thm.zeroonelawintro}, which yields the following dichotomy. 
\begin{itemize}
	\item[\it Case 1:] $\alpha^s(x,y)=0$, for $\mu$-a.e. every pair $(x,y)$ in the same stable manifold;
	\item[\it Case 2:] $\alpha^s(x,y)>0$, for $\mu$-a.e. every pair $(x,y)$ in the same stable manifold. 
\end{itemize} 
The proof of this theorem will be given in Section \ref{s.zeroouum} and relies on a martingale argument.  
In words, we either have positive angles almost everywhere along the stable manifold of almost every point, or we have zero angles almost everywhere along the stable manifold of almost every point. 

The rest of the proof is divided into two parts:
\begin{itemize}
	\item $[$Case 1 + $\mathrm{supp}(\mu)=\TT]\implies$ joint integrability;
	\item Case 2$\implies\mu$ is SRB.
\end{itemize}   %Our proof can be splitted in two parts: the first part is a dichotomy about joint integrability of $E^s$ and $E^u$ and a type of ``transversality'' condition; the second part is about using this transversality condition to prove that $\mu$ is SRB. 

\subsection*{Part 1: joint integrability.} % and the transversality condition.}

 Given a point $x\in \TT$ we define 
%\[
%\mathcal{P}(x)\eqdef\left\{y\in\cW^s(x):\alpha^s(x,y)=0\right\},%\quad %\mathcal{N}(x)\eqdef\left\{y\in\cW^s(x):\alpha^s(x,y)>0\right\}
%\]
%where $DH^s_{x,y}$ is the derivative of the stable holonomy between the local center stable of $x$ and of $y$ at the point $x$ (recall \S\ref{sss.holo}). This derivative allows us to study how $E^s$ and $E^u$ jointly integrate or not in an infinitesimal level.  
the \emph{Bad set} as
\[
\mathbf{B}\eqdef \{x \in \TT:\mu_x^s\{\alpha^s_x>0\}=0\},
\]
where $\alpha_x^s\eqdef [y \in \mathcal{W}^s(x)\mapsto \alpha^s(x,y)]$, and $\mu^s_x$ is the conditional measure of $x$ along a stable manifold (see \S\ref{ss.badset} for more details). The \textit{Bad set} is the set of points such that there is an ``infinitesimal'' joint integrability with almost every other point in its stable manifold.  

The first part is an argument by contraposition that goes as follows (see  Proposition \ref{p.mesuredubadset}):
\begin{itemize}
	\item if there is no joint integrability, then for any $x\in \TT$, the set $\{\alpha_x^s>0\}$ is open and dense within the stable manifold $\mathcal{W}^s(x)$ (see Lemma \ref{lemm dicht});
	\item if, furthermore, the \textit{Bad set} $\mathbf{B}$ has measure $1$, the continuity of the angles implies that whenever $\alpha^s(x_0,y_0)>0$, there exists a small foliated chart $\mathcal{V}$ around $y_0$ for $\mathcal{W}^s$ of measure $0$ (see Lemma \ref{lemme Bruno});
	\item hence, combining the two previous points, if there is no joint integrability and $\mathbf{B}$ has full measure, then the support of $\mu$ must have empty interior. 
\end{itemize}
The last point is the only place in the paper where the support condition on $\mu$ is used. 
%In Proposition \ref{p.mesuredubadset} we prove that if $\mu$ has full support and $\mu(B_{\mu}) >0$ then $E^s$ and $E^u$ are jointly integrable.  In particular,  if $E^s$ and $E^u$ are not jointly integrable, then for any fully supported $u$-Gibbs measure $\mu$ there are ``angles'' almost everywhere (i.e., $\mu(B_{\mu}) = 0$).  The proofs related to this part are completed in Section~\ref{s.main_thm}. The transversality we will use in the paper is given by these angles.

%\subsubsection*{A zero or one law for spreading transversality}

%Still regarding the kind of transversality between $E^s$ and $E^u$ required for our arguments, an important technical step is achieved in Section~\ref{s.zeroouum} where we use an idea from \cite{BRH} to prove a \emph{zero or one law for the existence of angles} between $DH^s_{x,y}E^u(x)$ and $E^u(y)$. This argument uses a martingale convergence theorem to prove a dichotomy: either one has $\mu^s_x(\cP(x))=0$ for a.e. $x$ or $\mu^s_x(\cP(x))=1$ for a.e. $x$.

\subsection*{Part 2: transversality implies SRB}

Most of this paper is dedicated to proving that if $\mu$ is a $u$-Gibbs measure such that $\mu(\mathbf{B}) = 0$ then $\mu$ is SRB (Theorem \ref{mainthm.technique}).

To check that a $u$-Gibbs measure $\mu$ is SRB, one can consider the disintegration of this measure along center-unstable manifolds $\{\mu^{cu}_x\}$ and then quotient it by the strong unstable manifolds $\{\hat{\mu}^c_x\}$. These are called the transverse measures. Then $\mu$ is SRB if and only if the transverse measures are equivalent to Lebesgue.  In this approach, it is really convenient to consider certain parameterizations of center-unstable manifolds that simplify the dynamics, the so-called \emph{normal forms} (see Section \ref{section formes normales}).  These coordinates allow us to identify the quotient measures $\{\hat{\mu}^c_x\}$ with measures $\{\hat{\nu}^c_x\}$ in $\mathbb{R}$.  To conclude that the measures $\{\hat{\mu}^c_x\}$ are equivalent to Lebesgue, we will show that the measures $\{ \hat{\nu}^c_x\}$ are proportional to the Lebesgue measure on $\mathbb{R}$, where we say that two locally finite Borel measures $\nu,\eta$ are proportional, and we indicate it by $\nu\propto\eta$, if $\nu=c\eta$ for some $c>0$.   

To do so, it is enough to prove that for $\mu$-a.e. $x\in\TT$, $\hat{\nu}^c_x$ is invariant by translation. Actually, thanks to an argument due originally to Katok-Spatzier \cite{KatokSpa} and Kalinin-Katok \cite{KalininKatok} (see also \cite[Proposition 7.1]{BRH}), which is a beautiful mixture of ergodic and Lie theoretic arguments, it is enough to prove something weaker: \emph{for a set $G$ of points $x\in\TT$ with positive measure, there exist affine maps $\psi$ with uniformly bounded derivatives and arbitrarily small translational parts such that $\hat \nu_x^c\propto\psi_\ast\hat\nu_x^c$}. This is Lemma \ref{l.lema6.1} (largely inspired by \cite[Lemma 6.1]{BRH_Little}).

%This can be done by proving that $\hat{\nu}^c_x$ is,  for many points $x$, ``invariant'' by many affine maps (invariant here means that the measure class is preserved). The construction of these affine maps is at the core of this paper. The goal is to prove that there is a set of positive measure of points $\hat{x}$ such that for any $\varepsilon>0$ there exists an affine map $\psi:\mathbb{R} \to \mathbb{R}$ such that $\hat{\nu}^c_{\hat{x}}$ is ``invariant'' by $\psi$. Moreover, $\psi$ has slope bounded by a constant independent of $\varepsilon$ and  $|\psi(0)| \approx \varepsilon$ (Lemma \ref{l.lema6.1}, which is largely inspired by Lemma 7.1 from Brown-Rodriguez Hertz \cite{BRH}). This idea of getting invariance by translations through invariance under many affine maps with controlled slope and small translational part goes back to the works of Katok-Saptzier \cite{KatokSpa} and Kalinin-Katok \cite{KalininKatok}.

The construction of these affine maps is where our key arguments are located. Our strategy is to use the \emph{$Y$-configurations} introduced by Eskin-Mirzakhani \cite{EskinMirzakhani} (see also Eskin-Lindenstrauss \cite{EskinLind}). Below we outline this proof, splitting the explanation in two parts, for the sake of clarity. First, we describe $Y$-configurations. Then we explain how to use them in order to get invariance by affine maps in the way we described in last paragraph.

\subsubsection*{$Y$-configurations} 

See Figure~\ref{f.coupledconfig}. Fixing $\ell\in \mathbb{N}$ we find points $x$ and $y$ that are typical for $\mu$ such that:
\begin{enumerate}
\item $y\in \cW^s_{\mathrm{loc}}(x)$ and  $d(x,y) \approx \|Df^\ell|_{E^s}\|$;
\item if $x_{-\ell} = f^{-\ell}(x)$ and $y_{-\ell} = f^{-\ell}(y)$,  then $d(x_{-\ell},y_{-\ell}) \approx 1$ and the angle $\alpha^s(x_{-\ell},y_{-\ell})$ is more than some constant $\frac{1}{C}$.
\end{enumerate} 
As a consequence of points (1) and (2), $\alpha^s(x,y)\geq \frac{1}{C} \frac{\|Df^\ell(x_{-\ell})|_{E^c}\|}{\|Df^\ell(x_{-\ell})|_{E^u}\|}$. Normal forms help a lot here. This computation is performed inside the proof of Lemma~\ref{control centre expa}.

We can then find points $x^u\in \cW^u_{\mathrm{loc}}(x)$ and $y^u \in \cW^u_{\mathrm{loc}}(y)$ such that:
\begin{itemize}
\item there exists a point $z^u$ with $\{z^u\} = \cW^s_{\mathrm{loc}}(x^u) \cap \cW^c_{\mathrm{loc}}(y^u)$;
\item $d(z^u, y^u) \approx \frac{\|Df^\ell(x_{-\ell})|_{E^c}\|}{\|Df^\ell(x_{-\ell})|_{E^u}\|}$.
\end{itemize}

So far we obtained points $x^u$ and $y^u$ with some estimate of the displacement in the center direction that we get when projecting $x^u$ to $\cW^u_{\mathrm{loc}}(y)$ by center stable holonomy. Next, fixing $\eps>0$, we define the stopping time 
\[
\tau(\ell)= \tau(x,x^u,\varepsilon,\ell)\eqdef \inf \left\{n\in \mathbb{N}: \frac{\|Df^\ell(x_{-\ell})_{E^c}\|}{\|Df^\ell(x_{-\ell})|_{E^u}\|} \|Df^n(x^u)|_{E^c}\| \geq \varepsilon\right\}
\]
and consider the points $f^{\tau(\ell)}(x^u)$ and $f^{\tau(\ell)}(y^u)$.  The choice of the stopping time $\tau(\ell)$ is such that 
\begin{equation}
\label{eq.distanceestimate}
d(f^{\tau(\ell)}(x^u), f^{\tau(\ell)}(y^u)) \approx \varepsilon,
\end{equation}
see \S\ref{sss.drifestimates}.  Essentially, these points are related to the translational part of size $\varepsilon$ of the affine maps that we desire to construct. To control the derivative of the affine maps and to actually obtain an ``invariance'' of the measures $\hat{\nu}^c_x$, as we mentioned before, we also define another stopping time
\[
t(\ell)=t(x,x^u,\varepsilon,\ell) \eqdef \inf \left\{n\in \mathbb{N} : \frac{\|Df^{n}(x)|_{E^c}\|}{\|Df^{\tau(\ell)}(x^u)|_{E^c}\|} \geq 1\right\}.
\]
In particular, this condition implies that
\begin{equation}
\label{e.tempomaroto}
\|Df^{t(\ell)}(x)|_{E^c}\| \approx \|Df^{\tau(\ell)}(x^u)|_{E^c}\|,
\end{equation}
see \S \ref{subsec.stoppingtimes}.
\begin{figure}[h]
	\centering
	\begin{tikzpicture}[scale=.6]
	\fill[black!30!white, opacity=.6] (10,0) -- (9.5,0.5) to[bend left] (9.5,1) -- (10,0);
	\fill[green!30!white, opacity=.5] (-1,1)--(1,-1)--(11,-1)--(9,1)--(-1,1);
	\draw[green!40!black, thick] (0,0)--(10,0) node[midway,below]{\tiny $\cW^s_{\loc}(x_{-\ell})$};
	\draw[red!80!black,thick] (-1,1)--(1,-1) node[left]{\tiny $\cW^u_{\loc}(x_{-\ell})$};
	\draw[violet, thick] (9,1)--(11,-1) node[right]{\tiny $H^s_{x_{-\ell},y_{-\ell}}(\cW^u_{\loc}(x_{-\ell}))$};
	\fill[fill=black] (9,1) circle (2pt);
	\draw (9.1,1.4) node[left]{\small $z^u_{-\ell}$};
	\draw[red!80!black, thick] (11,-2)node[right]{\tiny $\cW^u_{\loc}(y_{-\ell})$}--(9,2);
	\fill[fill=black] (9,2) node[right]{$y^u_{-\ell}$} circle (2pt); 
	\draw[dotted] (11,-2)--(11,-1);
	\draw[dotted] (9,1)--(9,2);
	\fill[green!30!white, opacity=.5] (4,4)--(1,4.5)--(3,4.5)--(6,4);
	\draw[green!40!black, thick] (4,4)--(6,4);
	\draw[green!40!black, thick] (4.7,10)--(5.3,10);
	\draw[green!40!black, thick] (-0.2,14)--(0.2,14);
	%	\draw[blue!40!black, thick] (4.8,6)--(4.9,6);
	\draw[red!80!black,thick] (4,4)--(1,4.5); %node[midway, below]{$\alpha^u$};
	\draw[red!80!black, thick] (6,4)--(3,5);
	%\draw[blue!40!black, thick] (0.5,7)--(0.75,7);
	\draw[orange, thick] (3,4.5)--(3,5);
	%\draw[red!70!white, thick] (0.75,7)--(0.75,9) node[midway, left]{$\varepsilon\sim$};
	%\draw[black!70!white, opacity=.6,->] (1,4.5)--(0.5,7);
	%\draw[black!70!white, opacity=.6,->](3,4.5)--(0.75,7);
	%\draw[black!70!white, opacity=.6,->] (3,5)--(0.75,9);
	\draw[densely dotted, black!70!white, opacity=.6,->] (0,0).. controls (2,1) and (3.7,3) .. (4,4) node[midway, right]{$f^\ell$};
	\draw[densely dotted, black!70!white, opacity=.6,->] (4,4)--(4.7,10) node[midway, left]{$f^{t(\ell)}$};
	\draw[densely dotted, black!70!white, opacity=.6,->] (1,4.5)--(-0.1,14) node[near end, left]{$f^{\tau(\ell)}$};
	\draw[densely dotted, black!70!white, opacity=.6,->] (3,4.5)--(0.2,14) node[near end, right]{$f^{\tau(\ell)}$};
	\draw[densely dotted, black!70!white, opacity=.6,->] (3,5)--(0.2,16) node[near end, right]{};
	\draw[densely dotted, black!70!white, opacity=.6,->] (10,0).. controls (9,0) and (6,3) .. (6,4) node[midway, left]{$f^\ell\ $};
	\draw[densely dotted, black!70!white, opacity=.6,->] (6,4)--(5.3,10) node[midway, right]{$f^{t(\ell)}$};
	%\draw[black!70!white, opacity=.6,->] (4,4)--(4.8,6);
	%\draw[black!70!white, opacity=.6,->] (6,4)--(4.9,6);
	\fill[fill=black] (0,0) node[left]{\small $x_{-\ell}$} circle (2pt);
	\draw[violet, thick] (6,4)--(3,4.5); 
	\draw[violet,thick]  (9.7,0.8) node[right]{\tiny $\alpha^s(x_{-\ell},y_{-\ell})$};
	\fill[fill=black] (10,0) node[right]{\small $y_{-\ell}$} circle (2pt);
	\fill[fill=black] (1,4.5) node[left]{\small $x^u$} circle (2pt);
	\draw[orange, thick] (0.2,14)--(0.2,16) node[left, midway]{\small $\asymp \varepsilon$};
	%\fill[fill=orange] (2.2,15) node[right]{\small $\asymp \varepsilon$};
	\fill[fill=black] (-0.1,14) node[left]{\small $f^{\tau(\ell)}(x^u)$} circle (2pt);
	\fill[fill=black] (0.2,14) node[right]{\small $f^{\tau(\ell)}(z^u)$} circle (2pt);
	\fill[fill=black] (0.2,16) node[left]{\small $f^{\tau(\ell)}(y^u)$} circle (2pt);
	\fill[fill=black] (-1,1) node[left]{\small $x^u_{-\ell}$} circle (2pt);
	\fill[fill=black] (6,4) node[right]{\small $y$} circle (2pt);
	\fill[fill=black] (4,4) circle (2pt);
	\draw (4,4) node[below]{\small $x$};
	\fill[fill=black] (4.7,10) circle (2pt);
	\draw (4.8,10) node[left]{\small $f^{t(\ell)}(x)$};
	\fill[fill=black] (5.3,10) circle (2pt);
	\draw (5.3,10) node[right]{\small $f^{t(\ell)}(y)$};
	%\fill[fill=black] (4.8,6) node[left]{\small $f^{t(\ell)}(x)$} circle (2pt);
	%\fill[fill=black] (4.9,6) node[right]{\small $f^{t(\ell)}(y)$} circle (2pt);
	\draw[green!40!black, thick] (1,4.5)--(3,4.5);
	\fill[fill=black] (3,4.5) circle (2pt);
	\draw (3.1,4.8) node[left]{\small $z^u$};
	\fill[fill=black] (3,5)  circle (2pt);
	\draw (3,5.3) node[right]{\small $y^u$};
	%\fill[fill=black] (0.75,7) node[right]{\tiny $f^{\tau(\ell)}(z)$} circle (2pt); 
	%\fill[fill=black] (0.75,9) node[right]{\tiny $f^{\tau(\ell)}(y^u)$} circle (2pt);
	%\fill[fill=black] (0.5,7) node[left]{\tiny $f^{\tau(\ell)}(x^u)$} circle (2pt);  
	%\draw (9.5,0.5) to[bend left] (9.5,1);
	\end{tikzpicture}
	\caption{\label{f.coupledconfig} The points $x,x^u, x_{-\ell}, f^{\tau(l)}(x^u)$ and $f^{t(\ell)}(x)$ are called a $Y$-configuration (similarly the points involving $y$'s).} 
\end{figure}

\subsubsection*{How to use $Y$-configurations to get invariance by affine maps}

By construction of measures $\{\hat{\nu}^c_p\}_{p\in\TT}$ we understand how they change under three basic moves.
\begin{enumerate}
\item \emph{Applying the dynamics}. For $\mu$-a.e. $p\in\TT$,  
$\hat{\nu}^c_{f^n(p)}\propto\Lambda_*\hat{\nu}^c_p$, for every $n\in\N$, where $\Lambda(s)=\|Df^n(p)|_{E^c}\| s$ is a linear map of $\R$ (see Lemma~\ref{chang dyna}). 
\item \emph{Moving along unstable manifolds}. For $q\in \cW^u(p)$ then $\hat{\nu}^c_q\propto L_\ast \hat{\nu}^c_p$ where $L(s)=\beta s$ is a linear map,  for some $C_u^{-1}\leq\beta\leq C_u$, and $C_u=C_u(d_u(p,q))>1$ is a number which is bounded from above and it depends on the unstable distance $d_u(p,q)$ (see Lemma~\ref{changt inst}).
\item \emph{Moving along center manifolds}. For $q\in \cW^c(p)$ then $\hat{\nu}^c_q\propto \Psi_\ast \hat{\nu}^c_p$ where $\Psi(s)=as+b$, with $C_c^{-1}\leq a\leq C_c$ for some constant $C_c=C_c(d_c(p,q))>1$ bounded from above with $d_c(p,q)$, and with $b\approx d_c(p,q)$ (see Lemma~\ref{chanfemt centrale}).
\end{enumerate}

Now we look at the points on the top part of Figure~\ref{f.coupledconfig} and deduce two facts.

\begin{enumerate}[label=(\alph*)]
\item $\hat{\nu}^c_{f^{\tau(\ell)}(x^u)}\propto(\Lambda_1)_*\hat{\nu}^c_{f^{t(\ell)}(x)}$ for some \emph{linear} map $\Lambda_1\colon\R\to\R$ with derivative bounded independently of $\ell$.
\item $\hat{\nu}^c_{f^{\tau(\ell)}(y^u)}\propto(\Lambda_2)_*\hat{\nu}^c_{f^{t(\ell)}(y)}$ for some \emph{linear} map $\Lambda_2\colon\R\to\R$ with derivative bounded independently of $\ell$.
%\item $\hat{\nu}^c_{f^{\tau(\ell)}(y^u)}\propto \Psi_*\hat{\nu}^c_{f^{\tau(\ell)}(z^u)}$ for some \emph{affine} map $\Lambda_1:_\R\to\R$ with derivative bounded independently of $\ell$ and translational part $\approx\eps$.
\end{enumerate}

There is a subtlety here: stopping times $\tau(\ell)$ and $t(\ell)$ depend on $x,x^u$ and not on $y,y^u$. This is treated by our \emph{synchronization estimates} (Lemma \ref{lemme synchro}). Now note that as $\ell\to\infty$, $\tau(\ell),t(\ell)\to\infty$ so $d(f^{\tau(\ell)}(x^u),f^{\tau(\ell)}(z^u))\to 0$ and $d(f^{t(\ell)}(x),f^{t(\ell)}(y))\to 0$. If we knew that the family of measures $\hat \nu^c_z$ were continuous with $z$ then we could hope to compare $\hat{\nu}^c_{f^{\tau(\ell)}(x^u)}$ with $\hat{\nu}^c_{f^{\tau(\ell)}(y^u)}$ and take accumulation points to construct the desired set $G$.

%Also, when we move the base point $p$ along a strong unstable manifold the changing on the measure $\hat{\nu}^c_p$ is also linear . In Figure~\ref{f.coupledconfig} we can consider this latter linear map with slope independent of $\ell$. This implies that the measures $\hat{\nu}^c_{f^{\tau(\ell)}(x^u)}$ and $\hat{\nu}^c_{f^{t(\ell)}(x)}$ are proportional up to a linear map, whose slope is proportional to $1$ (recall \eqref{e.tempomaroto}). The same can be said about $\hat{\nu}^c_{f^{\tau(\ell)}(y^u)}$ and $\hat{\nu}^c_{f^{t(\ell)}(y)}$.

%
%Since we want to control and compare the measures $\hat{\nu}^c_z$ for different points, and many of 
But the objects we are working with are only measurable, so we must first fix a large Lusin set for which the map $z\mapsto \hat{\nu}^c_z$ is continuous (among other dynamical objects that appear in the proof). We want to do the constructions of all of the points mentioned above, in a way that all of them belong to this Lusin set. For this, it is essential to obtain \emph{quasi-isometric estimates} for the functions $\tau(\cdot)$ and $t(\cdot)$ (see Lemma \ref{l_qi_estimates}; see also Lemma~\ref{l.lemadosaci} where the quasi-isometric estimates for stopping times are used in a crucial way).

By continuity, we have that $\hat{\nu}^c_{f^{t(\ell)}(x)}\approx\hat{\nu}^c_{f^{t(\ell)}(y)}$ and $\hat{\nu}^c_{f^{\tau(\ell)}(x^u)}\approx\hat{\nu}^c_{f^{\tau(\ell)}(z^u)}$. We conclude from this that $\hat{\nu}^c_{f^{\tau(\ell)}(x^u)}$ is \emph{almost proportional to} $\hat{\nu}^c_{f^{\tau(\ell)}(y^u)}$, \emph{up to some linear map with controlled derivative}.

After considering a subsequence $\ell_k$ we obtain points $p$ and $q$ (obtained as the limit of $f^{\tau(\ell_k)}(x^u)$ and $f^{\tau(\ell_k)}(y^u))$) with the properties that:
\begin{itemize}
\item $q\in \mathcal{W}^c_{\mathrm{loc}}(p)$;
\item$d(p,q) \approx \varepsilon$;
\item $p$ and $q$ belong to the Lusin set;
\item $\hat{\nu}^c_q\propto \hat \Lambda_\ast \hat{\nu}^c_p$, $\hat \Lambda$ being linear with uniformly bounded derivative (see \S\ref{ss.mainreduction}).
\end{itemize}   

Finally we explained that $\hat{\nu}^c_p\propto \Psi_\ast\hat{\nu}^c_q$ for an affine map $\Psi$ with uniformly bounded derivative and a translational part of order $\eps$: we find $\hat{\nu}^c_p\propto(\Psi\hat \Lambda)_\ast \hat{\nu}^c_p$: $\Psi\hat \Lambda$ is the desired affine map (with bounded slope and translational part $\approx\eps$).

%
%Since normal forms along one dimensional expanding foliations provide us with an \emph{affine structure} (see Proposition~\ref{p.kk}, which comes from \cite{KalininKatok}) the family $\{\hat{\nu}^c_p\}_{p\in\TT}$ also has the property for points on the same center leaf, the changing is affine. Thus for the points $p$ and $q$ constructed above, $\hat{\nu}^c_p$ and $\hat{\nu}^c_q$ are also proportional to each other up to an affine map, whose translational part is proportional to $d_c(p,q)$ (see Lemma~\ref{chanfemt centrale}). Combining the two informations we deduce that there exists an affine map $\psi :\mathbb{R} \to \mathbb{R}$ such that $\hat{\nu}^c_p$ is ``invariant'' by $\psi$ and $\psi$ has bounded slope and translational part of order $\varepsilon$. Actually, when viewed with normal forms, $\psi$ take $p$ to $q$, that is, $0$ corresponds to $p$ and $\psi(0)$ corresponds to $q$.

\subsubsection*{A technical difficulty}

There is one delicate point in the above argument. Since the center stable foliation is not absolutely continuous, in general, we cannot choose the points $x^u$ and $y^u$ at the same time in the Lusin set and such that $y^u \in \cW^{cs}_{\mathrm{loc}}(x^u)$.  To overcome this difficulty, we introduce the notion of matched $Y$-configurations in Section~\ref{s.coupled} (the picture changes just a little bit from the one described above) and use it to prove our theorem (see Section \ref{s.end_proof}).

\bigskip

Let us also emphasize that the control we obtain on the distance between the points $z^u$ and $y^u$ is possible because $H^s$ is $C^1$, so essentially we can control this distance by looking at the angle between $DH^s_{x,y} E^u(x)$ and $E^u(y)$. This allows us to explicitly define the stopping time $\tau$ in our construction. This is not possible in Katz’s proof since the holonomies are only Hölder. To get around this problem, he defines a much more complicated stopping time built from an operator using ideas from Eskin-Mirzakhani.

\section{A zero-one law for angles}\label{s.zeroouum}

The goal of this section is to prove Theorem \ref{thm.zeroonelawintro}. In a slightly more general context, i.e., for some $f\in\mathcal{PH}^2(\TT)$ with $C^1$ stable holonomies (for instance, if $f$ is close to some $f_0\in\cA^2_m(\TT)$, as we saw in Lemma~\ref{c one holon}), we introduce the angle function $\alpha^s$, which measures the ``twist'' of unstable manifolds along stable manifolds.
\begin{defi}[The angle function]\label{def_angle-function}
	Let $x\in\TT$ and $y\in \mathcal{W}^s(x)$. The \emph{angle function} is the assignment $(x,y)\mapsto \alpha^s(x,y)\eqdef\angle(DH^s_{x,y}(x)E^u(x),E^u(y))$.
\end{defi} 
%establish a dichotomy for the existence of angles between $DH^s_{x,y}E^u(x)$ and $E^u(y)$ for \emph{almost every pair} $y\in\cWs(x)$. More precisely, given an ergodic invariant measure $\mu$ for some $f\in\cA^2(\TT)$ with $C^1$ stable holonomies (for instance, if $f$ is close to some $f_0\in\cA^1_m(\TT)$, as we saw in Lemma~\ref{c one holon}) then either for $\mu$ almost every $x$ and $\mu^s_x$ almost every $y\in\cWs(x)$ the holonomy map $DH^s_{x,y}$ twists the bundle $E^u$, and we have infinitesimal transversality, or not. 

%The place of this technical result in our global strategy is fundamental but subtle: in the absence of joint integrability, it 

%In this section we establish Theorem \ref{thm.zeroonelawintro} in a slightly more general form; it
In this context, we have the following zero-one law, whose proof is inspired by the work of Brown-Hertz \cite[Lemma 7.1]{BRH_Little}. 
\begin{theorem}[A zero-one law for angles]\label{th_0-1}
	Let $f\in \mathcal{PH}^2(\TT)$ be %an Anosov diffeomorphism with a one-dimensional stable bundle $E^s$ and a two-dimensional unstable bundle $E^{cu}=E^{c}\oplus E^{u}$, and whose stable holonomies $H^s$ are $C^1$. 
a partially hyperbolic diffeomorphism with a %one-dimensional stable bundle $E^s$ and a two-dimensional unstable bundle $E^{cu}=E^{c}\oplus E^{u}$
	splitting $T\TT=E^s\oplus E^c \oplus E^u$, whose stable holonomies $H^s$ along the strong stable foliation $\cWs$ are $C^1$.  Fix an ergodic $f$-invariant measure $\mu$. Let $\xi^s$ be a measurable partition subordinate to $\cW^s$ and $\{\mu^s_x\}_x$ be a system of conditional measures relative to $\xi^s$. Then the following dichotomy holds:
	\begin{enumerate}
		\item either for $\mu$-almost every $x\in \TT$,
		$$\mu^s_x\left\{y\in\xi^s(x):\alpha^s(x,y)=0\right\}=1;$$
		
		\item or for $\mu$-almost every $x\in \TT$,
		$$\mu^s_x\left\{y\in\xi^s(x):\alpha^s(x,y)=0\right\}=0.$$
	\end{enumerate}
\end{theorem}
The second alternative will allow us to build ``twisted'' quadrilaterals (such as the one depicted in Figure~\ref{fig.quadrilatero}) with points $x,y$ in some good Lusin set while for $x^u$ and $y^u$ we will find points of the Lusin set arbitrarily close to them, so that the twisted quadrilateral will be part of a \emph{matched $Y$-configurations} associated to $x$ and $y$.    

\subsection{Conditional expectation, $\sigma$-algebras and martingales}\label{sss_martitou}
The proof of Theorem \ref{th_0-1}  requires some preliminaires about martingales. 
If $\xi$ is a measurable partition of a measurable space $X$ then we let $\cF_\xi$ denote the $\sigma$-algebra generated by unions of atoms of $\xi$. Let $\{\mu_x\}_x$ be a system of conditional measures of $\mu$ with respect to $\xi$. We define the \emph{conditional expectation} of $\varphi\in L^1(X,\mu)$ as the following $L^1$-function 
$$\E_\mu[\varphi\,|\,\cF_\xi]\colon x \mapsto\int_{\xi(x)}\varphi\,d\mu^\xi_x.$$

Note that if $(\xi_n)_n$ is an increasing sequence of partitions (in the sense that $\xi_n\prec\xi_{n+1}$) then we have $\cF_{\xi_n}\dans\cF_{\xi_{n+1}}$.

The following result is a consequence of the increasing martingale theorem (for which we refer to \cite[Theorem 5.5, p.126]{EinsWar}).

\begin{theorem}[Increasing martingale theorem]\label{th.martingette}
	Let $(\xi_n)_{n\in\N}$ be an increasing sequence of measurable partitions of $X$ such that $\bigvee_{n=0}^\infty\xi_n$ is the partition into points.
	Then, for every $\varphi\in L^1(X,\mu)$ and $\mu$-almost every $x\in X$, we have
	$$\lim_{n\to\infty}\E_\mu\big[\varphi\,\big|\,\cF_{\xi_n}\big](x)=\varphi(x).$$
\end{theorem}

%\subsection{Superposition property of subordinate partitions}

%Let $\xi$ be a measurable partition subordinate to $\cW_f^s$. We let $\xi^s_n$ denote $f^n\xi^s$. 

\subsection{Proof of the zero-one law}\label{ss_proof_01}

We are now ready to prove Theorem \ref{th_0-1}.  
	We fix $\xi^s$, a measurable partition subordinate to $\cW^s$. For $n\in\N$, set $\xi^s_n\eqdef f^n\xi^s$. Systems of conditional measures relative to $\xi^s$ and $\xi^s_n$ are denoted respectively by $\{\mu^s_x\}_x$ and $\{\mu^s_{n,x}\}_x$. For $\mu$-almost every $x\in \TT$ and every $n\in\N$, set
	\begin{itemize}
		\item $\cP^{\xi^s}(x)\eqdef\left\{y\in\xi^s(x):\alpha^s(x,y)=0\right\}$,
		\item $\cP^{\xi^s}_n(x)\eqdef\left\{y\in\xi_n^s(x):\alpha^s(x,y)=0\right\}$.
	\end{itemize}
	By definition of $\xi^s_n$ and invariance of the unstable bundle and of the stable foliation, we have the following commutation relation for $\mu$-almost every $x\in \TT$:
	\begin{equation}\label{eq_commutation_relation_px}
	f^n\cP^{\xi^s}\left(f^{-n}(x)\right)=\cP^{\xi^s}_n(x).
	\end{equation}

	Let $A\eqdef\{x\in \TT:\mu^s_x[\cP^{\xi^s}(x)]>0\}$. This is a Borel set. We must prove the following dichotomy:
	\begin{enumerate}
		\item\label{premi item} either $\mu(A)=1$, and for $\mu$-almost every $x\in \TT$, $\mu^s_x[\cP^{\xi^s}(x)]=1$;
		\item or $\mu(A)=0$.
	\end{enumerate}
	
	So let us suppose $\mu(A)>0$ and prove \eqref{premi item}. With that goal in mind, we claim that $A$ is an $f$-invariant set (mod 0). Indeed, take a $\mu$ generic point $x\in A$. Notice that $\cP^{\xi^s}_1(x)\subset\cP^{\xi^s}(x)$, since $\xi^s_1(x)\subset\xi^s(x)$. Also, recall from Lemma~\ref{l_inv_con_meas} that 
	\[
	\mu^s_{1,f(x)}=f_*\mu^s_x.
	\]
	Combining this with \eqref{eq_commutation_relation_px} one obtains that 
	\[
	\mu^s_{1,f(x)}(\cP^{\xi^s}_1(f(x)))=\mu^s_x(f^{-1}(\cP^{\xi^s}_1(f(x))))=\mu^s_x(\cP^{\xi^s}(x))>0.	
	\]
	Thus, from the superposition property (Lemma~\ref{l.superposition} and Remark~\ref{remark.superposition}) one gets
	\[
	\mu^s_{f(x)}(\cP^{\xi^s}(f(x)))\geq\mu^s_{f(x)}(\cP^{\xi^s}_1(f(x)))=\mu^s_{1,f(x)}(\cP^{\xi^s}_1(f(x)))\mu^s_{f(x)}(\xi^s_1(f(x)))>0,
	\]
	which implies that $f(x)\in A$, proving our claim.
	
	From the ergodicity of $\mu$, we have $\mu(A)=1$. We define functions in $L^1(\TT,\mu)$ by
	\begin{itemize}
		\item $\phi\colon x \mapsto \mu^s_x[\cP^{\xi^s}(x)]$;
		\item $\phi_{n}\colon x \mapsto \mu^s_{n,x}[\cP^{\xi^s}_n(x)]=
		\mu^s_{n,x}[\cP^{\xi^s}(x)]$, for each $n \in \N$. 
	\end{itemize}

For $\mu$-almost every $x\in\TT$ and every $n\in\N$ we consider
	
	\begin{itemize}
		\item $\cF^s_{x}=\cF^s_{\xi^s(x)}=\{\emptyset,\xi^s(x)\}$, the trivial $\sigma$-algebra over $\xi^s(x)$;
		\item $\cF^s_{n,x}=\cF^s_{n,\xi^s(x)}$, the $\sigma$-algebra generated by unions of atoms $\xi_n^s(y)$, $y \in \xi^s(x)$ (note that $\cF^s_{n,x}\dans\cF^s_{n+1,x}$);
		\item $\cF^s_{\infty,x}=\cF^s_{\infty,\xi^s(x)}$, the smallest $\sigma$-algebra containing $\bigcup_{n=0}^\infty\cF^s_{n,x}$. This is the Borel $\sigma$-algebra of $\xi^s(x)$.
	\end{itemize}

	For $\mu$-almost every $x  \in \T^3$ and $n \in \N$, we define the function $\psi_{n,x}\in L^1(\xi^s(x),\mu_x^s)$ by the following formula, for $\mu^s_x$-almost every $y\in\xi^s(x)$:
	\begin{equation*} 
	\psi_{n,x}(y)=\E_{\mu_x^s}\big[\mathbf{1}_{\cP^{\xi^s}(x)}\,\big|\,\cF^s_{n,x}\big](y)=\mu^s_{n,y}\big[\cP^{\xi^s}(x)\big].
	\end{equation*}
	Note that $\psi_{n,x}(y)=\phi_n(y)$ for all $n$ and $\mu^s_x$-almost every $y\in\cP^{\xi^s}(x)$ (note that in that case $\cP^{\xi^s}(x)=\cP^{\xi^s}(y)$).

	On the one hand, for $\mu$-almost every $x \in \T^3$, by the increasing martingale theorem (we apply Theorem \ref{th.martingette} to the probability space $(\xi^s(x),\mu_x^s)$), $\psi_{n,x}$ converges to $\mathbf{1}_{\cP^{\xi^s}(x)}$ $\mu^s_{x}$-almost surely as $n \to +\infty$. In particular,  
	\begin{equation}\label{psi n x y}
	\text{for }\mu^s_{x}\text{-a.e. } y \in \cP^{\xi^s}(x),\quad \psi_{n,x}(y)=\phi_{n}(y)\to 1\text{ as }n \to +\infty.
	\end{equation} 
	Let us define the set 
	$$
	\mathcal{S}\eqdef\{x \in \T^3:\phi_n(x)\to 1\text{ as }n\to +\infty\}.
	$$ 
	%Note that $\mathcal{S}=\cap_{m \geq 1}\cup_{n_0 \in \N}\cap_{n\geq n_0}\phi_n^{-1}\big(\big[1-\frac 1m,1\big]\big)$; moreover, f
	For each $n \in \N$, $\phi_n$ is measurable, hence $\mathcal{S}$  %=\cap_{m \geq 1}\cup_{n_0 \in \N}\cap_{n\geq n_0}\phi_n^{-1}\big(\big[1-\frac 1m,1\big]\big)$  
	is a Borel set. Assume that  $\mu[\mathcal{S}]=0$. Since $\mu[\mathcal{S}]=\int \mu_x^s[\mathcal{S}] d\mu(x)$, we would then have $\mu_x^s[\mathcal{S}]=0$, for $\mu$-a.e. $x \in \T^3$. But \eqref{psi n x y} implies that for $\mu$-a.e. $x \in A$,  $\mu_x^s[\mathcal{S}]\geq \mu_x^s[\mathcal{S}\cap\cP^{\xi^s}(x)]=\mu_x^s[\cP^{\xi^s}(x)]>0$, and by our assumption that $\mu[A]>0$, %the function $x \mapsto \mu_x^s[\mathcal{S}]$ is positive on a positive measure set, 
	we reach a contradiction. Therefore, %it holds 
	\begin{equation*}%\label{mesure s}
	\mu[\mathcal{S}]=\mu\{x \in \T^3:\phi_n(x)\to 1\text{ as }n\to +\infty\}>0.
	\end{equation*} 
	
	On the other hand, we deduce from \eqref{eq_commutation_relation_px} and from Lemma \ref{l_inv_con_meas} that
	$$\phi_n(x)=\mu^s_{n,x}[\cP^{\xi^s}_n(x)]=f^n_\ast\mu^s_{f^{-n}(x)}[f^n\cP(f^{-n}(x))]=\phi(f^{-n}(x)).$$
	
	The latter proves that $\phi\circ f^{-k}$ converges to $1$ $\mu$-almost surely on $\mathcal{S}$ as $k \to +\infty$. Therefore, by considering Ces\`aro averages $\frac 1n\sum_{k=0}^{n-1} \phi \circ f^{-k}(x)$ for a $\mu$-generic point $x \in  \mathcal{S}$, and by Birkhoff's theorem, we conclude that $\int_\TT\phi\, d\mu=1$. As $\phi$ takes values in $[0,1]$, we must have $\phi(x)=1$, for $\mu$-almost every $x\in \TT$.
\qed

\begin{remark}\label{remark general zero un}
It is clear that $\alpha^s(x,y)=0$ is an equivalence relation on stable leaves, so 
Theorem \ref{th_0-1} can be generalized as follows. Let $f$ be as in Theorem \ref{th_0-1}. Let $\mathcal{R}$ be a measurable equivalence relation on stable leaves (i.e., such that $x\mathcal{R}y\,\Rightarrow\, y \in \cWs(x)$) such that $x\mathcal{R}y\,\Rightarrow\, f^{n}(x)\mathcal{R}f^{n}(y)$, for any $n \in \N$. Fix a measurable partition $\xi^s$ subordinate to $\cW^s$, a system $\{\mu^s_x\}_x$ of conditional measures relative to $\xi^s$, and for $x \in \TT$, let $\cP^{\xi^s}(x)$ be its $\mathcal{R}$-equivalence class. Then the following dichotomy holds:
\begin{enumerate}
	\item either for $\mu$-almost every $x\in M$, 
	$\mu^s_x[\cP^{\xi^s}(x)]=1$;
	
	\item or for $\mu$-almost every $x\in M$,
	$\mu^s_x[\cP^{\xi^s}(x)]=0$.
\end{enumerate} 
\end{remark}

\section{Joint integrability and the \emph{Bad set}}\label{s.main_thm}

We start in this section the formal proof of Theorem~\ref{mainthm.dicotomia}. So we let $f\in\cA^2(\TT)$ be an Anosov diffeomorphism, strongly partially hyperbolic with expanding center and $C^1$ stable bundle. Recall from Lemma~\ref{c one holon} that this is always satisfied when $f$ is close to a conservative map $f_0$. Let $\mu$ denote an ergodic $u$-Gibbs measure, with full support.

We are going to reduce the proof of Theorem~\ref{mainthm.dicotomia} to a more technical version of the result by studying the set of points $x$ for which one sees almost no twist of the bundle $E^u$ by the application of the stable holonomy. Our main technical result says that if the measure of this \emph{Bad set} is zero then $\mu$ is SRB.

\subsection{The \emph{Bad set}}\label{ss.badset}

%We can formalize in a quantitative way the idea of the stable holonomy map twisting the unstable bundle.
%Let us introduce the following relation on $\TT$.
%\begin{align*}
%x \sim y \quad&\Leftrightarrow\quad y \in \cWs(x)\text{ and } \alpha^s(x,y)= 0,\\
%&\Leftrightarrow\quad y \in \cWs(x)\text{ and } DH_{x,y}^s E^u(x)=E^u(y). 
%\end{align*} 

For any $x \in \TT$, we denote 
$$
\cP(x)\eqdef\Big\{y \in \cWs(x):\alpha^s(x,y)=0\Big\},\quad \mathcal{N}(x)\eqdef \Big\{y\in \cWs(x):\alpha^s(x,y)>0 \Big\}.
$$ 

Observe that given any measurable partition $\xi^s$ subordinate to $\cW^s$ we have that for $\mu$-a.e. $x\in\TT$, $\cP^{\xi^s}(x)=\cP(x)\cap\xi^s(x)$, where $\cP^{\xi^s}(x)$ was defined in the proof of Theorem~\ref{th_0-1}, our zero-one law: see \S \ref{ss_proof_01}.

\begin{remark}\label{claim equiv}
	Since the unstable bundle $E^u$ is invariant under $Df$, and holonomy maps are equivariant with the dynamics, i.e.,
	\begin{equation}
	\label{e.equivariancia}
	f\circ H^s_{x,y}=H^s_{f(x),f(y)}\circ f,
	\end{equation}  
	we have $\forall\, x,y,$ 
%	 
%The following property also follows easily from \eqref{e.equivariancia}:
	\begin{equation}
	\alpha^s(x,y)=0\quad \Longleftrightarrow\quad \forall\, n \in \Z,\ \alpha^s(f^n(x),f^n(y))=0. 
	\end{equation}
	%Notice that in the above equation we have used our notation for orbits \eqref{e.orbita}.
\end{remark}

For any $x \in \TT$, the function $\alpha^s(x,\cdot)$ on $\mathcal{W}^s(x)$ is continuous. Therefore, the sets $\cP(x)$ and $\mathcal{N}(x)$ are respectively closed and open. 
Moreover, by Remark \ref{claim equiv}, it holds 
\begin{equation}\label{eq invar}
f^n(\cP(x))=\cP\big(f^n(x)\big),\quad f^n(\mathcal{N}(x))=\mathcal{N}(f^n(x)),\quad \forall\, n \in \Z.
\end{equation}

\begin{defi}[Bad set]\label{def bad set}	Let $\xi^s$ be a measurable partition subordinate to the stable foliation $\cW^s$, and let $\{\mu_x^s\}_{x \in \TT}$ be a system of conditional measures relative to $\xi^s$. The \emph{Bad set} $\mathbf{B}=\mathbf{B}(\xi^s,\mu)$ is defined as
$$
\mathbf{B}\eqdef\big\{x\in\TT:\mu^s_x(\mathcal{N}(x))=0\big\}.
$$
\end{defi}

Using our zero-one law, and the results established in \S \ref{ss_proof_01}, we have the following properties.

\begin{enumerate}
\item $\mathbf{B}$ is equal to the set $A$ introduced in \S \ref{ss_proof_01};
\item $\mathbf{B}$ is measurable and $f$-invariant, and it satisfies that for all $\ell\in\Z$
\begin{equation}\label{eq_changing_partition}
\mathbf{B}(\xi^s,\mu)=\mathbf{B}(\xi^s_\ell,\mu),
\end{equation}
where we recall $\xi^s_\ell=f^\ell(\xi^s)$ is still a measurable partition subordinate to $\cW^s$;
\item $\mathbf{B}$ has measure $0$ or $1$.
\end{enumerate}

A priori the \emph{Bad set} $\mathbf{B}=\mathbf{B}(\xi^s,\mu)$ depends on the particular choice of the measurable partition subordinate to $\cW^s$. Nevertheless, the next lemma shows that the Bad set associated to another subordinate partition is equal to $\mathbf{B}$ modulo sets of measure $0$.

\begin{lemma}
\label{l.badsetinvariante}
For every measurable partition $\eta^s$ subordinate to $\cW^s$ we have 
\[
\mu\left(\mathbf{B}(\xi^s,\mu)\right)=\mu\left(\mathbf{B}(\eta^s,\mu)\right)\in\{0,1\}.
\]
\end{lemma}	
\begin{proof}
For $\ell\in\Z$, let us denote $\xi^s_{\ell}\eqdef f^\ell(\xi^s)$. Then, all the partitions $\xi^s_{\ell}$ are measurable partitions subordinate to $\cW^s$ and by \eqref{eq_changing_partition} we have that for all $\ell\in\Z$, $\mathbf{B}(\xi^s,\mu)=\mathbf{B}(\xi^s_\ell,\mu)$. These sets have measure $0$ or $1$; assume they have measure $1$. Consider a $\mu$-generic point $x$ so that
\begin{itemize}
\item $\eta^s(x)$ contains an open neighbourhood of $x$ in $\cW^s(x)$ so $\mu^{\xi^s_m}_x(\eta^s(x))>0$ for all $m$ (see Corollary \ref{coro.superposition});
\item $\eta^s(x)$ has diameter less than $1$ and;
\item $\mu_x^{\xi^s_m}(\cN(x))=0$.
\end{itemize}

Using Lemma \ref{lem.af1} (specifically Remark \ref{rem.af1}) there exists $m<0$ such that $\eta^s(x)\dans\xi^{s}_m(x)$ so we can apply the superposition property and
\[
\mu^{\eta^s}_x(\cN(x))=\frac{\mu^{\xi^s_m}_{x}(\cN(x)\cap\eta^s(x))}{\mu^{\xi^s_m}_{x}(\eta^s(x))}\leq\frac{\mu^{\xi^s_m}_{x}(\cN(x))}{\mu^{\xi^s_m}_{x}(\eta^s(x))}=0.
\]
As a result, $x\in \mathbf{B}\left(\eta^s,\mu\right)$. This proves that $\mu\left( \mathbf{B}(\eta^s,\mu)\right)=1$, confirming the second assertion and concluding.
\end{proof}

Although we always need to fix a partition to speak about the bad set, Lemma~\ref{l.badsetinvariante} allows us to speak about \emph{the measure of the Bad set} without fixing a particular choice of a partition. 

\begin{remark}\label{rem_other_partitions}
The conclusion of Lemma \ref{l.badsetinvariante} remains valid even if the measurable partition $\eta^s$ is only required to have the following properties for $\mu$-a.e. $x\in\TT$
\begin{enumerate}
\item $\eta^s(x)\dans\cW^s(x)$;
\item $\eta^s(x)$ contains an open neighbourhood of $\cW^s(x)$ in the internal topology.
\item $\eta^s(x)$ has uniformly bounded diameter;
\end{enumerate}
Indeed the proof of Lemma \ref{l.badsetinvariante} does not use that $\eta^s$ is decreasing. 

There are many (non decreasing) measurable partitions satisfying these three properties. For example, consider any finite foliated atlas $\cA$ for $\cW^s$ composed of open charts $U_i$ such that $\mu(\partial U_i)=0$. Define a finite partition mod $0$ of $\TT$ as
$$\mathcal{Q}=\bigvee_i\left\{U_i,\TT\setminus \overline{U_i}\right\},$$
Each atom $\mathcal{Q}(x)$ is an open set included inside a foliated chart for $\cW^s$. As a result it is trivially foliated by stable plaques: this defines a measurable partition of $\mathcal{Q}(x)$. We obtain a refinment of $\mathcal{Q}$, denoted by $\eta^s$, that satisfies the desired properties, and is referred to as the \emph{measurable partition associated to} $\cA$.
\end{remark}

\subsection{Joint integrability}

When the Bad set has positive measure, by ergodicity we have that for almost every $x$, we see almost no twist of $E^u$ along $\cWs(x)$. This can be read as an infinitesimal form of joint integrability between $E^s$ and $E^u$. In this paragraph we shall improve this to actual joint integrability.

\begin{prop}
	\label{p.mesuredubadset}
Let $f\colon\mathbb{T}^3\to\mathbb{T}^3$ be a $C^2$ Anosov diffeomorphism, strongly partially hyperbolic with expanding center and $C^1$ stable holonomies. Let $\mu$ be a fully supported ergodic $f$-invariant measure. If the \emph{Bad set} has full measure  then $E^s$ and $E^u$ are jointly integrable.	
\end{prop} 

It is worth to point out that the proof of this proposition is the only place in our argument towards Theorem~\ref{mainthm.dicotomia} where the full support assumption of the $u$-Gibbs measure is used. Moreover, in Proposition~\ref{p.mesuredubadset} the ergodic invariant measure $\mu$ does not need to be $u$-Gibbs.

\subsubsection{Local joint integrability}

In the proof of Proposition~\ref{p.mesuredubadset} we shall apply a criterion for joint integrability that comes from \cite{didier2003stability}. See also \cite[\S 2.3]{HHUBook}.

\begin{defi}
Given $x\in\TT$ we say that the bundles $E^s$ and $E^u$ are \emph{jointly integrable at $x$} if there exists $\delta,\eps>0$ such that for each $z\in\cWs_\delta(x)$ and $y\in\cWu_\delta(x)$ it holds
\[
\cWu_\eps(z)\cap\cWs(y)\neq\emptyset.
\]
\end{defi}

The result below follows directly from \cite[Lemma 5]{didier2003stability} (see also \cite[Lemma 2.3.7]{HHUBook}). 

\begin{lemma}
	\label{l.didier}
Let $f\in\cA^2(\TT)$. Assume that $E^s$ and $E^u$ are jointly integrable at each $x\in\TT$, with uniform constants $\delta,\eps$. Then, $f$ is jointly integrable (in the sense of Definition~\ref{def.joint}).
\end{lemma}

\subsubsection{Proof of Proposition~\ref{p.mesuredubadset}}

This Proposition follows directly from Lemma~\ref{lemm dicht} and Lemma~\ref{lemme Bruno} below. 

On the topological level, we have:
\begin{lemma}\label{lemm dicht}
Let $f\colon\mathbb{T}^3\to\mathbb{T}^3$ be a $C^2$ Anosov diffeomorphism, strongly partially hyperbolic with expanding center and $C^1$ stable holonomies.

	Then the following dichotomy holds:
	\begin{itemize}
		\item either $E^s$ and $E^u$ are jointly integrable;
		\item or for any $x \in \TT$, the set $\mathcal{N}(x)$ is open and dense in $ \cWs(x)$. 
	\end{itemize}
\end{lemma}

\begin{proof}	
	Let us assume that there exists $x_0 \in\T^3$ such that $\cP(x_0)$  contains a non-trivial open interval, i.e., that for some $x \in \mathcal{W}^s(x_0)$, and $\varepsilon>0$, it holds $\mathcal{W}_\varepsilon^s(x)\subset \cP(x_0)$. 
	Since $\cP(x)=\cP(x_0)$, we also have $\cP(x)\supset \mathcal{W}_\varepsilon^s(x)$, hence by \eqref{eq invar}, for any integer $n \geq 1$, % it holds 
	\begin{equation}\label{dynamique passeee}
	\cP\big(f^{-n}(x)\big)\supset f^{-n}\big(\mathcal{W}_\varepsilon^s(x)\big).
	\end{equation} 
	By compactness, we can take a subsequence $(f^{-n_k}(x))_{k \geq 0}$ such that $f^{-n_k}(x) \to y$ as $k \to +\infty$, for some point $y \in \T^3$. Since the restriction of $f^{-1}$ to stable leaves is uniformly expanding, we deduce from \eqref{dynamique passeee} that 
	\begin{equation}\label{dynamique passeee hehe}
	\cP(y)=\mathcal{W}^s(y).  
	\end{equation} 
	
	Let us now show that the same holds for any point, i.e., that $\cP(z)=\mathcal{W}^s(z)$, for any $z \in \T^3$. 
	Fix $z \in \T^3$ and $z' \in \mathcal{W}^s(z)$. By minimality of the stable foliation, the leaf $\mathcal{W}^s(y)$ is dense in $\T^3$, hence there exists a sequence $(y_n)_{n \geq 0}\in (\mathcal{W}^s(y))^{\N}$ such that $\lim_{n \to +\infty} y_n=z$. For each integer $n \geq 0$, we let $y_n'\eqdef H_{z,z'}^s(y_n)\in \mathcal{W}^{cu}(z')\cap \mathcal{W}^s(y_n)$. In particular, we also have $\lim_{n \to +\infty} y_n'=H_{z,z'}^s(\lim_{n} y_n)=z'$. By continuity of the angle function $\alpha^s$, we deduce that 
	$$
	\alpha^s(z,z')=\lim_{n \to +\infty}\alpha^s(y_n,y_n')=0,
	$$
	i.e., $z' \in \mathcal{P}(z)$.  Since $z' \in \mathcal{W}^s(z)$ was chosen arbitrarily within $\mathcal{W}^s(z)$, we deduce that $\mathcal{W}^s(z)=\mathcal{P}(z)$, for all $z \in \T^3$.

	Now, fix arbitrarily $x \in \T^3$. Let $0<\delta<\eps$ be chosen so that the stable holonomy map $H^s_{x,z}\colon\cWcu_\delta(x)\to\cWcu_\eps(z)$ is well defined for every $z\in\cWs_\delta(x)$.
	Notice that for any $y \in  \mathcal{W}_\delta^u(x)$, for any $z \in \mathcal{W}_\delta^s(x)$, if we set $y'\eqdef H_{x,z}^s(y)\in \mathcal{W}_\eps^s(y)$, we have 
	$$
	\angle \big(DH_{x,z}^s E^u(y),E^u(y')\big)=\alpha^s(y,y')=0. 
	$$
	In particular, $H_{x,z}^s\big(\mathcal{W}_\delta^u(x)\big)$ is a $C^1$ curve that is everywhere tangent to $E^u$ and at {\color{blue} $y'$} it is tangent to $E^u(y')$. By the unique integrability of the $E^u$ bundle, we conclude that 	
	
	% $\mathcal{W}^u(v')$, and hence, 
	$$
	H_{x,z}^s\big(\mathcal{W}_{\mathrm{loc}}^u(x)\big) \subset \mathcal{W}^u(y'). 
	$$
	
Hence, we obtain that for any $x\in \T^3$,  any $z\in \cW^s_{\delta}(x)$ and $y\in \cW^u_{\delta}(x)$ we have
\[
\cW^u_{\eps}(z) \cap \cW^s_{\eps}(y) \neq \emptyset. 
\]
This proves that $f$ fulfils the assumption of Lemma~\ref{l.didier}. Therefore, $f$ is jointly integrable.\qedhere

\end{proof}

%Let us fix an ergodic $u$-Gibbs measure $\mu$ for $f$. By definition, the support $\mathrm{supp}(\mu)$ is $u$-saturated, i.e., 
%\begin{equation}\label{saturraaaaion}
%x \in \mathrm{supp}(\mu)\ \implies\ \cW^u(x)\subset \mathrm{supp}(\mu).
%\end{equation}

%\begin{defi}[\textit{Bad} set]
%	We let $\mathbf{B}\subset \T^3$ be the set of all points $x \in \T^3$ such that the conditional measure $\mu_x^{s}$ gives zero measure to $\mathcal{N}_x$, i.e., for any $y \in \mathcal{W}_f^s(x)$ and  for any foliation box $\mathcal{B}\ni y$ for $\mathcal{W}_f^s$, we have  $\mu_{\mathcal{B},y}^s(\mathcal{N}_x)=0$. By a slight abuse of notation, we write this as $\mu_x^{s}(\mathcal{N}_x)=0$. 
%\end{defi}

%\begin{remark}
%	As the measure $\mu$ is $f$-invariant, it follows from \eqref{inv desting} and \eqref{eq invar} that for any $x \in\mathbf{B}$, and for any integer $n \in \Z$,
%	$$
%	\mu_{f^{n}(x)}^{s}(\mathcal{N}_{f^n(x)})=\mu_{x}^{s}(\mathcal{N}_{x})=0,
%	$$
%	hence $\mathbf{B}$ is $f$-invariant. By the ergodicity of $\mu$, we deduce that $\mu(\mathbf{B})=0$ or $1$. 
%\end{remark}

The lemma below concludes the proof of Proposition~\ref{p.mesuredubadset}.

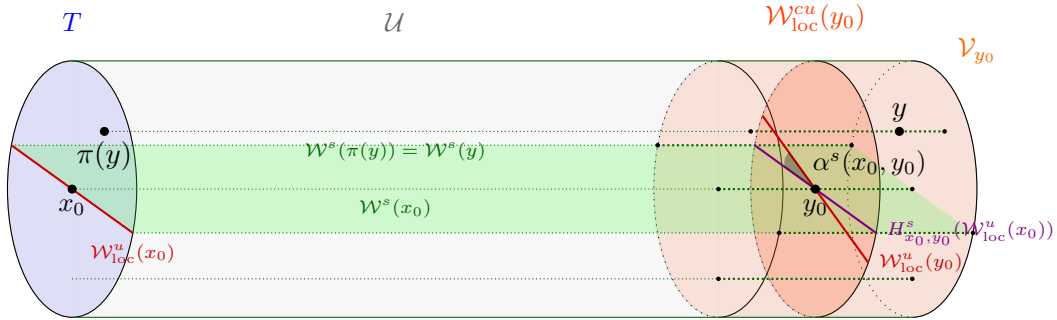
\begin{figure}[h]
	\centering
	\begin{tikzpicture}[scale=.85] 
	\fill[black!30!white, opacity=.1] (-10,-2)--(3,-2)--(3,2)--(-10,2)--(-10,-2);
	\draw[fill=black!30!white, opacity=.1] ($(-10, 0) + (90:1cm and 2cm)$(P) arc
	(90:270:1cm and 2cm);
	\draw[fill=black!30!white, opacity=.1] ($(3, 0) + (-90:1cm and 2cm)$(S) arc
	(-90:90:1cm and 2cm);
	\draw[fill=blue,opacity=.1] (-10,0) ellipse (1cm and 2cm);
	\draw (-10,0) ellipse (1cm and 2cm);
	%\draw (0,0) ellipse (1cm and 2cm);
	\coordinate (P) at ($(0, 0) + (90:1cm and 2cm)$);
	\draw[fill=red!50!orange,opacity=.12] ($(0, 0) + (90:1cm and 2cm)$(P) arc
	(90:270:1cm and 2cm);
	\draw[dotted] ($(0, 0) + (90:1cm and 2cm)$(P) arc
	(90:270:1cm and 2cm);
	\coordinate (Q) at ($(0, 0) + (270:1cm and 2cm)$);
	\draw ($(0, 0) + (270:1cm and 2cm)$(Q) arc
	(-90:90:1cm and 2cm);
	\coordinate (S) at ($(3, 0) + (90:1cm and 2cm)$);
	\draw[dotted] ($(3, 0) + (90:1cm and 2cm)$(S) arc
	(90:270:1cm and 2cm);
	\coordinate (R) at ($(3, 0) + (270:1cm and 2cm)$);
	\draw[fill=red!50!orange,opacity=.12] ($(3, 0) + (270:1cm and 2cm)$(Q) arc
	(-90:90:1cm and 2cm);
	\draw ($(3, 0) + (270:1cm and 2cm)$(Q) arc
	(-90:90:1cm and 2cm);
	\coordinate (Z) at ($(1.5, 0) + (90:1cm and 2cm)$);
	\draw[fill=red!50!orange,opacity=.2] ($(1.5, 0) + (90:1cm and 2cm)$(S) arc
	(90:270:1cm and 2cm);
	\draw[dotted] ($(1.5, 0) + (90:1cm and 2cm)$(S) arc
	(90:270:1cm and 2cm);
	\coordinate (Y) at ($(1.5, 0) + (270:1cm and 2cm)$);
	\draw[fill=red!50!orange,opacity=.2] ($(1.5, 0) + (270:1cm and 2cm)$(Q) arc
	(-90:90:1cm and 2cm);
	\draw ($(1.5, 0) + (270:1cm and 2cm)$(Q) arc
	(-90:90:1cm and 2cm);
	\fill[fill=red!50!orange,opacity=.12] (0,2)--(3,2)--(3,-2)--(0,-2)--(0,2);
	\draw[green!40!black] (-10,2)--(3,2);
	\draw[green!40!black] (-10,-2)--(3,-2);
	\draw[densely dotted, green!40!black] (-10,0)--(0,0);
	\fill[green!80!white, opacity=.2] ($(-10, 0) + (160:1cm and 2cm)$)--($(-10, 0) + (340:1cm and 2cm)$)--($(3, 0) + (340:1cm and 2cm)$)--($(3, 0) + (160:1cm and 2cm)$)--($(-10, 0) + (160:1cm and 2cm)$);
	\draw[densely dotted, green!40!black] ($(-10, 0) + (340:1cm and 2cm)$)--($(0, 0) + (340:1cm and 2cm)$);
	\draw[densely dotted, thick, green!40!black] ($(0, 0) + (340:1cm and 2cm)$)--($(3, 0) + (340:1cm and 2cm)$);
	\draw[densely dotted, green!40!black] ($(-10, 0) + (160:1cm and 2cm)$)--($(0, 0) + (160:1cm and 2cm)$);
	\draw[densely dotted, thick, green!40!black] ($(0, 0) + (160:1cm and 2cm)$)--($(3, 0) + (160:1cm and 2cm)$);
	\draw[densely dotted, thick, green!40!black] (0,0)--(3,0); 
	\draw[densely dotted, green!40!black] (-10,-1.4)--(0,-1.4);
	\draw[densely dotted, thick, green!40!black] (0,-1.4)--(3,-1.4);
	\draw[densely dotted, green!40!black] (-9.5,0.9)--(0.5,0.9);
	\draw[densely dotted, thick, green!40!black] (0.5,0.9)--(3.5,0.9);
	\draw[thick, red!80!black] ($(-10, 0) + (160:1cm and 2cm)$)--($(-10, 0) + (340:1cm and 2cm)$) node[below]{\tiny $\mathcal{W}_{\mathrm{loc}}^u(x_0)$};
	\draw[thick, violet] ($(1.5, 0) + (160:1cm and 2cm)$)--($(1.5, 0) + (340:1cm and 2cm)$)  node[right]{\tiny $H^s_{x_0,y_0}(\mathcal{W}_{\mathrm{loc}}^u(x_0))$};
	\draw[thick, red!80!black] ($(1.5, 0) + (145:1cm and 2cm)$)--($(1.5, 0) + (325:1cm and 2cm)$) node[right]{\tiny $\mathcal{W}_{\mathrm{loc}}^u(y_0)$};
	\fill[black, opacity=0.3] (1.5,0) -- ($(1.5, 0) + (160:0.5cm and 1cm)$) to[bend left] ($(1.5, 0) + (145:0.5cm and 1cm)$) -- (1.5,0);
	\fill[fill=black] (-10,0) node[below]{\small $x_0$} circle (2pt); 
	\fill[fill=black] (0,0) node[below]{} circle (1pt); 
	\fill[fill=black] (3,0) node[below]{} circle (1pt); 
	\fill[fill=black] (-9.5,0.9) node[below]{$\pi(y)$} circle (2pt); 
	\fill[fill=black] (2.8,0.9) node[above]{$y$} circle (2pt); 
	\fill[fill=black] (0,-1.4) node[below]{} circle (1pt);
	\fill[fill=black] (3,-1.4) node[below]{} circle (1pt); 
	\fill[fill=black] (0.5,0.9) node[below]{} circle (1pt);
	\fill[fill=black] (3.5,0.9) node[below]{} circle (1pt);
	\fill[fill=black] ($(0, 0) + (160:1cm and 2cm)$) node[below]{} circle (1pt); 
	\fill[fill=black] ($(3, 0) + (160:1cm and 2cm)$) node[below]{} circle (1pt); 
	\fill[fill=black] ($(0, 0) + (340:1cm and 2cm)$) node[below]{} circle (1pt); 
	\fill[fill=black] ($(3, 0) + (340:1cm and 2cm)$) node[below]{} circle (1pt); 
	\fill[fill=black] (-5,0) node[below]{\tiny \color{green!40!black} $\mathcal{W}^s(x_0)$}; 
	\fill[fill=black] (-5,2.3) node[above]{\small \color{white!40!black} $\mathcal{U}$}; 
	\fill[fill=black] (-5,0.9) node[below]{\tiny \color{green!40!black} $\mathcal{W}^s(\pi(y))=\mathcal{W}^s(y)$}; 
	\fill[fill=black] (-10,2.3) node[above]{\small \color{blue}$T$};
	\fill[fill=black] (4,1.8) node[above]{\small \color{orange!80!red}$\mathcal{V}_{y_0}$};
	\fill[fill=black] (1.5,2.3) node[above]{\small \color{orange!50!red}$\mathcal{W}^{cu}_{\mathrm{loc}}(y_0)$};
	\fill[fill=black] (1.5,0) node[below]{\small $y_0$} circle (2pt); 
	\fill[fill=black] (1.3,0.4) node[right]{\small $\alpha^s(x_0,y_0)$}; 
	\end{tikzpicture}
	\caption{\label{dich angle supp} Case where $E^s\oplus E^u$ is not integrable and $\nu(\mathbf{B}_\nu)=1$.}
\end{figure}

\begin{lemma}\label{lemme Bruno}
Let $f\colon\mathbb{T}^3\to\mathbb{T}^3$ be a $C^2$ Anosov diffeomorphism, strongly partially hyperbolic with expanding center and $C^1$ stable holonomies.
Assume that $E^s$ and $E^u$ are not jointly integrable. Then, for any ergodic 
	$f$-invariant Borel probability measure $\nu$ on $\TT$, we have the  dichotomy: 
	\begin{itemize}
		\item either the \emph{Bad set} associated to $\nu$ has zero measure;
		\item or $\mathrm{supp}(\nu)$ has empty interior. 
	\end{itemize}
	%If $\mu(\mathbf{B})=1$, then $E_f^s$ and $E_f^u$ are jointly integrable. 
\end{lemma}

\begin{proof}
	Fix $\nu$ as in the statement of the lemma, and let $x_0 \in \T^3$. Assume the first possibility does not occur. Then, the \emph{Bad set} has full measure. As we assume that $E^s$ and $E^u$ are not jointly integrable, by Lemma \ref{lemm dicht}, the set $\mathcal{N}(x_0)$ is open and dense in $\cWs(x_0)$. Fix $y_0 \in \mathcal{N}(x_0)$, i.e., such that $\alpha^s(x_0,y_0)\neq 0$. % Without loss of generality, we may assume that $y_0$ is close to $x_0$; indeed, for each $n \in \N$, $\alpha^s(f^n (x_0),f^n (y_0))\neq 0$, and $\lim_{n \to+\infty}d_{\cW^s}(f^n (x_0),f^n (y_0))= 0$. Moreover, b
	By the continuity of the angle function $\alpha^s$, there exist:
\begin{itemize}
	\item a foliated chart $\mathcal{U}$ for $\mathcal{W}^s$ containing $x_0,y_0$; stable plaques of $\mathcal{U}$ are denoted by $\{\mathcal{U}(x)\}_x$;
	\item a transversal $T$ for $\mathcal{W}^s$ at $x_0$;
	\item a neighbourhood $\mathcal{V}_{y_0}\subset \mathcal{U}$ of $y_0$;
	\item a projection $\pi\colon \mathcal{V}_{y_0} \ni y\to \pi(y)\in T$ along stable plaques of $\mathcal{U}$; 
\end{itemize}
such that 
$$
\forall\, y \in \mathcal{V}_{y_0},\quad \alpha^s(\pi(y),y)>0. 
$$
%$\sigma>0$ and a small neighbourhood $\mathcal{U}$ of $x_0$ such that for any $x \in \mathcal{U}$, if $y(x)\eqdef H_{x_0,y_0}^s(x)\in \mathcal{W}_{\mathrm{loc}}^{cu}(y_0)$ denotes the projection along stable leaves of $x$ on $\mathcal{W}_{\mathrm{loc}}^{cu}(y_0)$, then the set  $\mathcal{J}_x\eqdef\mathcal{W}_{\sigma}^s(y(x))$ satisfies $\mathcal{J}_x\subset \mathcal{N}(x)$. We set $\mathcal{V}_{y_0}\eqdef\cup_{x \in \mathcal{U}} \mathcal{J}_x$. Let us consider a small foliation box $\mathcal{B}$ for $\cWs$, with $(\mathcal{U}\cup\mathcal{V}_{y_0})\subset \mathcal{B}$, and l
The open set $\cU$ is obtained as a sufficiently small neighbourhood of a path inside $\cW^s$ from $x_0$ to $y_0$ (which must be trivially foliated by $\cW^s$ by Reeb's stability theorem). See Figure \ref{dich angle supp}. 

We claim that $\nu(\cV_{y_0})=0$. To see this, we consider a finite foliated atlas $\cA$ for $\cW^s$ such that
\begin{enumerate}
\item $\overline{\cU}$ is included inside a foliated chart $\cU'$ of $\cA$ with $\mu(\partial\cU')=0$;
\item every other foliated chart of $\cA$ is disjoint from $\mathcal{\cU}$, and its boundary has measure $0$.
\end{enumerate}
To obtain such an atlas, first observe that from our construction, $\cU$ can be obtained so that its closure lies inside an open set $\cU'$ trivially foliated by $\cW^s$. Next, we cover the compact set $\TT\setminus\cU'$ by finitely many open balls $B_i$ that trivialize $\cW^s$, have $\mu(\partial B_i)=0$ and are disjoint from $\cU$. Let $\eta$ be the measurable partition associated to $\cA$ (see Remark \ref{rem_other_partitions}). By construction, stable plaques $\cU(x)$ are contained inside atoms of $\eta$ and we have that for almost every $x\in T$, 
$$\eta(x)\cap\cV_{y_0}=\cU(x)\cap\cV_{y_0}\dans\cN(x),$$
hence $\nu_x^\eta(\eta(x)\cap\cV_{y_0})\leq\nu_x^\eta(\cN(x))=0$ for $\mu$-a.e. $x$. By the definition of conditional measures, this implies that $\nu(\cV_{y_0})=0$.
In other words, any point $y \in \mathcal{N}(x_0)$ has an open neighbourhood $\mathcal{V}_{y}\ni y$ such that $\T^3 \setminus \mathcal{V}_{y}$ has full measure
	therefore,
	$$
	\mathrm{supp}(\nu)\subset \T^3 \setminus \bigcup_{y\in \mathcal{N}(x_0)}\mathcal{V}_{y}.
	$$
	As $\mathcal{N}(x_0)$ is open and dense in $\cWs(x_0)$, and $\cWs(x_0)$ is dense in $\TT$, the set $\bigcup_{y\in \mathcal{N}(x_0)}\mathcal{V}_{y}$ is open and dense in $\TT$. Thus, $\mathrm{supp}(\nu)$ has empty interior. 
\end{proof}

\subsection{Main technical theorem}

We shall now show that our main result, Theorem~\ref{mainthm.dicotomia}, can be reduced to a more technical statement involving the \emph{Bad set}.

\begin{theorem}
	\label{mainthm.technique}
	Let $f\colon\mathbb{T}^3\to\mathbb{T}^3$ be a $C^2$ Anosov diffeomorphism, strongly partially hyperbolic %so that $f$ expands uniformly the central bundle.
	with expanding center and $C^1$ stable holonomies. Let $\mu$ be an ergodic $u$-Gibbs measure. If the \emph{Bad set} of $\mu$ has zero measure then $\mu$ is an SRB measure. 	
\end{theorem}

\subsubsection{Proof of Theorem~\ref{mainthm.dicotomia} assuming Theorem~\ref{mainthm.technique}}
Recall that by Lemma~\ref{c one holon} for every $f_0\in\cA^2_m(\TT)$ there exists a small neighbourhood $\cU(f_0)$ in $\dif$ so that every $f\in\cU(f_0)$ is an Anosov diffeomorphism, strongly partially hyperbolic with expanding center and $C^1$ stable holonomies. Let $\cU\eqdef\cup_{f_0\in\cA^2_m(\TT)}\cU(f_0)$. Take $f\in\cU$. Then $f$ fulfils the assumptions of Theorem~\ref{mainthm.technique}. Assume that $E^s$ and $E^u$ are not jointly integrable and let $\mu$ be a fully supported ergodic $u$-Gibbs measure. Then, Lemma~\ref{lemme Bruno} implies that the \emph{Bad set} has zero measure. By Theorem~\ref{mainthm.technique} it follows that $\mu$ is SRB, concluding.\qed

\bigskip

It is important to remark that in this main technical result we \emph{do not assume} that our $u$-Gibbs measure is fully supported. The result says that the existence of positive angles $\alpha^s(x,y)$ for almost every $y\in\cW^s(x)$ and almost every $x\in\TT$ suffices to convert $\mu$ into an SRB measure. 

%From now on, we let $f$ denote a map satisfying the assumptions of Theorem~\ref{mainthm.technique} and $\mu$ is an ergodic $u$-Gibbs measure whose Bad set has zero measure. Our goal is to show that $\mu$ is SRB.

\section{Normal forms}\label{section formes normales}

In this section, our goal is to prove Theorem~\ref{thm.normalformsintro}. This section is independent of the rest of the paper and we want to emphasize that we do not require any regularity of holonomies. We begin by revisiting the normal forms along one dimensional expanding foliations, as developped by Kalinin-Katok \cite{KalininKatok}. Then, using strongly that $E^u$ is $C^1$ along the center-unstable manifold, we extend this construction to two dimensions. In the final part of the section, we study the change of coordinates for points on the same leaf of $\cW^{cu}$.

Let $f\colon\T^3 \to \T^3$ be a $C^2$ Anosov diffeomorphism with a decomposition $T \T^3 = E^s \oplus E^c \oplus E^u$, where $E^c$ is uniformly expanded.  

Let us recall some notations that will be frequently used in this section.  
\begin{itemize}
%	\color{teal}
	\item Given a point $x\in \mathbb{T}^3$ we will write $\lambda^*_x \eqdef \|Df(x)|_{E^*}\|$, for $* = c, u$. Observe that $\|Df^{-1}(x)|_{E^*}\|=(\lambda^*_{x_{-1}})^{-1}$.
	
	\item Given $\ell\in \Z$ we write $x_\ell= f^\ell(x)$. 
	
	\item Given $n \in \Z$, and $*= c, u$, we write 
	\begin{equation}\label{recall equation cocycle}
		\lambda_x^*(n)\eqdef \|Df^n(x)|_{E^*}\|=\left\{
		\begin{array}{ll}
			\prod_{\ell=0}^{n-1}\lambda_{x_\ell}^*,&\text{ when }n \geq 0;\\
			\prod_{\ell=1}^{-n}\frac{1}{\lambda_{x_{-\ell}}^*}=\frac{1}{\lambda_{x_{n}}^*(-n)},&\text{ when }n < 0.
		\end{array}
		\right.
	\end{equation}
%	\color{black}
\end{itemize}

\subsection{One dimensional normal forms}

Given $x\in \TT$ and $y\in\cW^{cu}(x)$, we consider the functions 
\begin{equation}\label{eq.rhonormalforms}
	\rho_x^*(y)\eqdef\prod_{\ell=1}^{+\infty}\frac{\lambda^*_{x_{-\ell}}}{\lambda^*_{y_{-\ell}}}=\lim_{n \to +\infty}\frac{\lambda_y^*(-n)}{\lambda_x^*(-n)},	\:\:\textrm{for}\:\:*=c,u.
\end{equation}
It follows readily from the definition that $\rho_y^*(x)=(\rho_x^*(y))^{-1}$. Furthermore, for any $z \in \cW^{cu}(x)=\cW^{cu}(y)$, we have
\begin{equation}\label{quotient rho x y rho x z}
	\rho_x^*(z)=\rho_y^*(z)\rho_x^*(y).
\end{equation}
The basic distortion result (Lemma~\ref{l.distortionbasic}) implies that $\rho_x^*\colon \cW^{cu}(x)\to(0,+\infty)$ is a continuous map that depends continuously on the base point (for further details we refer to \cite{KalininKatok}). 

We then define
\begin{equation}\label{eq.kkunstable}
	\cH^*_x(y)\eqdef\int_x^y\rho_x^*(\hat{y})\,d\hat{y},
\end{equation}
which gives a $C^1$ diffeomorphism from $\cW^*(x)$ to the real line. Now, consider the inverse map $\Phi^*_x=(\cH^*_x)^{-1}$. Here's the proposition.

\begin{prop}[Kalinin-Katok -- see \cite{KalininKatok}, Section 3.1]
	\label{p.kk}
  For $\ast=c,u$, The family $\{\Phi^*_x\}_{x\in\TT}$ of $C^1$ diffeomorphisms $\Phi^*_x\colon\R\to\cW^*(x)$ satisfies  
	\begin{enumerate}
		\item $f\circ\Phi^*_x(s)=\Phi^*_{x_1}(\lambda^*_xs)$, for all $s \in \R$;
		\item $\Phi^*_x(0)=x$;
		\item $D\Phi^*_x(0)e_1=v^*(x)$, where $e_1$ denotes the unitary tangent vector field of the real line $\R$.
%		\item\label{point quatre prop six un} If $y\in\cW^*(x)$ then $(\Phi^*_y)^{-1}\circ\Phi^*_x\colon s \mapsto {\color{red}a_{x,y}^*s+b_{x,y}^*}$, for some real numbers {\color{red}$a_{x,y}^*,b_{x,y}^*$} (see Lemma \ref{lemme changt coords affine} below for the expression of $a_{x,y}^*$). 
	\end{enumerate}
\end{prop}

For any $x\in \TT$ and $y \in \mathcal{W}^*(x)$, % $q=\Phi_x^*(s)=\Phi_{x'}^*(s') \in \cW^*(x)=\cW^*(y)$,  $s,s'=s'(s)\in \mathbb{R}$, where
we define the \emph{change of charts} as
\begin{equation}\label{defi holono h phi}
	\mathcal{H}_{x,y}^*\eqdef \mathcal{H}^*_y\circ(\mathcal{H}^*_x)^{-1}=(\Phi^*_y)^{-1}\circ\Phi^*_x\colon \mathbb{R}\to \mathbb{R}.
\end{equation}

\begin{lemma}\label{lemme changt coords affine}
	For any $x\in \TT$ and $y \in \mathcal{W}^*(x)$, % $q=\Phi_x^*(s)=\Phi_{x'}^*(s') \in \cW^*(x)=\cW^*(y)$,  $s,s'=s'(s)\in \mathbb{R}$, where
	%let us define $\mathcal{H}_{x,y}^*\eqdef(\Phi^*_y)^{-1}\circ\Phi^*_x\colon \mathbb{R}\to \mathbb{R}$. Then,
	and for any $s\in \mathbb{R}$, it holds 
	\begin{equation}\label{equ vingt deux}
		\mathcal{H}_{x,y}^*(s)-\mathcal{H}_{x,y}^*(0)=\rho_{y}^*(x)\cdot s. %, \quad \text{i.e., }a_{x,y}^*=\rho_y^*(x).%,\, b_{x,y}^*=h^2_{x,y}^*(0).
	\end{equation}
	In other words, the change of one-dimensional normal form coordinates is affine with derivative $\rho^*_y(x)$. 
\end{lemma}
\begin{proof}
	%For each $s \in \mathbb{R}$, let us denote $p^*(s)\eqdef \Phi_x^*(s)
	%By definition, for any $s\in \mathbb{R}$, we have $h^2_{x,y}^*(s)=\mathcal{H}_{y}^*\circ (\mathcal{H}_x^*)^{-1}(s)$. 
	Fix $s \in \mathbb{R}$, and let $z\eqdef (\mathcal{H}_x^*)^{-1}(s)=(\mathcal{H}_y^*)^{-1}(\cH_{x,y}^*(s))\in \cW^*(x)=\cW^*(y)$. By \eqref{eq.kkunstable}. By differentiating $\mathcal{H}_{x,y}^*=\mathcal{H}_{y}^*\circ (\mathcal{H}_x^*)^{-1}$ at $s$, we get
	$$
	\frac{d\mathcal{H}_{x,y}^*}{ds}(s)=\frac{(\mathcal{H}_{y}^*)'(z)}{(\mathcal{H}_{x}^*)'(z)}=\frac{\rho_y^*(z)}{\rho_x^*(z)}=\rho_y^*(x),
	$$
	thus $\mathcal{H}_{x,y}^*(s)-\mathcal{H}_{x,y}^*(0)=\int_{0}^s \rho_y^*(x)\, dt=\rho_{y}^*(x)\cdot s$. 
\end{proof}

 %Since $E^c$ is also uniformly expanding the construction of $\{\Phi^c_x\}_{x\in\TT}$ is identical, with the obvious changes. 

%In the following, for $*=c,u,cu$, we denote 
%\begin{equation}\label{def pi etoile}
%\Pi^*:=\{(x,y):x \in \TT,\, y \in \mathcal{W}^{*}(x)\}\subset (\TT)^2.
%\end{equation}

\subsection{Two dimensional normal forms}

Recall that there are unitary vector fields $\TT\ni x\mapsto v^*(x)\in E^*(x)$ such that 
\[
Df(x)v^*(x)=\lambda^*_xv^*(x_1).
\] 
Notice that since $f$ is $C^2$, the bundle $E^u$ is $C^1$ inside $\cW^{cu}$ and since $\|v^u(x)\|>0$ the map $x\mapsto\lambda^u_x\in\R$ is $C^1$ in restriction to center unstable manifolds. 
Moreover, for a fixed $R>0$ the $C^1$ norm 
\[
\sup\{\|D(\lambda^u|_{\cW^{cu}(x)})(y)\|:y\in \cW^{cu}_R(x),\:x\in M\}
\]
is bounded by some uniform constant $C=C(R,f)>0$. These results follow from \cite{PSW}.

We restate here Theorem~\ref{thm.normalformsintro}, which gives a center-unstable version of Proposition~\ref{p.kk}.

\begin{theorem}
	\label{thm.normalforms}
	There exists a family of $C^1$ diffeomorphisms $\Phi_x \colon\R^2 \to \mathcal{W}^{cu}(x)$, depending continuously on $x$ such that:
	\begin{enumerate}
		\item\label{ppprop un} $f \circ \Phi_x(t,s) = \Phi_{f(x)}(\lambda^u_x t,\lambda^c_x s)$, for all $(s,t)\in \R^2$;
		\item\label{ppprop deux} $\Phi_x(0,0) = x$;
		\item\label{ppprop trois} $D\Phi_x(0,0)e_1=v^u(x)$ and $D\Phi_x(0,0)e_2=v^c(x)$;
		\item\label{ppprop quatre} $\Phi_x(\cdot,\cdot)$ depends continuously with the choice of $x$ in the $C^1$-topology; %;\marginpar{Affine structure?}
		\item\label{ppprop cinq} for all $s\in \R$, $\Phi_x(\R\times \{s\})=\mathcal{W}^u(\Phi_x(0,s))$, and $\Phi_x(\{0\}\times \R)=\mathcal{W}^c(x)$ (see Figure \ref{2D NF}).
		%\item\label{affine structures on w cu} \color{teal} for any $x \in \TT$ and $y \in \mathcal{W}^{cu}(x)$, the change of coordinates $\Phi_{y}^{-1}\circ \Phi_x$ is affine (see Proposition \ref{propo affine struct} below for a more precise statement).\color{black}
	\end{enumerate}
\end{theorem}

The drawback of our construction is that the natural generalization of Lemma~\ref{lemme changt coords affine} is no longer true, as the remark below demonstrates.

\begin{remark}
	\label{rem.martinsage}
	Note that in general, we cannot expect the change of normal charts $\mathcal{H}_{x,y}\eqdef\Phi_y^{-1}\circ \Phi_x\colon \R^2 \to \R^2$ to be affine for any $x\in \TT$ and $y \in \mathcal{W}^{cu}(x)$. 
	
	Suppose it were the case. We claim that it implies that for any $x \in \TT$, the center foliation $\mathcal{W}^c$ is $C^1$ within the center-unstable leaf $\mathcal{W}^{cu}(x)$. Indeed, for any $y \in \mathcal{W}^{cu}(x)$, by Theorem \ref{thm.normalforms}\eqref{ppprop cinq}, we have
	$$
	\Phi_x^{-1}(\mathcal{W}^c(y))=\mathcal{H}_{y,x}\circ \Phi_y^{-1}(\mathcal{W}^c(y))=\mathcal{H}_{y,x}(\{0\}\times \R),
	$$
	which is a straight line since the map $\mathcal{H}_{y,x}$ is affine. Since two distinct center leaves cannot cross, and $\Phi_x^{-1}(\mathcal{W}^c(x))=\{0\}\times \R$ is vertical, it follows that $\Phi_x^{-1}(\mathcal{W}^c(y))$ is vertical for any $y \in \mathcal{W}^{cu}(x)$. In particular, since the normal chart $\Phi_{x}$ is $C^1$, it does imply that $\mathcal{W}^c$ is $C^1$ within $\mathcal{W}^{cu}(x)$. 
	
	However, regularity of the center foliation $\mathcal{W}^c$ is a rare phenomenon. For instance, in \cite[Lemma 3]{GogolevPath}, it is shown that for $r \geq 2$, there exsits a $C^1$-open and $C^r$-dense set $\mathcal{V}\subset \mathcal{A}_m^r(\TT)$ of the set of $C^r$ conservative Anosov diffeomorphisms on $\TT$ considered in the present work such that for $f\in \mathcal{A}_m^r(\TT)$, the center foliation $\mathcal{W}^c$ is Lipschitz inside the center-unstable foliation $\mathcal{W}^{cu}$ if and only if $f \in \mathcal{A}_m^r(\TT)\setminus \mathcal{V}$. 
\end{remark}

We now proceed to the proof of Theorem~\ref{thm.normalforms}.

\subsubsection{A foliated chart}

A natural way of obtaining a parametrization $(t,s)$ of $\cWcu(x)$ is to consider the arc-length parameter $t$ along the curve $\cWu(\Phi^c_x(s))$. As we shall see, it is not hard to prove that this indeed gives a $C^1$ identification with nice properties (for instance horizontal lines are mapped onto unstable manifolds). The difficulty for proving Theorem~\ref{thm.normalforms} is that in these coordinates the map $f$ \emph{does not acts linearly}. For this reason in the forthcoming paragraphs we will perform suitable reparametrizations.

More formally, let $\phi\colon\R\times\cW^{cu}(x)\to\cW^{cu}(x)$ denote the flow of the vector field $x\mapsto v^u(x)$. Notice that $v^u|_{\cW^{cu}(x)}$ is $C^1$, with its $C^1$ norm depending continuously on $x$ on compact subsets of $\cW^{cu}(x)$. With this flow at hand we define a map $\Gamma_x\colon\R^2\to\cW^{cu}(x)$ by
\[
\Gamma_x(t,s)\eqdef\phi^t(\Phi^c_x(s)),\quad \forall\, (t,s)\in \R^2,
\] 
where $\{\Phi^c_x\}_{x\in\TT}$ are Kalinin-Katok's normal forms in $\cW^c$ (see Theorem \ref{thm.normalforms}).
\begin{lemma}
	\label{l.cartasparavoce}
For each $x\in\TT$, the map $\Gamma_x$ is a $C^1$ diffeomorphism whose $C^1$ norm depends continuously with $x$.
\end{lemma}
\begin{proof}
Let us show that this defines a $C^1$ diffeomorphism. Indeed, $\Gamma_x$ is injective and surjective due to Lemmas~\ref{l.coerenciaum}~and~\ref{l.coorenciadois} respectively. Therefore it suffices to prove that $\Gamma_x$ is $C^1$ with invertible derivative at each point.

Choosing an arbitrary coordinate system locally in $\cW^{cu}(x)$ we can think of $\Gamma_x$ as a map from $\R^2$ to $\R^2$. Now, observe that the Jacobian matrix of the map $(t,s)\mapsto\phi^t(\Phi^c_x(s))$ is the $2\times 2$ matrix whose columns are the vectors $v^u(\phi^t(\Phi^c_x(s)))$ and $D\phi^t(\Phi^c_x(s))\frac{d}{ds}\Phi^c_x(s)$. Since $\phi^t$ is a $C^1$ flow and $\Phi^c_x$ is a $C^1$ curve this proves that $\Gamma_x$ has continuous derivative. Moreover, since $v^u(\Phi^c_x(s))$ and $\frac{d}{ds}\Phi^c_x(s)$ are transverse and since $D\phi^t(\Phi^c_x(s))$ is a linear isomorphism, we have that the vectors
\[
v^u(\phi^t(\Phi^c_x(s)))=D\phi^t(\Phi^c_x(s))v^u(\Phi^c_x(s))\:\:\:\textrm{and}\:\:\:D\phi^t(\Phi^c_x(s))\frac{d}{ds}\Phi^c_x(s)
\]  
are also transverse. 
Thus the Jacobian matrix of $\Gamma_x$ is invertible, proving the assertion. The continuous dependence of the $C^1$ norm follows from that of $v^u$ and $\Phi^c$.
\end{proof}

\subsubsection{Preliminary construction}

We shall start the construction of the map $\Phi$ with a slight modification of the construction of the foliated chart $\Gamma_x$. Instead of using the arc-length parameter along unstable manifolds we shall use normal forms. This choice will allow us to perform the required reparametrization for linearizing the action of $f|_{\cWcu(x)}$.    

\begin{lemma}
	\label{l.obagulhoehcum}
	The map $\Psi_x\colon\R^2\to\cW^{cu}(x)$ given by $\Psi_x(t,s)\eqdef\Phi^u_{\Phi^c_x(s)}(t)$ is a $C^1$ diffeomorphism, with $C^1$ norm depending continuously from $x$.
\end{lemma}
\begin{proof}
	Notice that $\Psi_x$ is bijective because we can define directly an inverse map. Indeed, given $y\in\cW^{cu}(x)$ we consider the point $y^c\in\cW^c(x)$, whose existence is ensured by Lemmas~\ref{l.coerenciaum}~and~\ref{l.coorenciadois}, such that 
	\[
	\{y^c\}=\cW^u(y)\cap\cW^c(x).
	\]
	Define then a map $\Xi_x\colon\cW^{cu}(x)\to\R^2$ by 
	\[
	\Xi_x(y)\eqdef(\cH^u_{y^c}(y),\cH^c_{x}(y^c)).
	\] 
	From the definitions we have
	\[
	\Psi_x\circ\Xi_x(y)=\Phi^u_{\Phi^c_x(\cH^c_x(y^c))}(\cH^u_{y^c}(y))=\Phi^u_{y^c}(\cH^u_{y^c}(y))=y.
	\]
	In a similar way one sees that $\Xi_x\circ\Psi_x(t,s)=(t,s)$. It suffices then to check that $\Xi_x$ is $C^1$ with invertible derivative at each point. 
		
	To prove that $\Xi_x$ is a $C^1$ diffeomorphism then it suffices to establish that $\Xi_x\circ\Gamma_x\colon \R^2\to\R^2$ is a $C^1$ diffeomorphism. To compute the derivative of this map fix a point $y=\Gamma_x(t,s)\in\cW^{cu}(x)$. Then, since $\Gamma_x$ is a foliated chart, we have $y^c=\Gamma_x(0,s)$. Thus,
	\[
	\Xi_x\circ\Gamma_x(t,s)=(\cH^u_{\Gamma_x(0,s)}(\Gamma_x(t,s)),\cH^c_x(\Gamma_x(0,s))).
	\]  
	Note that $(t,s)\mapsto\cH^c_x(\Gamma_x(0,s))=s$ is a $C^1$ function. Moreover, as
	\[
	\cH^u_{\Gamma_x(0,s)}(\Gamma_x(t,s))=\int_{\Gamma_x(0,s)}^{\Gamma_x(t,s)}\rho_{\Gamma_x(0,s)}(\Gamma_x(r,s))dr,
	\]
	Leibniz rule will imply that $\Xi_x\circ\Gamma_x$ is a $C^1$ map as long as we prove that  $(t,s)\mapsto\rho_{\Gamma_x(0,s)}(\Gamma_x(t,s))\in(0,+\infty)$ is $C^1$. To verify this assertion, recall that
	\[
	\rho_{\Gamma_x(0,s)}^u(\Gamma_x(t,s))=\prod_{\ell=0}^{+\infty}\frac{\|Df^{-1}(f^{-\ell}(\Gamma_x(t,s)))|_{E^u}\|}{\|Df^{-1}(f^{-\ell}(\Gamma_x(0,s)))|_{E^u}\|}.
	\]
	Consider the auxiliary function $g(t,s)=\log\rho_{\Gamma_x(0,s)}^u(\Gamma_x(t,s))$. Then we can write
	\[
	g(t,s)=\sum_{\ell=0}^{+\infty}\left(\log\lambda^u_{f^{-\ell-1}(\Gamma_x(0,s))}-\log\lambda^u_{f^{-\ell-1}(\Gamma_x(t,s))}\right).
	\]  
	First notice that as $\lambda^u_{(.)}\colon\cW^{cu}\to\R$ is $C^1$ and since $d(f^{-\ell}(\Gamma_x(0,s)),f^{-\ell}(\Gamma_x(t,s)))\to 0$ exponentially fast (with uniform rate, depending only on $f$), we deduce that $g$ is the uniform limit on compact sets of the sequence of partial sums. Consider the function $g_{\ell}(t,s)=\log\lambda^u_{f^{-\ell-1}(\Gamma_x(t,s))}$. By the chain rule,
	\[
	Dg_{\ell}(t,s)=\frac{D\lambda^u_{f^{-\ell-1}(\Gamma_x(t,s))}Df^{-\ell-1}(\Gamma_x(t,s))D\Gamma_x(t,s)}{\lambda^u_{f^{-\ell-1}(\Gamma_x(t,s))}}.
	\]
	Therefore, for $(t,s)\in B_R(0)$ there is a uniform constant $C=C(R,f)>0$ so that 
	\[
	\|Dg_{\ell}(t,s)\|\leq  C\|Df^{-\ell-1}(\Gamma_x(t,s))|_{E^{cu}}\|.
	\]
	The right hand side above has a uniform bound decreasing exponentially as $\ell$ increases for $\|Df^{-\ell}(x)|_{E^{cu}}\|\leq e^{-\chi^c_2\ell}$, for every $x\in\TT$, with $\chi_2^c>0$ (recall our notations from \S\ref{sss.hypestimates}). Thus,
	\[
	\|Dg_{\ell}(t,s)\|\leq  Ce^{-\chi^c_2\ell}.
	\] 
	An analogous estimate holds for the function $\tilde{g}_\ell(t,s)=\log\lambda^u_{\Gamma_x(0,s)}$. This proves that the derivative of the truncated series defining $g$ also converges uniformly on compact sets. By elementary calculus (see for instance Proposition 1.41 of \cite{barreira2012ordinary}) this proves that $g$ is $C^1$, and therefore we conclude that $(t,s)\mapsto\rho_{\Gamma_x(0,s)}^u(\Gamma_x(t,s))\in(0,+\infty)$ is $C^1$, as desired. To complete the proof of the lemma, observe that the derivative of $\Xi_x\circ\Gamma_x$ is an upper triangular matrix with non-zero entries in the diagonal. By the inverse function theorem, as $\Xi_x\circ\Gamma_x$ is bijective, this proves that this map is in fact a $C^1$ diffeomorphism. This ends the proof of the lemma.  
\end{proof}

%For simplicity, suppose that $f$ preserves the orientation of all the bundles $E^*$, for $* = s,c, u$. Let $\{\Phi^c_x\colon\R \to \mathcal{W}^{c}(x)\}_{x\in \T^3}$ be the family of one normal forms for center leaves, i.e., such that for each $x\in \TT$ and $s\in \R$, $f \circ \Phi_x^c(s) = \Phi_{f(x)}^c(\lambda^c_x s)$. Fix a continuous unitary vector field $v^c(\cdot)$. Using $v^c(x)$ we may identify $E^c(x)$ with $\R$.  Thus, we have a family of parametrizations $\Phi^c_x\colon \R \to \mathcal{W}^c(x)$ that verifies:
%\begin{itemize}
%	\item for any $s\in \R$, $f\circ \Phi^c_x(s) = \Phi^c_{x_1}(\lambda^c_x s)$;
%	\item $\Phi^c_x(0) = x$ and $D\Phi^c_x(0) = \mathrm{Id}$;
%	\item for any $y \in \mathcal{W}^c(x)$, the map $(\Phi^c_y)^{-1} \circ \Phi^c_x$ is affine. 
%\end{itemize}

%Let $v^u(\cdot)$ be a unitary vector field that generates the bundle $E^u$.  Consider the vector field $v^u_x\colon \R\ni s \mapsto v^u(\Phi^c_x(s))\in E^u(x)$.  In what follows, for $j \geq 1$, we will write $s_{-j}(x) = \Phi^c_{x_{-j}}\left((\lambda^c_{x_{-j}} \cdots \lambda^c_{x_{-1}})^{-1} s\right)\in \mathcal{W}^c(x_{-j})$.

\subsubsection{A reparametrization function}

Here again, the diffeomorphism $\Psi_x$ constructed in Lemma~\ref{l.obagulhoehcum} does not satisfy all the requirements we need. Indeed, to obtain condition \eqref{ppprop un} from Theorem~\ref{thm.normalforms} it is necessary to perform a suitable reparametrization. The design of such a map is the content of next lemma.

\begin{lemma}
	\label{lem.changeofubasis}
	For each $x\in \T^3$, there exists a $C^1$ function $\beta_x\colon \R \to \R^+$ such that 
	\begin{equation}
	\label{eq.oneiterate}
	Df(\Phi^c_x(s))\cdot \beta_x(s) v^u\left(\Phi^c_x(s)\right)= \lambda^u_x\beta_{x_1}(\lambda^c_x s) v^u\left(\Phi^c_{x_1}(\lambda^c_x s)\right).  
	\end{equation}
	
\end{lemma}

\begin{proof}
	Let us fix a point $x\in\TT$. Consider the sequence of functions $\R \ni s\mapsto\psi^x_{-\ell}(s)\in\cW^c(x_{-\ell})$ where 
	\[
	\psi^x_{-\ell}(s)\eqdef\Phi^c_{x_{-\ell}}\left((\lambda^c_{x_{-\ell}}\cdots\lambda^c_{x_{-1}})^{-1}s\right).
	\]
	For coherence of notation also make the convention that $\psi^x_0(s)=\Phi^c_x(s)$. 

\begin{remark}\label{rem_true_orbit}
Note that by property of Kalinin-Katok's normal forms, we have that if $y=\psi^x_0(s)=\Phi^c_x(s)$ for some $s\in\R$ and $\ell\in\N$
$$f^{-\ell}(y)=\psi^x_{-\ell}(s).$$
\end{remark}

	Using this sequence we define $h_{n,x}\colon \R\to\R$,
	\[
	h_{n,x}(s) \eqdef\sum_{\ell=1}^n\left(\log\lambda^u_{\psi^x_{-\ell}(s)}-\log\lambda^u_{x_{-\ell}}\right).
	\]
	Notice that $y\mapsto\lambda^u_y$ is uniformly bounded from above and from below, and the lower bound is larger than $1$. Thus
	\[
	|\log\lambda^u_{\psi^x_{-\ell}(s)}-\log\lambda^u_{x_{-\ell}}|\leq |\lambda^u_{\psi^x_{-\ell}(s)}-\lambda^u_{x_{-\ell}}|.
	\]    
	Now, as $y\mapsto\lambda^c_y$ is also uniformly bounded with a lower bound larger than $1$, since the family $\{\Phi^c_x\}$ is continuous in the $C^1$ topology and $\Phi^c_{x_{-\ell}}(0)=x_{-\ell}$ it follows that $d(\psi^x_{-\ell}(s),x_{-\ell})\to 0$ exponentially fast with uniform constants (keep in mind Remark \ref{rem_true_orbit}). Moreover, as $y\mapsto\lambda^u_y$ is uniformly $C^1$ inside $\cW^{cu}$ there exists some constant $C>0$ such that 
	\[
	|\lambda^u_{\psi^x_{-\ell}(s)}-\lambda^u_{x_{-\ell}}|\leq Cd(\psi^x_{-\ell}(s),x_{-\ell}).
	\] 
	This proves that $h_{n,x}$ converges uniformly to a continuous function $h_x$. We define $\beta_{x}(s)\eqdef e^{h_x(s)}$. Observe that with Notation \eqref{eq.rhonormalforms},  we can write
	\begin{equation}\label{eq.betadefinition}
	\beta_{x}(s)=\prod_{\ell=1}^{+\infty}\frac{\lambda^u_{\psi^x_{-\ell}(s)}}{\lambda^u_{x_{-\ell}}}=\rho_{\Phi_x^c(s)}^u(x)=\frac{1}{\rho_{x}^u(\Phi_x^c(s))}.
	\end{equation}
	
	We claim that $\lambda^u_{\Phi^c_x(s)}\beta_x(s)=\lambda^u_x\beta_{x_1}(\lambda^c_xs)$. This formula immediately gives us \eqref{eq.oneiterate}. Indeed note that by property of Kalinin-Katok's normal forms
\begin{align*}
Df(\Phi^c_x(s))\cdot v^u(\Phi^c_x(s))&=\lambda^u_{\Phi^c_x(s)}v^u(f(\Phi^c_x(s)))\\
                                &=\lambda^u_{\Phi^c_x(s)}v^u(\Phi^c_{x_1}(\lambda^c_x s)),
\end{align*}
and combining this equality with the claim provides \eqref{eq.oneiterate}.	
	
To prove the claim we first remark that 
	\[
	\psi^{x_1}_{-\ell}(\lambda^c_xs)=\Phi^c_{x_{-\ell+1}}\left((\lambda^c_{x_{-\ell+1}}\cdots\lambda^c_{x_{-1}}\lambda^c_x)^{-1}\lambda^c_xs\right)=\psi^x_{-\ell+1}(s).
	\]
	We deduce that
	\[
	\beta_{x_1}(\lambda^c_xs)=\prod_{\ell=1}^{+\infty}\frac{\lambda^u_{\psi^{x_1}_{-\ell}(\lambda^c_xs)}}{\lambda^u_{x_{-\ell+1}}}=\prod_{\ell=1}^{\infty}\frac{\lambda^u_{\psi^x_{-\ell+1}(s)}}{\lambda^u_{x_{-\ell+1}}}=\frac{\lambda^u_{\Phi^c_x(s)}}{\lambda^u_x}\beta_x(s),
	\]
	proving our claim. 
	
	To finish the lemma it remains to show that $\beta_x$ is a $C^1$ function, depending continuously (with respect to the local $C^1$ topology) on $x$. With that goal in mind we compute
	\[
	\frac{d}{ds}\psi^x_{-\ell}(s)=(\lambda^c_{x_{-\ell}}\cdots\lambda^c_{x_{-1}})^{-1}D\Phi^c_{x_{-\ell}}\left((\lambda^c_{x_{-\ell}}\cdots\lambda^c_{x_{-1}})^{-1}s\right)e_1.
	\] 
	Applying thus the chain rule we obtain
	\begin{align*}
	\frac{d}{ds}\lambda^u_{\psi^x_{-\ell}(s)}&=D(\lambda^u_{\psi^x_{-\ell}(s)})\frac{d}{ds}\psi^x_{-\ell}(s)\nonumber\\
	&=(\lambda^c_{x_{-\ell}}\cdots\lambda^c_{x_{-1}})^{-1}D(\lambda^u_{\psi^x_{-\ell}(s)})D\Phi^c_{x_{-\ell}}\left((\lambda^c_{x_{-\ell}}\cdots\lambda^c_{x_{-1}})^{-1}s\right)e_1,\nonumber
	\end{align*}
	where $D(\lambda^u_y)$ denotes the derivative of the function $y\mapsto\lambda^u_y$ at $y$. As $(\lambda^c_{x_{-\ell}}\cdots\lambda^c_{x_{-1}})^{-1}\to 0$ exponentially fast with uniform constants, the left hand side above also vanishes exponentially fast because $\lambda^u_y$ is $C^1$ and $x\mapsto\max_{|s|\leq 1}\|D\Phi^c_x(s)\|$ is uniformly bounded. As a consequence the sequence 
	\[
	h_{n,x}^{\prime}(s)=\sum_{\ell=1}^n\frac{\frac{d}{ds}\lambda^u_{\psi^x_{-\ell}(s)}}{\lambda^u_{_{\psi^x_{-\ell}(s)}}}
	\]
	is uniformly bounded by a convergent geometric series, because $\lambda^u_y$ is lower bounded by a constant larger than $1$.
	Since $h_{n,x}\to h_x$ uniformly on compact sets, by elementary calculus (see for instance \cite{barreira2012ordinary} Proposition 1.41) this proves that $h_x$ is $C^1$ with $h_x^\prime=\lim_{n \to +\infty}h^\prime_{n,x}$. 
	Moreover, the uniform estimates obtained show that the $C^1$ norm of $f$ changes continuously with $x$. This completes the proof of the lemma.
\end{proof}

\subsubsection{Defining the normal form}\label{sss_def_normal_form}

With Lemmas~\ref{l.obagulhoehcum} and \ref{lem.changeofubasis} at hand we are in position to define our two dimensional parametrization of center-unstable manifolds $\cW^{cu}(x)$. For each $x\in\TT$ we consider
\begin{equation*}
\Phi_x\colon\left\{
\begin{array}{rcl}
\R^2&\to&\cW^{cu}(x),\nonumber\\
(t,s)&\mapsto&\Psi_x(\beta_x(s)t,s)=\Phi^u_{\Phi_x^c(s)}(\beta_x(s)t).
\end{array} 
\right.
\end{equation*}
Notice that, by Lemma~\ref{lem.changeofubasis}, $s\mapsto\beta_{x}(s)$ is a $C^1$ positive function. It follows that $(t,s)\mapsto(\beta_{x}(s)t,s)$ is a $C^1$ diffeomorphism of $\R^2$. Therefore, applying Lemma~\ref{l.obagulhoehcum} we obtain that $\Phi_x$ is a $C^1$ diffeomorphism. We denote $\cH_x\eqdef\Phi_x^{-1}$.
\begin{proof}[Proof of Theorem~\ref{thm.normalforms}]
Properties \eqref{ppprop deux}-\eqref{ppprop quatre} are automatic from the construction of $\Phi_x$. To prove \eqref{ppprop un} we apply Proposition~\ref{p.kk} and obtain 
\[
f(\Phi_x(t,s))=f\left(\Phi^u_{\Phi^c_x(s)}(\beta_{x}(s)t)\right)=\Phi^u_{f\circ\Phi^c_x(s)}\left(\lambda^u_{\Phi^c_x(s)}\beta_{x}(s)t\right).
\]
Now, equality \eqref{eq.oneiterate} implies that 
\[
\beta_{x}(s)\lambda^u_{\Phi^c_x(s)}v^u(f\circ\Phi^c_x(x))=\lambda^u_x\beta_{x_1}(\lambda^c_x s) v^u\left(\Phi^c_{x_1}(\lambda^c_x s)\right).
\]
Proposition~\ref{p.kk} applied with $\Phi^c_x$ shows that $f\circ\Phi^c_x(s)=\Phi^c_{x_1}(\lambda^c_xs)$, where $x_1=f(x)$. Combining we get $\beta_x(s)\lambda_{\Phi^c_x(s)}=\lambda^u_x\beta_{x_1}(\lambda^c_xs)$. One deduces then
\[
f(\Phi_x(t,s))=\Phi^u_{\Phi^c_{x_1}(\lambda^c_xs)}\left(\lambda^u_x\beta_{x_1}(\lambda^c_xs)t\right).
\]
The very definition of $\Phi_x$ says that 
\[
\Phi^u_{\Phi^c_{x_1}(\lambda^c_xs)}\left(\lambda^u_x\beta_{x_1}(\lambda^c_xs)t\right)=\Phi_{x_1}(\lambda^u_xt,{\color{blue}\lambda^c_xs}).\qedhere
\]

\end{proof}

\begin{remark}
	\label{rem.horizontal}
	It follows readily from our construction that the inverse map $\cH_x\colon\cWcu(x)\to\R^2$ sends each unstable manifold $\cWu(y)$, for $y\in\cWc(x)$ onto the horizontal line $\{s=\cH_x^c(y)\}$. 
\end{remark}

\subsection{Change of normal form coordinates}

%Recall that for $*=c,u,cu$, we denote $\Pi^*:=\{(x,y):x \in \TT,\, y \in \mathcal{W}^{*}(x)\}$. 
As in \eqref{defi holono h phi} and Remark~\ref{rem.martinsage}, for any $x\in \TT$ and $y \in \mathcal{W}^{cu}(x)$, % $q=\Phi_x^*(s)=\Phi_{x'}^*(s') \in \cW^*(x)=\cW^*(y)$,  $s,s'=s'(s)\in \mathbb{R}$, where
we define the change of charts
\begin{equation}\label{defi holono h phi bis}
\mathcal{H}_{x,y}\eqdef \mathcal{H}_y\circ \mathcal{H}_x^{-1}=\Phi_y^{-1}\circ\Phi_x\colon
\left\{
\begin{array}{rcl}
\mathbb{R}^2&\to&\mathbb{R}^2,\\
(t,s)&\mapsto&(h^1_{x,y}(t,s),h^2_{x,y}(t,s)).
\end{array} 
\right.
\end{equation}
The function $h^2_{x,y}$ only depends on the second coordinate,  as demonstrated by the following lemma.
\begin{lemma}\label{preserv hori fol}
	For any $x\in \TT$ and $y \in \mathcal{W}^{cu}(x)$, the diffeomorphism $\mathcal{H}_{x,y}$ preserves the foliation of $\mathbb{R}^2$ into horizontal lines $\{s=\mathrm{cst}\}$. 
\end{lemma}

\begin{proof}
This is a straightforward consequence of the fact that normal forms sends horizontal lines to unstable manifolds (see Remark~\ref{rem.horizontal}). More precisely, given $s\in\R$, using item (5) of Theorem~\ref{thm.normalforms} we have
\[
\cW^u(\Phi^c_x(s))=\Phi_x(\R\times\{s\}).
\]
Now define $\hat{y}=\cW^c(y)\cap\cW^u(\Phi^c_x(s))$ (recall Lemma~\ref{l.coorenciadois}).  Applying item (5) of Theorem~\ref{thm.normalforms} once more we obtain
\[
\cH_y(\cW^u(\hat{y}))=\R\times\{\cH^c_y(\hat{y})\}
\]
Since $\cW^u(\hat{y})=\cW^u(\Phi^c_x(s))$, this shows that 
\[
\cH_{x,y}(\R\times\{s\})=\R\times\{\cH^c_y(\hat{y})\},
\]  
which concludes the proof.
\end{proof}

\begin{lemma}\label{c one dep h x y}
	The diffeomorphism $\mathcal{H}_{x,y}(\cdot,\cdot)$ depends continuously in the $C^1$ topology with the pair $(x,y)$, $x\in \TT$ and $y \in \mathcal{W}^{cu}(x)$. Moreover, for any $\bar x \in \TT$, and for $(\bar t,\bar s)\in \mathbb{R}^2$, 
	\begin{enumerate}
		\item\label{permier point six trois} $%\lim_{\Pi^{cu}\ni (x,y)\to (\bar x,\bar x),\, \mathbb{R}^2\ni(t,s)\to (\bar t,\bar s)} 
		\mathcal{H}_{x,y}(t,s)\to(\bar t,\bar s)$ as $(x,y)\to (\bar x,\bar x)$ and $(t,s)\to (\bar t,\bar s)$;
		\item\label{deuxierme point six trois} %\lim_{\Pi^{cu}\ni (x,y)\to (\bar x,\bar x),\, \mathbb{R}^2\ni(t,s)\to (\bar t,\bar s)} 
		$D\mathcal{H}_{x,y}(t,s)\to\mathrm{Id}_{\mathbb{R}^2}$ as $(x,y)\to (\bar x,\bar x)$ and $(t,s)\to (\bar t,\bar s)$.
	\end{enumerate}
%and the convergence is uniform on compact subsets of $\Pi^{cu}\times \R^2$. 
\end{lemma}

\begin{proof}
	It follows immediately from Theorem \ref{thm.normalforms}\eqref{ppprop quatre}, since $\mathcal{H}_{x,y}=\mathcal{H}_{y}\circ \mathcal{H}_{x}^{-1}$, and 
	$D\mathcal{H}_{x,y}(t,s)=D\mathcal{H}_{y}(\mathcal{H}_{x}^{-1}(t,s))\circ D\mathcal{H}_{ x}^{-1}(t,s)$, with 
	\begin{equation*}
	%\lim_{\TT\ni x\to \bar x,\, \mathbb{R}^2\ni(t,s)\to (\bar t,\bar s)} 
	D\mathcal{H}_{x}(t,s)\to D\mathcal{H}_{\bar x}(\bar t,\bar s),
	\end{equation*}
	as $x\to \bar x$ and $(t,s)\to (\bar t,\bar s)$, and 
	%\lim_{\Pi^{cu}\ni (x,y)\to (\bar x,\bar x),\, \mathbb{R}^2\ni(t,s)\to (\bar t,\bar s)} 
	$$
	D\mathcal{H}_{y}(\mathcal{H}_{x}^{-1}(t,s))\to D\mathcal{H}_{\bar x}(\mathcal{H}_{\bar x}^{-1}(\bar t,\bar s))=(D\mathcal{H}_{\bar x}^{-1}(\bar t,\bar s))^{-1},
	$$
	as $(x,y)\to (\bar x,\bar x)$ and $(t,s)\to (\bar t,\bar s)$. 
\end{proof} 
\subsubsection{Change of normal charts for two points in the same center leaf}

We first study the change of normal charts between two points of the same center leaf.
%The next lemma explicits the map $\mathcal{H}_{x,y}$ when $y \in \mathcal{W}^c(x)$.
\begin{lemma}\label{l.change_center_charts}
	Let $x\in \TT$ and $y\in \mathcal{W}^{c}(x)$. For any $(t,s)\in\R^2$, it holds
	\begin{enumerate}
		\item $h^1_{x,y}(t,s)=\rho^u_y(x)t$;
		\item $h^2_{x,y}(s)=\cH^c_{x,y}(s)=\rho^c_y(x)s+\cH^c_y(x)$.
	\end{enumerate}
%	\begin{equation*}
%	\mathcal{H}_{x,y}(t,s)=\begin{bmatrix}
%	\rho_{y}^u(x) & 0\\
%	0 & \rho_{y}^c(x)
%	\end{bmatrix}
%	\begin{bmatrix}
%	t\\
%	s
%	\end{bmatrix}+\begin{bmatrix}
%	0\\
%	\mathcal{H}_{y}^c(x)
%	\end{bmatrix},
%	\end{equation*}
%	\begin{enumerate}
%		\item $h^1_{x,y}(t,s)=\rho_{y}^u(x)\cdot t$,
%		\item $h^2_{x,y}(s)=\rho_y^c(x)\cdot s+\mathcal{H}_{y}^c(x)$,
%	\end{enumerate} 
\end{lemma}

\begin{proof}
	We start by considering the function $h^1_{x,y}$. As $x$ and $y$ belong to the same center leaf, for any $s \in \mathbb{R}$, it holds $h^1_{x,y}(0,s)=0$, and 
	$$
	\Phi^c_x(s)=\Phi_x(0,s)=\Phi_y(h^1_{x,y}(0,s),h^2_{x,y}(s))=\Phi_{y}(0,h^2_{x,y}(s))=\Phi^c_y(h^2_{x,y}(s)),
	$$
	hence by \eqref{equ vingt deux}, 
	$$
	h^2_{x,y}(s)=(\Phi_y^c)^{-1} \circ \Phi_x^c(s)=\mathcal{H}_{x,y}^c(s).
	$$
	The second claimed equality follows from Lemma~\ref{lemme changt coords affine}.
	
	Now, let us fix $(t,s)\in \mathbb{R}^2$, and denote $t'\eqdef h^1_{x,y}(t,s)$ and $s'\eqdef h^2_{x,y}(s)$ for simplicity. 
	By definition, and as $\Phi^c_y(h^2_{x,y}(s))=\Phi^c_x(s)$, we obtain
	\begin{equation*}%\label{eq defi phi x phi y}
	\Phi^u_{\Phi^c_x(s)}(\beta_x(s)t)=\Phi_x(t,s)=\Phi_y(t',s')=\Phi^u_{\Phi^c_y(s')}(\beta_y(s')t'),
	\end{equation*}
	hence 
	\begin{equation}\label{egalite beta s t}
	\beta_x(s)t=\beta_y(s')t'.
	\end{equation}
On the other hand, by \eqref{eq.betadefinition} we have 
\begin{align}
	\beta_y(s')&=\rho^u_{\Phi^c_y(s')}(y)=\rho^u_{\Phi^x_x(s)}(y)\nonumber\\
&=\rho^u_{\Phi^c_x(s)}(x)\rho^u_x(y)=\beta_x(s)\rho^u_x(y),	
\end{align}
where in the second line we applied \eqref{quotient rho x y rho x z}. Combining this with \eqref{egalite beta s t} we deduce $t'=\rho^u_y(x)t$, concluding.
%
%
%	By \eqref{eq.betadefinition}, we also have 
%	$\beta_{x}(s)=(\rho_x^u(y(s)))^{-1}$, $\beta_{y}(s')=(\rho_y^u(y(s)))^{-1}$, hence \eqref{egalite beta s t}-\eqref{quotient rho x y rho x z} yield
%	%(recall that with the notation of \S \ref{sss_def_normal_form} $\psi^x_{-\ell}(s)=\psi^x_{-\ell}(s)=y'_{-\ell}$)
%	$$
%	h^1_{x,y}(t,s)=t'=\frac{\beta_x(s)}{\beta_y(s')}\cdot t=\frac{\rho_y^u(y(s))}{\rho_x^u(y(s))}\cdot t=\rho_y^u(x) \cdot t.%=\prod_{\ell=1}^\infty\frac{\lambda^u_{x_{-\ell}}}{\lambda^u_{y_{-\ell}}}\cdot t.
%	$$
\end{proof}
\subsubsection{Change of normal forms for two points in the same unstable leaf}
Let us now study the change of normal charts $\mathcal{H}_{x,x'}\colon (t,s)\mapsto (h^1_{x,x'}(t,s),h^2_{x,x'}(s))$ between two points $x,x'$ of the same unstable leaf. Unlike the previous case, here we do not get an affine map. 
\begin{lemma}\label{chang of coord}
	Let $x\in \TT$ and $x'\in \mathcal{W}^u(x)$. Then, there exists a $C^1$ function $a_{x,x'}:\R\to\R$ such that  for any $(t,s)\in \mathbb{R}^2$, %recall that $(h^1_{x,x'}(t,s),h^2_{x,x'}(s))=\mathcal{H}_{x,x'}(t,s)$. I
	it holds %, and let $z\eqdef \Phi_x(t,s)=\Phi_x'(t',s')\in \cWcu(x)$. Then
\[
h^1_{x,x'}(t,s)=\rho^u_{x'}(x)t+a_{x,x'}(s).
\]

%	\begin{equation*}
%	\mathcal{H}_{x,x'}(t,s)=\begin{bmatrix}
%	\rho_{x'}^u(x) & 0\\
%	0 & \rho_{x'}^c(x)
%	\end{bmatrix}\begin{bmatrix}
%	t \\ 
%	s
%	\end{bmatrix}+\begin{bmatrix}
%	a_{x,x'}(s) \\
%	0
%	\end{bmatrix},
%	\end{equation*}
%	\begin{enumerate}
%		\item\label{popoint uno} $h^1_{x,x'}(t,s)= \rho_{x'}^u(x)\cdot t + a_{x,x'}(s)$,
%		\item\label{popoint duo} $h^2_{x,x'}(s)=\rho_{x'}^c(x)\cdot s$,
%	\end{enumerate} 
%	with $\rho_{x'}^*(x)=\prod_{\ell=1}^{+\infty} \frac{\lambda^u_{x'_{-\ell}}}{\lambda^u_{x_{-\ell}}}$, $*=c,u$, as in \eqref{eq.rhonormalforms}, and 
%	$$
%	a_{x,x'}\colon s \mapsto\int_{\mathcal{W}^u(y',y)} \rho_{x'}^u(\hat y) \, d\hat{y},\quad a_{x,x'}(0)=\mathcal{H}_{x'}^u(x),
%	$$ 
%	where in the integral, $\hat y$ ranges over the segment $\mathcal{W}^u(y',y)\subset \mathcal{W}^u(y)$ of unstable manifold connecting the points $y'\eqdef\Phi_{x'}^c(\rho_{x'}^c(x)s)$ and $y\eqdef \Phi_x^c(s)$.
\end{lemma}

\begin{proof}
	Consider $(t,s)\in \mathbb{R}^2$, and let us abbreviate $t'\eqdef h^1_{x,x'}(t,s)$, $s'\eqdef h^2_{x,x'}(s)$,  $y=y(s)\eqdef \Phi_x^c(s)$, $y'=y'(s)\eqdef \Phi_{x'}^c(s')$, and $z=z(t,s)\eqdef \Phi_x(t,s)=\Phi_{x'}(t',s')$. 
	Recall that $z=\Phi_x(t,s) = \Phi^u_{y}(\beta_x(s) t)$. By \eqref{eq.kkunstable}, we thus obtain 
	\begin{equation}\label{eq.cceq1}
	\beta_x(s) t = \int_y^z \rho_y^u(\hat{y})\, d\hat{y}.
	\end{equation}
	Similarly, $z=\Phi_{x'}(t',s') = \Phi^u_{y'}(\beta_{x'}(s')t')$, and 
	\begin{equation}\label{eq.cceq2}
	\beta_x(s') t' = \int_{y'}^z \rho_{y'}^u(\hat{y})\, d\hat{y},
	\end{equation} 
	Observe that $y'\in \cWu(y)$. 
	%Recall also that for any pair $(p,q)\in \Pi^u$, we have  
	%\[
	%\rho_p^u(q) = \displaystyle \prod_{l=1}^{+\infty} \frac{\lambda^u_{p_{-l}}}{\lambda^u_{q_{-l}}}.
	%\]
	By \eqref{quotient rho x y rho x z}, for any point $\hat{y}\in \cWu(y)=\cWu(y')$, we have $\rho_{y'}^u(\hat{y}) = \rho_y^u(\hat{y})  \rho_{y'}^u(y)$. Hence,
	\begin{align*}
	\beta_{x'}(s') t' & = \int_{y'}^z \rho_{y'}^u(\hat{y})\, d\hat{y} = \int_{y'}^z \rho_{y'}^u(y) \rho_y^u(\hat{y}) \, d\hat{y} \\
	& =  \rho_{y'}^u(y) \left( \int_{y'}^z \rho_y^u(\hat{y})\,  d\hat{y}\right) = \rho_{y'}^u(y) \left( \int_{y'}^y \rho_y^u(\hat{y}) \, d\hat{y} + \int_y^z \rho_y^u(\hat{y})\,  d\hat{y} \right)\\
	& = \rho_{y'}^u(y) \beta_x(s) t + \rho_{y'}^u(y) \int_{y'}^y \rho_y^u(\hat{y})\,  d\hat{y},
	\end{align*} 
	where in the last equality we used \eqref{eq.cceq1}.  We conclude that
	\[
	t' = \rho_{y'}(y) \frac{\beta_x(s)}{\beta_{x'}(s')} \cdot t + \frac{\rho_{y'}^u(y)}{\beta_{x'}(s')} \int_{y'}^y \rho_y^u(\hat{y}) \, d\hat{y}.
	\]
	Set $a_{x,x'}(s)\eqdef  \frac{\rho_{y'}^u(y)}{\beta_{x'}(s')} \int_{y'}^y \rho_y^u(\hat{y}) \, d\hat{y}$. This is the translation part of the change of coordinates.  Recall that
	\[
	\displaystyle \beta_x(s) = \rho_{\Phi_x^u(s)}^u(x)=\rho_{y}^u(x).
	\]
	Similarly,
	\[
	\beta_{x'}(s') =\rho_{\Phi_{x'}^u(s')}^u(x')=\rho_{y'}^u(x')=(\rho_{x'}^u(y'))^{-1}.% \displaystyle \prod_{l=1}^{+\infty} \frac{\lambda^u_{y_{-l}'}}{\lambda^u_{x_{-l}'}}.
	\]
	Therefore, by \eqref{quotient rho x y rho x z}, it holds 
	\[
	\displaystyle \rho_{y'}^u(y) \frac{\beta_x(s)}{\beta_{x'}(s')}=\rho_y^u(x)\rho_{y'}^u(y)\rho_{x'}^u(y')=\rho_y^u(x)\rho_{x'}^u(y)=\rho_{x'}^u(x). %\left(\prod_{l=1}^{+\infty} \frac{\lambda^u_{y_{-l}'}}{\lambda^u_{y_{-l}}}\right).  \left(\prod_{l=1}^{+\infty} \frac{\lambda^u_{y_{-l}} \lambda^u_{x_{-l}'}}{\lambda^u_{x_{-l}} \lambda^u_{y_{-l}'}} \right) = \prod_{l=1}^{+\infty} \frac{\lambda^u_{x_{-l}'}}{\lambda^u_{x_{-l}}} = \rho_{x'}(x).
	\]
	We conclude that $t' = \rho_{x'}^u(x) \cdot t + a_{x,x'}(s)$. Moreover,  
	\begin{align*}
	a_{x,x'}(s) &\eqdef\frac{\rho_{y'}^u(y)}{\beta_{x'}(s')} \int_{y'}^y \rho_y^u(\hat{y}) \, d\hat{y}=\frac{\rho_{x'}^u(x)}{\beta_{x}(s)} \int_{y'}^y \rho_y^u(\hat{y}) \, d\hat{y}\\
	&=\frac{\rho_{x'}^u(x)}{\rho_y^u(x)} \int_{y'}^y \rho_y^u(\hat{y}) \, d\hat{y}=\rho_{x'}^u(y)\int_{y'}^y \rho_y^u(\hat{y}) \, d\hat{y}=\int_{y'}^y \rho_y^u(\hat{y})\rho_{x'}^u(y) \, d\hat{y}\\
	&=\int_{\mathcal{W}^u(y',y)} \rho_{x'}^u(\hat y) \, d\hat{y},
	\end{align*}
	where in the last integral, $\hat y$ ranges over the segment $\mathcal{W}^u(y',y)\subset  \mathcal{W}^u(y)$ of unstable manifold connecting the points $y'$ and $y$.
	\end{proof}
	
	Let us now study the map $h^2_{x,x'}\colon \R \to \R$. The first step is to characterize it in terms of unstable holonomies. 
	
	\begin{lemma}
	\label{lem.lindoefacil}
Fix $x\in\TT$ and $x'\in\cW^u(x)$ and let  $H^u_{x,x'}:\cW^c(x)\to\cW^c(x')$ denote the unstable holonomy map. Define the map $L^u_{x,x'}:\R\to\R$ by
\[
L^u_{x,x'}(s)\eqdef\cH^c_{x'}\circ H^u_{x,x'}\circ\Phi^c_x(s).
\]
Then, we have $h^2_{x,x'}(s)=L^u_{x,x'}(s)$.	
	\end{lemma}
	\begin{proof}
Consider an $s\in\R$ and define $y=\Phi^c_x(s)$ and $y'=H^u_{x,x'}(y)$. Set $s'\eqdef\cH^c_{x'}(y')$.  Notice that, by definition, $s'=L^u_{x,x'}(s)$. By Remark~\ref{rem.horizontal} we know that $\Phi_x(\R\times\{s\})=\cW^u(y)$, and similarly $\Phi_{x'}$ sends $\R\times\{s'\}$ onto $\cW^u(y')$. Since $\cW^u(y)=\cW^u(y')$ this shows that 
\[
\cH_{x'}\circ\Phi_x(\R\times\{s\})=\R\times\{s'\},
\] 
and therefore $h^2_{x,x'}(s)=s'$. 
	\end{proof}

We now prove that unstable holonomies, when conjugated by normal forms (as in the previous lemma), become a linear maps of the real line. 

\begin{lemma}
	\label{lem.calculholonomie}
Fix $x\in\TT$ and $x'\in\cW^u(x)$. Then, $L^u_{x,x'}(s)=\rho^c_{x'}(x)s.$
\end{lemma}
\begin{proof}
For ease of notation in this proof, we will denote $L(s)=L^u_{x,x'}(s)$. By definition, we have that
\[
\Phi^c_{x'}\circ L(s)=H^u_{x,x'}\circ\Phi^c_x(s).
\]	
Now, consider an integer $k\in\Z$ and define the functions $g_k\colon \mathbb{R} \to \mathbb{R},s \mapsto \lambda_{x}^c(k) \cdot s$ (recall Notation \eqref{recall equation cocycle}). Similarly let $g'_k:\R\to\R, s\mapsto\lambda^c_{x'}(k) \cdot s$. 	By applying $f^k$ to both sides of the above equation and using the equivariance properties of normal forms and holonomy maps, we obtain
\[
\Phi^c_{x'_k}(g'_k\circ L(s))=H^u_{x_k,x'_k}\circ\Phi^c_{x_k}(g_k(s)).
\]
Taking derivatives with respect to $s\in \mathbb{R}$ on both sides we get
\begin{equation}\label{eg diff at x k}
D\Phi_{x_k'}^c(g_k'(s))\lambda_{x'}^c(k) \frac{dL}{ds}(s)=	DH_{x_k,x_k'}^u(\Phi_{x_k}^c(g_k(s)))D\Phi_{x_k}^c(g_k(s))\lambda_x^c(k).
\end{equation}
Recall that $\det D\Phi_{x_k}^c(g_k(s))=\frac{1}{\rho_{x_k}^c(\Phi_{x_k}^c(g_k(s)))}$, and $\det D\Phi_{x_k'}^c(g_k'(s))=\frac{1}{\rho_{x_k'}^c(\Phi_{x_k'}^c(g_k'(s)))}$. 
Passing to the determinant in \eqref{eg diff at x k}, we deduce that 
\begin{equation}\label{expression derivee s' a s}
	\frac{dL}{ds}(s)=\frac{\rho_{x_k'}^c(\Phi_{x_k'}^c(g_k'(s)))}{\rho_{x_k}^c(\Phi_{x_k}^c(s_k(s)))} \cdot \frac{\lambda_x^c(k)}{\lambda_{x'}^c(k)} \cdot \det DH_{x_k,x_k'}^u(\Phi_{x_k}^c(g_k(s))). 
\end{equation}
We now let $k\to-\infty$. By compactness, up to considering a subsequence, we can assume that $x_k,x_k'\to p\in \TT$ as $k \to -\infty$. Moreover, $g_k(s),g_k'(s)\to 0$ as $k \to-\infty$, and by the properties of $\Phi_{(\cdot)}^c(\cdot)$ stated in Proposition \ref{p.kk}, we have $\Phi_{x_k}^c(g_k(s)),\Phi_{x_k'}^c(g_k'(s))\to \Phi_q^c(0)=q$ as $k \to -\infty$. By continuity of $\rho_{(\cdot)}^c(\cdot)$, we deduce that 
$$
\lim_{k \to-\infty}\frac{\rho_{x_k'}^c(\Phi_{x_k'}^c(g_k'(s)))}{\rho_{x_k}^c(\Phi_{x_k}^c(g_k(s)))}= \frac{\rho_q^c(q)}{\rho_q^c(q)}=1. 
$$
Since the holonomy maps $H_{x_k,x_k'}^u$ converge uniformly to $H_{q,q}^u=\mathrm{Id}|_{\mathcal{W}^c(q)}$ in the $C^1$ topology as $k \to -\infty$, we also have $\lim_{k \to -\infty}\det DH_{x_k,x_k'}^u(\Phi_{x_k}^c(g_k(s)))=1$. 

By \eqref{expression derivee s' a s}, \eqref{recall equation cocycle} and \eqref{eq.rhonormalforms}, we conclude that 
$$
\frac{dL}{ds}(s)=\lim_{k \to -\infty}\frac{\lambda_x^c(k)}{\lambda_{x'}^c(k)}=\prod_{\ell=1}^{+\infty}
\frac{\lambda_{x_{-\ell}'}^c}{\lambda_{x_{-\ell}}^c}=\rho_{x'}^c(x).
$$
This completes the proof.
\end{proof}

\section{Leaf-wise quotient measures}\label{s.leafwise}

Let us denote by $\cM(\R)$ the set of locally finite Borel measures on the real line. Using the families of parametrizations $\{\Phi_x\},\{\Phi^c_x\}_{x\in\TT}$ we will construct a family of elements of $\cM(\R)$ which will be the main object of study in this paper. We summarize the result of the construction with the following statement. 

\begin{theorem}
\label{t.leafwise}
Let $f\in\cA^2(\TT)$ be an Anosov diffeomorphism with expanding center. Let $\mu$ be an ergodic $u$-Gibbs measure of $f$. Then, there exists a family $\{\hat{\nu}^c_x\}_{x\in\TT}\subset\cM(\R)$ of lacally finite Borel measures on the real line with the following property: if, for $\mu$ almost every $x$, the measure $\hat{\nu}^c_x$ is proportional to the Lebesgue measure of $\R$ then $\mu$ is an SRB measure. 
\end{theorem}

The family $\{\hat{\nu}^c_x\}_{x\in\TT}$ is referred to as \emph{leaf-wise quotient measures}. Once Theorem~\ref{t.leafwise} is established, our main result is reduced to prove that, for $\mu$ almost every $x$, each $\hat{\nu}^c_x$ is a translation-invariant measure on the real line. The proposition below summarizes the properties of the family $\{\hat{\nu}^c_x\}_{x\in\TT}\subset\cM(\R)$ that enable us to achieve this result.

\begin{prop}[Basic moves]
	\label{p.basicmoves}
The family $\{\hat{\nu}^c_x\}_{x\in\TT}\subset\cM(\R)$ satisfies the following properties on a full measure subset $\Lambda$ of $\TT$, which is $\cW^u$-saturated. For every $x\in\Lambda$
\begin{enumerate}
\item if $y\in\cW^c(x)\cap\Lambda$ then $\hat{\nu}^c_x\propto(\cH^c_{x,y})_*\hat{\nu}^c_y$, where $\cH^c_{x,y}:\R\to\R$ is the change of normal form coordinates on the manifold $\cW^c(x)$, which is an affine map. 
\item $\hat{\nu}^c_{f(x)}\propto(\Lambda^c_x)_*\hat{\nu}^c_x$, where $\Lambda^c_x:\R\to\R$ is the linear map $\Lambda^c_x(s)=\lambda^c_xs$.
\item if $x^{\prime}\in\cW^u(x)$ then $\hat{\nu}^c_{x^{\prime}}\propto (L_{x,x^{\prime}})_*\hat{\nu}^c_x$, where $L_{x,x^{\prime}}:\R\to\R$ is the linear map $L_{x,x^{\prime}}(s)=\rho^c_{x^{\prime}}(x)s$.
\end{enumerate} 
\end{prop}
 
The remainder of this section is dedicated to proving Theorem~\ref{t.leafwise} and Proposition~\ref{p.basicmoves}. Thus, we fix $f\in\cA^2(\TT)$ and $\mu$ an ergodic $u$-Gibbs measure. 
 
\subsection{Conditional measures of $\mu$ in normal forms}

We will show that when we push the conditional measures of $\mu$ along center unstable manifolds using the normal forms and subsequently disintegrate this measure along the horizontal lines (corresponding to unstable manifolds), we obtain measures \emph{proportional} to the Lebesgue measure. his fact plays a fundamental role in the proofs of Theorem~\ref{t.leafwise} and Proposition~\ref{p.basicmoves}.

\subsubsection{Some preliminary facts about  the conditional measures $\mu^{cu}_{n,x}$}

Let us start by recalling the notation from \S\ref{sub.sub.conditional} and stating some basic lemmas that will be useful throughout this section. Start with $\xi^{cu}_0$, an increasing measurable partition subordinate to the uniformly expanded foliation $\mathcal{W}^{cu}$. When we apply the dynamics, we obtain a sequence $\xi^{cu}_n\eqdef f^n(\xi^{cu}_0)$ of measurable partitions that remain subordinate to $\mathcal{W}^{cu}$. Moreover, for every $n\in\N$, we have $\xi^{cu}_{n+1}\prec \xi^{cu}_n$. Now consider $\{\mu^{cu}_{n,x}\}_{x\in\mathbb{T}^3}$ to be the disintegration of $\mu$ with respect to the partition $\xi^{cu}_n$. Similar to Lemma~\ref{l.superposition}, we find that the following emph{superposition property} holds for $\mu$-a.e. $x\in \TT$:
\begin{equation}
	\label{e.superposition bis}
	\textrm{if}\:m>n,\:\textrm{then}\:\mu^{cu}_{n,x}(A)=\frac{\mu^{cu}_{m,x}(A)}{\mu^{cu}_{m,x}(\xi^{cu}_n(x))},
\end{equation}
for every measurable set $A\subset\xi^{cu}_n(x)$. 
Furthermore, as $\mu$ is $f$-invariant, the distintegrations $\{\mu^{cu}_{n,x}\}_{n\in \N,\,x\in\mathbb{T}^3}$ satisfy:
\begin{lemma}
	\label{l_inv_con_meas deux}
	For every $n\in\N$ and for $\mu$-almost every $x\in \TT$, we have
	$$\mu^{cu}_{n+1,f(x)}=f_* \mu^{cu}_{n,x}.$$
\end{lemma}

\begin{proof}
The proof follows a similar approach to that of Lemma \ref{l_inv_con_meas}. 
\end{proof}

By intersecting with the center foliation $\mathcal{W}^c$, or unstable foliation $\mathcal{W}^u$, we can refine each partition $\xi^{cu}_n$, $n \in \N$. More precisely, for $\star=c,u$, we consider a partition $\xi^{cu}_n\prec \xi^{\star}_n=\{\xi^{\star}_n(y)\}_{y \in \TT}$, with  
	\[
	\xi^{\star}_{n}(y)\eqdef\mathcal{W}^{\star}(y)\cap\xi^{cu}_n(y),\quad \forall\, y\in \TT.
	\] 
By Remark \ref{rem_subfoliation_subordinate}, we obtain: 
	\begin{lemma}
		\label{l.poincarebendixon}
		For $\star=c,u$, the family of sets $\xi^{\star}_n$ is a measurable partition subordinate to $\mathcal{W}^{\star}$.
	\end{lemma}
	
	For every $n\in\N$ and for $\mu$-almost every $x\in \TT$, we define $\zeta^\star_n(x)\eqdef\mathcal{H}_x(\xi^\star_n(x))\dans\R^2$, and $\nu_{n,x}^{cu}\eqdef(\mathcal{H}_x)_* \mu^{cu}_{n,x}$. Note that due to item \eqref{ppprop un} of Theorem~\ref{thm.normalforms}, for each $x \in \TT$, it holds 
\begin{equation}\label{conjug forme normale}
	N_x\circ \mathcal{H}_x|_{\mathcal{W}^{cu}(x)}=\mathcal{H}_{f(x)}\circ f|_{\mathcal{W}^{cu}(x)},\quad \text{where }N_x\colon (t,s)\mapsto (\lambda_x^u t,\lambda_x^c s). 
\end{equation}

As a result of \eqref{conjug forme normale} and Lemma \ref{l_inv_con_meas deux}, we conclude:
\begin{lemma}
	\label{l_inv_con_meas trois}
	For every $n\in\N$ and for $\mu$-almost every $x\in \TT$, it holds
	$$\nu^{cu}_{n+1,f(x)}=(N_x)_* \nu^{cu}_{n,x}.$$
\end{lemma}

\subsubsection{Disintegration of the conditional measures in normal forms}

As in \S\ref{sub.sub.conditional} when we disintegrate $\mu^{cu}_{n,x}$ along the atoms $\{\xi^u_{n}(y)\}_{y\in\xi^{cu}_n(x)}$ we get a probability measure $\mu^c_{n,x}$ over $\xi^c_{n}(x)$ and a family $\{\mu^u_{n,y}\}_{y\in\xi^{cu}_n(x)}$, each of which is a probability measure on $\xi^u_n(y)$.

Now, let $\cH_x:\cW^{cu}(x)\to\R^2$ be the inverse map of the normal form coordinate chart. Define the measures 
\[
\nu^{\star}_{n,x}\eqdef(\cH_x)_*\mu^{\star}_{n,x},\:\:\:\textrm{for}\:\:\:\star=c,cu,
\]
and
\[
\nu^{u}_{n,s}\eqdef(\cH_x)_*\mu^{u}_{n,s},\:\:\:\textrm{for each}\:\:\:s=\cH^c_x(y),\:\:\:\textrm{with}\:\:\: y\in\xi^c_n(x). 
\]

A remarkable property of $u$-Gibbs measures is that their conditional measures along unstable manifolds in normal forms are \emph{proportional} to Lebesgue.
	
	\begin{lemma}\label{Lebesg}
		For $\mu$-a.e. $x\in \TT$ and $\nu_{n,x}^{cu}$-a.e. $(t,s)\in\zeta_n^{cu}(x)$, there exists a constant $\gamma_{n,x}(s)>0$ (independent of $t$) such that for every Borel set $B\subset\zeta^u_n(\Phi^c_x(s))$, the following holds
		\[
		\nu^u_{n,s}(B)=\gamma_{n,x}(s)\operatorname{Leb}(B).
		\]
		\end{lemma}
	
	\begin{proof}
		Given that $\mu$ is a $u$-Gibbs measure, for $\mu$-a.e. $x\in \TT$, and $\mu_{n,x}^{cu}$-a.e. $y\in \xi_n^{cu}(x)$, the measure $\mu_{n,y}^u$ is absolutely continuous. Its density, up to a constant, is given by (referring to Notation \eqref{eq.rhonormalforms})
		\begin{equation*} 
			\rho_y^u \colon \xi_n^ u(y)\ni z \mapsto \rho_y^u(z)=\prod_{j=0}^{+\infty} \frac{\|Df^{-1}(f^{-j}(z))|_{E^u}\|}{\|Df^{-1}(f^{-j}(y))|_{E^u}\|}.
		\end{equation*}
		Let $(t,s)\eqdef \mathcal{H}_x(y)$, so that $\nu_{n,s}^u=(\mathcal{H}_x)_*\mu_{n,y}^u$. Restricted to $\R\times\{s\}$, $\Phi_x=\mathcal{H}_x^{-1}$ is given by the map  $\Phi^u_{\Phi_x^c(s)}(\beta_x(s)t)$, for  some $\beta_x(s)\in \R$. Therefore, up to a constant that depends only on $s$, the density of $\nu_{n,s}^u=(\mathcal{H}_x)_*\mu_{n,y}^u$ is given by 
		\begin{equation*} 
			\tilde \rho_s^u \colon t \mapsto \rho_y^u(\Phi_{\Phi_x^c(s)}^u(t)) D\Phi_{\Phi_x^c(s)}^u(t).
		\end{equation*}
		Since $D\Phi_{\Phi_x^c(s)}^u(t)=(D\mathcal{H}_y^u)^{-1}(\Phi_{\Phi_x^c(s)}^u(t))=\rho_y^u(\Phi_{\Phi_x^c(s)}^u(t))^{-1}$, hence $\tilde \rho_s^u$ is constant. 
	\end{proof}

\subsection{Construction of leaf-wise quotient	 measures}
   The measures we define here are similar to the quotient measures $\nu^c_{n,x}$, but we will use unstable segments of a fixed length together with a normalization that allows for stabilization when $n$ is sufficiently large large enough (see Lemma~\ref{l.superposition bis} below). Let $I\dans\R$ be any bounded interval centered at $0$. It determines the length of unstable segments used in defining the leaf-wise quotient measures. Given a bounded Borel set $B\dans\R$ and a $\mu$-typical point $x\in\TT$ there exists $n_0=n_0(x,I,B)$ such that for every $n\geq n_0$, $I\times B\dans\zeta_n^{cu}(x)$ and $I\times [-1,1]\dans\zeta_n^{cu}(x)$: see Lemma \ref{lem.af1}.  We first show that the choice of length of unstable segments, and thus the interval $I\subset\R$ will not modify the definition. 
	 
	\begin{lemma}\label{l.indep_on_I}
		For any other interval $J\subset\R$ centered at $0$, we have
		$$\frac{\nu_{n,x}^{cu} (I \times B)}{\nu_{n,x}^{cu} (I \times [-1,1])}=\frac{\nu_{n,x}^{cu} (J \times B)}{\nu_{n,x}^{cu} (J \times [-1,1])},$$
		for every sufficiently large $n$. 
	\end{lemma}
	
	\begin{proof}
	Consider real numbers $\alpha$ and  $\beta$ such that the affine map $\psi(t)=\alpha t+\beta$ satisfies $\psi(I)=J$. Then, applying the disintegration in conditional measures and using Lemma~\ref{Lebesg} we obtain
	\begin{align}
		\frac{\nu_{n,x}^{cu} (J \times B)}{\nu_{n,x}^{cu} (J \times [-1,1])}&=\frac{\int_{B}\nu^u_{n,s}(J)d\nu^c_{n,x}(s)}{\int_{[-1,1]}\nu^u_{n,s}(J)d\nu^c_{n,x}(s)}\nonumber\\
		&=\frac{\int_{B}\gamma_{n,x}(s)\operatorname{Leb}(\alpha I+\beta)d\nu^c_{n,x}(s)}{\int_{[-1,1]}\gamma_{n,x}(s)\operatorname{Leb}(\alpha I+\beta)d\nu^c_{n,x}(s)}\nonumber\\
		&=\frac{\int_{B}\gamma_{n,x}(s)\operatorname{Leb}(I)d\nu^c_{n,x}(s)}{\int_{[-1,1]}\gamma_{n,x}(s)\operatorname{Leb}(I)d\nu^c_{n,x}(s)}\nonumber\\
		&=\frac{\int_{B}\nu^u_{n,s}(I)d\nu^c_{n,x}(s)}{\int_{[-1,1]}\nu^u_{n,s}(I)d\nu^c_{n,x}(s)}\nonumber\\
		&=\frac{\nu_{n,x}^{cu} (I \times B)}{\nu_{n,x}^{cu} (I \times [-1,1])}\nonumber.
	\end{align} 		
This equality holds for sufficiently large $n$, ensuring that $J\times B$ and  $J\times [-1,1]$ are both subsets of $\zeta_n^{cu}(x)$. Notice that this can always be achieved due to Lemma~\ref{lem.af1}. This completes the proof.
	\end{proof}

Given $B\subset\R$ a bounded Borel measurable subset, and given $n\in\N$ we define the following number

	\begin{equation}
		\label{eq.les_nunu}
\hat \nu_{n,x}^c(B)\eqdef\frac{\nu_{n,x}^{cu} (I \times B)}{\nu_{n,x}^{cu} (I \times [-1,1])}.
	\end{equation}
	
We now use the superposition property to show that this sequence stabilizes  for large $n$.
	
	\begin{lemma}
		\label{l.superposition bis}
		For $\mu$-almost every $x\in\mathbb{T}^3$, every bounded Borel set $B\subset \R$, and every pair of integers $m>n\geq n_0(x,I,B)\in\N$ we have $\hat{\nu}^c_{m,x}(B)=\hat{\nu}^c_{n,x}(B)$.
	\end{lemma}
	\begin{proof}
		Firstly, observe that for each pair $m>n$, there exists a $\mu$-full measure set of points $x\in \TT$ for which $\mu^{cu}_{m,x}(\xi^{cu}_{n,x}(x))>0$. This follows from the superposition property \eqref{e.superposition bis}. By taking the countable intersection of all these sets, and applying \eqref{e.superposition bis} again, we deduce that for every bounded Borel set $B\subset \R$, and $m,n\geq n_0(x,I,B)$, so that $I\times B,I\times [-1,1]\dans\zeta_n^{cu}(x)\dans\zeta_m^{cu}(x)$, the following holds
		\begin{align*}
			\hat{\nu}^c_{n,x}(B)&=\frac{\mu^{cu}_{n,x}(\Phi_x(I\times B))}{\mu^{cu}_{n,x}(\Phi_x(I\times [-1,1]))} \\
			&=\frac{\mu^{cu}_{m,x}(\Phi_x(I\times B))}{\mu^{cu}_{m,x}(\xi^{cu}_{n,x}(x))}\times\frac{\mu^{cu}_{m,x}(\xi^{cu}_{n,x}(x))}{\mu^{cu}_{m,x}(\Phi_x(I\times B))}=\hat{\nu}^c_{m,x}(B).\qedhere
		\end{align*} 
	\end{proof}
	
	Hence, the construction above defines a family of measures on $\R$.
	
	\begin{defi}[Leaf-wise quotient measure]
		For $\mu$-a.e. $x \in \TT$, let $\hat{\nu}^c_x$ be the unique locally finite Borel measure defined in $\R$ so that, for each bounded Borel $B\subset \R$, one has 
		\[
		\hat{\nu}^c_x(B)\eqdef\hat{\nu}^c_{n_0(x,B),x}(B).
		\]
	\end{defi}
	By Lemma~\ref{l.superposition bis} the right hand side above is well defined.

\begin{remark}[Choice of normalization]
	\label{rem.normaliza}
	By construction, we have that $\hat{\nu}^c_x([-1,1])=1$.
\end{remark}

\begin{remark}\label{remark.superposition2}
	By construction, for $\mu$-a.e. $x \in \TT$, $\hat \nu_x^c$ gives positive measure to any open neighbourhood of $0$. This follows from the definition and Corollary \ref{coro.superposition}
\end{remark}

\subsection{Obtaining the SRB property: Proof of Theorem~\ref{t.leafwise}}
	To complete the proof of Theorem \ref{t.leafwise} we assume that for $\mu$ almost every $x\in\TT$ the measure $\hat{\nu}^c_x$ is proportional to Lebesgue on $\R$.  To conclude that $\mu$ is an SRB measure it suffices to show that the conditional measures $\mu^{cu}_{0,x}$ are absolutely continuous with respect to $\operatorname{Leb}_{\cW^{cu}(x)}$ for $\mu$ almost every $x$. To achieve this goal, consider $a<0<b$ with $b-a$ large enough so that, denoting $I=[a,b]$, we have 
	\[
	\xi^{cu}_0(x)\subset\Phi_x(I\times I),\:\textrm{for}\:\mu\:\textrm{almost every}\:x\in\TT.
	\] 
	This can be done because the atoms of $\xi^{cu}_0$ have uniformly bounded diameter and the $C^1$ norm of $\Phi_x$ depends continuously on $x$.  Let $R=\max_{x\in\TT}\operatorname{diam}(\Phi_x(I\times I))$. Then, applying Lemma~\ref{lem.af1} with this number we get that for $\mu$ almost every $x$, there exists $n_0(x,I)$ such that if $n\geq n_0(x,I)$ then 
	\begin{enumerate}
		\item $I\subset\zeta^c_n(x)$
		\item $I\times\{s\}\subset\zeta^u_{n}(s)$, for every $s\in I$.		
	\end{enumerate}
Now, given such $x$ and $n\geq n_0(x,I)$ take $s\in I$. We observe
\begin{align}
1=\nu^u_{n,s}(\zeta^u_n(s))&\geq\nu^u_{n,s}(I\times\{s\})\nonumber\\
&=\gamma_{n,x}(s)\operatorname{Leb}(I)\nonumber\\
&=\gamma_{n,x}(s)(b-a).
\end{align}
Therefore,
\[
0<\gamma_{n,x}(s)<\frac{1}{b-a}.
\]
We claim that $\nu^c_{n,x}|_I$ is absolutely continuous with respect to the leaf-wise quotient measure $\hat{\nu}^c_x|_I$. Indeed, let us denote for simplicity $\alpha=(\nu^{cu}_{n,x}(I\times [-1,1]))^{-1}$. Then, if $B\subset I$ we have
\begin{align}
\hat{\nu}^c_x(B)&=\alpha\nu^{cu}_{n,x}(I\times B)=\alpha\int_B\nu^u_{n,s}d\nu^c_{n,x}(s)\nonumber\\
&=\alpha\operatorname{Leb}(I)\int_B\gamma_{n,x}(s)d\nu^c_{n,x}(s).
\end{align}
Since $\gamma_{n,x}>0$ if $\hat{\nu}^c_x(B)=0$ then $\nu^c_{n,x}(B)=0$, which proves our claim. Since, by assumption, $\hat{\nu}^c_x\propto\operatorname{Leb}$ we conclude that indeed $\nu^c_{n,x}|_I$ is absolutely continuous with $\operatorname{Leb}|_I$. As $d\nu^{cu}_{n,x}(t,s)=d\nu^u_{n,s}(t)d\nu^c_{n,x}(s)$, it follows immediately from Lemma~\ref{Lebesg} that $\nu^{cu}_{n,x}|_{I\times I}$ is absolutely continuous with respect to the two dimensional Lebesgue measure $\operatorname{Leb}|_{I\times I}$. Since $\Phi_x$ is $C^1$, we get that $\mu^{cu}_{n,x}|_{\Phi_x(I\times I)}$ is absolutely continuous with respect to $\operatorname{Leb}_{\cW^{cu}(x)}$. By the superposition property we deduce that $\mu^{cu}_{0,x}$ is absolutely continuous, as desired. The theorem is proved. \qed

	%	Recall that $\mu$ is SRB if and only if, for $\mu$-a.e. $x\in \TT$ $\mu_{x}^{cu}$ is absolutely continuous with respect to Lebesgue on $\xi^{cu}(x)$. Recall also that the definition of leaf-wise measures is independent of the initial choice of an interval $I$ (by Lemma \ref{l.indep_on_I}) and that atoms of $\xi^{cu}$ have uniform diameter. Hence we may as well choose $I$ large enough in such a way that for $\mu$-a.e. $x\in\TT$ we have 
		%$$\zeta^{cu}(x)=\mathcal{H}_x(\xi^{cu}(x))\dans I\times I.$$
		
		%Let us use this $I$ in the definition of leaf-wise measures. Let $n$ large enough so that $I\times I\dans\zeta^{cu}_{n}(x)$. Using disintegration if $\hat\nu^c_x$ is Lebesgue measure then $\nu^{cu}_{n,x}$ restricted to $I\times I$ is equivalent to Lebesgue measure (recall Lemma \ref{Lebesg}). Hence, $\nu^{cu}_{n,x}$, when restricted to $\zeta^{cu}(x)\dans I\times I$ is equivalent to Lebesgue. But the superposition property implies that $\nu^{cu}_{n,x}|_{\zeta^{cu}(x)}$ is equivalent to $\nu^{cu}_x$, which as a consequence must be equivalent to Lebesgue. Using that $\Phi_x$ is a $C^1$-diffeomorphism so preserve Lebesgue class, we deduce that $\mu^{cu}_x$ is equivalent to Lebesgue for $\mu$-a.e. $x\in\TT$. This proves that $\mu$ is SRB.

\subsection{Basic moves for leaf-wise measures}\label{ss.basic}

We now proceed with the proof of  Proposition~\ref{p.basicmoves}. Our goal is to investigate how the measures $\hat{\nu}^c_x$ change as we move the base point $x$. In this analysis, the understanding of how the normal form coordinates change with the base point on the same center-unstable leaf plays a crucial role. 
%In the sequence of lemmas below we shall see that the proportionality class of leaf-wise measures remains the same \emph{up to some linear or affine map} when we move appropriately the base point.  

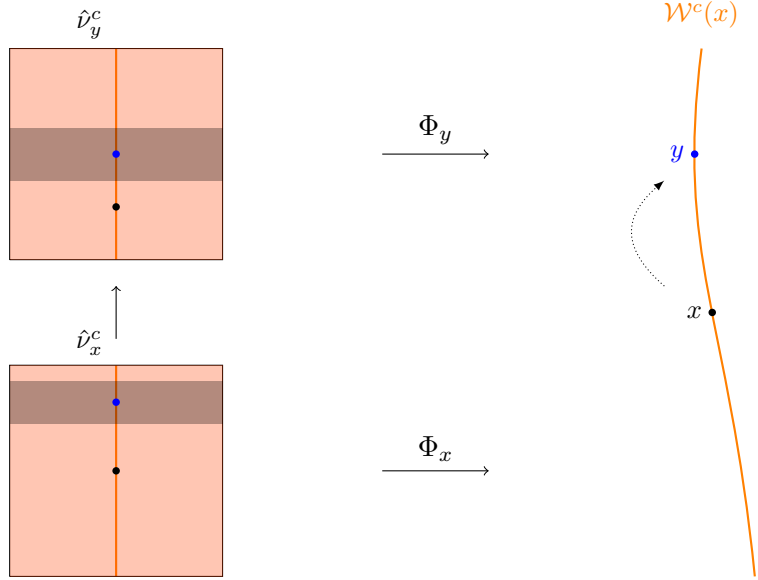
\begin{figure}[h]
	\centering
	\begin{tikzpicture}[scale=.7] 
		\draw[draw=orange, thick] (-8,-2)--(-8,2);
		\draw[draw=orange, thick] (-8,-8)--(-8,-4);
		\draw[->,draw=black] (-8,-3.5)--(-8,-2.5);
		\draw (-10,-2)--(-10,2)--(-6,2)--(-6,-2)--(-10,-2);
		\fill[red!50!orange,opacity=.3] (-10,-2)--(-10,2)--(-6,2)--(-6,-2)--(-10,-2); 
		\fill[black,opacity=.3] (-10,-0.5)--(-10,0.5)--(-6,0.5)--(-6,-0.5)--(-10,-0.5); 
		\draw (-10,-8)--(-10,-4)--(-6,-4)--(-6,-8)--(-10,-8);
		\fill[red!50!orange,opacity=.3] (-10,-8)--(-10,-4)--(-6,-4)--(-6,-8)--(-10,-8);
		\fill[black,opacity=.3] (-10,-5.1)--(-10,-4.3)--(-6,-4.3)--(-6,-5.1)--(-10,-5.1);
		\draw[->,draw=black] (-3,0)--(-1,0) node[midway,above]{$\Phi_y$};
		\draw[->,draw=black] (-3,-6)--(-1,-6) node[midway,above]{$\Phi_x$};
		\draw[draw=orange, thick] (4,-8)  .. controls (3.5,-3) and (2.5,-2) .. (3,2);
		\draw[draw=black, densely dotted, -latex] (2.3,-2.5)  .. controls (1.5,-1.8) and (1.5,-1.2) .. (2.3,-0.5);
		\fill[fill=blue] (2.87,0) node[left]{\color{blue}\small $y$} circle (2pt); 
		\fill[fill=black] (3.2,-3) node[left]{\small $x$} circle (2pt); 
		\fill[fill=black] (3,2.2) node[above]{\color{orange}\small $\mathcal{W}^c(x)$};
		\fill[fill=black] (-8.5,2) node[above]{\small $\hat \nu_y^c$};  
		\fill[fill=black] (-8.5,-4) node[above]{\small $\hat \nu_x^c$};  
		\fill[fill=blue] (-8,0) node[left]{} circle (2pt);
		\fill[fill=black] (-8,-1) node[left]{} circle (2pt);
		\fill[fill=black] (-8,-6) node[left]{} circle (2pt);
		\fill[fill=blue] (-8,-4.7) node[left]{} circle (2pt);
	\end{tikzpicture}
	\caption{\label{first basic move} Moving the base point within a central leaf changes the leaf-wise measure by an affine map up to renormalization.}
\end{figure}

\subsubsection{First basic move: moving along center manifolds}\label{sss.center}

In the following lemma we show that for points on the same center leaf there exists an affine map that makes the corresponding leaf-wise quotient measure proportional to each other: see Figure \ref{first basic move}. It is important to note that the derivative of this affine map is determined by normal form coordinates along the center manifold.
	\begin{lemma}
		\label{chanfemt centrale}
		For $\mu$-a.e. $x\in \TT$, and $y\in \mathcal{W}^c(x)$, we have
		$$
		%\propto 
		\hat \nu_{y}^c\propto  (h^2_{x,y})_* \hat \nu_{x}^c,
		$$ 
where $h^2_{x,y}\colon \R\to \R$ is the affine map defined by $s\mapsto \rho_y^c(x)\cdot s+\mathcal{H}_{y}^c(x)$, corresponding to the action of the change of normal forms on the second coordinate given by Lemma \ref{l.change_center_charts}. 
	\end{lemma}
	
	\begin{proof}
		Let $I$ be a bounded  interval centered at $0$ and $B$ be a bounded Borel subset of $\R$. Consider $n\in\N$, large enough so that $y\in\xi^{cu}_n(x)$ and 
		\[
		\Phi_y(I\times B)\cup\Phi_y(I\times [-1,1])\dans\xi^{cu}_n(y)=\xi^{cu}_n(x).
		\]
       This is possible after Lemma \ref{lem.af1}. Recall the normal form coordinate change $\cH_{x,y}(t,s)=\cH_y\circ\Phi_x=(h^1_{x,y}(t,s),h^2_{x,y}(s))$, which is given by (see Lemma~\ref{l.change_center_charts})
       \[
       h^1_{x,y}(t,s)=\rho^u_y(x)t\:\:\:\textrm{and}\:\:\:h^2_{x,y}(s)=\rho^c_y(x)s+\cH^c_y(x).
       \]
       Let $\tilde{B}=h^2_{x,y}(B)$, $I_1=[-1,1]$ and $\tilde{I}_1=h^2_{x,y}(I_1)$. Then, we can apply Lemma~\ref{Lebesg} as before to argue that
       \begin{align}
       	\hat{\nu}^c_x(B)&=\frac{\mu^{cu}_{n,x}(\Phi_x(I\times B))}{\mu^{cu}_{n,x}(\Phi_x(I\times I_1))}=\frac{\mu^{cu}_{n,y}(\Phi_y\circ\cH_{x,y}(I\times B))}{\mu^{cu}_{n,y}(\Phi\circ\cH_{x,y}(I\times I_1))}\nonumber\\
       	&=\frac{\nu^{cu}_{n,y}(\rho^u_y(x)I\times\tilde{B})}{\nu^{cu}_{n,y}(\rho^u_y(x)I\times\tilde{I}_1)}=\frac{\int_{\tilde{B}}\nu^u_{n,s}(\rho^u_y(x)I)d\nu^c_{n,x}(s)}{\int_{\tilde{I}_1}\nu^u_{n,s}(\rho^u_y(x)I)d\nu^c_{n,x}(s)}\nonumber\\
       	&=\frac{\int_{\tilde{B}}\nu^u_{n,s}(I)d\nu^c_{n,x}(s)}{\int_{\tilde{I}_1}\nu^u_{n,s}(I)d\nu^c_{n,x}(s)}=\frac{\nu^{cu}_{n,x}(I\times\tilde{B})}{\nu^{cu}_{n,x}(I\times\tilde{I}_1)}\nonumber\\
       	&=\alpha(x,y)\hat{\nu}^c_y(\tilde{B}),\nonumber
       \end{align}	 
		where $\alpha(x,y)=\frac{\nu^{cu}_{n,x}(I\times I_1)}{\nu^{cu}_{n,x}(I\times\tilde{I}_1)}$. This is exactly the desired conclusion. 
		\end{proof}

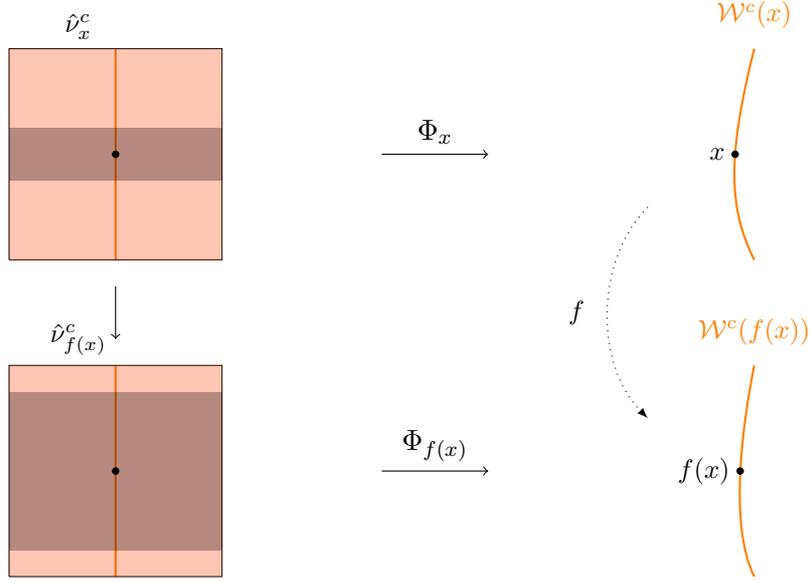
\begin{figure}[h]
	\centering
	\begin{tikzpicture}[scale=.7] 
		\draw[draw=orange, thick] (-8,-2)--(-8,2);
		\draw[draw=orange, thick] (-8,-8)--(-8,-4);
		\draw[->,draw=black] (-8,-2.5)--(-8,-3.5);
		\draw (-10,-2)--(-10,2)--(-6,2)--(-6,-2)--(-10,-2);
		\fill[red!50!orange,opacity=.3] (-10,-2)--(-10,2)--(-6,2)--(-6,-2)--(-10,-2); 
		\fill[black,opacity=.3] (-10,-0.5)--(-10,0.5)--(-6,0.5)--(-6,-0.5)--(-10,-0.5); 
		\draw (-10,-8)--(-10,-4)--(-6,-4)--(-6,-8)--(-10,-8);
		\fill[red!50!orange,opacity=.3] (-10,-8)--(-10,-4)--(-6,-4)--(-6,-8)--(-10,-8);
		\fill[black,opacity=.3] (-10,-7.5)--(-10,-4.5)--(-6,-4.5)--(-6,-7.5)--(-10,-7.5);
		\draw[->,draw=black] (-3,0)--(-1,0) node[midway,above]{$\Phi_x$};
		\draw[->,draw=black] (-3,-6)--(-1,-6) node[midway,above]{$\Phi_{f(x)}$};  
		\draw[draw=orange, thick] (4,-2)  .. controls (3.5,-1) and (3.5,0) .. (4,2);
		\fill[fill=black] (3.64,0) node[left]{\small $x$} circle (2pt); 
		\fill[fill=black] (4,2.2) node[above]{\color{orange}\small $\mathcal{W}^c(x)$};
		\draw[draw=orange, thick] (4,-8)  .. controls (3.5,-7) and (3.8,-5) .. (4,-4);
		\fill[fill=black] (3.735,-6) node[left]{\small $f(x)$} circle (2pt); 
		\fill[fill=black] (1,-3) node[left]{\small $f$};
		\fill[fill=black] (4,-3.8) node[above]{\color{orange}\small $\mathcal{W}^c(f(x))$};
		\draw[draw=black, dotted, -latex] (2,-1)  .. controls (1,-2) and (1,-4) .. (2,-5);
		\fill[fill=black] (-8.7,2) node[above]{\small $\hat \nu_x^c$};  
		\fill[fill=black] (-8.7,-4) node[above]{\small $\hat \nu_{f(x)}^c$};  
		\fill[fill=black] (-8,0) node[left]{} circle (2pt);
		\fill[fill=black] (-8,-6) node[left]{} circle (2pt); 
	\end{tikzpicture}
	\caption{\label{second basic move} Applying the dynamics to the base point changes the leaf-wise measure by a linear map given by the differential of $f$ along $E^c$ up to renormalization.}
\end{figure}

\subsubsection{Second basic move: applying the dynamics}\label{sss.appl_dyn}
When we push $x$ to $f(x)$ the change on the measures is linear and the slope of the linear map is given by the derivative at $x$ along the center direction.   Recall that for $*=c,u$, we denote by $\Lambda_x^*$ the linear map $ t \mapsto \lambda_x^*\cdot t$ (where $\lambda_x^*\eqdef  \|Df(x)|_{E^*}\|$). 
\begin{lemma}\label{chang dyna}
	For $\mu$-a.e. $x \in \TT$, we have 
	$$ 
	\hat \nu_{f(x)}^c\propto (\Lambda_x^c)_* \hat \nu_{x}^c. 
	$$ 
\end{lemma}

\begin{proof}
	
		Let $I$ be a bounded interval centered at $0$ and $B$ be a bounded Borel subset of $\R$. Fix $n \in \N$ large so that $I\times B,I\times[-1,1]\dans\zeta^{cu}_{n+1}(f(x))$ (see Lemma \ref{lem.af1}). By definition, we have 
		$$
		\hat\nu_{n+1,f(x)}^c(B)=\frac{\nu_{n+1,f(x)}^{cu} (I\times B)}{\nu_{n+1,f(x)}^{cu} (I\times  [-1,1])}, 
		$$ 
		and by Lemmas \ref{l_inv_con_meas trois} and \ref{Lebesg}
		\begin{align}
		\frac{\nu_{n+1,f(x)}^{cu} (I\times B)}{\nu_{n+1,f(x)}^{cu} (I\times  [-1,1])}&=\frac{\nu_{n,x}^{cu} (N_x^{-1}(I\times B))}{\nu_{n,x}^{cu} (N_x^{-1}(I\times [-1,1]))}\nonumber\\
		&=\frac{\int_{(\Lambda^c_x)^{-1}(B)}\nu^u_{n,s}((\lambda^u_x)^{-1}I)d\nu^c_{n,x}(s)}{\int_{(\Lambda^c_x)^{-1}([-1,1])}\nu^u_{n,s}((\lambda^u_x)^{-1}I)d\nu^c_{n,x}(s)}\nonumber\\
		&=\frac{\nu^{cu}_{n,x}(I\times(\Lambda^c_x)^{-1}(B)}{\nu^{cu}_{n,x}(I\times(\Lambda^c_x)^{-1}([-1,1]))}\nonumber\\
		&=\alpha(x)\hat{\nu}^c_x((\Lambda^c_x)^{-1}(B)),
		\end{align}
		where $\alpha(x)=\hat{\nu}^c_x((\Lambda^c_x)^{-1}([-1,1]))^{-1}$. We deduce that
		$$\hat \nu_{f(x)}^c(B)\propto (\Lambda_x^c)_\ast \hat \nu_{x}^c(B).\qedhere$$
\end{proof}

\begin{figure}[h!]
	\centering
	\begin{tikzpicture}[scale=.7] 
		\draw[draw=orange, thick] (-8,-2)--(-8,2); 
		\draw[draw=orange, thick] (4,-2)--(4,2);
		\draw[draw=red!80!black, thick] (-10,0)--(-6,0);  
		\draw[draw=red!80!black, thick] (2,0)--(6,0); 
		\draw[->,draw=black] (-8,-3.5)--(-8,-4.5) node[midway, left]{$\Phi_x$};
		\draw[->,draw=black] (4,-3.5)--(4,-4.5) node[midway, right]{$\Phi_{x'}$};
		\draw (-10,-2)--(-10,2)--(-6,2)--(-6,-2)--(-10,-2);
		\fill[red!50!orange,opacity=.3] (-10,-2)--(-10,2)--(-6,2)--(-6,-2)--(-10,-2); 
		\fill[black,opacity=.3] (-10,-0.5)--(-10,0.5)--(-6,0.5)--(-6,-0.5)--(-10,-0.5); 
		\fill[red!50!orange,opacity=.3] (2,-2)--(2,2)--(6,2)--(6,-2)--(2,-2); 
		\fill[black,opacity=.3] (2,-0.6)--(2,0.6)--(6,0.6)--(6,-0.6)--(2,-0.6); 
		\draw (2,-2)--(2,2)--(6,2)--(6,-2)--(2,-2); 
		\draw[->,draw=black] (-3,0)--(-1,0); 
		\fill[fill=blue] (4,0) node[left]{} circle (2pt); 
		\fill[fill=black] (2.7,0) node[left]{} circle (2pt); 
		\fill[fill=black] (-8.5,2) node[above]{\small $\hat \nu_x^c$};  
		\fill[fill=black] (3.5,2) node[above]{\small $\hat \nu_{x'}^c$};   
		\draw[draw=red!80!black, thick, densely dotted] (-10,-6.7)  .. controls (-7,-6.1) and (0,-6.3) .. (5,-6.1);
		\draw[draw=red!80!black, thick] (-10,-7)  .. controls (-7,-6.5) and (0,-6.3) .. (5,-7);
		\draw[draw=red!80!black, thick, densely dotted] (-10,-7.3)  .. controls (-7,-6.7) and (0,-6.8) .. (5,-7.7);
		\fill[fill=black] (-6,-7) node[left]{$x$}; 
		\fill[fill=blue] (1,-7) node[left]{\color{blue} $x'$}; 
		\draw[draw=black, densely dotted, -latex] (-5,-6)  .. controls (-3,-4) and (-2,-4) .. (0,-6);
		\fill[fill=red!80!black] (-10,-7) node[left]{\color{red!80!black} \small $\mathcal{W}^u(x)$}; 
		\draw[draw=green!40!black, thick] (-6.3,-6)  .. controls (-5.8,-7) and (-5.8,-7) .. (-5.9,-7.7);
		\draw[draw=green!40!black, thick] (0.8,-6)  .. controls (1.15,-7) and (1.15,-7) .. (1,-7.7);
		\fill[fill=black] (-8,0) node[left]{} circle (2pt);
		\fill[fill=blue] (-7,0) node[left]{} circle (2pt); 
		\fill[fill=black] (-6,-6.62) circle (2pt); 
		\fill[fill=blue] (1,-6.63) circle (2pt); 
	\end{tikzpicture}
	\caption{\label{third basic move} Moving the base point within an unstable leaf changes the leaf-wise measure by a linear map %given by unstable holonomy 
		up to renormalization.}
\end{figure}
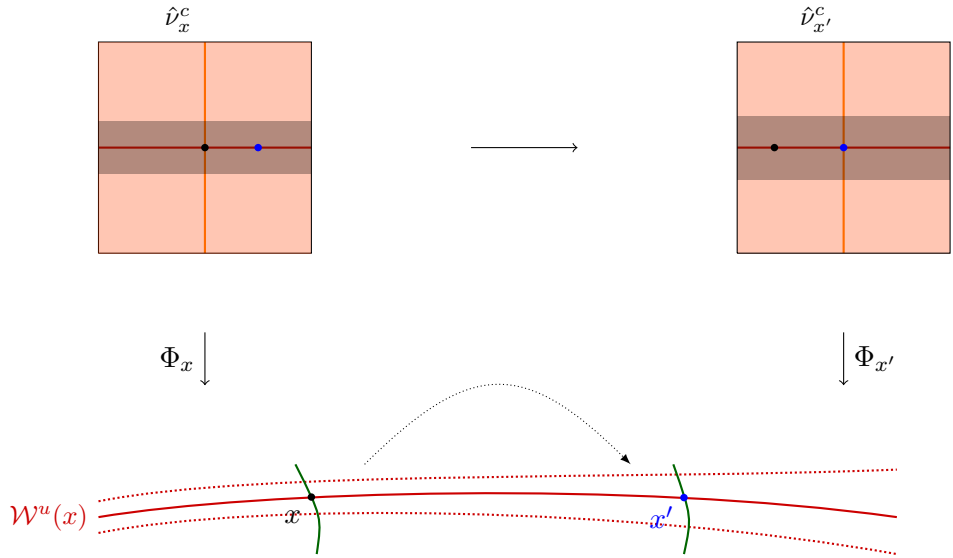

\subsubsection{Third basic move: moving along unstable manifolds} For points on the same strong unstable leaf the change is also linear, but the slope is determined by the derivative of the unstable holonomy map, which coincides with the derivative of the linear map giving the action of the change in normal form coordinates on the second variable. The lemma below completes the proof of Proposition~\ref{p.basicmoves}. 

\begin{lemma}\label{changt inst}
		For $\mu$-a.e. $x \in \TT$, and $x'\in  \cW^u(x)$, it holds 
		$$
		\hat \nu_{x}^c\propto(L^u_{x,x'})_* \hat \nu_{x'}^c,
		$$
		for the linear map $L^u_{x,x'}\colon s\mapsto \rho_{x'}^c(x)\cdot s$ corresponding to the action of the change of normal coordinates on the second variable given by Lemmas~\ref{lem.lindoefacil} and \ref{lem.calculholonomie}.
	\end{lemma}
	\begin{proof}
		For $\mu$-a.e. $x\in \TT$, and $x'\in  \cW^u(x)$, we proceed as follows. Let $I$ be any bounded interval centered at $0$ and $B$ be any bounded Borel subset of $\R$. Then, there exists $n$ large so that $x'\in\xi_n^{cu}(x)$ and $I\times B,I\times[-1,1]\times B\dans\zeta^{cu}_n(x)$. This is possible by Lemma \ref{lem.af1}.
		As before, let $\cH_{x,x'}(t,s)=(h^1_{x,x'}(t,s),h^2_{x,x'}(s))$ be the change in normal form coordinate. By Lemmas~\ref{chang of coord}, \ref{lem.lindoefacil} and \ref{lem.calculholonomie} we have
		\[
		h^1_{x,x'}(t,s)=\rho^u_{x'}(x)t+a_{x,x'}(s),\:\:\:\textrm{and}\:\:\:h^2_{x,x'}(s)=\rho^c_{x'}(x)s.
		\]
		Now, introduce $\tilde{B}=h^2_{x,x'}(B)$, $\tilde{I}_1=h^2_{x,x'}(I_1)$, where $I_1=[-1,1]$. By Lemma~\ref{Lebesg}:
		\begin{align}
			\hat{\nu}^c_{x'}(B)&=\frac{\mu^{cu}_{n,x}(\Phi_x\circ\cH_{x,x'}(I\times B))}{\mu^{cu}_{n,x}(\Phi_x\circ\cH_{x,x'}(I\times I_1))}\nonumber\\
			&=\frac{\nu^{cu}_{n,x}(h^1_{x,x'}(I\times B)\times\tilde{B})}{\nu^{cu}_{n,x}(h^1_{x,x'}(I\times I_1)\times\tilde{I}_1)}\nonumber\\
			&=\frac{\int_{\tilde{B}}\nu^u_{n,s}(\rho^u_{x'}(x)I+a_{x,x'}(\rho^c_x(x')s))d\nu^c_{n,x}(s)}{\int_{\tilde{I}_1}\nu^u_{n,s}(\rho^u_{x'}(x)I+a_{x,x'}(\rho^c_x(x')s))d\nu^c_{n,x}(s)}\nonumber\\
			&=\frac{\int_{\tilde{B}}\nu^u_{n,s}(I)d\nu^c_{n,x}(s)}{\int_{\tilde{I}_1}\nu^u_{n,s}(I)d\nu^c_{n,x}(s)}\nonumber\\
			&=\frac{\nu^{cu}_{n,x}(I\times\tilde{B})}{\nu^{cu}_{n,x}(I\times\tilde{I}_1)}\nonumber\\
			&=\alpha(x,x')\hat{\nu}^c_x(\tilde{B}),
		\end{align}
	where $\alpha(x,x')=\left(\hat{\nu}^c_x(\tilde{I}_1)\right)^{-1}$. This completes the proof of the lemma.
	\end{proof}

\section{Invariance by affine maps}\label{s.Main_texh_lemma}
%
%Theorem \ref{mainthm.fullsupport} is a consequence of Theorem \ref{mainthm.technique} and Proposition~\ref{p.mesuredubadset}.
%
%\subsubsection{Proof of Theorem~\ref{mainthm.fullsupport}}\label{sub subs proof}
% 
% Recall that by Lemma \ref{c one holon}, for every $g\in\cA^2_m(\TT)$, there exists a small neighbourhood $\cU(g)$ in $\dif$ so that every $f\in\cU(g)$ is an Anosov diffeomorphism, strongly partially hyperbolic with expanding center and $C^1$ stable holonomies. 
%Let $\cU=\cup_{g\in\cA^2_m(\TT)}\cU(g)$. Take $f\in\cU$. Then, $f$ fulfils the assumptions of Theorem~\ref{mainthm.technique}. Assume that $E^s$ and $E^u$ are not jointly integrable, and let $\mu$ be a fully supported ergodic $u$-Gibbs measure. Then, Proposition~\ref{p.mesuredubadset} implies $\mu(\mathbf{B}_\mu)=0$ and Theorem~\ref{mainthm.technique} shows that $\mu$ is SRB, concluding the proof. \qed 
%

Once we have constructed the \textit{leaf-wise quotient measures} $\{\hat{\nu}^c_x\}_{x\in\TT}$, Theorem~\ref{t.leafwise} tells us that the proof of Theorem~\ref{mainthm.technique} reduces to show that $\hat{\nu}^c_x$ is a multiple of the Lebesgue measure of the real line.  As in \cite{BRH} this can be achieved by proving that $\hat{\nu}^c_x$ is, for many points $x$, invariant by affine maps with controlled slope and small translational part.  More precisely, the lemma below is the analogue of \cite[Proposition  7.1]{BRH} in our context.

\begin{lemma}
	\label{l.lema6.1}
	There exist constants $M_0>0$ and $\delta_0\in(0,1)$ such that for every $\eps>0$ sufficiently small one can find a compact set $G(\eps)\subset\mathbb{T}^3$ so that $\mu(G(\eps))\geq\delta_0$ and for every $p\in G(\eps)$ there exists an affine map $\psi\colon\R\to\R$ satisfying 
	\begin{enumerate}
		\item $\frac{1}{M_0}<|\psi^{\prime}(0)|<M_0$;
		\item $\frac{\eps}{M_0}<|\psi(0)|<\eps M_0$;
		\item $\psi_{*}\hat \nu_p^c\propto\hat \nu_p^c$.
	\end{enumerate} 
Furthermore, writing $G_0\eqdef \{p \in \TT: p \in G(\frac 1N)\text{ for infinitely many } N \in \N\}$, we have $\mu(G_0)\geq\delta_0$. 
\end{lemma}

Exactly as in \cite{BRH}, the lemma above implies Theorem~\ref{mainthm.technique}. The proof  is a direct adaptation of \cite[Lemma 3.10]{KalininKatok} and \cite[Lemma 7.3]{BRH}; for the convenience of the reader, we provide this beautiful argument below and we refer to \cite{BRH,KalininKatok} for other applications of the same argument.    
\begin{proof}[Proof of Theorem~\ref{mainthm.technique} assuming Lemma~\ref{l.lema6.1}]
	%The proof is a direct adaptation of \cite[Lemma 3.10]{KK} and \cite[Lemma 7.3]{BRH}, and we refer to these papers for the details. 
	%; for the convenience of the reader, let us outline the main steps of the proof. 
	Let $\mathrm{Aff}(\R)$ denote the group of invertible affine maps of $\R$, and for $p \in \TT$, let  $\mathcal{A}(p)\subset \mathrm{Aff}(\R)$ be the subgroup of affine maps $\psi\colon \R \to \R$ such that $\psi_* \hat \nu_p^c \propto {\color{blue}\hat \nu_p^c}$. 
	
	We claim that $\cA(p)$ is a closed subgroup of $\mathrm{Aff}(\R)$. Let $\psi_n\in\cA(p)$ converge to $\psi\in\mathrm{Aff}(\R)$. Since each element of $\mathrm{Aff}(\R)$ is a homeomorphism of the real line and since the convergence in $\mathrm{Aff}(\R)$ implies converge in the compact-open topology we have, for each continuous function with compact support $\phi\colon\R\to\R$ that 
	\[
	\int\phi\circ\psi_n\, d\hat{\nu}^c_p\to\int\phi\circ\psi\, d\hat{\nu}^c_p.
	\] 
	This implies that $(\psi_n)_*\hat{\nu}^c_p\to\psi_*\hat{\nu}^c_p$. On the other hand, for each $n$ we have
	\begin{equation}
	\label{e.psin}
	(\psi_n)_*\hat{\nu}^c_p=c_n\hat{\nu}^c_p,
	\end{equation}
	for some constant $c_n$, which can therefore be obtained by 
	\[
	c_n=\frac{\hat{\nu}^c_p(\psi_n^{-1}(K))}{\hat{\nu}^c_p(K)},
	\]
	for any measurable set $K\subset\R$ with finite positive measure with respect to both $\hat{\nu}^c_p$ and $(\psi_n)_*\hat{\nu}^c_p$. Now, as all the measures are locally finite we can require further that such compact set $K$ is also a continuity set (of finite measure) for $\psi_*\hat{\nu}^c_p$. In particular, since $(\psi_n)_*\hat{\nu}^c_p(K)\to\psi_*\hat{\nu}^c_p(K)$, we deduce that $c_n$ converges to some positive real number $c$. We deduce from \eqref{e.psin} and uniqueness of limits that $\psi_*\hat{\nu}^c_p=c\hat{\nu}^c_p$. This proves that $\psi\in\cA(p)$ and establishes our claim. This implies in particular that $\cA(p)$ is a Lie group.
	
	By Lemma \ref{l.lema6.1}, for every $p \in G_0$, $\mathcal{A}(p)$ contains a sequence $(\psi_j)_{j \in \N}$ of elements $\psi_j \colon t \mapsto \lambda_j t + v_j$, with $v_j \neq 0$, $\lim_{j \to +\infty}v_j=0$, and $\lim_{j \to +\infty}\lambda_j=\lambda$, for some $\lambda \neq 0$. %Considering a converging subsequence $(\psi_{\varphi(j)})_{j \in \N}$, we deduce that 
	In particular, $\mathcal{A}(p)$ contains the homothety $h_\lambda \colon t \mapsto \lambda  t$, and $( h_\lambda^{-1}\circ\psi_{j})_{j \in \N}$ converges to the identity within $\mathcal{A}(p)$. Therefore, $\mathcal{A}(p)$ is not discrete and must be of dimension $1$ or $2$. This implies that the identity component $\mathcal{A}^0(p)\subset \mathcal{A}(p)$ contains a one-parameter subgroup of $\mathrm{Aff}(\R)$: such a group consists of translations, or is conjugate to homothety.
	
	We now claim that the groups $\cA^0(p)$ are isomorphic for a.e $p\in\TT$. Indeed, recall the linear map $\Lambda^c_p\colon \R\to\R$ given by $\Lambda^c_p(t)=\lambda^c_pt$, where $\lambda^c_p=\|Df(p)|_{E^c}\|$. Then, it follows from Lemma~\ref{chang dyna} that $\psi\in\cA^0(p)$ if, and only if,
	$\Lambda^c_p\circ\psi\circ(\Lambda^c_p)^{-1}\in\cA^0(f(p))$, proving our claim.
	
    Since isomorphisms classes of closed subgroups of $\mathrm{Aff}(\R)$ form a separable space and since $\mu(G_0)>0$, the claim then follows {\color{red}by} ergodicity. In particular, for $\mu$-a.e. $p\in \TT$, $\mathcal{A}^0(p)$ contains a one-parameter subgroup of $\mathrm{Aff}(\R)$. 
	
	Assume by contradiction that $\mathcal{A}^0(p)$ were conjugate to homothety for a positive measure set of $p\in \TT$. Then, by ergodicity, for $\mu$-a.e. $p \in\TT$, the action of $\mathcal{A}^0(p)$ on $\R$ would have a unique fixed point $t(p)\in \R$. Since $\cA^0(p)$ contains affine maps with arbitrarily small (and non-zero) translational part, we have $t(p)\neq 0$ for a set of positive measure. 
	
	Besides that, as we observed above
	$
	\cA^0(f^n(p))=\{\Lambda^c_{p,n}\circ\psi\circ(\Lambda^c_{p,n})^{-1}:\psi\in\cA^0(p)\},
	$	
	where $\Lambda^c_{p,n}(t)=\lambda^c_p(n)t$ (recall \eqref{e.derivada}). As a consequence we deduce
	\begin{equation}
	\label{e.explodecoração}
	|t(f^n(p))|=\lambda^c_p(n)\cdot |t(p)|.
	\end{equation}
	Consider a positive measure compact subset $K$ of $G_0$ where the measurable function $p\mapsto t(p)$ is continuous and bounded. By Poincaré recurrence for some $p\in K$ we have $p_{n_k}=f^{n_k}(p)\in K$ for infinitely many iterates $n_k\in\N$. By compactness, we can assume $p_{n_k}\to q\in K$. However, \eqref{e.explodecoração} is incompatible with the boundedness of $t|_K$. We have thus reached a contradiction.
	
	Therefore, $\mathcal{A}^0(p)$ contains the group of translations, for $\mu$-a.e. $p\in \TT$. For $t \in \R$, $p \in \TT$, we let $g=g_t\colon s \mapsto s+t$ and $c(p,t)\eqdef \hat \nu_p^c([-t-1,-t+1])$, so that 
	\begin{equation}\label{action transl derivee Radon Niko}
	\frac{dg_*\hat\nu_p^c}{d\hat\nu_p^c}=c(p,t). 
	\end{equation}
	
 Let us see some properties of the function $(p,t)\mapsto c(p,t)$. 
Firstly, for $\mu$-a.e. $p\in \TT$, $\mathcal{A}^0(p)$ contains all translations. This implies that $\hat\nu_p^c$ has no atom, which implies that $c(p,\cdot)$ is continuous. We claim that 
\begin{equation}
\label{e.ospesnochaoeacabeçanasnuvens}
c(p,t)=c(p_n,\lambda^c_p(n)t),
\end{equation}	
where $p_n=f^n(p)$. To prove the claim, consider $\psi=\Lambda^c_{p,n}\circ g\circ(\Lambda^c_{p,n})^{-1}$ and observe that $\psi(s)=s+\lambda^c_p(n)t$. From the definition it follows that
\[
c(p_n,\lambda^c_p(n))=\hat{\nu}^c_{p_n}(\psi^{-1}([-1,1]))=\frac{\hat{\nu}^c_{p_n}(\psi^{-1}([-1,1]))}{\hat{\nu}^c_{p_n}([-1,1])}.
\] 
In the last equality we have used our normalization choice $\hat{\nu}^c_{p_n}([-1,1])=1$ for leaf-wise quotient measures. Applying now Lemma~\ref{chanfemt centrale} we deduce
\begin{align}
\frac{\hat{\nu}^c_{p_n}(\psi^{-1}([-1,1]))}{\hat{\nu}^c_{p_n}([-1,1])}&=\frac{\hat{\nu}^c_{p}((\lambda^c_p(n))^{-1}\times\psi^{-1}([-1,1]))}{\hat{\nu}^c_{p}((\lambda^c_p(n))^{-1}\times[-1,1])}\nonumber\\
&=\frac{\hat{\nu}^c_{p}(g^{-1}(\lambda^c_p(n))^{-1}\times([-1,1])))}{\hat{\nu}^c_{p}((\lambda^c_p(n))^{-1}\times[-1,1])}\nonumber\\
&=\frac{g_*\hat{\nu}^c_{p}(\lambda^c_p(n))^{-1}\times([-1,1]))}{\hat{\nu}^c_{p}((\lambda^c_p(n))^{-1}\times[-1,1])}\nonumber\\
&=c(p,t),
\end{align}
where on the second equality we applied the definition of $\psi$ and in the last equality we have used the fact that the measures $g_*\hat{\nu}^c_p$ and $\hat{\nu}^c_p$ are proportional by a factor precisely equal to $c(p,t)$. This establishes \eqref{e.ospesnochaoeacabeçanasnuvens}.	

We now give the final argument for completing the proof. For each $\varepsilon>0$, let $r>0$ be chosen such that the set
	$$
	B_{r,\varepsilon}\eqdef \big\{p\in \TT: |c(q,t)-1|<\varepsilon,\quad \forall\, |t|<r\big\}
	$$
	satisfies $\mu(B_{r,\varepsilon})>0$. By ergodicity, $\mu$-a.e. point $p\in \TT$ visits $B_{r,\varepsilon}$ infinitely many times both in future and past; but \eqref{e.ospesnochaoeacabeçanasnuvens} implies that $|c(p,t)-1|<\varepsilon$ for all {\color{blue}$t\in \R$} and $\mu$-a.e. $p \in \TT$. Letting $\varepsilon\to 0$, we conclude that $c(p,t)=1$ for all $t\in \R$ and $\mu$-a.e. $p \in \TT$, hence, by \eqref{action transl derivee Radon Niko}, $\hat\nu_p^c$ is invariant under the group of translations, for $\mu$-a.e. $p\in \TT$, hence it is proportional to Lebesgue measure.  
\end{proof}

\subsection{Drift along the center}

Recall that the normal forms $\{\Phi^c_x\colon\R\to\cW^c(x)\}_{x\in\TT}$ give us parametrizations of the center manifolds whose change of coordinates are affine maps. Using these changes of coordinates one can build the maps $\psi\colon\R\to\R$ promised in Lemma~\ref{l.lema6.1}. Indeed, we claim that it suffices to prove the result below.

\begin{prop}
	\label{p.drift}
	There exist constants $M>0$ and $\delta_0\in(0,1)$ such that for every $\eps>0$ sufficiently small one can find a compact set $G=G(\eps)\subset\mathbb{T}^3$ so that $\mu(G)\geq\delta_0$ and for every $p\in G$ there exists a point $q\in \cW_1^c(p)$, so that 
	\begin{equation}
	\label{e.drift}
	M^{-1}\eps\leq |\cH^c_p(q)|\leq M\eps\:\:\textrm{and}\:\:\:\hat \nu^c_{q}\propto B_{*}\hat \nu^c_{p},
	\end{equation}
	for some linear map $B\colon\R\to\R$ of the form $s\mapsto \beta \cdot s$ so that $M^{-1}<|\beta|<M$. 
	
\end{prop}
\subsubsection{Proof of Lemma~\ref{l.lema6.1} assuming Proposition~\ref{p.drift}}

We only need to show that the set $G$ of Proposition~\ref{p.drift} satisfies the claims in Lemma~\ref{l.lema6.1} for a suitably chosen constant $M$ (which will be perhaps a bit larger than the one already given by the proposition). To see this, take $q\in \cW_1^c(p)$ for some $p\in\mathbb{T}^3$ and let us, for the sake of this proof, denote by $\xi=\cH^c_{p,q}=\cH^c_p\circ\Phi^c_{q}\colon\R\to\R$ the affine change of normal form coordinates along the center direction.
Because normal forms depends continuously with respect to the base point, we
may enlarge $M$ if necessary so that $|\xi'(0)|\in (M^{-1},M)$. 

Now, take $G$ the set given by Proposition~\ref{p.drift} and let $p\in G$. We shall construct the affine map $\psi\colon\R\to\R$ claimed by the lemma. For this take $q\in \cW^c_1(p)$ given by Proposition~\ref{p.drift} and notice that, at one hand, Lemma~\ref{chanfemt centrale} gives us that $\hat{\nu}^c_{q}\propto\xi_{*}\hat{\nu}^c_p$. On the other hand, Proposition~\ref{p.drift} says that for some linear map $B\colon\R\to\R$ with derivative $\beta$ bounded in between $M^{-1}$ and $M$ we have $\hat{\nu}^c_{q}\propto B_{*}\hat{\nu}^c_p$. These two properties give us that $\hat{\nu}^c_p\propto\psi_{*}\hat{\nu}^c_p$, where $\psi=B^{-1}\xi$ is an affine map. The bounds we have on $\beta$ and on $|\xi'(0)|$ give
\[
|\psi'(0)|\in(M^{-2},M^2).
\]
Moreover, by Proposition~\ref{p.drift}
\[
|\psi(0)|=|B^{-1}\xi(0)|=|B^{-1}\cH^c_p(q)|\in(M^{-2}\eps,M^2\eps).
\]
This shows that $\psi$ satisfies all the requirements in Lemma~\ref{l.lema6.1} with $M_0=M^2$, thus completing the proof.\qed

\subsubsection{The Lusin set}\label{Lusin lusin}
We now move on to the proof of Proposition~\ref{p.drift}, where our key arguments are concentrated. Our approach here draws inspiration from Eskin-Lindenstrauss' work \cite{EskinLind}. To offer some insight of our reasonning, recall from Lemma~\ref{chang dyna} that the measures $\hat{\nu}^c_x$ change linearly when we shift the base point from $x$ to $f^n(x)$, with the slope of this linear map being $\|Df^n(x)|_{E^c}\|$. Our strategy involves identifying specific dynamical configurations in which points, \emph{almost} on the same center leaf, drift apart by a center distance proportional to $\eps$. In these configurations, their corresponding leaf-wise quotient measures are \emph{almost} proportional to each other, modulo a linear map whose slope can be controlled using the dynamics. To pinpoint the precise set where ``almost'' turns into ``equality'', we take limits. To achieve this we must consider points belonging to compact sets restricted to which the assignment $x\mapsto\hat{\nu}^c_x$ has good properties.  

%\begin{notation}\label{notation constantes k j}
%	Let us introduce some notation, based on Lemmas \ref{chanfemt centrale}-\ref{chang dyna}-\ref{changt inst}. For $\mu$-a.e. $x \in  \TT$, we let $J_x>0$ be the constant such that (recall Lemma \ref{chang dyna}) 
%	\begin{equation}\label{constant k x}
%	\hat \nu_{f(x)}^c=J_x\times (\Lambda_x^c)_* \hat \nu_{x}^c. 
%	\end{equation}
%	For $\mu$-a.e. $x\in \TT$, and $y=\Phi_x(s)\in \mathcal{W}^c(x)$ with $|s|$ small, we let $K_{x,y}^c>0$ be the constant such that (recall Lemma \ref{chanfemt centrale})
%	$$ 
%	\hat \nu_{y}^c=K_{x,y}^c \times (\psi_{x,y})_* \hat \nu_{x}^c. 
%	$$ 
%	Finally, for $\mu$-a.e. $x \in  \TT$, and for $x'\in \cW^u(x)$, we let $K_{x,x'}^u>0$ be the constant such that (recall Lemma \ref{changt inst}) 
%	\begin{equation}\label{constant j x x prime}
%	\hat \nu_{x'}^c=K_{x,x'}^u \times(L_{x,x'}^c)_* \hat \nu_{x}^c. 
%	\end{equation}
%\end{notation}

With that goal in mind, we denote by $C^0_c(\R)$ the space of continuous functions with compact support of the real line. 

\begin{lemma}
	\label{l.lusin}
For every $\delta>0$ there exists a compact set $A\subset\TT$ with $\mu(A)>1-\delta$ such that for all $\{x_n\}_{n\in\N}\dans A$ converging to $x\in A$ the following holds 
\[
\int\varphi d\hat{\nu}^c_{x_n}\to\int\varphi d\hat{\nu}^c_x,\:\:\:\:\textrm{for every}\:\:\varphi\in C^0_c(\R).
\]
\end{lemma}
\begin{proof}
Let us denote $K\eqdef\operatorname{supp}(\varphi)$. We first claim that, given  $\varphi\in C^0_c(\R)$, the function $x\in\TT\mapsto\int\varphi d\hat{\nu}^c_x\in\R$ is measurable.  Indeed, consider the set
\[
\Lambda(\eps)=\{x\in\TT;\cW^c_{\eps}(x)\subset\xi^c_0(x)\}.
\]
Since $\mu(\Lambda(\eps))\to 1$ as $\eps\to 0$ it suffices to check that, for a given $\eps>0$, the restriction of the function to the corresponding set $\Lambda(\eps)$ is measurable (notice that the set $\Lambda(\eps)$ itself is measurable). By compactness of $K$ and uniform continuity of the $C^1$-norm of our normal forms in compact sets, Lemma \ref{lem.af1} ensures the existence of an $n\in\N$ such that
\begin{equation}
\label{e.knozeta}
K\subset\zeta^c_n(x),\:\:\:\textrm{for every}\:\:\:x\in\Lambda(\eps).
\end{equation}
Define $\hat{\varphi}:\R^2\to\R$ as $\hat{\varphi}(t,s)=\varphi(s)$. Then, denoting $\alpha(x)=1/\nu^{cu}_{n,x}(I\times[-1,1])$, it follows from (\ref{e.knozeta}) that
\[
\int\varphi d\hat{\nu}^c_x=\alpha(x)\int\hat{\varphi}\circ\cH_x(z)d\mu^{cu}_{n,x}(z),
\]
for every $x\in\Lambda(\eps)$. Since the right-hand side above depends measurably on $x\in\Lambda(\eps)$ our claim is proved.

As a second step in our proof, we claim the existence of a countable dense subset $\{\varphi_\ell\}_{\ell\in\N}\subset C^0_c(\R)$ with the following property: for a sequence $\{x_n\}_{n\in\N}\dans\TT$ converging to a point $x\in\TT$, if
\[
\int\varphi_\ell d\hat{\nu}^c_{x_n}\to\int\varphi_\ell d\hat{\nu}^c_x,
\] 
for all $\ell\in\N$, then we can conclude that
\[
\int\varphi d\hat{\nu}^c_{x_n}\to\int\varphi d\hat{\nu}^c_x,
\] 
for every $\varphi\in C^0_c(\R)$. 

To construct such a dense subset, start with a countable dense subset $\cF$ of $C^0_c(\R)$ with the following property: if $\varphi$ is in $C^0_c(\R)$ and $K=\operatorname{supp}(\varphi)$ then define $\hat{K}=[\inf K-1,\sup K+1]$: for every $\eps>0$ small enough there exists $\psi\in\cF$ such that $\|\varphi-\psi\|_{\infty}<\eps$ and $\operatorname{supp}(\psi)\subset\hat{K}$. The existence of $\cF$ follows by standard arguments. 

Now, augment $\cF$ by adding, for each positive integer $n$, a continuous bump function that equals $1$ inside the interval $[-n,n]$ and $0$ outside $(-n-1,n+1)$. The resulting set, still denoted as $\cF$, remains countable and dense, but now has the property that for any compact set $K\subset\R$, there exists some function $\psi\in\cF$ such that $\psi(x)\geq 1$ for every $x\in K$. To show that $\cF$ meets the claim's requirements, consider $\varphi$ in $ C^0_c(\R)$ and let $K=\operatorname{supp}(\varphi)$. Choose $\psi\in\cF$ such that $\psi|_{\hat{K}}\geq 1$, where $\hat{K}=[\inf K-1,\sup K+1]$. Then,
\[
\hat{\nu}^c_{x_n}(\hat{K})\leq\int\psi d\hat{\nu}^c_{x_n}.
\]      
By assumption, the right-hand side above converges to $\int\psi d\hat{\nu}^c_x$. Thus, $\hat{\nu}^c_{x_n}(\hat{K})$ is bounded by some number $\beta(K)>0$, independent of $n$. Now, given $\eps>0$, choose $\hat{\psi}$ in the countable dense family such that 
\[
\|\hat{\psi}-\varphi\|_\infty<\frac{\eps}{3\beta(K)},
\]
and with $\operatorname{supp}(\hat{\psi})\subset\hat{K}$. 
Since $\hat{\psi}$ is in $\cF$ there exists $n_0\in\N$ such that if $n\geq n_0$  
\[
\left|\int\hat{\psi}d\hat{\nu}^c_{x_n}-\int\hat{\psi}d\hat{\nu}^c_x\right|<\frac{\eps}{3}.
\]
Thus, if $n\ge n_0$ we have
\begin{align}
	\left|\int\varphi d\hat{\nu}^c_{x_n}-\int\varphi d\hat{\nu}^c_x\right|&\leq\left|\int\varphi d\hat{\nu}^c_{x_n}-\int\hat{\psi}d\hat{\nu}^c_{x_n}\right|+\left|\int\hat{\psi}d\hat{\nu}^c_{x_n}-\int\hat{\psi}d\hat{\nu}^c_x\right|\nonumber\\
	&+\left|\int\hat{\psi}d\hat{\nu}^c_x-\int\varphi d\hat{\nu}^c_x\right|\nonumber\\
	&<\frac{\eps \hat{\nu}^c_{x_n}(\hat{K})}{3\beta(K)}+\frac{\eps}{3}+\frac{\eps\hat{\nu}^c_x(\hat{K})}{3\beta(K)}\nonumber\\
	&\leq\eps\nonumber.
\end{align}
This proves our second claim. To conclude the proof of the lemma, take $\cF=\{\varphi_\ell\}_{\ell\in\N}$ as the countable dense family of our previous claim. For each $\ell$, according to our first claim, the function $g_\ell:x\in\TT\mapsto\int\varphi_\ell d\hat{\nu}^c_x$ is measurable. By Lusin's theorem there exists a compact set $A_\ell\subset\TT$ with measure $\mu(A_\ell)>1-\delta/2^{\ell+1}$ and such that  $g_\ell|_{A_\ell}$ is continuous. Define the compact set $A=\bigcap_{\ell=0}^{\infty}A_\ell$. Then, $\mu(A)>1-\delta$ and if a sequence $\{x_n\}_{n\in\N}\subset A$ converges to $x$, we have $g_\ell(x_n)\to g_\ell(x)$ for every $\ell$, and the lemma follows from the above claims. 
\end{proof}
	
Let us take $\cL\subset\TT$ a compact set with large measure (we shall give precise estimates below) such that the conclusion of Lemma~\ref{l.lusin} holds for $A=\cL$. We shall require further the following properties:	
	
\begin{enumerate}
	\item For $\sigma=s,c,u,cu$ given $\xi^{\sigma}$, a measurable partition subordinate to $\cW^\sigma$, there exists $r_0>0$ so that for every $x\in\cL$, 
	$$\cW^\sigma_{r_0}(x)\dans\xi^\sigma(x).$$
	\item For every $x\in\cL$, $\xi^u(x)$ is an interval whose size $r^u(x)$ {\color{blue}varies} continuously with $x\in\cL$ (see Remark \ref{rem_subord_interval}).
%	\item the functions $x \mapsto J_x$, $(x,y)\mapsto K^c_{x,y}$, and $(x,x^{\prime})\mapsto K_{x,x'}^u$ introduced in Notation~\ref{notation constantes k j} 
	%$x\mapsto K_x$ where $K_x$ was introduced in Lemma~\ref{chang dyna} and $(x,y)\mapsto K^c_{x,y}$ from Lemma~\ref{chanfemt centrale} 
%	are all continuous in restriction to $\cL$.  
%	\item there exists $r_1>0$ such that is $0<s<r_1$ then $\hat{\nu}^c_x([-s,s])>0.$
\end{enumerate}

The continuity of the function $\cL\ni x\mapsto\hat{\nu}^c_x$ together with Remark \ref{remark.superposition2} imply the following.
\begin{corollary}
	\label{c.lusin}
For each $M>1$ there exists $c=c(M)>0$ such that if $B:\R\to\R$ is a linear map with derivative in $[M^{-1},M]$ then $\hat{\nu}^c_x(B^{-1}[-1,1])\in[c^{-1},c]$ for each $x\in\cL$.
\end{corollary}
\begin{proof}
Assume by contradiction that a lower bound does not hold. Denote $J_0=[-M^{-1},M^{-1}]$. Then, for each $n$ there must exist $x_n\dans\cL$ such that $\hat{\nu}^c_{x_n}(J_0)\leq 1/n$. We can assume, by compactness of $\cL$, that $x_n\to x\in\cL$. Since $y\in\cL\mapsto\hat{\nu}^c_y$ is continuous, we have that 
\[
\hat{\nu}^c_x(J_0)\leq\liminf_{n\to\infty}\hat{\nu}^c_{x_n}(J_0)=0,
\] 
which violates Remark \ref{remark.superposition2}. By a similar reasoning, considering $J_1=[-M,M]$ and assuming that the upper bound does not hold we would obtain a point $x\in\cL$ for which $\hat{\nu}^c_x(J_1)=\infty$, which is also impossible by local finiteness. 
\end{proof}

\subsubsection{The main proposition}\label{sss_exp-drift}

Let us fix two numbers $\delta$ and $\delta_0$ such that
$$0<\delta_0<\frac {1}{10},\,\,\,\,\,\,\,\,\,\,\,\text{and}\,\,\,\,\,\,\,\,\,\,\,\delta=3\delta_0.$$
We will require $\mu(\cL)> 1-\delta_0$. The constant $\delta_0$ is the constant appearing in Lemma \ref{l.lema6.1} and Proposition \ref{p.drift}. Further assumption will be given on $\delta$ and $\delta_0$ in Section \ref{s.end_proof}, and more precisely in \S \ref{delta_delta0}.

Given $M,\eps>0$ set
\begin{equation} \label{eq.definitiongem}
G(\eps,M)\eqdef
\left\{p\in\TT:
\begin{array}{l}\exists\, q\in\cW^c_1(p),\,\exists\, B(s)=\beta\cdot s\text{ linear map such that}\\
\hat \nu^c_{q}\propto B_{*}\hat \nu^c_{p}\\
 M^{-1}\eps\leq|\cH^c_p(q)|\leq M\eps\\
  M^{-1}\leq|\beta|\leq M
\end{array}
\right\}.
\end{equation}
Notice that $G(\eps,M)$ is nothing but the set of points $p\in\TT$ for which the conclusion of Proposition~\ref{p.drift} holds with constants $\eps,M$. Recall that in order to prove Proposition \ref{p.drift}, and thus Lemma \ref{l.lema6.1}, we must show that $\mu(G(\eps,M))\geq\delta_0$.

Given this choice of parameters, Proposition~\ref{p.drift} follows from the statement below. 

\begin{prop}
	\label{p.levraidrift}
	There exists $M>0$ such that for $\eps>0$ small enough and every compact set $K_{00}$ with $\mu(K_{00})>1-2\delta_0$,
	$$K_{00}\cap\cL\cap G(\eps, M)\neq\emptyset.$$
\end{prop} 

\begin{proof}[Proof that Proposition \ref{p.levraidrift} $\Longrightarrow$ Proposition \ref{p.drift}] Let $M>0$ be the constant given in the statement above. Assume by contradiction that Proposition~\ref{p.drift} does not hold with the constants $M$ and $\delta_0$. Then, for some small $\eps>0$ we must have $\mu(G(\eps,M))<\delta_0$. By regularity of $\mu$ there exists an open neighbourhood $U$ of $G(\eps,M)$ such that $\mu(U)<\delta_0$. Let $K_{00}\eqdef\TT\setminus U$. This set is compact and satisfies $\mu(K_{00})\geq 1-\delta_0>1-2\delta_0$. Proposition~\ref{p.levraidrift} then implies that $K_{00}\cap G(\eps,M)\neq\emptyset$, which is absurd. 
\end{proof}

In order to start the proof of Proposition~\ref{p.levraidrift}, fix $\eps>0$ (independently of $\delta$: notice that while $\delta$ is fixed we need to take $\eps\to 0$). More assumptions on $\eps$ will be given later on. This is the expected size of the drift we want to see along the center direction. Given a compact set $K_{00}$ whose measure is larger than $1-2\delta_0$, notice that the compact set
\[
K_0\eqdef K_{00}\cap\cL.
\] 
has measure $\mu(K_0)>1-3\delta_0=1-\delta$. We shall prove that $K_0\cap G(\eps,M)\neq\emptyset$. This will occupy the rest of the paper.

\section{Stopping times, $Y$-configurations, quadrilaterals and synchronization}\label{s.Yconfigurations}
The goal of this section is to introduce the key dynamical ingredients involved in the proof of Proposition~\ref{p.levraidrift}, which is the main part of our implementation of an exponential drift argument. We also use this section to derive some estimates relating these ingredients.

\subsection{Stopping times}\label{subsec.stoppingtimes}  
We introduce below the stopping time functions. They are devised to measure the appropriate time length of the top part of Figure~\ref{f.coupledconfig} so that we get the precise drift we want along the center direction.   

\subsubsection{Definition of stopping times}
Recall our concise notations for derivatives
\[
\lambda^\star_x(n)=\|Df^n(x)|_{E^\star}\|\:\:\:\:\star=c,s,u,
\]
and 
\[
d^\ell_x=\frac{\lambda^c_{x_{-\ell}}(\ell)}{\lambda^u_{x_{-\ell}}(\ell)}
\]
introduced in \eqref{e.derivada} of \S\ref{sss.notaderiva} and \eqref{e.domonationalongy} of \S\ref{sss.hypestimates}, respectively, where $x_n=f^n(x)$ is our concise notation for orbit points introduced in \eqref{e.orbita} of \S\ref{sss.notaderiva}.
Given $x\in \T^3$, $x^u\in \cW^u(x)$, $\varepsilon>0$, and  $\ell \in \N$, we define
\[
\tau(\ell)=\tau(x,x^u,\varepsilon,\ell)\eqdef\inf\Big\{n\in\N:d^\ell_x\times\lambda^c_{x^u}(n)\geq \varepsilon\Big\}.
\]
We also define
\[
t(\ell)=t(x,x^u,\varepsilon,\ell)\eqdef\inf\Big\{n\in\N:\frac{\lambda^c_x(n)}{\lambda^c_{x^u}(\tau(\ell))}\geq 1\Big\}.
 \]
These functions are called the \textit{stopping times}.  

%\begin{remark}\label{rem_stop-times-tend-infty}
%Using that the expansion along center manifolds is uniform and strictly smaller than the expansion along strong unstable manifolds it is easy to see from the definition of stopping times that $\tau(\ell)\to\infty$ and $t(\ell)\to\infty$ as $\ell\to\infty$. We go further in Lemma \ref{l_qi_estimates}, by proving quasi-isometric estimates.
%\end{remark}

\subsubsection{Quasi-isometric estimates}

In the following, we fix $\eps>0$. Recall that for $x\in \TT$, $x^u \in \cWu(x)$, we abbreviate $\tau(\ell)=\tau(x,x^u,\eps,\ell)$ and $t(\ell)=t(x,x^u,\eps,\ell)$. 
\begin{lemma}[Quasi-isometric estimates]\label{l_qi_estimates}
	There exists $\Theta=\Theta(f)>1$ and $A=A(f)>0$ so that given $\ell,m\in\N$, $x\in\TT$ and $x^u\in\cW^u(x)$ the following holds 
	\begin{enumerate}
		\item $\Theta^{-1}m-A<\tau(\ell+m)-\tau(\ell)<\Theta m+A$ and 
		\item $\Theta^{-1}m-A<t(\ell+m)-t(\ell)<\Theta m+A$.
	\end{enumerate}
\end{lemma}

\begin{proof}
	In this proof we shall make use of the constants of hyperbolicity of $f$ introduced in \S\ref{sss.hypestimates}.	
	With these constants at hand, we can now develop the argument for the \emph{quasi-isometric estimates}. For this, fix $x\in\TT$ and $x^u\in\cW^u(x)$. Given some $\eps>0$, let us consider $\tau(\ell)=\tau(x,x^u,\eps,\ell)$ and $t(\ell)=t(x,x^u,\eps,\ell)$. From the definition of $\tau$, we have that 
	\[
	d^\ell_x\lambda^c_{x^u}\left(\tau(\ell)\right)\geq\eps\:\:\:\textrm{and}\:\:\:d^\ell_x\lambda^c_{x^u}\left(\tau(\ell)-1\right)<\eps,\:\:\:\textrm{for each}\:\:\ell\in\N.
	\]
	Thus, 
	\begin{equation}
	\label{e.odriftchegou}
	\eps\leq d^\ell_x\lambda^c_{x^u}\left(\tau(\ell)\right)<e^{\chi^c_1}\eps,\:\:\:\textrm{for each}\:\:\ell\in\N.
	\end{equation}
	Fix $\ell,m\in\N$. On the one hand, by the cocycle property \eqref{e.cocycle} we have that $d^{\ell+m}_x=d^\ell_xd^{m}_{x_{\ell}}$ and therefore we can use \eqref{e.odriftchegou} to obtain
	\[
	\eps e^{m\chi^d_1}<d^{\ell+m}_x\lambda^c_{x^u}\left(\tau(\ell)\right)<\eps e^{\chi^c_1+m\chi^d_2}.
	\]
	On the other hand, we can use the cocycle property once more to write $\lambda^c_{f^{\tau(\ell)}(x^u)}\left(\tau(\ell+m)-\tau(\ell)\right)=\frac{\lambda^c_{x^u}\left(\tau(\ell+m)\right)}{\lambda^c_{x^u}\left(\tau(\ell)\right)}$ and thus {\color{blue}obtain}
	\[
	e^{\chi^c_2\left(\tau(\ell+m)-\tau(\ell)\right)}<\frac{\lambda^c_{x^u}\left(\tau(\ell+m)\right)}{\lambda^c_{x^u}\left(\tau(\ell)\right)}<e^{\chi^c_1\left(\tau(\ell+m)-\tau(\ell)\right)}.
	\]
    Therefore, since $d^{\ell+m}_x\lambda^c_{x^u}\left(\tau(\ell+m)\right)=d^{\ell+m}_x\lambda^c_{x^u}\left(\tau(\ell)\right)\frac{\lambda^c_{x^u}\left(\tau(\ell+m)\right)}{\lambda^c_{x^u}\left(\tau(\ell)\right)}$ we can combine the two above inequalities and obtain
    \[
    \eps e^{m\chi_1^d+\chi_2^c\left(\tau(\ell+m)-\tau(\ell)\right)}<d^{\ell+m}_x\lambda^c_{x^u}\left(\tau(\ell+m)\right)<\eps e^{\chi_1^c+m\chi_2^d+\chi_1^c\left(\tau(\ell+m)-\tau(\ell)\right)}.
    \]
    Putting $\ell+m$ instead of $\ell$ in \eqref{e.odriftchegou} and combining with this we get the inequalities
    \[
    \eps e^{\chi_1^c+m\chi_2^d+\chi_1^c\left(\tau(\ell+m)-\tau(\ell)\right)}>\eps\:\:\:\textrm{and}\:\:\:\eps e^{-\chi_1^c+m\chi_1^d+\chi_2^c\left(\tau(\ell+m)-\tau(\ell)\right)}<\eps.
    \]
    Dividing by $\eps$ and taking logarithms we deduce that 
    \begin{equation}
    \label{e.quasiisometric}
    m\left[\frac{-\chi_2^d}{\chi_1^c}\right]-1<\tau(m+\ell)-\tau(\ell)<m\left[\frac{-\chi_1^d}{\chi_2^c}\right]+\frac{\chi_1^c}{\chi_2^c}.
    \end{equation}
    Let us deal with the function $\ell\mapsto t(\ell)$. From its very definition we have an inequality analogous to \eqref{e.odriftchegou}:
    \begin{equation}
    \label{e.omundoehredondo}
    1\leq\frac{\lambda^c_x\left(t(\ell)\right)}{\lambda^c_{x^u}\left(\tau(\ell)\right)}<e^{\chi_1^c},\:\:\:\textrm{for every}\:\:x\in\TT,\:\ell\in\N.
    \end{equation}
    To simplify the remainder of the exposition, we shall denote 
    $\tau_{\ell,m}\eqdef\tau(\ell+m)-\tau(\ell)$ and $t_{\ell,m}=t(\ell+m)-t(\ell)$.
    Notice that 
    \[
    e^{\chi_2^ct_{\ell,m}-\chi_1^c\tau_{\ell,m}}<\frac{\lambda^c_{f^{t(\ell)}(x)}\left(t(\ell+m)-t(\ell)\right)}{\lambda^c_{f^{\tau(\ell)}(x^u)}\left(\tau(\ell+m)-\tau(\ell)\right)}<e^{\chi_1^ct_{\ell,m}-\chi_2^c\tau_{\ell,m}}.
    \]  
    Combining this with \eqref{e.omundoehredondo}, we can use the decomposition 
    \[
    \frac{\lambda^c_x\left(t(\ell+m)\right)}{\lambda^c_{x^u}\left(\tau(\ell+m)\right)}=\frac{\lambda^c_x\left(t(\ell)\right)\lambda^c_{f^{t(\ell)}(x)}\left(t(\ell+m)-t(\ell)\right)}{\lambda^c_{x^u}\left(\tau(\ell)\right)\lambda^c_{f^{\tau(\ell)}(x^u)}\left(\tau(\ell+m)-\tau(\ell)\right)}
    \]
    and conclude that 
    \[
    e^{\chi_2^ct_{\ell,m}-\chi_1^c\tau_{\ell,m}}<\frac{\lambda^c_x\left(t(\ell+m)\right)}{\lambda^c_{x^u}\left(\tau(\ell+m)\right)}<e^{\chi_1^ct_{\ell,m}-\chi_2^c\tau_{\ell,m}+\chi_1^c}.
    \]
    Now, using \eqref{e.omundoehredondo} with $\ell+m$ instead of $\ell$, the above inequality implies that
    \[
    e^{\chi_2^ct_{\ell,m}-\chi_1^c\tau_{\ell,m}}<e^{\chi_1^c}\:\:\:\textrm{and}\:\:\:1<e^{\chi_1^ct_{\ell,m}-\chi_2^c\tau_{\ell,m}+\chi_1^c},
    \]
    and thus
    \[
    e^{\chi_2^ct_{\ell,m}-\chi_1^c\tau_{\ell,m}-\chi_1^c}<1<e^{\chi_1^ct_{\ell,m}-\chi_2^c\tau_{\ell,m}+\chi_1^c}.
    \]
    Taking logarithms leads us to 
    \[
    \left[\frac{\chi_2^c}{\chi_1^c}\right]\tau_{\ell,m}-1<t_{\ell,m}<\left[\frac{\chi_1^c}{\chi_2^c}\right]\tau_{\ell,m}+\frac{\chi_1^c}{\chi_2^c},
    \]
    which combined with \eqref{e.quasiisometric} ends the proof.
\end{proof}

\subsection{$Y$-configurations}

We introduce below a dynamical ingredient inspired by \cite{EskinLind,EskinMirzakhani}. They allow us to ``decompose'' Figure~\ref{f.coupledconfig} in the ``$x$-side'' and in the ``$y$-side''. Notice that each side has a kind of $Y$-shape. The idea from \cite{EskinLind} for getting the points inside the Lusin set is to try to prove the existence of a large amount of these $Y$-shaped dynamical configurations and then try to find some of them which are linked through stable manifolds, as it appears in Figure~\ref{f.coupledconfig}. These measure theoretical arguments will be developed in Section~\ref{s.end_proof}. In this section, we introduce these configurations and establish \emph{synchronization and (center) drift estimates for them}.  

\subsubsection{Definition}
 Given $\eps>0$ and $\ell\in \N$, a \emph{$Y$-configuration} $Y=Y(x,x^u,\ell)$ is a quintuple of points $(x,x^u,x_{-\ell},x^u_{\tau},x_t)$, that depends on parameters $x,x^u\in\mathbb{T}^3$ and $\ell$ (the dependence on $\eps$ is implicit throughout the text), and such that
\begin{enumerate}
	\item $x^u\in\cW^u(x)$;
	\item $x_{-\ell}=f^{-\ell}(x)$;
	\item $x^u_{\tau}=f^{\tau}(x^u)$, where $\tau=\tau(x,x^u,\eps,\ell)$;
	\item $x_t=f^{t}(x^u)$, where $t=t(x,x^u,\eps,\ell)$.
\end{enumerate}
where $\tau(x,x^u,\eps,\ell)$ and $t(x,x^u,\eps,\ell)$ are the stopping times defined above. 

We call $\ell$ the \textit{length} of the $Y$-configuration. Moreover, given a set $\Lambda \subset \TT$, we say that a $Y$-configuration $(x,x^u,x_{-\ell},x^u_\tau,x_t)$ is $\Lambda$\emph{-good} if $x,x^u,x^u_\tau,x_t\in \Lambda$.

\begin{figure}[h!]
	\begin{tikzpicture}
	\draw[black, ->] (-4,-1)--(-4,7)node[right]{\text{Time}};
	\fill[fill=black] (-4,3) node[left]{\color{black}\small $0$} circle (1pt);
	\fill[fill=black] (-4,0) node[left]{\color{black}\small $-\ell$} circle (1pt);
	\fill[fill=black] (-4,5.5) node[left]{\color{black}\small $\tau(\ell)$} circle (1pt);
	\fill[fill=black] (-4,5) node[left]{\color{black}\small $t(\ell)$} circle (1pt);
	\draw[red!80!black, thick] (0,3)--(2,3)node[midway, below]{\color{red!80!black}\small $\mathcal{W}^u(x)$};
	\draw[black, densely dotted, ->] (0,3)--(-1,5.5)node[midway, left]{\tiny $\tau(\ell)=\tau(x,x^u,\varepsilon,\ell)$};
	\draw[black, densely dotted, ->] (2,3)--(3,5)node[midway, right]{\tiny $t(\ell)=t(x,x^u,\varepsilon,\ell)$};
	\draw[black, densely dotted, ->] (0,0)--(2,3)node[midway, right]{\tiny $\ell$};
	\fill[fill=black] (0,3) node[left]{\color{black}\small $x^u$} circle (1pt);
	\fill[fill=black] (0,0) node[left]{\color{black}\small $x_{-\ell}$} circle (1pt);
	\fill[fill=black] (2,3) node[right]{\color{black}\small $x$} circle (1pt);
	\fill[fill=black] (3,5) node[right]{\color{black}\small $f^{t(\ell)}(x)$} circle (1pt);
	\fill[fill=black] (-1,5.5) node[left]{\color{black}\small $f^{\tau(\ell)}(x^u)$} circle (1pt);
	\end{tikzpicture}
	\caption{\label{yconf} A $Y$-configuration.}
\end{figure}
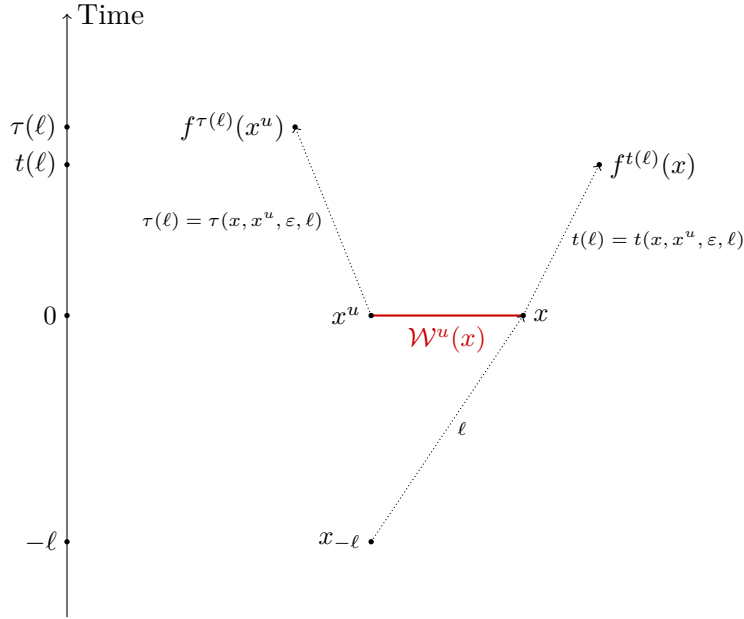

%To avoid overload of notation we leave implicit the dependence of the configuration on the parameters $\eps$ and $\ell$. Moreover, with the same purpose we use capital letter notation $X=(x,x^ux,\ell)$ as a short cut for the longer description with 5 points. In this notation the points $x_{-\ell}$, $b_x$ and $c_x$ are implicit. 
In \cite{EskinLind}, the authors defined a notion of \emph{pairs of coupled} $Y$-configurations. For technical reasons we shall replace this notion by that of pairs of \emph{matched} $Y$-configurations that will appear in the next section.

\subsection{Quadrilaterals and synchronization}\label{sss.quadri_couple} 
Another important aspect of Figure~\ref{f.coupledconfig} is located at its middle, where we have a kind of a quadrilateral, twisted along the unstable direction. This, indeed, is the third dynamical ingredient in our implementation of an exponential drift argument. It will allow us to use the angle condition and to define the aforementioned notion of matching of $Y$-configurations.

\subsubsection{Quadrilaterals}

A \textit{quadrilateral} is a quadruple $(x,x^u,y,y^u)\in(\TT)^4$ such that
\begin{enumerate}
\item $y\in \cWs(x)$;
\item $x^u\in\cW^u(x)$ belongs to the domain of a center-holonomy map $H^{cs}_{x,y}$;
\item $y^u=H^{cs}_{x,y}{\color{blue}(x^u)}\in\cWu(y)\cap\cWcs(x^u)$.
\end{enumerate}
For such a quadrilateral, we define the point $z^u\eqdef H^{s}_{x,y}(x^u)$, so that $z^u\in \mathcal{W}^s(x^u)\cap \cWc(y^u)$. Moreover, given $C>1$, $\ell\in \N$, 
we say that $(x,x^u,y,y^u)$ is a \emph{$(C,\ell)$-quadrilateral} if, besides, 
\begin{enumerate}
	\item[(4)] $C^{-1}<d(x_{-\ell},y_{-\ell})<1$,
	\item[(5)] $C^{-1}<\alpha^s(x_{-\ell},y_{-\ell})<C$ and
	\item[(6)] $C^{-1}<d(x,x^u)<C$, 
\end{enumerate}
where, as before, $x_{-\ell}=f^{-\ell}(x)$,  $y_{-\ell}=f^{-\ell}(y)$. 

\begin{figure}[h]
	\centering
	\begin{tikzpicture}[scale=.6]
	\fill[black!30!white, opacity=.6] (10,0) -- (9.5,0.5) to[bend left] (9.5,1) -- (10,0);
	\fill[green!30!white, opacity=.5] (-1,1)--(1,-1)--(11,-1)--(9,1)--(-1,1);
	\draw[green!40!black, thick] (0,0)--(10,0) node[midway,below]{\tiny $\cW^s_{\loc}(x_{-\ell})$};
	\draw[red!80!black,thick] (-1,1)--(1,-1) node[left]{\tiny $\cW^u_{\loc}(x_{-\ell})$};
	\draw[violet, thick] (9,1)--(11,-1) node[right]{\tiny $H^s_{x_{-\ell},y_{-\ell}}(\cW^u_{\loc}(x_{-\ell}))$};
	\fill[fill=black] (9,1) circle (2pt);
	\draw (9.1,1.4) node[left]{\small $z^u_{-\ell}$};
	\draw[red!80!black, thick] (11,-2)node[right]{\tiny $\cW^u_{\loc}(y_{-\ell})$}--(9,2);
	\fill[fill=black] (9,2) node[right]{$y^u_{-\ell}$} circle (2pt); 
	\draw[dotted] (11,-2)--(11,-1);
	\draw[dotted] (9,1)--(9,2);
	\fill[green!30!white, opacity=.5] (4,4)--(1,4.5)--(3,4.5)--(6,4);
	\draw[green!40!black, thick] (4,4)--(6,4);
	%	\draw[blue!40!black, thick] (4.8,6)--(4.9,6);
	\draw[red!80!black,thick] (4,4)--(1,4.5) node[midway, below]{$\gamma^u$};
	\draw[red!80!black, thick] (6,4)--(3,5);
	%\draw[blue!40!black, thick] (0.5,7)--(0.75,7);
	\draw[orange, thick] (3,4.5)--(3,5);
	%\draw[red!70!white, thick] (0.75,7)--(0.75,9) node[midway, left]{$\varepsilon\sim$};
	%\draw[black!70!white, opacity=.6,->] (1,4.5)--(0.5,7);
	%\draw[black!70!white, opacity=.6,->](3,4.5)--(0.75,7);
	%\draw[black!70!white, opacity=.6,->] (3,5)--(0.75,9);
	%\draw[densely dotted, black!70!white, opacity=.6,->] (0,0)--(4,4) node[midway, right]{$f^\ell$};
	\draw[densely dotted, black!70!white, opacity=.6,->] (0,0).. controls (2,1) and (3.7,3) .. (4,4) node[midway, right]{$f^\ell$};
	\draw[violet,thick]  (9.7,0.8) node[right]{\tiny $\alpha^s(x_{-\ell},y_{-\ell})$};
	\draw[densely dotted, black!70!white, opacity=.6,->] (10,0).. controls (9,0) and (6,3) .. (6,4) node[midway, left]{$f^\ell\ $};
	%\draw[black!70!white, opacity=.6,->] (4,4)--(4.8,6);
	%\draw[black!70!white, opacity=.6,->] (6,4)--(4.9,6);
	\fill[fill=black] (0,0) node[left]{\small $x_{-\ell}$} circle (2pt);
	\draw[violet, thick] (6,4)--(3,4.5); 
	\fill[fill=black] (10,0) node[right]{\small $y_{-\ell}$} circle (2pt);
	\fill[fill=black] (1,4.5) node[left]{\small $x^u$} circle (2pt);
	\fill[fill=black] (-1,1) node[left]{\small $x^u_{-\ell}$} circle (2pt);
	\fill[fill=black] (6,4) node[right]{\small $y$} circle (2pt);
	\fill[fill=black] (4,4) circle (2pt);
	\draw (4,4) node[below]{\small $x$};
	%\fill[fill=black] (4.8,6) node[left]{\small $f^{t(\ell)}(x)$} circle (2pt);
	%\fill[fill=black] (4.9,6) node[right]{\small $f^{t(\ell)}(y)$} circle (2pt);
	\draw[green!40!black, thick] (1,4.5)--(3,4.5);
	\fill[fill=black] (3,4.5) circle (2pt);
	\draw (3.1,4.8) node[left]{\small $z^u$};
	\fill[fill=black] (3,5)  circle (2pt);
	\draw (3.2,5) node[above]{\small $y^u$};
	%\fill[fill=black] (0.75,7) node[right]{\tiny $f^{\tau(\ell)}(z)$} circle (2pt); 
	%\fill[fill=black] (0.75,9) node[right]{\tiny $f^{\tau(\ell)}(y^u)$} circle (2pt);
	%\fill[fill=black] (0.5,7) node[left]{\tiny $f^{\tau(\ell)}(x^u)$} circle (2pt);  
	%\draw (9.5,0.5) to[bend left] (9.5,1);
	\end{tikzpicture}
	\caption{\label{f.quadrilatero} A quadrilateral and its pre-image by $f^\ell$.}
\end{figure}
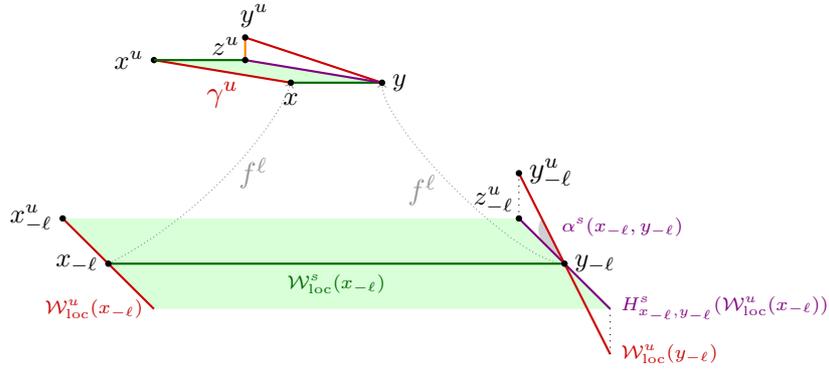

\subsubsection{Drift estimates}\label{sss.drifestimates}

%{\color{blue!50!red}
%	\begin{lemma}\label{control centre expa2}
%		Given $C>1$ and $T_0>0$ there exists a constant $\gamma>1$ depending only on $f,C,T_0$ such that for any pair %modified by at most $T_0$ and 
%		$(C,T_0)$-coupled $Y$-configurations $Y(x,x^u,\varepsilon,\ell)$, $Y(y,y^u,\varepsilon,\ell)$ of length $\ell$, it holds
%		\begin{equation}\label{derive vert2}
%		\frac{\varepsilon}{\gamma} < d_{\mathcal{W}^{c}}(f^{\tau'(\ell)}(y^u),f^{\tau'(\ell)}(z^u)) < \gamma \varepsilon.
%		\end{equation}
%	\end{lemma}}

In Figure~\ref{f.coupledconfig} we claim that the displacement along the center at the left top of the configuration is proportional to the parameter $\eps$. We shall now make this assertion precise in terms of normal form coordinates. We first introduce a useful notation that will be used throughout the text, and in particular in the next proof.

Given a set of real parameters $c_1,\ldots, c_n$ and two quantities $a$ and $b$ (that may or may not depend on other variables) we denote $a\asymp_{c_1,\ldots,c_n} b$ if there exists a real valued function $\rho=\rho(c_1,\ldots,c_n)\geq 1$ such that
\[
\rho^{-1} a\leq b\leq\rho a. 
\]

\begin{lemma}\label{control centre expa}
	Fix $\eps>0$. Given $C>1$ and $T>0$, there exists a constant $\kappa=\kappa(C,T)>1$ such that for every $\ell\in \N$ large enough, every 
 $(C,\ell)$-quadrilateral $(x,x^u,y,y^u)$, and every $\tau\geq 0$ such that $|\tau-\tau(y,y^u,\eps,\ell)|\leq T$ we have
	\begin{equation}\label{derive vert}
	\frac{\varepsilon}{\kappa} < \left|\cH^c_{f^{\tau}(z^u)}(f^{\tau}(y^u))\right| < \kappa \varepsilon.
	\end{equation}
\end{lemma}
\begin{proof}
Let $\gamma^u=[x,x^u]\subset\cW^u(x)$ be the segment of strong unstable manifold connecting the points $x$ and $x^u$. See Figure~\ref{f.quadrilatero}. Then, $f^{-\ell}\circ H^s_{x,y}(\gamma^u)$ is a $C^1$ curve joining $y_{-\ell}$ and $z^u_{-\ell}$.\footnote{Recall our notation convention for orbit points: $p_n=f^n(p)$ for $n\in\Z$.} Since, $f^{-\ell}\circ H^s_{x,y}(\gamma^u)=H^s_{x_{-\ell},y_{-\ell}}\circ f^{-\ell}(\gamma^u)$ the assumption $C^{-1}<\alpha^s(x_{-\ell},y_{-\ell})<C$ implies that the vector tangent to the curve $f^{-\ell}\circ H^s_{x,y}(\gamma^u)$ at $y_{-\ell}$ and the strong unstable direction $E^u(y_{-\ell})$ are transverse with an angle between $C^{-1}$ and $C$. Recall that $\cH_{y_{-\ell}}\colon\cW^{cu}(y_{-\ell})\to\R^2$ sends $\cW^u(y_{-\ell})$ onto the horizontal axis (see Theorem~\ref{thm.normalforms}). As the length of $f^{-\ell}\circ H^s_{x,y}(\gamma^u)$ is exponentially small with $\ell$ and the bundle $E^u$ is continuous, a standard compactness argument then ensures that for some constant $c_1=c_1(C)>0$ and for $\ell$ large enough, the curve $\cH_{y_{-\ell}}\circ f^{-\ell}\circ H^s_{x,y}(\gamma^u)$ on $\R^2$ is contained in the cone 
\[
\cC_1=\left\{(v_1,v_2)\in\R^2:c_1^{-1}|v_1|\leq|v_2|\leq c_1|v_1|\right\}.
\]
Now, observe that $\cH_y\circ H^s_{x,y}(\gamma^u)=\cH_y\circ f^{\ell}\circ\Phi_{y_{-\ell}}\circ \cH_{y_{-\ell}}\circ f^{-\ell}\circ H^s_{x,y}(\gamma^u)$ and that $\cH_y\circ f^{\ell}\circ\Phi_{y_{-\ell}}$ is the linear map $(t,s)\in\R^2\mapsto(\lambda_{y_{-\ell}}^u(\ell)t,\lambda_{y_{-\ell}}^c(\ell)s)$. These two observations put together imply that the curve $\cH_y\circ H^s_{x,y}(\gamma^u)$ is contained in the cone%\footnote{\color{teal}This argument is reminiscent of the proof of Lemma \ref{new lemme a x x prime}, where quotients of the form $\frac{\lambda_{*}^c(\ell)}{\lambda_{*}^u(\ell)}$ also appear, which describe how certain angles scale under the dynamics in normal coordinates.}
\[
\cC_2=\left\{(v_1,v_2)\in\R^2:c_1^{-1}\frac{\lambda_{y_{-\ell}}^c(\ell)}{\lambda_{y_{-\ell}}^u(\ell)}|v_1|\leq|v_2|\leq c_1\frac{\lambda_{y_{-\ell}}^c(\ell)}{\lambda_{y_{-\ell}}^u(\ell)}|v_1|\right\}.
\]

Now let $\gamma^c$ denote the piece of center manifold connecting $z^u$ to $y^u$. Then $\cH_y\circ\gamma^c$ connects the point $\cH_y(z^u)$, whose distance to the origin only depends on the constant $C$ and is bounded from above and below, to the point $\cH_y(y^u)$, which lies on the horizontal axis. Since $\cH_y(z^u)\in\cC_2$, we deduce that (see Figure~\ref{f.actionlineaire})
$$\length(\cH_y\circ\gamma^c)\asymp_C \frac{\lambda_{y_{-\ell}}^c(\ell)}{\lambda_{y_{-\ell}}^u(\ell)}.$$

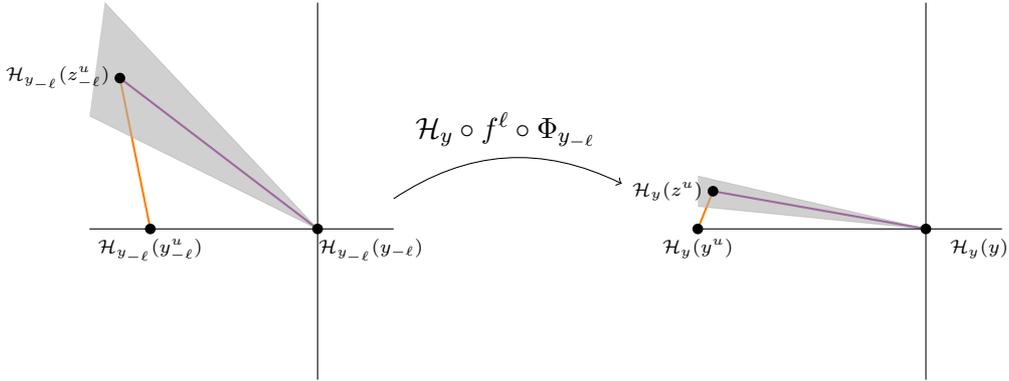
\begin{figure}[h!]
	\centering
	\begin{tikzpicture}
	\draw[orange, thick] (-2.6,2)--(-2.2,0);
	\draw[violet, thick] (0,0)--(-2.6,2);
	\draw[black!30!white, opacity=.6] (0,0)--(-3,1.5);
	\draw[black!30!white, opacity=.6] (0,0)--(-2.8,3);
	\fill[black!30!white, opacity=.6] (0,0)--(-3,1.5)--(-2.8,3)--(0,0);
	\draw (-3,0)--(1,0);
	\draw (0,-2)--(0,3);
	\fill[fill=black] (0,0) circle (2pt);
	\draw (0.7,0) node[below]{\tiny $\cH_{y_{-\ell}}(y_{-\ell})$};
	\fill[fill=black] (-2.2,0) node[below]{\tiny $\cH_{y_{-\ell}}(y^u_{-\ell})$}  circle (2pt);
    \fill[fill=black] (-2.6,2) node[left]{\tiny $\cH_{y_{-\ell}}(z^u_{-\ell})$}  circle (2pt);
    \begin{scope}[xshift=8cm]
   \draw[orange, thick] (-3,0)--(-2.8,0.5);
   \draw[violet, thick] (0,0)--(-2.8,0.5);
   \draw[black!30!white, opacity=.6] (0,0)--(-3,0.3);
   \draw[black!30!white, opacity=.6] (0,0)--(-3,0.7);
   \fill[black!30!white, opacity=.6] (0,0)--(-3,0.3)--(-3,0.7)--(0,0);
   \draw (-3,0)--(1,0);
   \draw (0,-2)--(0,3);
   \fill[fill=black] (0,0) circle (2pt);
   \draw (0.7,0) node[below]{\tiny $\cH_{y}(y)$};
   \fill[fill=black] (-3,0) node[below]{\tiny $\cH_{y}(y^u)$}  circle (2pt);
   \fill[fill=black] (-2.8,.5) node[left]{\tiny $\cH_{y}(z^u)$}  circle (2pt);
    \end{scope}
    \draw[->] (1,.4) to[bend left] node[midway, above]{$\cH_y\circ f^{\ell}\circ\Phi_{y_{-\ell}}$} (4,.6) ;
	\end{tikzpicture}
	\caption{\label{f.actionlineaire}On normal form coordinates the dynamics on $\cW^{cu}$ acts as a diagonal matrix, with stronger expansion on the horizontal.}
\end{figure} 

Since $C^{-1}<d^u(x,x^u)<C$ and we have $C^1$ holonomies, we also have that $C^{-1}<d(z^u,y)<C$ (upon enlarging the constant $C$ if necessary in order to take into account the action of holonomy maps). Recall that normal forms change continuously in the $C^1$ topology, so another compactness argument ensures that the $C^1$ norm of $\cH_{z^u}\circ\Phi_y|_{W^{cu}_{2C}(y)}$ is bounded by some constant depending on $C$. Thus
$$\length(\cH_{z^u}\circ\gamma^c)\asymp_C \frac{\lambda_{y_{-\ell}}^c(\ell)}{\lambda_{y_{-\ell}}^u(\ell)}.$$

On the other hand, by the construction of normal forms, $\cH_{z^u}\circ\gamma^c\subset\R^2$ is a vertical segment with length $|\cH^c_{z^u}(y^u)|$. This proves that
\[
\left|\cH^c_{z^u}(y^u)\right|\asymp_C \frac{\lambda_{y_{-\ell}}^c(\ell)}{\lambda_{y_{-\ell}}^u(\ell)}.
\]  
As we did before it is possible to bound uniformly from above the $C^1$-norms  of all $\cH^c_\ast$ in restriction to center segments of uniform radius so (upon enlarging $\ell$ if necessary) we also have

\[
d_c(y^u,z^u)\asymp_C \frac{\lambda_{y_{-\ell}}^c(\ell)}{\lambda_{y_{-\ell}}^u(\ell)},\,\,\,\,\,\,\,\,\text{and}\,\,\,\,\,\,\,\,\left|\cH^c_{y^u}(z^u)\right|\asymp_C \frac{\lambda_{y_{-\ell}}^c(\ell)}{\lambda_{y_{-\ell}}^u(\ell)}.
\]

Since the dynamics on normal forms acts linearly, this implies 
$$\left|\cH^c_{f^{\tau(\ell)}(y^u)}{(f^{\tau(\ell)}(z^u))}\right|\asymp_C \frac{\lambda_{y_{-\ell}}^c(\ell)}{\lambda_{y_{-\ell}}^u(\ell)}\times\lambda^c_{y^u}(\tau(\ell))$$
where $\tau(\ell)=\tau(y,y^u,\eps,\ell)$ is the stopping time. By definition this implies that
\begin{equation}
\label{e.gapdoseba}
\left|\cH^c_{f^{\tau(\ell)}(y^u)}{(f^{\tau(\ell)}(z^u))}\right|\asymp_C\eps.
\end{equation}
Since $|\tau-\tau(\ell)|\leq T$  we obtain

$$\left|\cH^c_{f^{\tau}(y^u)}(f^{\tau}(z^u))\right|\asymp_{C,T}\eps.\qedhere$$ 
\end{proof}

\begin{corollary}
\label{c.gapdoseba}
There exists a constant $\kappa=\kappa(C,T)$ such that for every $\tau$ satisfying $|\tau-\tau(y,y^u,\eps,\ell)|\leq T$ it holds 
\[
d_c(f^{\tau}(y^u),f^{\tau}(z^u))\leq\kappa\eps.
\]
\end{corollary}
\begin{proof}
This follows from the proof of Lemma~\ref{control centre expa} and the fact that normal forms $\{\cH^c_x\}_{x\in\TT}$ are $C^1$ maps with continuously varying $C^1$ norm on compact sets.
\end{proof}

\subsubsection{Synchronization estimates}
 The last set of estimates we need concerns the oscillation of the stopping times from one side to the other when we have a $(C,\ell)$-quadrilateral.

\begin{lemma}[Synchronization for quadrilaterals]\label{lemme synchro}
	For any $C>1$, there exists a constant $T_0>0$ depending only on $f,C$ such that for any $\ell\in \N$, and any $(C,\ell)$-quadrilateral $(x,x^u,y,y^u)$, it holds 
	$$
	|\tau(y,y^u,\varepsilon,\ell)-\tau(x,x^u,\varepsilon,\ell)|<T_0,\quad |t(y,y^u,\varepsilon,\ell)-t(x,x^u,\varepsilon,\ell)| < T_0. 
	$$ 
\end{lemma}

\begin{proof}
Recall the constants $\chi_2^c,\chi_1^c>0$ introduced in \S\ref{sss.hypestimates}. They satisfy $e^{k\chi_2^c}\leq\lambda^c_x(k)\leq e^{k\chi_1^c}$ for every $k\in\N$ and every $x\in\TT$ where $\lambda^c_x(k)=\|Df^k(x)|_{E^c}\|$ is our concise notation for derivatives from \eqref{e.derivada}. For the sake of simplicity, denote $\tau(\ell)=\tau(x,x^u,\eps,\ell)$ and $\tau^{\prime}(\ell)=\tau(y,y^u,\ell,\eps)$. Assume that $\tau(\ell)\geq\tau'(\ell)$. By definition
\begin{equation}
\label{e.tau}
d^\ell_y\lambda^c_{y^u}(\tau'(\ell))\geq\eps.
\end{equation}

Let us write for $k\in\N$
\begin{equation}\label{eq_telecscopouille1}
\frac{d^\ell_x\lambda^c_{x^u}(\tau'(\ell)+k)}{ d^\ell_y\lambda^c_{y^u}(\tau'(\ell))}=\frac{d^\ell_x}{d^\ell_y}\times\frac{\lambda^c_{x^u}(\tau'(\ell))}{\lambda^c_{y^u}(\tau'(\ell))}\times \lambda^c_{x^u_{\tau'(\ell)}}(k).
\end{equation}

We now bound from below each factor appearing in \eqref{eq_telecscopouille1}. We treat the third factor by observing that $\lambda^c_{x^u_{\tau'(\ell)}}(k)\geq e^{k\chi_2^c}$.  To treat the first factor, we use that $y_{-\ell}\in \cW^s_1(x_{-\ell})$ so the distortion control given by Corollary~\ref{c.distortionbasic} gives $\frac{d^\ell_y}{d^\ell_x}\geq C_0^{-1}$ %\annotation{Seb. Have to change the conditions of the distortion control (we need to go further the case $d(x,y)\leq 1$}.
Finally, in order to treat the second factor, we notice that Corollary~\ref{c.gapdoseba} provides that 
 \[
 d_c(f^j(z^u),f^j(y^u))\leq C\eps,
 \]
for some constant $C=C(f)>0$ and for every $j=0,\dots,\tau^\prime(\ell)$. With no loss of generality we can assume that $C\eps<1$. Applying the distortion control of Corollary~\ref{c.distortionbasic2} we obtain $\frac{\lambda^c_{x^u}(\tau'(\ell))}{\lambda^c_{y^u}(\tau'(\ell))}\geq C_0^{-1}$. These lower bounds together with  \eqref{eq_telecscopouille1} provide the following estimate
$$d^\ell_x\lambda^c_{x^u}(\tau'(\ell)+k)\geq \frac{e^{k\chi_2^c}}{C_0^2}\eps\geq\eps$$
as soon as we choose $k=k(f)$ such that
\begin{equation}\label{eq_choicek}
e^{k\chi_2^c}\geq C_0^2.
\end{equation}

Hence by definition of the stopping time, we obtain $\tau'(\ell)+k\geq\tau(\ell)$ so $\tau(\ell)-\tau'(\ell)\leq k$. A symmetric argument shows that $\tau'(\ell)-\tau(\ell)\leq k$ if $\tau'(\ell)\geq\tau(\ell)$.

We consider now the stopping time $t$. As above, we denote $t(\ell)=t(x,x^u,\eps,\ell)$ and $t^{\prime}(\ell)=t^{\prime}(y,y^u,\eps,\ell)$. By definition

\begin{equation}
\label{e.tau}
\frac{\lambda^c_{y}(t'(\ell)}{\lambda^c_{y^u}(\tau'(\ell))}\geq 1.
\end{equation}

Let us write for $k'\in\N$:
\begin{align*}\label{eq_telecscopouille2}
\frac{\lambda^c_{x}(t'(\ell)+k')}{\lambda^c_{x^u}(\tau(\ell))} {\frac{\lambda^c_{y^u}(\tau'(\ell))}{\lambda^c_{y}(t'(\ell)}}&=\frac{\lambda^c_{y^u}(\tau'(\ell))}{\lambda^c_{x^u}(\tau'(\ell))}\times\frac{\lambda^c_{x}(t'(\ell)}{\lambda^c_{y}(t'(\ell)}
\\
 &\quad\times  \lambda^c_{x_{t'(\ell)}}(k')\times\frac{1}{\lambda^c_{x^u_{\tau'(\ell)}}(\tau(\ell)-\tau'(\ell))}.
\end{align*}
We now bound from below each factor of the product above. We have already seen how to treat the first factor $\frac{\lambda^c_{x^u}(\tau'(\ell))}{\lambda^c_{y^u}(\tau'(\ell))}\geq C_0^{-1}$. To treat the second factor, we use that $y\in \cW^s_1(x)$ and apply Corollary~\ref{c.distortionbasic2} to get $\frac{\lambda^c_{x}(t'(\ell))}{\lambda^c_{y}(t'(\ell))}\geq C_0^{-1}$. We treat the third factor by observing that $\lambda^c_{x_{t'(\ell)}}(k')\geq e^{k'\chi_2^c}$. Note that $|\tau(\ell)-\tau'(\ell)|\leq k$ so the last bound follows from $\frac{1}{\lambda^c_{x^u_{\tau'(\ell)}}(\tau(\ell)-\tau'(\ell))}\geq e^{-k\chi_1^c}$. Finally we find
\begin{equation}
\frac{\lambda^c_{x}(t'(\ell)+k')}{\lambda^c_{x^u}(\tau(\ell))}\geq \frac{e^{k'\chi_2^c}}{C_0^2e^{k\chi_1^c}}\times\frac{\lambda^c_{y}(t'(\ell)}{\lambda^c_{y^u}(\tau'(\ell))}\geq 1,
\end{equation}
as soon as we choose $k'=k'(f)$ such that

\begin{equation}
\label{e.klinha}
e^{k'\chi_2^c}\geq C_0^2e^{k\chi_1^c}.
\end{equation}

Hence by definition we must have $t(\ell)\leq t'(\ell)+k'$ so $t(\ell)-t'(\ell)\leq k'$. Here again a symmetric argument yields $t'(\ell)-t(\ell)\leq k'$. Hence
\begin{equation}\label{eq_T0}
T_0=\max(k,k'),
\end{equation}
is the desired constant. This ends the proof.
\end{proof}

\section{Matching of $Y$-configurations}\label{s.coupled}

%In this section we shall perform a detailed analysis on the geometry (in normal form coordinates) of coupled $Y$-configurations and also on the stopping time functions. In particular, we will estimate the drift along the center on the top part of the configuration as well as their effect on leaf-wise quotient measures. 

In this section we develop a concept devised specifically to address the technical difficulty to implement the exponential drift idea we want to employ, a difficulty which was described at the end of Section~\ref{heuristics}. Namely, due to lack of absolute continuity of center stable holonomies we cannot ensure that the points $x^u$ and $y^u$ of Figure~\ref{f.coupledconfig} belong both to the Lusin set. What we can actually prove is that small perturbations of these two points can indeed be put inside the Lusin set. This will lead us to the notion of matched configurations.

\subsection{Matching of dynamical balls}

We start the formal definition of matching in our scenario by introducing intervals along the unstable manifold which measure the amount of perturbation of the points $x^u$ and $y^u$ which are allowed, without breaking the estimates we performed for quadrilaterals in the previous section. 

\subsubsection{Unstable dynamical balls}

Although the notion of dynamical balls is quite standard in ergodic theory, we use this name in this paper for a more specific object, adapted to our needs.   

\begin{defi}\label{d.unstaball}
Let $\eps>0$, $\ell\in\N$, $x\in\mathbb{T}^3$ and $x^u\in\cW^u(x)$. Let $\tau(\ell)=\tau(x,x^u,\eps,\ell)$. The \emph{$(\eps,\ell)$-unstable dynamical ball} at $x^u$ is defined as
\[
J(x^u)=J(x^u,\eps,\ell)\eqdef f^{-\tau(\ell)}\left(\cW^u_{1}(f^{\tau(\ell)}(x^u))\right).
\] 
\end{defi}

\begin{remark}\label{r.unstaball} By definition, for every $a\in J(x^u)$ and $j\in\{0,\ldots,\tau(x,x^u,\eps,\ell)\}$
$$d_u(f^j(x^u),f^j(a))<1.$$
\end{remark}

\subsubsection{Synchronization inside a dynamical ball} Before we define the notion of matching of $Y$-configuration it is useful to study the oscillation of stopping times inside an unstable dynamical ball.

Let $T_0>0$ be the constant obtained in Lemma~\ref{lemme synchro}.
\begin{lemma}
	\label{l.synchroinstable}
For every $a\in J(x^u)$ it holds $|\tau(x,a,\eps,\ell)-\tau(x,x^u,\eps,\ell)|<T_0$ and $|t(x,a,\eps,\ell)-t(x,x^u,\eps,\ell)|<T_0$.
\end{lemma}
\begin{proof}
The proof is almost identical to that of Lemma \ref{lemme synchro}. In particular  constants $0<\chi_2^c<\chi_1^c$ are those defined in \S \ref{sss.hypestimates} and $k$ and $k'$, those defined by \eqref{eq_choicek} and \eqref{e.klinha}. Let us denote $\tau_a(l)\eqdef\tau(x,a,\eps,\ell)$. Suppose first that $\tau_a(\ell)\geq\tau(\ell).$

\begin{equation}\label{eq_telecscopouille2}
\frac{d^\ell_x\lambda^c_a(\tau(\ell)+k)}{ d^\ell_x\lambda^c_{x^u}(\tau(\ell))}=\frac{\lambda^c_a(\tau(\ell))}{\lambda^c_{x^u}(\tau(\ell))}\times \lambda^c_{a_{\tau(\ell)}}(k).
\end{equation}

We have $\lambda^c_{a_{\tau(\ell)}}(k)\geq e^{k\chi_2^c}$. Finally we have by definition $f^{\tau(\ell)}(a)\in \cW^u_1(f^{\tau(\ell)}(x^u))$ so by the distortion control \eqref{e.distortionbasic} we have $\frac{\lambda^c_a(\tau(\ell))}{\lambda^c_{x^u}(\tau(\ell))}\geq C_0^{-1}$ and, by choice of $k$,
$$d^\ell_x\lambda^c_a(\tau(\ell)+k)\geq \frac{e^{k\chi_2^c}}{C_0}d^\ell_x\lambda^c_{x^u}(\tau(\ell))\geq\eps.$$

Hence $\tau_a(\ell)-\tau(\ell)\leq k$. A symmetric argument yields $\tau(\ell)-\tau_a(\ell)\leq k.$

We consider now the stopping time $t$. As above, we denote $t_a(\ell)\eqdef\tau(x,a,\eps,\ell)$ and $t(\ell)=t(x,a,\eps,\ell)$. Let us first suppose $t_a(\ell)\geq t(\ell)$. We have
\begin{align*}\label{eq_telecscopouille2}
\frac{\lambda^c_x(t(\ell)+k')}{\lambda^c_a(\tau_a(\ell))}\times {\frac{\lambda^c_{x^u}(\tau(\ell))}{\lambda^c_x(t(\ell))}}&=\frac{\lambda^c_{x^u}(\tau(\ell))}{\lambda^c_a(\tau(\ell))}\times \lambda^c_{x_{t'(\ell)}}(k')
\\
  &\times \frac{1}{\lambda^c_{x^u_{\tau(\ell)}}(\tau(\ell)-\tau'(\ell))}.
\end{align*}
The first factor is $\geq C_0^{-1}$. The second one is $\geq e^{k'\chi_2^c}$. The third one is $\geq e^{-k{\chi_1^c}}$. Hence by our choice of $k'$,
$$\frac{\lambda^c_x(t(\ell)+k')}{\lambda^c_a(\tau_a(\ell))}\geq\frac{e^{k'\chi_2^c}}{C_0e^{k\chi_1^c}}\times \frac{\lambda^c_{x^u}(\tau(\ell))}{\lambda^c_x(t(\ell))}\geq 1.$$

This proves that $t_a(\ell)-t(\ell)\leq k'$. Again, a symmetric argument gives $t_a(\ell)-t(\ell)\leq k'$ if $t(\ell)\geq t_a(\ell)$. This ends the proof of the lemma because recall that $T_0=\max(k,k')$.
%
% For simplicity let us denote $\tau_a\eqdef\tau(x,a,\eps,\ell)$. Assume first that $\tau_a>\tau$ and set $n=\tau+1$. Notice that by definitio
%\[
%|f^n(J(x^u))|=|f\left(W^u_\rho(x_\tau)\right)|\leq 2d_1\rho.
%\]	
%We can assume without lost of generality that $\rho>0$ is small so that $2d_1\rho\leq 1$. We are thus allowed to apply the basic distortion estimate \eqref{e.distortionbasic} to get
%\[
%C_0^{-1}\leq\frac{\|Df^n(x^u)|_{E^c}\|}{\|Df^n(a)|_{E^c}\|}\leq C_0.
%\]
%By \eqref{e.tau} again, as $n>\tau$, we have
%\[
%d^\ell_x\|Df^n(x^u)|_{E^c}\|>\eps,
%\]
%and therefore
%\[
%\frac{\eps}{d^\ell_x}<\|Df^n(x^u)|_{E^c}\|\leq C_0\|Df^n(a)|_{E^c}\|,
%\]
%which implies that
%\[
%d^\ell_x\|Df^{n+k}(a)|_{E^c}\|\geq C_0^{-1}e^{\chi_3k}\eps>\eps.
%\]
%As before, this leads to $n+k>\tau_a$ and so
%\[
%\tau_a-\tau<k+1.
%\]
%If now $\tau>\tau_a$ we set $n=\tau_a+1$, interchange the roles of $a$ and $x^u$ in the above reasoning and conclude in the same way that 
%\[
%\tau-\tau_a<k+1.
%\]  
%Since $T_0>k+1$ the proof is complete.
%
%\textbf{MUST TREAT t}
\end{proof}

\subsubsection{Matching of dynamical balls and distortion control} We first define the notion of \emph{matched unstable dynamical balls}  and study their geometric properties.

\begin{defi}[Matched unstable dynamical balls]\label{d.match_db}
When $(x,x^u,y,y^u)$ is a $(C,\ell)$-quadrilateral, then we say that $J(x^u)=J(x^u,\eps,\ell)$ and $J(y^u)=J(y^u,\eps,\ell)$ are $(C,\ell)$-\emph{matched}.
\end{defi}

We will need the following distortion control for matched unstable dynamical balls.

\begin{prop}\label{p_matched-dist_control}
	For $\eps$ small enough, there exists $\kappa_3=\kappa_3(f,\eps)$ such that if $\ell$ is sufficiently large,  for every $(C,\ell)$-matched dynamical balls $J(x^u)$ and $J(y^u)$, and every $a\in J(x^u)$, $b\in J(y^u)$, we have
	$$\kappa_3^{-1}\leq\frac{\lambda^\ast_b(j)}{\lambda^\ast_a(j)}\leq\kappa_3,$$
	for every integer $0\leq j\leq\max(\tau(x,x^u,\eps,\ell),\tau(y,y^u,\eps,\ell))$ and $\ast=c,u$.
\end{prop}

\begin{proof}
	Let us assume that $\tau(\ell)\leq\tau'(\ell)$, where $\tau(\ell)=\tau(x,x^u,\eps,\ell)$ and $\tau'(\ell)=\tau(y,y^u,\eps,\ell)$. Let us first notice that by Lemma \ref{lemme synchro},  we have $\tau'\leq\tau+T_0$.
	
	Let $j\leq\tau(\ell)$ so in particular
	$$|f^{j}(J(x^u))|,|f^{j}(J(y^u))|<1.$$
Let $z^u=H^s_{x,y}(x^u)$ so $d(x^u,z^u)<1$, for $\ell$ large enough,  and, by Corollary~\ref{c.gapdoseba}, $d(f^{j}(y^u),f^{j}(z^u))\leq d(f^{\tau'}(y^u),f^{\tau'}(z^u))\asymp_{C,T_0}\eps<1$.
	
	Let $\phi_\ast\eqdef\log\|Df|_{E^{\ast}}\|$ and $c_\ast,\theta_\ast$ be the Hölder constant and exponent of $\phi_\ast$, for $\ast = c,u$. We have
	$$\left|\log\frac{\lambda^\ast_b(j)}{\lambda^\ast_a(j)}\right|\leq\sum_{i=0}^{j-1}|\phi_\ast(f^i(a))-\phi_\ast(f^i(b))|.$$
	Hence,
	\begin{align*}
		|\phi_\ast(f^i(a))-\phi_\ast(f^i(b))|&\leq  |\phi_\ast(f^i(a))-\phi_\ast(f^i(x^u))|+|\phi_\ast(f^i(x^u))-\phi_\ast(f^i(z^u))|
		\\
		&+ |\phi_\ast(f^i(z^u))-\phi_\ast(f^i(y^u))|+|\phi_\ast(f^i(y^u))-\phi_\ast(f^i(b))|,
	\end{align*}
	Note that for $i<j$, $d(f^i(a),f^i(x^u)),d(f^i(b),f^i(y^u))\leq e^{\chi^u_1(i-j)}$. On the other hand, we have $d(f^i(x^u),f^i(z^u))\leq e^{i\chi^s_2}$. Finally $d(f^i(z^u),f^i(y^u))\leq e^{\chi^c_1(i-j)}$. It follows that
	$$|\phi_\ast(f^i(a))-\phi_\ast(f^i(b))|\leq c_\ast\left(2e^{\theta_\ast\chi^u_1(i-j)}+e^{i\theta_\ast\chi^s_2}+e^{\theta_\ast\chi^c_1(i-j)}\right).$$
	
	By summing over $i$ we deduced that $\big|\log\frac{\lambda^\ast_b(j)}{\lambda^\ast_a(j)}\big|$ is uniformly bounded from above by a constant depending only on $f$ and $\rho$. Of course we can now bound these quotients for $j$ up to  $\tau'(\ell)$ by using that
	$$d_0^{-T_0}\frac{\lambda^\ast_b(j)}{\lambda^\ast_a(j)}\leq \frac{\lambda^\ast_b(j+T_0)}{\lambda^\ast_a(j+T_0)}\leq d_0^{T_0}\frac{\lambda^\ast_b(j)}{\lambda^\ast_a(j)},$$
where $d_0=\|Df\|/m(Df)$ (recall \S\ref{sss.notaderiva}).
\end{proof}

\begin{corollary}
	\label{c_size_match}
	There exists $\kappa_4=\kappa_4(f,\eps)$ such that if $\ell$ is sufficiently large, for every $(C,\ell)$ matched dynamical balls $J(x^u)$ and $J(y^u)$, 
	$$\kappa_4^{-1}\leq\frac{|f^j(J(y^u))|}{|f^j(J(x^u))|}\leq\kappa_4,$$
	for every integer $0\leq j\leq\max(\tau(x,x^u,\eps,\ell),\tau(y,y^u,\eps,\ell))$.
\end{corollary}

\begin{proof}
	As before we denote $\tau(\ell)=\tau(x,x^u,\eps,\ell)$ and $\tau'(\ell)=\tau(y,y^u,\eps,\ell)$.
	Suppose $\tau(\ell)\leq\tau'(\ell)\leq\tau(\ell)+T_0$ (where $T_0$ is defined by Lemma \ref{lemme synchro}. In particular $m(Df)^{T_0}\leq |f^{\tau(\ell)}(J(y^u))|\leq 1$ and $|f^{\tau(\ell)}(J(x^u))|=1$.
	
	For $j\in\{0,\ldots,\tau(\ell)\}$, we notice that
	$$\frac{|f^j(J(x^u))|}{|f^j(J(y^u))|}=\frac{\int_{f^{\tau(\ell)}(J(x^u))}\lambda^u_q(k)dq}{\int_{f^{\tau(\ell)}(J(y^u))}\lambda^u_q(k)dq},$$
where $k=j-\tau(\ell)\leq 0$.	
By the bounded distortion along the strong unstable manifolds \eqref{e.distortionbasic} we find
\begin{align*}
C_0^{-2}\frac{\lambda^u_{f^{\tau(\ell)}(x^u)}(k)}{\lambda^u_{f^{\tau(\ell)}(y^u)}(k)}\frac{|f^{\tau(\ell)}(J(x^u))|}{|f^{\tau(\ell)}(J(y^u))|}&\leq \frac{|f^j(J(x^u))|}{|f^j(J(y^u))|}\\ &\leq C_0^2\frac{\lambda^u_{f^{\tau(\ell)}(x^u)}(k)}{\lambda^u_{f^{\tau(\ell)}(y^u)}(k)}\times\frac{|f^{\tau(\ell)}(J(x^u))|}{|f^{\tau(\ell)}(J(y^u))|}.
\end{align*}
Notice that 
\[
\lambda^u_{f^{\tau(\ell)}(x^u)}(k)=\frac{\lambda^u_{x^u}(-k)}{\lambda^u_{x^u}(\tau(\ell))},
\]
and a similar equation holds with $y^u$ instead of $x^u$. Therefore, we can apply Proposition \ref{p_matched-dist_control} to obtain
$$C_0^{-2}\kappa_3^{-2}\leq\frac{|f^j(J(x^u))|}{|f^j(J(y^u))|}\leq C_0^2\|Df\|^{T_0}(\kappa_3)^2.$$

We conclude the proof of the corollary by noting that 
	$$m(Df)^{2T_0}\frac{|f^j(J(x^u))|}{|f^j(J(y^u))|}\leq\frac{|f^{j+T_0}(J(x^u))|}{|f^{j+T_0}(J(y^u))|}\leq\|Df\|^{2T_0}\frac{|f^j(J(x^u))|}{|f^j(J(y^u))|}.\qedhere$$
\end{proof}

\subsection{Matching of $Y$-configurations}\label{match_Y_conf} We now define the notion of matched $Y$-configurations and show that our drift argument (Proposition \ref{p.levraidrift}) boils down to constructing arbitrarily long pairs of matched $Y$-configurations.
\subsubsection{Matching} We first give the main definition.
\begin{defi}[Matched $Y$-configurations]\label{d.matched} Let ${\color{blue}\cL}\dans\T^3$, $C>0$,
 $\ell\in\N$. Let $X=X(x,x^u,\ell)$ and $Y=Y(y,y^u,\ell)$ be two $Y$ configurations of length $\ell$. We say that $X$ and $Y$ are \emph{$({\color{blue}\cL},C,T)$-matched} if there exist $a,b\in\T^3$, and $\tau,t\in\N$ such that
\begin{enumerate}
\item $(x,x^u,y,y^u)$ is a $(C,\ell)$-quadrilateral;
\item $a\in J(x^u,\eps,\ell)$ and $b\in J(y^u,\eps,\ell)$;
\item\label{point point trois} $|\tau-\tau(x,x^u,\eps,\ell)|\leq T$ and $|t-t(x,x^u,\eps,\ell)|\leq T$;
\item $a,b,x,y,f^\tau(a),f^\tau(b),f^t(x),f^t(y)\in{\color{blue}\cL}$.
\end{enumerate}
\end{defi} 

\begin{figure}[h!]
\begin{tikzpicture}
\draw[thick, red!80!black] (2,4.5) .. controls (2.5,4.5) and (3,3)..(4,1.5).. controls (5,0) and (5,0) .. (6,-1) node[below]{\small$\cWu_1(y)$};
\fill[green!80!white, opacity=.2] (0,2)--(4,0)--(6,0.5)--(2,2.5)--(0,2);
\draw[thick, red!80!black] (0,2)--(4,0) node[below]{\small$\cWu_1(x)$};
\draw[thick, violet] (2,2.5)--(6,0.50);
\draw[dotted] (2,4.5)--(2,2.5);
\draw[dotted] (6,0.5)--(6,-1);
\draw[red!40!black, ultra thick] (2.65,4.2) node[above]{\small$J(y^u)$};
\draw[red!40!black, ultra thick] (2.35,4.25).. controls (2.47,4.3) and (3.4,2.5) .. (3.4,2.5);
\draw (3.2,2.85) node{$\bullet$};
\draw (3.2,2.85) node[right]{$b$};
\begin{scope}[xshift=-.4cm, yshift=.2cm]
\draw[red!40!black, ultra thick] (0.5,1.75)--(1.5,1.25) node[below]{\small$J(x^u)$};
\draw[thick, green!40!black] (1,1.5)--(3,2) node[midway,above]{\small$s$};
\draw[thick, orange] (3,2)--(3,3.7);
\draw (1,1.5) node{$\bullet$};
\draw (1,1.5) node[below]{$x^u$};
\draw (3,2) node{$\bullet$};
\draw (3,2) node[below]{$z^u$};
\draw (3,3.7) node{$\bullet$};
\draw (3,3.8) node[right]{$y^u$};
\draw (0.7,1.65) node{$\bullet$};
\draw (0.7,1.55) node[left]{$a$};
\end{scope}
\draw[thick, green!40!black] (2,1)--(4,1.5) node[midway, below]{$s$};
\draw (2,1) node{$\bullet$};
\draw (2,1) node[below]{$x$};
\draw (4,1.5) node{$\bullet$};
\draw (4,1.5) node[above]{$y$};
\draw[->, dashed, opacity=.5] (2.5,2.6) to[bend left] (1.5,3.5);
\draw[orange] (1.4,3.4) node[above]{\small$\cWc(y^u)$};
\end{tikzpicture}
\caption{\label{f.matchedconfig} For matched $Y$-configurations we can only put inside the Lusin set ``small perturbations'' $a$ and $b$ of the endpoints $x^y$ and $y^u$ (respectively) of the quadrilateral.}
\end{figure}

\begin{remark}\label{rem_symmetry}
At first sight this definition is not symmetric (see item \eqref{point point trois}). But the synchronization estimate (Lemma \ref{lemme synchro}) implies that,
$$|\tau-\tau(y,y^u,\eps,\ell)|\leq T+T_0,\text{          and                           }|t-t(y,y^u,\eps,\ell)|\leq T+T_0.$$
\end{remark}

\subsubsection{Finding pairs of long and matched $Y$-configurations} We are now ready to state our main technical lemma which, as we shall see, implies Proposition \ref{p.levraidrift}.

Let $\cL$ be the Lusin set defined in \S \ref{Lusin lusin} and $\delta=3\delta_0>0$ be the constant fixed in \S \ref{sss_exp-drift}: we shall give further assumption on $\delta$ later on.

\begin{lemma}
	\label{l.fatality}
Let $K_0\subset\cL$ be a compact set of measure $\mu(K_0)>1-\delta$.	There exist constants $C=C(\delta)$ and $T=T(\delta)$ and an infinite subset $\cD\subset\N$ such that for every $\ell\in\cD$ there exists a pair $(X,Y)$ of $(K_0,C,T)$-matched $Y$-configurations of length $\ell$.
\end{lemma}

The end of the section is devoted to proving that Lemma~\ref{l.fatality} implies Proposition~\ref{p.levraidrift}. The proof of Lemma \ref{l.fatality} will be the object of Section \ref{s.end_proof}.

\subsection{Asymptotic control of leaf-wise measures for matched configurations}\label{ss_control_leafwise_match}

The goal of this paragraph is to study how leaf-wise quotient measures change along a pair of matched $Y$-configurations, and what happens when we have sequences of longer and longer pairs. To simplify the exposition, we shall break this explanation in two parts. First we deal with a single $Y$-configuration, and then we treat the full situation.

\subsubsection{Control along a $Y$-configuration}\label{sss_controlY}

The lemma below is essentially an easy corollary of Lemmas~\ref{chang dyna} and \ref{changt inst}.

\begin{lemma}\label{coro controle affine var measu}
	For $x\in \T^3$, $x^u\in \mathcal{W}_r^u(x)$, for some $r>0$, $\varepsilon>0$, and $\ell \in \N$ as above, it holds 
	$$
	\hat \nu_{f^{\tau(\ell)}(x^u)}^c\propto (A_{x,x^u,\ell})_* \hat \nu_{f^{t(\ell)}(x)}^c,
	$$ 
%	for a constant (recall Notation \ref{notation constantes k j})
%	\begin{equation*}
%	C_{x,x^u,\ell} \eqdef \prod_{k=0}^{\tau(\ell)-1} J_{f^k(x^u)} K_{x,x^u}^u \prod_{k=0}^{t(\ell)-1} J_{f^k(x)}^{-1}>0,
%	\end{equation*} 
	for the linear map $A_{x,x^u,\ell} \colon s \mapsto a_{x,x^u,\ell} \cdot s$, with
	\begin{equation*}
	a_{x,x^u,\ell}\eqdef \rho_{x^u}^c(x)%\mathrm{Jac}(H_{x,x^u}^u) 
	 \frac{\lambda^c_x(t(\ell))}{\lambda^c_{x^u}(\tau(\ell))} .
	\end{equation*}
%	and for some  constant 
%	$$
%	C_{x,x^u,\ell} \eqdef \prod_{k=0}^{\tau(\ell)-1} K_{f^k(x^u)} J_{x,x^u} \prod_{k=0}^{t(\ell)-1} K_{f^k(x)}^{-1}>0.
%	$$ 
	Moreover, $|a_{x,x^u,\ell}|\in (a_0^{-1},a_0)$ for some constant $a_0=a_0(r)>1$ depending only on the upper bound $r>0$ on the distance along $\mathcal{W}^u$ between $x$ and $x^u$; in particular, the linear map $A_{x,x^u,\ell}$ is uniformly bounded away from $0$ and $\infty$, independently of $\ell$.  
\end{lemma}

\begin{proof}
	Applying Lemma \ref{chang dyna} and Lemma \ref{changt inst} we successively obtain:
	\begin{align*}
		\hat \nu_{f^{\tau(\ell)}(x^u)}^c&\propto (\Lambda_{\tau(\ell),x^u}^{c})_* \hat \nu_{x^u}^c\\
		&\propto   (L_{x,x^u}\circ \Lambda_{\tau(\ell),x^u}^{c})_* \hat \nu_{x}^c\\
		&\propto  (L_{x,x^u} \circ \Lambda_{\tau(\ell),x^u}^{c}\circ (\Lambda_{t(\ell),x}^{c})^{-1})_* \hat \nu_{f^{t(\ell)}(x)}^c,
	\end{align*}
%where for $w\in \TT$  and $n\in \N$, we let
%\begin{equation*}
%	J_{n,w}\eqdef  \prod_{k=0}^{n-1} J_{f^k(w)}.
%\end{equation*}
%	noting that one-dimensional linear maps commute, and
%	where, for each $y\in \TT$  and $n\in \N$, we define the multiplicative cocycles
%	\begin{align*}
%		K_{n,y}&\eqdef  \prod_{k=0}^{n-1} K_{f^k(y)},\\ 
%		\Lambda_{n,y}^{c}&\colon s \mapsto \|Df^n(y)|_{E_f^c}\| \cdot s= \Lambda_{f^{n-1}(y)}^c \circ \dots \circ \Lambda_{y}^c(s).
%	\end{align*}
where $\Lambda^c_{n,x}$ denotes the linear map
$s\mapsto\lambda^c_x(n)s$, and $L_{x,x^u}$ denotes the linear map $s \mapsto \rho_{x^u}^c(x) \cdot s$. Thus, $A_{x,x^u,\ell}\eqdef\Lambda_{\tau(\ell),x^u}^{c} \circ L_{x,x^u}\circ (\Lambda_{t(\ell),x}^{c})^{-1}$  is equal to 
	$$
	s \mapsto \rho_{x^u}^c(x)  \frac{\lambda^c_x(t(\ell))}{\lambda^c_{x^u}(\tau(\ell))} \cdot s.
	$$ 
	By the fact that $d_{u}(x,x^u)<r$, and by the definition of $t(\ell)$, $|A_{x,x^u,\ell}'(0)|$ is uniformly bounded, depending  only on $r$ and $f$, but not on $\ell$, which concludes the proof. 
\end{proof}

\subsubsection{Control for matched configurations}
We now deal with the full picture of sequences of pairs of good and matched configurations. 
For the next result, we refer to \S \ref{Lusin lusin} for the definition of the Lusin set $\mathcal{L}$, to \S \ref{sss.quadri_couple} and to \S \ref{match_Y_conf} for that of quadrilaterals and $(C,\ell)$-matched $Y$-configurations respectively.

The result of this subsection is the core of the proof that Lemma~\ref{l.fatality} implies Proposition~\ref{p.levraidrift}. \emph{Until the end of this section, the notation $x_n$ will NOT stand for $f^n(x)$ but for the usual notation of sequences}.

\begin{lemma}\label{coro depl mes}
		Let $\eps>0$ $(\ell_n)$ be an increasing sequence of integers. Suppose there exist constants $T,C>0$ (independent of $\eps$), and two sequences $X_n=(x_n,x^u_n,\ell_n)$ and $Y_n=(y_n,y^u_n,\ell_n)$ of $Y$-configurations of length $\ell_n$ that are $(\cL,C,T)$-matched. Let $\tau_n,t_n,a_n,b_n$ be the objects given by the condition of matching and  consider the sequences 
		\[
		f^{\tau_n}(a_n),f^{\tau_n}(b_n)\in\mathcal{L}, \text{           and           }
		f^{\tau_n}(x^u_n),f^{\tau_n}(y^u_n)\in\T^3.
		\]
		Assume all these sequences converge to points $a_{\infty}$, $b_{\infty}$, $p$ and $q$ respectively. Then $q\in \cW^c(p)$ and there exists $\gamma=\gamma(C,T)>1$ and $M=M(C,T)$ such that
		\begin{equation}\label{affine param centr}
		\frac{\varepsilon}{\gamma} \leq|\mathcal{H}_{p}^c(q)|\leq\gamma \varepsilon,
		\end{equation}
		and
		\begin{equation}\label{affine invar meas}
		\hat \nu_{a_{\infty}}^c\propto B_* \hat \nu_{b_{\infty}}^c.
		\end{equation}
where $B(s)= \beta \cdot s$, is a linear map that satisfies $\frac 1M<|\beta|<M$.

	\end{lemma}

%}

\begin{proof}
Note first that $\tau_n$ and $t_n$ tend to infinity: indeed, this is a consequence of the matching condition and of the fact that $\ell_n\to\infty$ so the stopping times also tend to infinity (by Lemma \ref{l_qi_estimates}).

Now suppose the sequences $f^{\tau_n}(a_n),f^{\tau_n}(b_n),f^{\tau_n}(x_n),f^{\tau_n}(y_n)$ converge to $a_\infty,b_\infty,p,q$ respectively. 
Let $z_n^u=H^s_{x_n,y_n}(x_n^u)\in\cW^s(x_n)$. The condition of matching implies that $d(z_n^u,x^u_n)$ is uniformly bounded so $d(f^{\tau_n}(z_n),f^{\tau_n}(x_n))\to 0$ as $n\to+\infty$ and $f^{\tau_n}(z_n^u)\to p$. Since normal forms changes continuously, we have that
\[
\cH^c_{f^{\tau_n}(z^u_n)}(f^{\tau_n}(y^u_n))\to\cH^c_p(q),\quad \text{as}\quad n \to +\infty. 
\]	
Thus \eqref{affine param centr} follows directly from Lemma~\ref{control centre expa} and Remark \ref{rem_symmetry}. Let us show \eqref{affine invar meas}. For this we use the matching condition: for all $n$, the points $f^{\tau_n}(a_n),f^{\tau_n}(b_n),f^{t_n}(x_n),f^{t_n}(y_n)$ all belong to the Lusin set $\cL$.

%For this notice that $\tilde{X}=(x_n,a_n,\ell_n)$ and $\tilde{Y}_n=(y_n,b_n,\ell_n)$ are also $\cL$-good $Y$-configurations. 
Combining Lemma~\ref{coro controle affine var measu} with Lemma~\ref{changt inst} and Remark \ref{rem_symmetry} we deduce that there exist linear maps $B_n\colon\R\to\R$ and $\tilde{B}_n\colon\R\to\R$ and constants $C_n,\tilde{C}_n>0$ such that the derivatives $B_n^\prime$ and $\tilde{B}_n^\prime$ lie inside $[M^{-1},M]$ for some constant $M=M(T,C)$, independent of $n$ and 
\[
\hat{\nu}^c_{f^{\tau_n}(a_n)}=C_n(B_n)_*\hat{\nu}^c_{f^{t_n}(x_n)}\:\:\:\textrm{and}\:\:\:\:\:\hat{\nu}^c_{f^{\tau_n}(b_n)}=\tilde{C}_n(\tilde{B}_n)_*\hat{\nu}^c_{f^{t_n}(y_n)}.
\] 
We claim that $C_n$ and $\tilde{C}_n$ are uniformly bounded away from $0$ and $\infty$. Indeed, let us consider $C_n$. Let $c=c(M)>0$ be the constant given by Corollary~\ref{c.lusin}. Then, denoting $I=[-1,1]$ we have that, by our choice of normalization (see Remark~\ref{rem.normaliza}) $\hat{\nu}^c_{f^{\tau_n}(a_n)}(I)=1$ and therefore
\[
C_n=\frac{1}{\hat{\nu}^c_{f^{t_n}(x_n)}(B_n^{-1}(I))}\in[c^{-1},c],
\]
for $f^{t_n}(x_n)\in\cL$ for each $n$. A similar argument treats the constants $\tilde{C}_n$. Therefore up to enlarging $M$, we can assume that $C_n,\tilde{C}_n\in[M^{-1},M]$.

The condition of matching implies that $y_n\in \cW^s(x_n)$ for every $n$ and that $d(x_n,y_n)$ is uniformly bounded. Now recall that $f^{t_n}(x_n),f^{t_n}(y_n)$ belong to the Lusin set $\cL$, which is compact. We can assume without loss of generality that there exists a point $x\in\cL$ such that $f^{t_n}(x_n)\to x$ and $f^{t_n}(y_n)\to x$. For simplicity, let us denote $\hat{a}_n=f^{\tau_n}(a_n)$, $\hat{b}_n=f^{\tau_n}(b_n)$, $\hat{x}_n=f^{t_n}(x_n)$ and $\hat{y}_n= f^{t_n}(y_n)$.  We can also assume that $C_n\to C$ and $\tilde{C}_n\to\tilde{C}$ and that there exist linear maps $B,\tilde{B}:\R\to\R$ such that $B_n\to B$ and $\tilde{B}_n\to\tilde{B}$ in the $C^1$ topology, i.e. the slopes of $B_n$ converge to the slope of $B$ and similarly for $\tilde{B}_n$ and $\tilde{B}$. 

Now, we claim that $\hat{\nu}^c_{a_{\infty}}=CB_*\hat{\nu}^c_x$ and $\hat{\nu}^c_{b_{\infty}}=\tilde{C}\tilde{B}_*\hat{\nu}^c_x$.  Once we prove this claim, the lemma will be established. To prove the claim, we first observe that 
\[
\int C_n(\varphi\circ B_n)d\hat{\nu}^c_{\hat{x}_n}\to\int C(\varphi\circ B)d\hat{\nu}^c_x,
\]
for every $\varphi\in C^0_c(\R)$. Indeed,  we can write
\begin{align}
\left|\int C_n(\varphi\circ B_n)d\hat{\nu}^c_{\hat{x}_n}-\int C(\varphi\circ B)d\hat{\nu}^c_x\right|&\leq\left|\int C_n(\varphi\circ B_n)d\hat{\nu}^c_{\hat{x}_n}-\int C(\varphi\circ B)d\hat{\nu}^c_{\hat{x}_n}\right|\nonumber\\
&+\left|\int C(\varphi\circ B)d\hat{\nu}^c_{\hat{x}_n}-\int C(\varphi\circ B)d\hat{\nu}^c_{x}\right|.\nonumber
\end{align}
The second term on the right-hand side above converges to zero as $n\to\infty$ because $\cL$ satisfies the conclusion of Lemma~\ref{l.lusin} and $\hat{x}_n\in\cL$ converges to $x$. The first term also converges to zero because $C_n(\varphi\circ B_n)$ converges uniformly to $C(\varphi\circ B)$ and $\hat{\nu}^c_{\hat{x}_n}(\operatorname{supp}(\varphi))$ is bounded independently of $n$. We use this observation to conclude the proof of the claim as follows. Given any $\varphi\in C^0_c(\R)$ we have that
\begin{align}
	\int\varphi d\hat{\nu}^c_{a_{\infty}}&=\lim_{n\to\infty}\int\varphi d\hat{\nu}^c_{\hat{a}_n}=\lim_{n\to\infty}\int C_n(\varphi\circ B_n) d\hat{\nu}^c_{\hat{x}_n}\nonumber\\
	&=\int C(\varphi\circ B)d\hat{\nu}^c_x,\nonumber
\end{align} 
proving that $\hat{\nu}^c_{a_{\infty}}=CB_*\hat{\nu}^c_x$. The proof that $\hat{\nu}^c_{b_{\infty}}=\tilde{C}\tilde{B}_*\hat{\nu}^c_x$ is similar so we omit it. 
The proof of the lemma is complete.
\end{proof}

With Lemma~\ref{coro depl mes} at hand we are now in position to reduce the proof of Theorem~\ref{mainthm.technique} to the proof of Lemma~\ref{l.fatality}.

\subsection{Proof that Lemma \ref{l.fatality} $\implies$ Proposition~\ref{p.levraidrift}}\label{ss.mainreduction}

Let $K_{00}$ be a compact set with $\mu(K_{00})>1-2\delta_0$. We apply Lemma \ref{l.fatality} to the compact set $K_0=K_{00}\cap\cL$ which has measure $\mu(K_0)>1-3\delta_0=1-\delta$. So let 
 $C=C(\delta)>0$, $T_0=T_0(\delta)>0$, $K_{0}=K_{00} \cap \mathcal{L}$,  and $\mathcal{D}$ be the objects given by Lemma \ref{l.fatality}.  

\begin{figure}[h!]
	\begin{tikzpicture}
	\begin{scope}[scale=.8]
  \fill[green!80!white, opacity=.2] (0,2)--(4,0)--(6,0.5)--(2,2.5)--(0,2);
  \draw[thick, red!80!black] (0,2)--(4,0) node[below]{\small$\cWu_1(x)$};
  \draw[thick, violet] (2,2.5)--(6,0.50);
  \draw[thick, red!80!black] (2,4.5) .. controls (2.5,4.5) and (3,3)..(4,1.5).. controls (5,0) and (5,0) .. (6,-1) node[below]{\small$\cWu_1(y)$};
  \draw[dotted] (2,4.5)--(2,2.5);
  \draw[dotted] (6,0.5)--(6,-1);
  \draw[red!40!black, ultra thick] (2.35,4.25).. controls (2.47,4.3) and (3.4,2.5) .. (3.4,2.5);
  \draw (3.2,2.85) node{$\bullet$};
  \draw (3.2,2.85) node[right]{$b$};
  \draw[red!40!black, ultra thick] (2.65,4.2) node[above]{\tiny$J(y^u)$};
  \begin{scope}[xshift=-.4cm, yshift=.2cm]
  \draw[red!40!black, ultra thick] (0.5,1.75)--(1.5,1.25);
  \draw[red!40!black] (0.5,1.75) node[left]{\tiny$J(x^u)$};
  \draw[thick, green!40!black] (1,1.5)--(3,2) node[midway,above]{\small$s$};
  \draw[thick, orange] (3,2)--(3,3.7);
  \draw (1,1.5) node{$\bullet$};
  \draw (1,1.5) node[below]{$x^u$};
  \draw (3,2) node{$\bullet$};
  \draw (3,2) node[below]{$z^u$};
  \draw (3,3.7) node{$\bullet$};
  \draw (3,3.8) node[right]{$y^u$};
  \draw (0.7,1.65) node{$\bullet$};
  \draw (0.7,1.55) node[left]{$a$};
  \end{scope}
  \draw[thick, green!40!black] (2,1)--(4,1.5) node[midway, below]{$s$};
  \draw (2,1) node{$\bullet$};
  \draw (2,1) node[below]{$x$};
  \draw (4,1.5) node{$\bullet$};
  \draw (4,1.5) node[above]{$y$};
  %\draw[->, dashed, opacity=.5] (2.5,2.6) to[bend left] (1.5,3.5);
  %\draw[orange] (1.4,3.4) node[above]{\small$\cWc(y^u)$};
	\end{scope}
	\fill[green!80!white, opacity=.2] (-1,7)--(0,7)--(1.5,10)--(.5,10)--(-1,7);
	\draw[orange, thick] (0.75,8.5)--(0.75,12);
    \draw[red!40!black, ultra thick] (-1,7)--(.5,10);
	\draw[thick, green!40!black] (-.25,8.5)--(.75,8.5) node[midway, below]{$s$};
	\draw (.25,9.5) node{$\bullet$};
	\draw (.25,9.5) node[left]{$a_{\tau}$};
	\fill[pink!60!white, opacity=.5] (0,7)-- (0,10.5)--(1.5,13.5)--(1.5,10)--(0,7);
 	\draw[red!40!black, ultra thick] (0,10.5)--(1.5,13.5);
 	\draw (0.75,12) node{$\bullet$};
 	\draw (0.75,12) node[right]{$y^u_{\tau}$};
 	\draw (.25,11) node{$\bullet$};
 	\draw (.25,11) node[right]{$b_{\tau}$};
 	\draw (0.75,8.5) node{$\bullet$};
 	\draw (0.75,8.5) node[right]{$z^u_{\tau}$};
 	\draw (-.25,8.5) node{$\bullet$};
 	\draw (-.25,8.5) node[left]{$x^u_{\tau}$};
 	\draw [->, dashed, opacity=.5] (1,2.5) to[bend left] node[midway, left]{$f^\tau$} (0,6.5);
 	\draw[red!40!black](-1,7) node[left]{\small$f^\tau(J(x^u))$};
 	\draw[red!40!black](1.5,13.5) node[left]{\small$f^\tau(J(y^u))$};
 	\draw[->, dashed, opacity=.5] (1.5,.8) to [bend right] node[midway, right]{$f^t$} (5,7);
 	\draw[->, dashed, opacity=.5] (3.3,1.3) to [bend right] node[midway, right]{$f^t$} (5.5,7);
 	\draw[green!40!black, thick] (5,7.25)--(5.5,7.25) node[midway, above]{$s$};
 	\draw (5,7) node[above]{$\bullet$};
 	\draw (5.5,7) node[above]{$\bullet$};
 	\draw (5,7.3) node[left]{$x_t$};
 	\draw (5.5,7.3) node[right]{$y_t$};
 	\draw[->, dotted, opacity=.6] (0.75,10) to[bend left] (2,11);
\draw[orange, thick] (2,11) node[right]{$\asymp\eps$};
   \draw[->, dotted, opacity=.6] (1.25,12) to[bend right] (2.5,13);
   \draw[pink!60!black] (2.5,13) node[above]{$\cWcu(z^u_{\tau})$};
   
 	\end{tikzpicture}
 	\caption{\label{fig.driftalongthecenter} For each $\ell_n$ we have a picture like this one (we have suppressed the dependence on $n$ for simplicity). The top-left part will converge to Figure~\ref{fig.nolimite}.}
\end{figure}
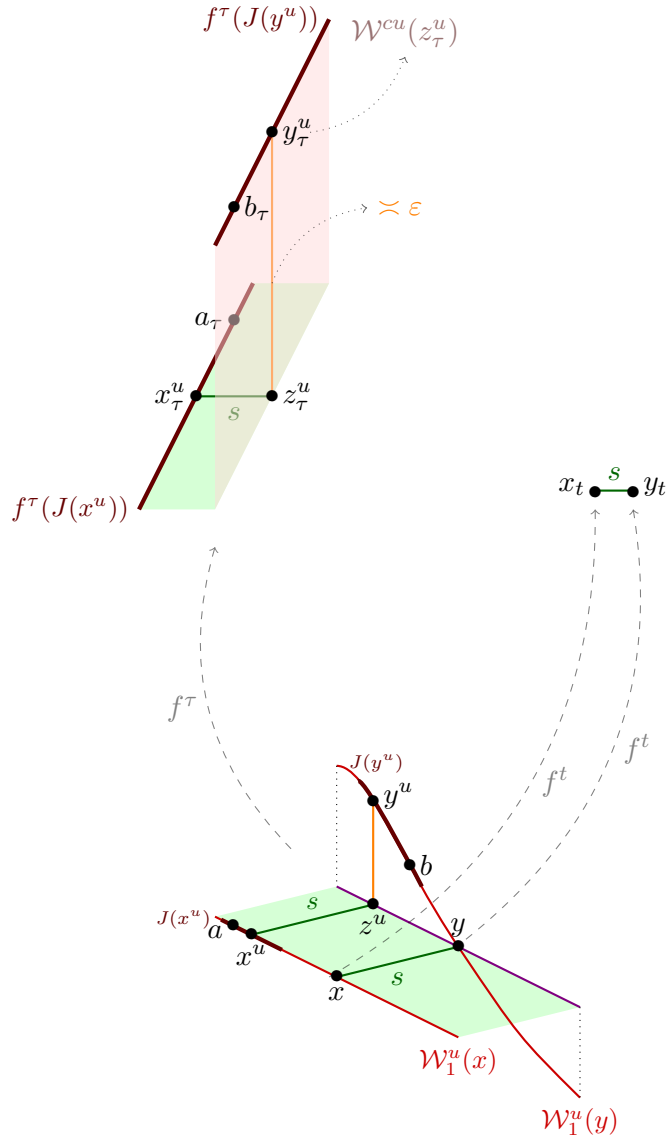

As the set $\cD$ is infinite, there exists a sequence $\ell_n\to+\infty$ of integers belonging to $\cD$. For each such integer, let $X_n=(x_n,x^u_n,\ell_n)$ and $Y_n=(y_n,y^u_n,\ell_n)$ be the pair of $(K_0,C,T)$-matched $Y$-configurations given by Lemma~\ref{l.fatality}. We let $a_n,b_n,\tau_n,t_n$ be the points and times corresponding to the pair $(X_n,Y_n)$ (see Definition \ref{d.matched}).  By definition, we have 
\[
f^{\tau_n}(a_n),f^{\tau_n}(b_n)\in K_0.
\]
We also consider the sequences 
\[
f^{\tau_n}(x^u_n),f^{\tau_n}(y^u_n).
\]

Upon extracting subsequences if necessary, we may assume that these four sequences converge respectively to points $a_{\infty}$, $b_{\infty}$, $p$ and $q$. Observe that $a_{\infty},b_{\infty}\in K_0$, but we do not know if the same holds for $p$ and $q$. However, we have good estimates for the distance between these points (see Figure~\ref{fig.driftalongthecenter}). Indeed, by Lemma~\ref{coro depl mes} there exists $\gamma=\gamma(C,T)>1$ and $M=M(C,T)$ such that

\begin{equation}
\frac{\varepsilon}{\gamma} \leq|\mathcal{H}_{p}^c(q)|\leq\gamma \varepsilon,
\end{equation}
and
\begin{equation}
\hat \nu_{a_{\infty}}^c\propto B_* \hat \nu_{b_{\infty}}^c.
\end{equation}
where $B(s)= \beta \cdot s$, satisfies $\frac 1M<|\beta|<M$. We shall prove that these conditions ensure that $b_{\infty}\in G(\eps,\tilde{M})$, for some constant $\tilde{M}$ to be defined later. This will show Proposition~\ref{p.levraidrift}.
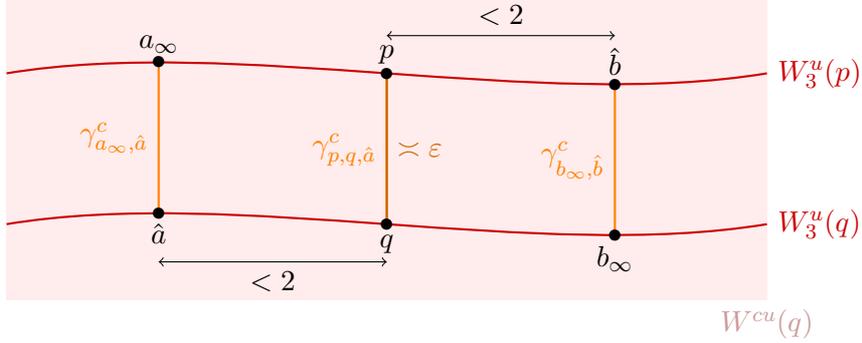
\begin{figure}[h!]
	\begin{tikzpicture}
	\fill[pink!60!white, opacity=.5] (-5,-2) rectangle (5,2);
	\draw[pink!80!black] (5,-2) node[below]{$W^{cu}(q)$};
	\draw[red!80!black, thick] (-5,1) .. controls (-2,1.5) and (2,0.5).. (5,1) node[right]{$W^u_{3}(p)$};
	\begin{scope}[yshift= -2cm]
    \draw[red!80!black, thick] (-5,1) .. controls (-2,1.5) and (2,0.5).. (5,1) node[right]{$W^u_{3}(q)$};
	\end{scope}
	\draw[orange, thick] (-3,1.14)--(-3,-.86) node[midway,left]{$\gamma^c_{a_{\infty},\hat{a}}$};
	\draw (-3,1.14) node{$\bullet$};
	\draw (-3,1.14) node[above]{$a_{\infty}$};
	\draw (-3,-.86) node{$\bullet$}; 
	\draw (-3,-.86) node[below]{$\hat{a}$};
	\begin{scope}[xshift=6cm, yshift=-0.3cm] 
	\draw[orange, thick] (-3,1.14)--(-3,-.86) node[midway,left]{$\gamma^c_{b_{\infty},\hat{b}}$};
	\draw (-3,1.14) node{$\bullet$};
	\draw (-3,1.14) node[above]{$\hat{b}$};
	\draw (-3,-.86) node{$\bullet$}; 
	\draw (-3,-.86) node[below]{$ b_{\infty}$};
	\end{scope}
	\begin{scope}[xshift=3cm, yshift=-0.15cm]
    \draw[orange, thick] (-3,1.14)--(-3,-.86) node[midway,left]{$\gamma^c_{p,q,\hat{a}}$};
    \draw[orange!80!black, thick] (-3,1.14)--(-3,-.86) node[midway,right]{$\asymp\eps$};
    \draw (-3,1.14) node{$\bullet$};
    \draw (-3,1.14) node[above]{$p$};
    \draw (-3,-.86) node{$\bullet$}; 
    \draw (-3,-.86) node[below]{$q$};
	\end{scope}
	\draw[<->] (0,1.5)--(3,1.5) node[midway,above]{$<2$};
    \draw[<->] (-3,-1.5)--(0,-1.5) node[midway,below]{$<2$};
	\end{tikzpicture}
\caption{\label{fig.nolimite} How to get invariance by an affine map using the exponential drift: we can move the leaf-wise measures from $b_{\infty}$ to $a_{\infty}$ with a linear map, and from $a_{\infty}$ to $\hat{b}$ with a linear map and from $\hat{b}$ back to $b_{\infty}$ with an affine map.}
\end{figure}
Notice that by construction the four points $a_{\infty},b_{\infty},p$ and $q$ belong to the same center unstable leaf. Let us consider the local strong unstable manifolds
$\cW^u_{3}(p)$ and $\cW^u_{3}(q)$. Notice that, also by construction we have $a_{\infty}\in\cW^u_{1}(p)$ and $b_{\infty}\in\cW^u_{1}(q)$. An application of Corollary \ref{c.gapdoseba} yields
\begin{equation}
\label{e.dcpq}
d_c(p,q)\leq\kappa\eps.
\end{equation}
Hence we can choose $\eps$ small enough so that the intersection point $\hat{a}=\cW^u(q)\cap \cW^c(a_\infty)$ given by Lemmas~\ref{l.coerenciaum} and \ref{l.coorenciadois} satisfies $d_u(\hat a,q)<2$. Similarly we can choose $\hat{b}=W^u(p)\cap \cW^c(b_\infty)$ so that $d_u(\hat b,p)<2$ (see Figure~\ref{fig.nolimite}). Denote by $\gamma^c_{a_{\infty},\hat{a}}$ the segment of center manifold joining the points $a_{\infty}$ and $\hat{a}$ and similarly consider the segments of center manifolds $\gamma^c_{p,q}$ and $\gamma^c_{b_{\infty},\hat{b}}$. By \eqref{e.dcpq} we have that
\[
\text{length}(\gamma^c_{p,q})=\text{length}(\Phi^c_p[0,\cH^c_p(q)])\asymp_{C,\gamma}\eps.
\]
Note that $d_u(a_{\infty},p)<2$, $d_u(b_{\infty},q)< 2$ and $d_c(p,q)\leq\kappa\eps<\rho_0$, where $\rho_0$ is the constant of  Lemma \ref{l.lemadorho0} (provided $\eps$ is choosen small enough). Hence Lemma \ref{l.lemadorho0} implies that the unstable holonomy maps $H^u_{p,a_{\infty}}$ and $H^u_{q,b_{\infty}}$ are bilipschitz with constants which depends only on $f$. Since $\gamma^c_{a_{\infty},\hat{a}}=H^u_{p,a_{\infty}}(\gamma^c_{p,q})$ and similarly $\gamma^c_{b_{\infty},\hat{b}}=H^u_{p,a_{\infty}}(\gamma^c_{p,q})$ we deduce that 
\[
\text{length}(\gamma^c_{a_{\infty},\hat{a}})\asymp_{C,\gamma}\eps\:\:\:\:\:\textrm{and}\:\:\:\:\:\text{length}(\gamma^c_{b_{\infty},\hat{b}})\asymp_{C,\gamma}\eps
\]   
Therefore, as $\eps$ is small we deduce that $\hat{b}\in \cW^c_1(b_{\infty})$. Using the uniform bound for the $C^1$ norm of the normal forms in segments of bounded length we have that
\begin{equation}
\label{e.dobchapeuprobinfinito}
\cH^c_{b_{\infty}}(\hat{b})\asymp_{C,\gamma}\eps.
\end{equation}
We apply Lemma~\ref{changt inst} to get a linear map $L\colon\R\to\R$ whose derivative satisfies $L^{\prime}(0)\in[C^{-1},C]$, for a constant $C=C(f)$ such that $\hat{\nu}_{\hat{b}}\propto L_*\hat{\nu}_{a_{\infty}}$. Therefore,
\[
\hat{\nu}_{\hat{b}}\propto (LB)_*\hat{\nu}_{b_{\infty}}
\]
and the linear map $L\circ B$ has a derivative bounded by $[C^{-1}M^{-1},CM]$. This together with \eqref{e.dobchapeuprobinfinito} implies that $b_{\infty}\in G(\eps,\tilde{M})$ for some constant $\tilde{M}=\tilde{M}(\gamma,C,M)$ (see \eqref{eq.definitiongem} for the definition of the set $G(\eps,\tilde{M})$), concluding the proof of Proposition~\ref{p.levraidrift}.\qed

\section{Construction of matched $Y$-configurations: end of the proof}\label{s.end_proof}
This section is devoted to the proof of Lemma \ref{l.fatality}. Recall that we reduced our main Theorem \ref{mainthm.dicotomia} to Lemma \ref{l.fatality}. Let us recall what we want to do. We assume that $\mu(\mathbf{B})=0$, where $\mathbf{B}=\mathbf{B}$ is the Bad set introduced in Definition \ref{def bad set}. In particular, it follows from the zero-one law (Theorem \ref{th_0-1}) that for $\mu$-a.e. $x\in \TT$, 
\begin{equation}\label{eq loi zero un}
\mu^s_x\left\{y\in\xi^s(x):\alpha^s(x,y)=0\right\}=0.
\end{equation}
where we recall that $\alpha^s(x,y)=\angle(DH^s_{x,y}(x)E^u(x),E^u(y))$ when $y\in\cW^s(x)$.  Given a large compact set $K_0$ we want to find arbitrarily long pairs of \emph{matched} $Y$-configurations which are $K_0$-good (meaning that their points belong to $K_0$).  

%\begin{theorem}[Joint integrability v.s. angles almost everywhere]\label{thm dich deux}
%	Let $f\colon \T^3\to \T^3$ be an Anosov diffeomorphism with a one-dimensional stable bundle $E^s$ and a two-dimensional unstable bundle $E^{cu}=E^{c}\oplus E^{u}$, and whose stable holonomies are $C^1$. The following dichotomy holds:
%\begin{equation}\label{eq loi zero un}
%%	\item either $\mu(\mathbf{B})=1$,
%%	\item or %for any ergodic $f$-invariant $u$-Gibbs measure $\mu$, it holds: 
%%\mu_x^s(\cP(x))=
%\mu^s_x\left\{y\in\xi^s(x):DH_{x,y}^s E^u(x)=E^u(y)\right\}=0.
%\end{equation}
%\end{theorem}

\subsection{Angle condition and absolute continuity}\label{ss.angle_condition}

We start this section by showing how to use the condition  $\mu(\mathbf{B})=0$ to obtain abundance of pairs of points $y\in\cWs_1(x)$ so that $\alpha^s(x,y)$ is uniformly bounded from below.

Before carrying on the proof recall that for $\ast = s,u,c,cu$, we have fixed measurable partitions $\xi^\ast$ subordinate to $\cW^\ast$  as well as disintegrations $\{\mu^{\ast}_x\}_x$ relative to $\xi^\ast$. We fixed a Lusin set $\cL$ of measure $\mu(\cL)>1-\delta$ (further assumptions on $\delta$ will be given later on) as given in \S \ref{Lusin lusin}. In particular there exists $r_0>0$ such that for every $x\in\cL$, 
\begin{equation}\label{eq_inner_}
\cW^\ast_{r_0}(x)\dans\xi^\ast(x).
\end{equation}

\subsubsection{A Markov type inequality}\label{sss.Markoukouille}
Our first ingredient will be a simple inequality \`a la Markov that will be used several times throughout the section.
\begin{lemma}
	\label{l.leminha}
	Let $(X,\cB,\mu)$ be a probability space, and $\eta\in (0,1)$. Let $\psi\colon X\to[0,1]$ be a measurable function with $\int \psi d\mu>1-\eta$. Let $B\eqdef\{x\in X:\psi(x)>1-\sqrt{\eta}\}$. Then, $\mu(B)>1-\sqrt{\eta}$. 
\end{lemma}
\begin{proof}
	We have	
	\[
	1-\eta<\int_B \psi d\mu+\int_{X\setminus B}\psi d\mu\leq \mu(B)+(1-\sqrt{\eta})(1-\mu(B))=1-\sqrt{\eta}(1-\mu(B)),
	\]
	which gives $\mu(B)>1-\sqrt{\eta}.$
\end{proof} 

%\subsubsection{Ergodic properties of $u$-Gibbs measures}\label{sss_Entropy_Lyap_dim} Recall that $f:\TT\to\TT$ has $1$-dimensional strong unstable and stable bundles. In what follows $\mu$ is an ergodic $u$-Gibbs measure of $f$. We let $h_\mu$ denote the \emph{$\mu$-entropy} of $f$; $\chi^u_\mu$, $\chi^s_\mu$ denote the \emph{Lyapunov exponents} of $\mu$ along $E^s$ and $E^u$ respectively; and 
%$d^s_\mu$ denote the dimension of $\mu$ in the \emph{stable dimension}. Recall that it is defined as
%\begin{equation}\label{def_stable_dimension}
%d^s_\mu=\lim_{r \to 0} \frac{\log \mu_{x}^s[\cWs_{r}(x)]}{\log r},
%\end{equation}
%for $\mu$-a.e. $x\in \TT$, where $\{\mu_x^s\}_x$ is a disintegration of $\mu$ with respect to any partition $\xi^s$ subordinate to $\cW^s$ (see \cite{LYII}).
%
%
%\begin{lemma}\label{l_entropy}
%We have $\chi^u_\mu>0,\,\chi^s_\mu<0,\,h_\mu>0$ and $d^s_\mu>0$
%\end{lemma}
%
%\begin{proof}
%That $\chi^u_\mu>0$ and $\chi^s_\mu<0$ follows from the uniform expansion and contraction of $\cW^u$ and $\cW^s$. In our context (where $\dim \cW^u=1$) have an inequality that holds for every $u$-Gibbs measures: $h_\mu\geq\chi_\mu^u>0$. We refer to \cite{LedStr} or to \cite[Proposition 2.3]{AY} for different approaches (this is also consequence of the general Ledrappier-Young formula: \cite{LYII}). This proves the first item.
%
%Since $f$ and $f^{-1}$ have the same $\mu$-entropy and since $-\chi^s_\mu$ is the unique positive Lyapunov exponent of $f^{-1}$, Ledrappier-Young's formula gives $0<h_\mu=-d^s_\mu\chi^s_\mu$. The second item follows.
%\end{proof}

\subsubsection{Bounding from below the angle function} We fix a measurable partition $\xi^s$ subordinate to $\cW^s$ and a disintegration $\{\mu^s_x\}_x$ of $\mu$ relative to $\xi^s$. Below we use the condition $\mu(\mathbf{B})=0$: by \eqref{eq loi zero un} it means that for $\mu$-a.e. $x\in\T^3$ and $\mu^s_x$-a.e. $y\in\xi^s(x)$, $\alpha^s(x,y)>0$.

%We fix a small constant $\eta$ so that
%\begin{equation}\label{eq_choice_eps}
%0<\delta<\eta<\left(\frac{1}{9}\right)^{\frac{1}{d_\mu^s}},
%\end{equation}
%where we recall that $d^s_\mu$ is the stable dimension of $\mu$. Note that $\d^s_\mu\leq 1$ so we clearly have $\eps<1/9$.

Until the end of the section we fix $r_0>0$ such that \eqref{eq_inner_} holds for all $x\in\cL$. We consider a constant $\eta=\eta(\delta)>0$ which goes to $0$ with $\delta$. This will be explicitly given later on.
\begin{lemma}\label{l.angle} Let $B\dans\cL$ of measure $\mu(B)>1-\eta$. Then there exists a measurable set $B'\dans B$ as well as a number $c=c(\eta)$ such that $0<c\leq r_0\eta$ and such that the following properties hold:

\begin{enumerate}
\item $\mu(B')>1-2\sqrt{\eta}$; and for $\mu$-a.e. $x\in B'$ there exists $y\in B\cap\xi^s(x)$ such that
%\item $d_{\cW^s}(x,y)>c$;
\item $\alpha^s(x,y)>c$.
\end{enumerate}
\end{lemma}

%\begin{lemma}[Sliding along the stable with uniform bounds on angles]
%	\label{l.angle}
%	For $\delta>0$ sufficiently small, there exists $c=c(\delta)>0$ (which tends to $0$ with $\delta$) such that given any set $L^{\radiusprime}\subset\mathbb{T}^3$ with $\mu(L^{\prime})>1-\delta$, there exists a subset $L\subset L^{\prime}$ with large measure ($\mu(L)>1-\delta'$, with $\delta'=\delta'(\delta)\to 0$ as $\delta\to 0$) satisfying: for every $q\in L$, there exists $q^{\prime}\in \mathcal{W}^s(q)\cap L^{\prime}$ such that $c(\delta)< d_{\mathcal{W}^s}(q,q^{\prime})<c(\delta)^{-1}$ and $\angle(DH^s_{q,q^{\prime}}(q)E^u(q),E^u(q^\prime))=|\alpha^s(q,q^{\prime})|> c(\delta)$.
%\end{lemma}

%\begin{remark}[{\color{blue} Pertinent?}]
%	Note that for some constant $c'(\delta)>0$, we also have $|\alpha^s(q,q^{\prime})|<c'(\delta)$, for all $q\in L$, and all $q^{\prime}\in \mathcal{W}^s(q)$ such that $c(\delta)< d_{\mathcal{W}^s}(q,q^{\prime})<c(\delta)^{-1}$. 
%\end{remark}

\begin{proof}
%First recall that by definition of the Lusin set $\cL$, if $A\dans\cL$ then we have $\cW^s_{r_0}(x)\dans\xi^s(x)$ for every $x\in A$.
	
	Let us consider the set $B_1\eqdef\{x \in \TT:\mu_x^s[B\cap\xi^s(x)]>1-\sqrt{\eta}\}$. It follows from Lemma \ref{l.leminha} that $\mu(B_1)>1-\sqrt{\eta}$. 
	
	Now, given $\alpha\geq 0$, we define for $\mu$-a.e. $x \in \TT$:
	\begin{align*}
		\mathscr{A}_\alpha(x)&\eqdef\{y \in \xi^s(x): \alpha^s(x,y)> \alpha\},\\
		\mathscr{A}_{\alpha,\eta}&\eqdef\{x\in \TT: \mu_x^s[\mathscr{A}_\alpha(x)]>1-\eta\}.
	\end{align*}
	Using that $\mu(\mathbf{B})=0$ we find $\mu_x^s[\mathscr{A}_0(x)]=1$,  for $\mu$-almost every $x\in \TT$. This implies that $\mu[\mathscr{A}_{0,\eta}]=1$. Then, there exists $\alpha=\alpha(\eta)>0$ such that the set $B_2\eqdef\mathscr{A}_{\alpha,\eta}$ has measure $\mu(B_2)>1-\eta$. 
	
%	Using the notations of \S \ref{sss_Entropy_Lyap_dim} as well as Lemma \ref{l_entropy} the stable dimension of $\mu$ satisfies $d^s_\mu>0$. Thus, 
%	for $\mu$-a.e. $x\in \TT$, there  exists $\rho(x)>0$, depending measurably on $x$, such that for $r\in (0,\rho(x))$, it holds $\mu_{x}^s[\cWs_r(x)]\leq r^{d^s_\mu/2}$. 
%	We then choose $\rho_0=\rho_0(\eta)>0$ sufficiently small such that the set $A_3\eqdef\{x\in \TT: \rho(x)>\rho_0\}$ has measure $\mu(A_3)>1-\eta$. By definition, for any $x \in A_3$, and for any $r\in (0,\rho_0)$, we have $\mu_x^s[\mathcal{W}_r^s(x)]\leq r^{d^s/2}$.
	
	Let us define the set 
	$$B'\eqdef B\cap B_1\cap B_2,$$
	so
	$$\mu(B')> 1- 2\eta-\sqrt{\eta}>1-2\sqrt{\eta}$$
	(provided $\eta<1/4$, so $2\eta<\sqrt{\eta}$). Then for any $x \in B'$, 
	\begin{itemize}
		\item $\mu_x^s(B\cap\xi^s(x))>1-\sqrt{\eta}$;
		\item $\mu_x^s\{y \in \xi^s(x): \alpha^s(x,y)> \alpha\}>1-\eta$;
	\end{itemize}
	Set $c=c(\eta)\eqdef\min\left(\alpha(\eta),r_0\eta\right)\in \big(0,r_0\eta\big]$. Then, for every $x \in B'$, we have
$$\mu_x^s\{y\in B\cap\xi^s(x): \alpha^s(x,y)>c\}\geq  1-\sqrt{\eta}-\eta>1-2\sqrt{\eta}>0.$$

So we conclude that for every $x\in B'$, then there exists $y\in\xi^s(x)$ such that $\alpha^s(x,y)>c$. This concludes the proof of the lemma.
\end{proof}

\subsection{Recurrence estimates: building long and good $Y$-configurations}\label{recurrence} In this subsection we fix a measurable set $K\dans\T^3$ of measure $\mu(K)>1-\delta$ for some small $\delta>0$. The main result of this paragraph is Proposition \ref{p.claim64}, that builds $K$-good $Y$-configurations of length $\ell$ for every integer $\ell$ inside a subset of $\N$ of positive density. We start by establishing some preliminary results from ergodic theory.

\subsubsection{An elementary quantitative recurrence estimate}\label{sss.elementary}

For the sake of clarity in the presentation, we introduce the following notion.

\begin{defi}
	Given $\gamma>0$, $n>0$ and $B$ a measurable set, we say that a point $x\in\mathbb{T}^3$ is \emph{$(\gamma,n)$-recurrent} to $B$ if $L>n$ implies that 
	\[
	\#\left\{\ell\in[0,L]:f^\ell(x)\in B\right\}> (1-\gamma)L.
	\]    
\end{defi} 

The following holds for any ergodic system $(f,\cB,\mu)$. 

\begin{lemma}
	\label{l.birkhorov}
	For every measurable set $B$ with $\mu(B)>1-\gamma$ there exist $T=T(\gamma)$ and a subset $B^{\circ}\subset B$ with $\mu(B^\circ)>1-\gamma$ so that any $x\in B^{\circ}$ is $(\gamma,T)$-recurrent to $B$.
	
\end{lemma}
\begin{proof}
	Let $B$ be a measurable set with $\mu(B)>1-\gamma$. We consider the sequence $(\varphi_n)_{n \in \N}$ of $L^1$ functions given by $\varphi_n(x)\eqdef \frac{1}{n}\sum_{k=0}^{n-1} \mathbf{1}_{B}(f^k(x))$. By Birkhoff's Theorem, the sequence $(\varphi_n)_{n \in \N}$ converges almost surely to the constant function $\mu(B)$. Moreover, by Egorov's Theorem, there exists a measurable subset $B^\circ\subset B$ of measure $\mu(B^\circ)>1-\gamma$ such that the sequence $(\varphi_n|_{B^\circ})_{n \in \N}$ converges uniformly to $\mu(B)$. Then, there exists $T>0$ such that for any $x \in B^\circ$ and $n>T$, $\varphi_n(x)> 1-\gamma$. Such an $x$ is $(\gamma,T)$-recurrent to $B$ by definition. 
\end{proof}

%For simplicity let us assume that $\mu(K_0)>1-0.2\delta^8$. Applying Lemma~\ref{l.birkhorov} we obtain a set $K\subset K_0$, with measure larger than $1-0.5\delta^8$, and a time $T_0=T_0(\delta)$ so that every point in $K$ is $(\delta,T_0)$-recurrent to $K_0$. We apply Lemma~\ref{l.birkhorov} once more, now to the set $K$ itself, obtaining a set $K_1\subset K$ with measure larger than $1-\delta^8$ and a time $T_1=T_1(\delta)$ (which we may take larger than $T_0$ if needed) so that every point in $K_1$ is $(\delta,T_1)$-recurrent to $K$. 

\subsubsection{Stopping times and return times to $K$ for pairs $(x,x^u)$} 
Recall that we fixed a measurable set $K\dans \TT$ with measure $\mu(K)>1-\delta$. Applying Lemma \ref{l.birkhorov} to $K$ yields {\color{blue}an} integer $T_1=T_1(\delta)$ and a measurable subset $K^\circ\dans K$ with measure $\mu(K^\circ)>1-\delta$ consisting of $(\delta,T_1)$-recurrent points to $K$. Up to enlarging $T_1$ we can assume 
\begin{equation}
\label{e.otemponãopara}
T_1>4T_0
\end{equation}
where $T_0$ is the constant given by the synchronization estimates of Lemmas~\ref{lemme synchro} and \ref{l.synchroinstable}.

Given a pair $\omega=(x,x^u),$ with $x^u\in\xi^u(x)$ we set
\begin{equation}\label{eq_leslongueurs}
E(\omega)\eqdef\{\ell\in\N:f^{\tau(\ell)}(x^u)\in K\:\textrm{and}\:f^{t(\ell)}(x)\in K\}.
\end{equation}

The objects $T_1$ and $K^\circ$ appearing in the next statement are the ones constructed in the previous paragraph.

\begin{lemma}
	\label{l.lemadosaci}
	There exists a constant $r>0$ which only depends on the quasi isometric estimates so that for every $L>T_1$, $x\in K^\circ$ and $x^u\in\xi^u(x)\cap K^\circ$ and $\omega=(x,x^u)$,
	$$\# \left(E(\omega)\cap [0,L]\right) > (1-r\delta)L.$$
\end{lemma}
\begin{proof}
Recall that the stopping times satisfy $t(\ell)\geq 0$ as well as the quasi-isometric estimate
$$\Theta^{-1}|\ell-m|-A<|t(\ell)-t(m)|<\Theta|\ell-m|+A,$$
for some constants $\Theta>1$ and $A>0$ depending only on $f$: see Lemma \ref{l_qi_estimates}. In particular for every $L>0$ we have $t(L)\in[0,L']$ where $L'=\Theta L+A$. By hypothesis, $x\in K^\circ$ so whenever $L>T_1$ (and thus $L'>T_1$), we have
$$\#\{k\in[0,L']:f^k(x)\notin K\}\leq \delta L'.$$

On the other hand, the quasi-isometric estimate also implies that an integer in $[0,L']$ has at most $2A\Theta$ preimages by $t$ so we have
$$\#\{\ell\in[0,L]:f^{t(\ell)}(x)\notin K\} \leq 2A\Theta\#\{k\in[0,L']:f^k(x)\notin K\}\leq 2A\Theta\delta (\Theta L+A).$$

This proves that for every $L>T_1$, the set $E_t(x)\eqdef\{\ell\in\N:f^{t(\ell)}(x)\in K\}$ has density $>1-r_1\delta$ inside $[0,L]$ for some constant $r_1>0$ depending only on $A$ and $\Theta$. Now, since $x^u$ also belongs to $K^\circ$, the same property also holds for the set $E_\tau(x^u)\eqdef\{\ell\in\N:f^{\tau(\ell)}(x^u)\in K\}$: it has density $>1-r_2\delta$ inside $[0,L]$ for $r_2>0$ depending only on the quasi-isometry constants of $\tau$ from Lemma \ref{l_qi_estimates}.

Finally we can estimate from below the density of $E(\omega)=E(x,x^u)=E_\tau(x^u)\cap E_t(x)$ inside $[0,L]$ using the following inequality
$$\#(E(\omega)\cap[0,L])\geq \#(E_\tau(x^u)\cap[0,L])-\#([0,L]\setminus E_t(x))>(1-(r_1+r_2)\delta)L.$$
This ends the proof of the lemma.
\end{proof}

\subsubsection{Space of pairs $(x,x^u)$}\label{sss_spair}
It will be useful to consider a measurable structure on the space of pairs
$$\Omega\eqdef\{\omega=(x,x^u):x\in\TT\:\textrm{and}\:x^u\in\xi^u(x)\}.$$
Note that $\Omega$ is contained inside the continuous submanifold of $\TT\times \TT$ defined as $Y=\{(x,x^u):x\in\TT\:\textrm{and}\:x^u\in \cW^u_1(x)\}$ (recall that $1$ is a uniform upper bound of the diamaters of atoms of $\xi^u$). The topology on $Y$ induced by the product topology of $\TT\times \TT$ provides it with a Borel $\sigma$-algebra. Its restriction to $\Omega$ is denoted by $\cA$. Let $\pi\colon \Omega\to\mathbb{T}^3$ be the projection on the first coordinate. So we have $B\in\cA$ if and only if its projection $\pi(B)$ is a measurable subset of $\mathbb{T}^3$, and $B\cap(\{x\}\times\xi^u(x))$ is a measurable subset of $\{x\}\times\xi^u(x)$. Hence we can define a measure $\nu$ on $\Omega$ by

$$\nu(B)\eqdef \int_{\pi(B)}\mu^u_x\left[B\cap\left(\{x\}\times\xi^u(x)\right)\right] d\mu(x).$$

\subsubsection{Construction of many long and good $Y$-configurations}\label{sss.construction_longandgood} Assume $\ell\in E(\omega)$  for $\omega=(x,x^u)$ with $x,x^u\in K^\circ$. Then $(x,x^u,x_{-\ell},f^{\tau(\ell)}(x^u),f^{t(\ell)}(x))$ is a $K$-good $Y$-configuration of length $\ell$. So constructing many long and good $Y$-configurations means constructing many integers $\ell\in\N$ for which the set of pairs $\omega=(x,x^u)\in K^\circ\times K^\circ$ with $\ell\in  E(\omega)$ has large measure for $\nu$.

More precisely, let us fix a continuous function $\eta\colon[0,1]\to[0,\infty)$ vanishing at $0$ and write abusively $\eta=\eta(\delta)$. An explicit construction will be given in the proof of the next proposition. We will assume that $\delta$ is small enough so that $\eta<1$. For $\ell\in\N$ and $x\in K^\circ$ let us define

\begin{equation}\label{eq_juju}
Q^u(x)\eqdef K^\circ\cap \xi^u(x),
\end{equation}
and
\begin{equation}\label{eq_jujuell}
Q^u(x,\ell)\eqdef\{x^u\in\xi^u(x): x^u\in K^\circ,\,\textrm{and}\,\ell\in E(x,x^u)\}\dans Q^u(x).
\end{equation}
This yields a set of $K$-good $Y$-configurations with length $\ell$. Next we define
\begin{equation}\label{eq_kelkel}
K(\ell)\eqdef\{x\in K^\circ: \mu^u_x[Q^u(x,\ell)]>1-\eta\}.
\end{equation}
We want to construct a large set of integers $\ell$ with $\mu[K(\ell)]$ is large enough. This is provided by the following statement, inspired by Eskin-Lindestrauss' paper \cite{EskinLind} and which is essentially a Fubini-like argument.

\begin{prop}[see Claim 6.4 in Eskin-Lindenstrauss \cite{EskinLind}]\label{p.claim64}	Set 
\begin{equation}\label{dede}
\cD\eqdef\{\ell\in\N:\mu[K(\ell)]>1-\eta\}.
\end{equation}
	 Then for every  $L>T_1$ we have
	\[
	\# (\cD\cap [0,L])>(1-\eta)L.
	\] 	
\end{prop}

\begin{proof}
We start the proof by considering the set
\[
K_1\eqdef\left\{x\in K^\circ:\mu^u_x[Q^u(x)]>1-\sqrt{\delta}\right\},
\]
so that after applying Lemma~\ref{l.leminha} with $\psi\colon x\mapsto\mu^u_x[Q^u(x)]=\mu^u_x[K^\circ\cap\xi^u(x)]$ we see that $\mu(K_1)>1-\sqrt{\delta}$. The set
$$B\eqdef\{\omega=(x,x^u)\in \Omega: x\in K_1,\,x^u\in Q^u(x)\}$$
is a measurable subspace of $\Omega$ of measure
$$\nu(B)=\int_{K_1}\mu^u_x[Q^u(x)]d\mu(x)>(1-\sqrt{\delta})^2=1-\delta',$$
for some $\delta'=\delta'(\delta)>0$ tending to zero with $\delta$.

Now let $L>T_1$ and consider the space $(I_L,m)$, for the set $I_L\eqdef[0,L]\cap\N$ endowed with the counting measure that we denote by $m$. Let $F\subset\Omega\times I_L$ be the set of pairs $(\omega,\ell)$ such that $\omega\in B$ and $\ell\in E(\omega)$. It follows from Lemma~\ref{l.lemadosaci} that for all $\omega\in B$, $m[E(\omega)]>1-r\delta$. Hence
$$(\nu\times m)(F)=\int_B m[E(\omega)]\,d\nu(\omega)>(1-r\delta)(1-\delta')\eqdef1-\delta'',$$
for some $\delta''=\delta''(\delta)$ that tends to zero with $\delta$.

On the other hand set $B(\ell)=\{\omega\in B: \ell\in E(\omega)\}$ so we have
$$F=\{(\omega,\ell)\in\Omega\times I_L:\,\omega\in B(\ell)\}=\bigcup_{\ell\in I_L} B(\ell)\times\{\ell\}.$$
Applying Fubini's theorem we get
	\begin{eqnarray}
	(\nu\times m)(F)&=&\int_{\Omega}\int_{I_L}\textbf{1}_F(\omega,\ell)\, dm(\ell)\, d\nu(\omega)=\int_{I_L}\int_{\Omega}\textbf{1}_F(\omega,\ell)\, d\nu(\omega)\, dm(\ell)\nonumber\\
	&=&\int_{I_L}\nu(B(\ell))\, dm(\ell)\nonumber.
	\end{eqnarray}
It follows from Lemma~\ref{l.leminha} that the set of integers $\cD_L\eqdef\{\ell\in I_L:\nu(B(\ell))>1-\sqrt{\delta''}\}$ has measure $>1-\sqrt{\delta''}$ for $m$. Note that 
$$\nu(B(\ell))=\int_{K_1}\mu^u_x[Q^u(x,\ell)]\, d\mu(x),$$
hence, another application of Lemma~\ref{l.leminha} yields $\mu\{x\in K_1:\mu^u_x[Q^u(x,\ell)]>1-\eta\}>1-\eta$ where $\eta\eqdef\sqrt[4]{\delta''}$ is the function we were looking for. Now since by definition $\{x\in K_1:\mu^u_x[Q^u(x,\ell)]>1-\eta\}\dans K(\ell)$ we conclude that $\mu[K(\ell)]>1-\eta$. This implies that $\cD_L\dans \cD\cap[0,L]$ where $\cD$ is the set defined in \eqref{dede}. We deduce that $\#(\cD\cap[0,L])>(1-\eta)L$ as claimed.
\end{proof}

%\subsubsection{Moving along unstable manifolds}\label{sss.move_unstable} Fix $\xi^u$, a measurable partition subordinate to $\cW^u$ and $\{\mu^u_x\}_x$, a disintegration of $\mu$ relative to $\xi^u$. In particular the unstable dimension of $\mu$ is equal to one. Arguing as in Lemma \ref{l.angle} we get the following
%
%\begin{lemma}\label{l.slide_unstable}
%Let $A\dans\cL$ of measure $\mu(A)>1-\eps$. Then there exists $A''\dans A$ of measure $\mu(A'')>1-2\sqrt{\eps}$ as well as a number $c_1=c_1(\eps)$ such that for every $x\in A''$
%$$\mu_x^u\left\{x^u\in\xi^u(x): d_{\cW^u}(x,x^u)>c_1(\eps)\,\,\,\textrm{and}\,\,\,x^u\in A\right\}>1-2\sqrt{\eps}.$$
%\end{lemma}

\subsection{Construction of matched $Y$-configurations}\label{sss_coupling_sync}
We are now ready to finish the proof of Lemma \ref{l.fatality}, and thus that of our main Theorem. We will build arbitrarily long pairs of good and matched $Y$-configurations. 
%We will first build the tails of the $Y$-configurations, then we will build quadrilaterals, and finally we will use  recurrence properties to synchronize our pair of $Y$-configurations.

% In this paragraph, we show how to use the properties proved above in order to build coupled and synchronized $Y$-configurations which are long and good. This will finish the proof of Lemma 

Let $K_0$ be a compact set included inside the Lusin set $\cL$ defined in \S \ref{Lusin lusin} of measure $\mu(K_0)>1-\delta$. 
%We assume that $\delta$ is small enough so that the number $\eta$ defined in \S \ref{sss.construction_longandgood} satisfies \eqref{eq_choice_eps}. 

We consider the measurable set 
$$K=(K_0)^\circ,$$
so points of $K$ are $(\delta,T)$-recurrent in $K_0$ for some $T=T(\delta)$, and $\mu(K)>1-\delta$. We apply the results of \S \ref{sss.elementary} to $K$. They yield a set $\cD\dans\N$ as defined in \eqref{eq_kelkel} and Proposition~\ref{p.claim64}. In particular, by definition, $\cD$ has density $>1-\eta$ in $[0,L]$ for $L$ large enough and for every $\ell\in\cD$, $\mu[K(\ell)]>1-\eta$ (where $\eta=\eta(\delta)<1$).

\subsubsection{Constructing the tails} Recall that for $x,y\in\TT$ and $\ell\in\N$ we let $x_{-\ell}$ and $y_{-\ell}$ denote $f^{-\ell}(x)$ and $f^{-\ell}(y)$ respectively.

\begin{lemma}\label{l_construction_tails}
For every $\ell\in\cD$, there exists $K'(\ell)\dans K(\ell)$ such that $\mu[K'(\ell)]>1-2\sqrt{\eta}$ and for every $x\in K'(\ell)$, there exists $y\in K(\ell)\cap\xi^s(x)$ such that
$$\alpha^s(x_{-\ell},y_{-\ell})>c(\eta),$$
where $c(\eta)$ is the constant introduced in Lemma \ref{l.angle}.
\end{lemma}

\begin{proof}
We apply Lemma \ref{l.angle} to $A=f^{-\ell}(K(\ell))$ which has measure $\mu(K(\ell))>1-\eta$. Then we can define $K'(\ell)=f^\ell(A')$ where $A'$ is the set provided by that lemma. It is clear that this set satisfies the desired properties (recall that $\xi^s$ is decreasing).
\end{proof}

\subsubsection{Constructing the quadrilaterals: the matching argument}\label{ss.matching} For $x\in K'(\ell)$, so $\xi^u(x)$ is an interval of size $r(x)$, we set
\begin{equation}\label{eq_theIu}I^u(x)\eqdef\left\{z\in\xi^u(x):\,d^u(x,z)> c(\eta),\,\,\,\textrm{and}\,\,\,d^u(z,\partial\xi^u(x))> c(\eta)\right\}.
\end{equation}
Recall that $0<c(\eta)\leq r_0\eta\leq |\xi^u(x)|$, hence $|I^u(x)|>(1-4\eta)|\xi^u(x)|$. Note furthermore that if $y\in\cL$ is close enough to $x$ then $H^{cs}_{x,y}(I^u(x))\dans\xi^u(y)$ (we use here that $r(y)=|\xi^u(y)|$ is uniformly continuous in $\cL$ and that center-stable holonomy maps converge uniformly to the identity as $x$ tends to $y$).

Since $\mu^u_x$ is absolutely continuous continuous with respect to the inner Lebesgue length $|.|$ of $\cW^u(x)$, with a uniform bound $\beta$ on the densities (recall Lemma~\ref{l.density_u-gibbs}) it follows that
	$$\frac{|Q^u(x,\ell)|}{|\xi^u(x)|}>1-\beta\eta,$$

\begin{prop}[Matching argument]\label{p.Eskin_Mirzakoukou}
	There exists $L_0\in\N$ such that for every $\ell\geq L_0$ and $x\in K'(\ell)$, if $y\in K(\ell)\cap\xi^s(x)$ is given by Lemma \ref{l_construction_tails} then there exist $x^u\in\xi^u(x)$ and $y^u\in\xi^u(y)$ such that the following properties hold
	\begin{enumerate}
		\item $d_u(x,x^u)>c(\eta)$;
		\item $y^u=H^{cs}_{x,y}(x^u)$;
		\item $J(x^u,\eps,\ell)\cap Q^u(x,\ell)\neq\emptyset$ and $J(y^u,\eps,\ell)\cap Q^u(y,\ell)\neq\emptyset$.
	\end{enumerate}
\end{prop}

The next paragraphs are devoted to the proof of this proposition.

\subsubsection{Good matching} Let $x,y\in\TT$ be points verifying  Lemma \ref{l_construction_tails}. Note that the lengths of dynamical balls $|J(z)|=|J(z,\eps,\ell)|$ tend uniformly to $0$ as $\ell\to\infty$. Hence we may suppose that $\ell$ is large enough so that for every $z\in I^u(x)$, $J(z)\dans\xi^u(x)$ and $J(H^{cs}_{x,y}(z))\dans\xi^u(y).$

Let $x_1,\ldots,x_k\in I^u(x)\dans\xi^u(x)$ and $J_1,\ldots,J_k\dans \cW^u(x)$ be the dynamical balls defined {\color{blue}by} 
$$J_i\eqdef J(x_i,\eps,\ell)\dans\xi^u(x).$$

We may assume that the following properties are satisfied
\begin{enumerate}
	\item for every $i\geq1$, $d(x_i,x)>c$;
	\item $I^u(x) \dans\bigcup_{i=1}^k J_i$; 
	\item for every $z\in I^u(x)$, $\#\{i:z\in J_i\}\leq 2$.
\end{enumerate}

Set
$$y_i \eqdef H^{cs}_{x,y}(x_i)\in\xi^u(x),\,\,\,\,\,\,\text{and}\,\,\,\,\,\,J_i'\eqdef J(y_i,\eps,\ell)\dans\xi^u(x).$$
The dynamical balls $J_i$ and $J_i'$ are matched. In particular Corollary \ref{c_size_match} implies that 
$$\kappa_4^{-1}|J_i|\leq|J_i'|\leq\kappa_4|J_i|.$$

Set
$$\tau_i\eqdef \tau(x,x_i,\eps,\ell),\,\,\,\,\,\,\text{and}\,\,\,\,\,\,\tau_i'\eqdef\tau(y,y_i,\eps,\ell),$$
so $|\tau_i-\tau_i'|< T_0$ for all $i$.

\begin{lemma}[Control of overlap]\label{l.overlap}
	There exists $m=m(f)\in\N$ such that for every $z\in\xi^u(y)$
	$$\#\{i:z\in J_i'\}\leq m.$$
\end{lemma}

\begin{proof}
	Let $J'_{i_1},\ldots,J'_{i_m}$ be the dynamical balls in $\cW^u(y)$ that contain $z$. Note that for every $j$, $|\tau'_{i_j}-\tau'_z|\leq T_0<T_1$ (see Lemma \ref{l.synchroinstable}). Let
	$$J\eqdef\bigcup_{j=1}^m J_{i_j},\,\,\,\,\,\,\,\,\,\,\,\, J'\eqdef\bigcup_{j=1}^m J_{i_j}',\,\,\,\,\,\,\text{and}\,\,\,\,\,\,\tau^\ast\eqdef\min_{j=1,\ldots m}\{\tau_{i_j},\tau_{i_j}'\}.$$
	
	We claim that the oscillations inside the set $\{\tau_{i_1},\dots,\tau_{i_m},\tau'_{i_1},\dots,\tau'_{i_m}\}$ are less than $T_1$.
Indeed,	by Lemma~\ref{lemme synchro}, for every $j$ we have that $|\tau_{i_j}-\tau'_{i_j}|<T_0$. Also, by Lemma~\ref{l.synchroinstable} we have $|\tau'_{i_j}-\tau(y,z,\eps,\ell)|<T_0$ and thus $|\tau'_{i_j}-\tau'_{i_k}|<2T_0$ and $|\tau_{i_j}-\tau_{i_k}|<4T_0$. We conclude that for every $\alpha\in\{\tau_{i_1},\dots,\tau_{i_m},\tau'_{i_1},\dots,\tau'_{i_m}\}$ it holds
	\begin{equation}
	\label{e.temposmodernos}
	|\alpha-\tau^*|<4T_0,
	\end{equation}
	which proves our claim due to our choice of $T_1$ in \eqref{e.otemponãopara}
	
	On the one hand $f^{\tau_\ast}(z)$ belongs to the intersection of intervals $f^{\tau_\ast}(J_i')$, which have length $\leq 2$, so $|f^{\tau_\ast}(J')|\leq 4$.

	With this claim established we can bound from below the length $|f^{\tau^*}(J)|$. In fact, because at most two intervals $J_i$ can overlap at the same time, we can estimate
	$$\big|f^{\tau^\ast}(J)\big|=\left|\bigcup_{j=1}^mf^{\tau^\ast}(J_{i_j})\right|\geq \frac{1}{2}\sum_{i=1}^m|f^{\tau^\ast}(J_{i_j})|.$$
	Now, by definition of the unstable dynamical ball we have $|f^{\tau_{i_j}}(J_{i_j})|=2$ and by \eqref{e.temposmodernos} $|\tau^{*}-\tau_{i_j}|<T_1$ for every $j=1,\dots,m$. Therefore, 
	\[
	|f^{\tau^*}(J_{i_j})|\geq\frac{2}{\|Df^{-T_1}\|}.
	\]
	Combining the last two inequalities one deduces that 
	\[
	|f^{\tau^\ast}(J)|\geq\frac{m}{\|Df^{-T_1}\|}.
	\]
	This implies in particular that there exist two points of $f^{\tau_\ast}(J)$ that are distant of at least $\frac{m}{2\|Df^{-T_1}\|}$. Since such a point is at distance $\leq 1$ of some $f^{\tau_\ast}(x_i)$ we deduce that there exists $j,l$ such that
	$$d(f^{\tau_\ast}(x_{i_j}),f^{\tau_\ast}(x_{i_l}))\geq \frac{m}{2\|Df^{-T_1}\|}-2.$$
%Now we can use Theorem \ref{thm_regularity}. From Corollary~\ref{c.gapdoseba} we have that 
%\begin{eqnarray}
%d(f^{\tau^\ast}(x_{i_j}),f^{\tau^\ast}(y_{i_j}))&\leq& d(f^{\tau^\ast}(x_{i_j}),f^{\tau^\ast}(z_{i_j}))+d(f^{\tau^\ast}(z_{i_j}),f^{\tau^\ast}(y_{i_j}))\nonumber\\
%&\leq& e^{\chi^s_2\tau^\ast}+\kappa\eps<2,\nonumber
%\end{eqnarray}
%provided that $\ell$ is large enough.

On the one hand $d(f^{\tau^\ast}(y_{i_j}),f^{\tau^\ast}(y_{i_l}))\leq 2$ (this follows from the definition of $\tau^\ast$ and the fact that $f^{\tau^\ast}(z)\in\cW^u_1(f^{\tau^\ast}(y_{i_j}))\cap \cW^u_1(f^{\tau^\ast}(y_{i_l}))$). On the other hand if $\ell$ is large enough and $\eps$ small enough we may ask $d(f^{\tau^\ast}(y_{i_j}),f^{\tau^\ast}(x_{i_j}))\leq\rho_0$ (where $\rho_0$ is the constant of Lemma \ref{thm_regularity}). Hence the Hölder regularity of center-stable holonomies provided by Lemma \ref{thm_regularity} yields
	\begin{align*}
		\frac{m}{2\|Df^{-T_1}\|}-2 &\leq d(f^{\tau_\ast}(x_{i_j}),f^{\tau_\ast}(x_{i_l}))\\
		&\leq C^{cs}d(f^{\tau_\ast}(y_{i_j}),f^{\tau_\ast}(y_{i_l}))^{\theta^{cs}}\\
		&\leq C^{cs} 2^{\theta^{cs}}.
	\end{align*}
	
	We obtain
	$$m\leq (C^{cs}2^{1+\theta^{cs}}+4)\|Df^{-T_1}\|.$$
	This upper bound only depends on $f$, which concludes the proof.
\end{proof}

\subsubsection{Proof of Proposition \ref{p.Eskin_Mirzakoukou}}\label{sss.proof_mirzakoukou} We are now ready to give a proof of Proposition \ref{p.Eskin_Mirzakoukou} which is a modification of that of \cite[Lemma 12.8]{EskinMirzakhani}.

Set
$$\mathcal{I}^x\eqdef\left\{i\in\{1,\ldots,k\}: J_i\cap Q^u(x,\ell)\neq\emptyset\right\},\,\,\,\,\,\,\text{}\,\,\,\,\,\,\mathcal{J}^x\eqdef\{1,\ldots,k\}\setminus\mathcal{I}^x,$$
and
$$\mathcal{I}^y\eqdef\left\{i\in\{1,\ldots,k\}: J_i'\cap Q^u(y,\ell)\neq\emptyset\right\},\,\,\,\,\,\,\text{and}\,\,\,\,\,\,\mathcal{J}^y\eqdef\{1,\ldots,k\}\setminus\mathcal{I}^y.$$

Let
$$Q\eqdef\{z\in Q^u(x,\ell):\forall\, i,(z\in J_i\Rightarrow i\in\mathcal{I}^y)\}.$$
It suffices to show that $Q\cap I^u(x)\neq\emptyset$.

Note that
$$(Q^u(x,\ell)\cap I^u(x))\setminus Q\dans\bigcup_{i\in\mathcal{I}^x\cap\mathcal{J}^y} Q^u(x,\ell)\cap J_i$$
so (supposing that $\ell$ is large enough so that $|\xi^u(y)|\leq 2|\xi^u(x)|$)
\begin{align*}
	|(Q^u(x,\ell)\cap I^u(x))\setminus Q|&\leq  \sum_{i\in\mathcal{I}^x\cap\mathcal{J}^y}|Q^u(x,\ell)\cap J_i|\leq \sum_{i\in\mathcal{I}^x\cap\mathcal{J}^y}|J_i|
	\leq\kappa_4\sum_{i\in\mathcal{J}^y}|J_i'|\\
	&\leq \kappa_4 m \left|\bigcup_{i\in\mathcal{J}^y} J_i'\right|\leq \kappa_4 m|\xi^u(y)\setminus Q^u(y,\ell)|\\
	&\leq \kappa_4 m \beta\eta|\xi^u(y)|\leq 2\kappa_4 m \beta\eta|\xi^u(x)|.
\end{align*}
Hence there exists $\beta'=\beta'(f)$ such that
$$|Q\cap I^u(x)|\geq |Q^u(x,\ell)\cap I^u(x)|-|(Q^u(x,\ell)\cap I^u(x))\setminus Q |\geq (1-\beta'\eta)|\xi^u(x)|>0,$$
as soon as $\eta<\frac{1}{\beta'}$.

Now let $a\in Q$. There exists $i$ such that $d(x,x_i)>c$, $a\in J(x_i,\eps,\ell)$ and $J(y_i,\eps,\ell)\cap Q^u(y,\ell)\neq\emptyset$ where $y_i=H^{cs}_{x,y}(x_i)$: the proof of the proposition is over.\quad \hfill $\square$

\subsubsection{The synchronization} We are now ready to finish the proof of Lemma \ref{l.fatality}. Let $\ell\in\cD$ such that $\ell\geq L_0$. Let $x\in K'(\ell)$ and $y\in K(\ell)\cap\xi^s(x)$ be given by Lemma \ref{l_construction_tails}. Let $x^u\in\xi^u(x)$ and $y^u\in\xi^u(y)$ be given by Proposition \ref{p.Eskin_Mirzakoukou}: there exist $a\in J(x^u,\eps,\ell)\cap Q^u(x,\ell)$ and $b\in J(x^u,\eps,\ell)\cap Q^u(y,\ell)$.

\begin{lemma}\label{constructio_Yconfiguration}
There exists $T>0$ independent of $\ell$ and $\tau,t>0$ such that
\begin{enumerate}
\item $|\tau-\tau(x,x^u,\eps,\ell)|\leq T$ and $|t-t(x,x^u,\eps,\ell)|\leq T$;
\item $f^{\tau}(a),f^{t}(x),f^{\tau}(b),f^{t}(y)\in K_0$.
\end{enumerate}
\end{lemma}

It follows from Lemma \ref{constructio_Yconfiguration} that for $\ell\in\cD$ larger than $L_0$, the $Y$-configurations $X=X(x,x^u,\ell)$ and $Y=Y(y,y^u,\ell)$ are $(K_0,C,T)$-matched. This proves Lemma \ref{l.fatality}.

\begin{proof}
Set
$$\tau(\ell)=\tau(x,a,\varepsilon,\ell),t(\ell)=t(x,a,\varepsilon,\ell),\tau'(\ell)=\tau(y,b,\varepsilon,\ell),t'(\ell)=t(y,b,\varepsilon,\ell).$$

By construction $(x,x^u,y,y^u)$ is a $(C,\ell)$-quadrilateral and $a\in J(x^u)$, $b\in J(y^u)$. Combining the results of Lemmas \ref{lemme synchro} and \ref{l.synchroinstable} we obtain
	$$
	|\tau(\ell)-\tau(x,x^u,\eps,\ell)|<T_0,\quad |t(\ell)-t(x,x^u,\eps,\ell)|<T_0$$
	as well as
	$$|\tau'(\ell)-\tau(x,x^u,\eps,\ell)|<2T_0,\quad |t'(\ell)-t(x,x^u,\eps,\ell)|<2 T_0,
	$$ 
and consequently
	$$
	|\tau'(\ell)-\tau(\ell)|<3T_0,\quad |t'(\ell)-t(\ell)| <3T_0,
	$$ 
where $T_0=T_0(\delta)$.

Proposition \ref{p.Eskin_Mirzakoukou} implies that
$$f^{\tau(\ell)}(a),f^{t(\ell)}(x),f^{\tau'(\ell)}(b),f^{t'(\ell)}(y)\in K,$$
which means that these points are $(\delta,T)$-recurrent inside the compact set $K_0$. Now we might have $\tau(\ell)\neq\tau'(\ell)$ or $t(\ell)\neq t'(\ell)$. Assume for example that $t(\ell)<t'(\ell)$. Let $T'=T'(\delta)>\max(T,T_0)$ such that
$$\frac{T_0}{T'}<\delta,$$
in such a way that $[t(\ell),t'(\ell)]$ has density $<3\delta$ inside $[t(\ell),t'(\ell)+T']$. Since $T'>T$, this implies that 
$$\left\{k\in[t(\ell),t'(\ell)+T']:f^k(x)\in K_0\right\} \cap [t'(\ell), t'(\ell) + T']$$
 has density $>1-4\delta$ in $[t'(\ell),t'(\ell)+T']$. On the other hand the set 
 $$\left\{k\in[t'(\ell),t'(\ell)+T']:f^k(y)\in K_0\right\}$$
has density $>1-\delta$ inside $[t'(\ell),t'(\ell)+T']$. Therefore these two sets must intersect (as soon as $\delta<1/5$) so there exists $t\in[t'(\ell),t'(\ell)+T']$ such that $f^t(x),f^t(y)\in K_0$. This integer $t$ satisfies
$$|t-t(x,x^u,\eps,\ell)|\leq|t-t'(\ell)|+|t'(\ell)-t(x,x^u,\eps,\ell)|<T'+2T_0$$
so $T=T'+2T_0$ is the desired constant. The same argument can be reproduced to treat $a$ and $b$ and find the number $\tau$.
\end{proof}
%
%
%
%
%
%This means that if $T_2=2T+2T_0$, there exists $t_0,\tau_0,t'(0)\tau_0'$
%
%
%
%
%Now let $T=\max(T_0,T_1)$: $\tau(y,y^u,\varepsilon,\ell)$ and $\tau(x,x^u,\varepsilon,\ell)$ belong to the same interval $I$ of size $2T$.
%
%Since $x$ and $y$ are $(T,\delta)$-recurrent, 

\subsubsection{Postliminary: choice of $\delta_0$ and $\delta$}\label{delta_delta0} Let us list the requirements we made on the constant $\delta_0$ and $\delta=3\delta_0$. We needed to require 
$$\delta<\min\big(\frac{1}{5},\frac{1}{r}\big),$$
where $r$ only depends on $f$ (it is defined in Lemma \ref{l.lemadosaci}).

We defined in Proposition \ref{p.claim64} a function $\eta(\delta)$ (of the order of $\delta^{1/4}$) that depends only on $a$ and required that
$$\eta<\min\big(\frac 14,\frac{1}{\beta},\frac{1}{\beta'}\big),$$
where $\beta$ is the constant of Lemma \ref{l.density_u-gibbs} and $\beta'$ appears in \S \ref{sss.proof_mirzakoukou}.

\section{Remarks on Gogolev-Kolmogorov-Maimon's perturbations}
\label{s.final_remarks}

%\subsection{Remarks on Gogolev-Kolmogorov-Maimon's perturbations}

In \cite{GKM} the authors consider two families of perturbations of the linear Anosov diffeomorphism $f_0\colon\TT\to\TT$ induced by the matrix
\[
A = 
\begin{bmatrix}
2 & 1 & 0\\
1 & 2 &1\\
0 & 1 & 1
\end{bmatrix}.
\]

The \emph{dissipative family} is given
$$f_{D,\eps} \begin{bmatrix}
x \\
y \\
z 
\end{bmatrix}=A\begin{bmatrix}
x \\
y \\
z 
\end{bmatrix} +\frac{\eps}{2\pi}\begin{bmatrix}
\sin(2\pi x) \\
0 \\
0 
\end{bmatrix}\,\,\,\,\,\text{mod}(1).$$

The Jacobian at the fixed point of $f_{D,\eps}$ is given by
$$\text{Jac} f_{D,\eps}\begin{bmatrix}
0 \\
0 \\
0 
\end{bmatrix}=\begin{vmatrix}
2+\eps & 1 & 0\\
1 & 2 &1\\
0 & 1 & 1
\end{vmatrix}
=1+\eps,$$
so $f_{D,\eps}$ is dissipative: it does not preserve any volume.

The \emph{conservative family} is given by

$$f_{C,\eps} \begin{bmatrix}
x \\
y \\
z 
\end{bmatrix}=A\begin{bmatrix}
x \\
y \\
z 
\end{bmatrix} +\frac{\eps}{2\pi}\begin{bmatrix}
\sin(2\pi x) \\
\sin(2\pi x) \\
0 
\end{bmatrix}\,\,\,\,\,\text{mod}(1).$$
The Jacobian at any point of $f_{C,\eps}$ is given by
$$\text{Jac} f_{C,\eps}\begin{bmatrix}
x \\
y \\
z 
\end{bmatrix}=\begin{vmatrix}
2+\eps\cos(2\pi x) & 1 & 0\\
1+\eps\cos(2\pi x) & 2 &1\\
0 & 1 & 1
\end{vmatrix}=1.
$$

We apply Gan-Shi's criterion for joint integrability obtained in \cite{GanShi} in order to prove that these two families of perturbations are accessible and hence Theorem \ref{mainthm.dicotomia} applied to both of them gives the following theorem.

\begin{theorem}\label{th_GKM}
If $|\eps|> 0$ is small enough $f_{D,\eps}$ or $f_{C,\eps}$ are accessible. Hence any fully supported ergodic  $u$-Gibbs measure for $f_{D,\eps}$ or $f_{C,\eps}$ is SRB (this is Lebesgue in the conservative case).
\end{theorem}

\begin{proof}
In our context, Gan-Shi's criterion for accessibility is the following. \emph{A $C^{1+\alpha}$ difeomorphism $f$ $C^1$-close to $f_0$ is not accessible if and only if its central Lyapunov exponents at every perdiodic points coincide with those of $f_0$.}

In order to prove the accessibility of $f_{D,\eps}$ and $f_{C,\eps}$ it is enough to prove that the central Lyapunov exponent at the fixed point differ from the one of $A$ (which is approximately $1.55$). This can be done as follows. The characteristic polynomial of the differential at the fixed point for the dissipative family is
$$P_{D,\eps}(X)=\begin{vmatrix}
2-X+\eps & 1 & 0\\
1 & 2-X &1\\
0 & 1 & 1-X
\end{vmatrix}=-X^3+(5+\eps)X^2-(6+3\eps)X+1+\eps.$$

The coefficients of $P_{D,\eps}(X)$ vary smoothly with $\eps$ and when $\eps=0$, this polynomial has three disctinct roots. So by the implicit function theorem, for small $|\eps|$, $P_{D,\eps}$ has three distinct roots $\lambda_1(\eps)<\lambda_2(\eps)<\lambda_3(\eps)$ that depend smoothly on $\eps$ and satisfy the  relations
$$ \begin{cases}
\lambda_1\lambda_2\lambda_3=1+\eps \\
\lambda_1+\lambda_2+\lambda_3=5+\eps \\
\lambda_1\lambda_2+\lambda_1\lambda_3+\lambda_2\lambda_3=6+3\eps
\end{cases} .$$

We must have $\lambda_2'(0)\neq 0$, which yields the accessibility of $f_{D,\eps}$. Indeed, if $\lambda_2'(0)=0$, then derivating the relations above gives (the functions below are evaluated at $0$)

$$ \begin{cases}
\lambda_2(\lambda_1'\lambda_3+\lambda_1\lambda_3')=1 \\
\lambda_1'+\lambda_3'=1 \\
\lambda_2(\lambda_1'+\lambda_3')+\lambda_1'\lambda_3+\lambda_1\lambda_3'=3
\end{cases} .$$

Combining the above one finds $\lambda_2+1/\lambda_2=3$, which contradicts $\lambda_2\simeq 1.55$ (so $\lambda_2+1/\lambda_2\simeq 2.20$).

The same computation works for the conservative family. Indeed, the characteristic polynomial of the differential at the fixed point for this family is
$$P_{C,\eps}(X)=\begin{vmatrix}
2-X+\eps & 1 & 0\\
1+\eps & 2-X &1\\
0 & 1 & 1-X
\end{vmatrix}=-X^3+(5+\eps)X^2-(6+2\eps)X+1.$$

So the relations between roots and coefficients give
$$ \begin{cases}
\lambda_1\lambda_2\lambda_3=1 \\
\lambda_1+\lambda_2+\lambda_3=5+\eps \\
\lambda_1\lambda_2+\lambda_1\lambda_3+\lambda_2\lambda_3=6+2\eps
\end{cases} .$$

The same argument as before provides $\lambda_2'(0)\neq 0$ (the calculations are left to the reader), which yields the accessibility of $f_{C,\eps}$.
\end{proof}

%
%\subsection{Non joint integrability}
%
%In a certain way, our strategy looks at how the direction $E^u$ oscillates along a stable manifold, this is done by using the derivative of the stable holonomies. Recall that in the $Y$-configuration,we had a parameter $l\in \mathbb{N}$. In particular, the displacement we obtain in the center direction between the points $y^u$ and $z^u$ is of the order
%\[
%\frac{\lambda^c_{y_{-l}}(l)}{\lambda^u_{y_{-l}}(l)}.
%\]
%In our setting this is enough to conclude that the $u$-Gibbs measure is SRB.  
%
%Let us remark that in more general settings, one can look at different scales of non joint integrability.  The most important property to make the strategy with the $Y$ configurations work is to obtain a good lower bound control on the distance between $y^u$ and $z^u$. This is the content, for example, of the QNI condition used by Katz. With this control, one can 

%\section{Hyperbolic attractors}
%\input{Sec9_attractors.tex}

\bibliographystyle{plain}
\bibliography{Bib_ALOS}

\begin{thebibliography}{10}

\bibitem{barreira2012ordinary}
L.~Barreira and C.~Valls.
\newblock {\em Ordinary Differential Equations: Qualitative Theory}, volume
  137.
\newblock American Mathematical Soc., 2012.

\bibitem{BenoistQuintI}
Y.~Benoist and J.-F. Quint.
\newblock Mesures stationnaires et ferm\'{e}s invariants des espaces
  homog\`enes.
\newblock {\em Ann. of Math. (2)}, 174(2):1111--1162, 2011.

\bibitem{BowenR}
R.~Bowen.
\newblock {\em Equilibrium states and the ergodic theory of {A}nosov
  diffeomorphisms}, volume 470 of {\em Lecture Notes in Mathematics}.
\newblock Springer-Verlag, Berlin, revised edition, 2008.
\newblock With a preface by David Ruelle, Edited by Jean-Ren\'{e} Chazottes.

\bibitem{BrinBuragoIvanov}
M.~Brin, D.~Burago, and S.~Ivanov.
\newblock Dynamical coherence of partially hyperbolic diffeomorphisms of the
  3-torus.
\newblock {\em J. Mod. Dyn.}, 3(1):1--11, 2009.

\bibitem{Brown_ensaios}
A.~Brown.
\newblock {\em Entropy, {L}yapunov exponents, and rigidity of group actions},
  volume~33 of {\em Ensaios Matem\'{a}ticos [Mathematical Surveys]}.
\newblock Sociedade Brasileira de Matem\'{a}tica, Rio de Janeiro, 2019.
\newblock With appendices by Dominique Malicet, Davi Obata, Bruno Santiago,
  Michele Triestino, S\'{e}bastien Alvarez and Mario Rold\'{a}n, Edited by
  Michele Triestino.

\bibitem{BRH_Little}
A.~Brown and F.~Rodriguez~Hertz.
\newblock Measure rigidity for random dynamics on surfaces with positive
  entropy.
\newblock {\em Preprint, [arXiv:1406.7201]}, 2015.

\bibitem{BRH}
A.~Brown and F.~Rodriguez~Hertz.
\newblock Measure rigidity for random dynamics on surfaces and related skew
  products.
\newblock {\em J. Amer. Math. Soc.}, 30(4):1055--1132, 2017.

\bibitem{CantatDujardin}
S.~Cantat and R.~Dujardin.
\newblock Random dynamics on real and projective surfaces.
\newblock {\em J. Reine Angew. Math., to appear}, 2023.

\bibitem{CrovisierPotrie}
S.~Crovisier and R.~Potrie.
\newblock {\em Introduction to Partially Hyperbolic Dynamics}.
\newblock 2015.

\bibitem{CrovisierPotrieSamba}
S.~Crovisier, R.~Potrie, and M.~Sambarino.
\newblock Finiteness of partially hyperbolic attractors with one-dimensional
  center.
\newblock {\em Ann. Sci. \'{E}c. Norm. Sup\'{e}r. (4)}, 53(3):559--588, 2020.

\bibitem{didier2003stability}
Ph~Didier.
\newblock Stability of accessibility.
\newblock {\em Ergodic Theory and Dynamical Systems}, 23(6):1717--1731, 2003.

\bibitem{EinsLind}
M.~Einsiedler and E.~Lindenstrauss.
\newblock Diagonal actions on locally homogeneous spaces.
\newblock {\em Homogeneous flows, moduli spaces and arithmetic}, 10:155--241,
  2010.

\bibitem{EinsWar}
M.~Einsiedler and T.~Ward.
\newblock {\em Ergodic theory with a view towards number theory}, volume 259 of
  {\em Graduate Texts in Mathematics}.
\newblock Springer-Verlag London, Ltd., London, 2011.

\bibitem{EskinLind}
A.~Eskin and E.~Lindenstrauss.
\newblock Random walks on locally homogeneous spaces.
\newblock {\em Preprint}, 1(2):6, 2018.

\bibitem{EskinMirzakhani}
A.~Eskin and M.~Mirzakhani.
\newblock Invariant and stationary measures for the {${\rm SL}(2,\Bbb R)$}
  action on moduli space.
\newblock {\em Publ. Math. Inst. Hautes \'{E}tudes Sci.}, 127:95--324, 2018.

\bibitem{EskinPotrieZha}
A.~Eskin, R.~Potrie, and Z.~Zhang.
\newblock Geometric properties of partially hyperbolic measures and
  applications to measure rigidity.
\newblock {\em Preprint, [arXiv:2302.12981]}, 2023.

\bibitem{GanShi}
S.~Gan and Y.~Shi.
\newblock Rigidity of center {L}yapunov exponents and {$su$}-integrability.
\newblock {\em Comment. Math. Helv.}, 95(3):569--592, 2020.

\bibitem{GogolevKalSad}
A.~Gogolev, B.~Kalinin, and V.~Sadovskaya.
\newblock Center foliation rigidity for partially hyperbolic toral
  diffeomorphisms.
\newblock {\em Preprint, [arxiv:1908.03177]}, 2019.

\bibitem{GKM}
A.~Gogolev, A.~Kolmogorov, and I~Maimon.
\newblock A numerical study of {G}ibbs {$u$}-measures for partially hyperbolic
  diffeomorphisms on {$\Bbb T^3$}.
\newblock {\em Exp. Math.}, 28(3):271--283, 2019.

\bibitem{GogolevPath}
Andrey Gogolev.
\newblock How typical are pathological foliations in partially hyperbolic
  dynamics: an example.
\newblock {\em Israel J. Math.}, 187:493--507, 2012.

\bibitem{NancyChichi}
N.~Guelman and S.~Martinchich.
\newblock Uniqueness of minimal unstable lamination for discretized {A}nosov
  flows.
\newblock {\em Math. Z.}, 300(2):1401--1419, 2022.

\bibitem{GuyKat}
M.~Guysinsky and A.~Katok.
\newblock Normal forms and invariant geometric structures for dynamical systems
  with invariant contracting foliations.
\newblock {\em Math. Res. Lett.}, 5(1-2):149--163, 1998.

\bibitem{HammPotrie}
A.~Hammerlindl and R.~Potrie.
\newblock Pointwise partial hyperbolicity in three-dimensional nilmanifolds.
\newblock {\em J. Lond. Math. Soc. (2)}, 89(3):853--875, 2014.

\bibitem{HirschPughShub}
M.~W. Hirsch, C.~C. Pugh, and M.~Shub.
\newblock {\em Invariant manifolds}.
\newblock Lecture Notes in Mathematics, Vol. 583. Springer-Verlag, Berlin-New
  York, 1977.

\bibitem{KalininKatok}
B.~Kalinin and A.~Katok.
\newblock Measure rigidity beyond uniform hyperbolicity: invariant measures for
  {C}artan actions on tori.
\newblock {\em J. Mod. Dyn.}, 1(1):123--146, 2007.

\bibitem{KalSadI}
B.~Kalinin and V.~Sadovskaya.
\newblock Normal forms on contracting foliations: smoothness and homogeneous
  structure.
\newblock {\em Geom. Dedicata}, 183:181--194, 2016.

\bibitem{KalSadII}
B.~Kalinin and V.~Sadovskaya.
\newblock Normal forms for non-uniform contractions.
\newblock {\em J. Mod. Dyn.}, 11:341--368, 2017.

\bibitem{Katok_Hasselblatt}
A.~Katok and B.~Hasselblatt.
\newblock {\em Introduction to the modern theory of dynamical systems},
  volume~54 of {\em Encyclopedia of Mathematics and its Applications}.
\newblock Cambridge University Press, Cambridge, 1995.
\newblock With a supplementary chapter by Katok and Leonardo Mendoza.

\bibitem{KatokSpa}
A.~Katok and R.~Spatzier.
\newblock Invariant measures for higher-rank hyperbolic abelian actions.
\newblock {\em Ergodic Theory and Dynamical Systems}, 16(4):751--778, 1996.

\bibitem{Katz}
A.~Katz.
\newblock Measure rigidity of {A}nosov flows via the factorization method.
\newblock {\em Geom. Funct. Anal.}, 33(2):468--540, 2023.

\bibitem{LedStr}
F.~Ledrappier and J.-M. Strelcyn.
\newblock A proof of the estimation from below in {P}esin's entropy formula.
\newblock {\em Ergodic Theory \& Dynam. Systems}, 2(2):203--219, 1982.

\bibitem{LYI}
F.~Ledrappier and L.-S. Young.
\newblock The metric entropy of diffeomorphisms. {I}. {C}haracterization of
  measures satisfying {P}esin's entropy formula.
\newblock {\em Ann. of Math. (2)}, 122(3):509--539, 1985.

\bibitem{Livsic}
A.~N. Liv\v{s}ic.
\newblock Certain properties of the homology of {$Y$}-systems.
\newblock {\em Mat. Zametki}, 10:555--564, 1971.

\bibitem{casanova}
S.~E Newhouse.
\newblock On codimension one anosov diffeomorphisms.
\newblock {\em American Journal of Mathematics}, 92(3):761--770, 1970.

\bibitem{Obata}
D.~Obata.
\newblock Open sets of partially hyperbolic skew products having a unique {SRB}
  measure.
\newblock {\em Adv. Math.}, 427:Paper No. 109136, 91, 2023.

\bibitem{Palis}
J.~Palis.
\newblock A global view on dynamics and a conjecture on the denseness of
  finitude of attractors.
\newblock {\em Ast\'erisque}, 261:335--347, 2000.

\bibitem{potrie}
R.~Potrie.
\newblock Partial hyperbolicity and foliations in $\mathbb{T}^3$.
\newblock {\em Journal of Modern Dynamics}, 9(01):81, 2015.

\bibitem{PSW}
C.~Pugh, M.~Shub, and A.~Wilkinson.
\newblock H\"{o}lder foliations.
\newblock {\em Duke Math. J.}, 86(3):517--546, 1997.

\bibitem{Rat2}
M.~Ratner.
\newblock On measure rigidity of unipotent subgroups of semisimple groups.
\newblock {\em Acta Math.}, 165(3-4):229--309, 1990.

\bibitem{Rat1}
M.~Ratner.
\newblock Strict measure rigidity for unipotent subgroups of solvable groups.
\newblock {\em Invent. Math.}, 101(2):449--482, 1990.

\bibitem{Rat3}
M.~Ratner.
\newblock On {R}aghunathan's measure conjecture.
\newblock {\em Ann. of Math. (2)}, 134(3):545--607, 1991.

\bibitem{Rat4}
M.~Ratner.
\newblock Raghunathan's topological conjecture and distributions of unipotent
  flows.
\newblock {\em Duke Math. J.}, 63(1):235--280, 1991.

\bibitem{HHUBook}
F.~Rodriguez~Hertz, J.~Rodriguez~Hertz, and R.~Ures.
\newblock Partially hyperbolic dynamics.
\newblock {\em Publica{\c{c}}oes Matem{\'a}ticas do IMPA}, 2011.

\bibitem{HertzHertzUres}
F.~Rodriguez~Hertz, M.~A. Rodriguez~Hertz, and R.~Ures.
\newblock Accessibility and stable ergodicity for partially hyperbolic
  diffeomorphisms with 1{D}-center bundle.
\newblock {\em Invent. Math.}, 172(2):353--381, 2008.

\bibitem{HertzUres}
J.~Rodriguez~Hertz and R.~Ures.
\newblock On the three-legged accessibility property.
\newblock In {\em New trends in one-dimensional dynamics}, volume 285 of {\em
  Springer Proc. Math. Stat.}, pages 239--248. Springer, Cham, 2019.

\bibitem{Rok_meas-th}
V.~A. Rokhlin.
\newblock On the fundamental ideas of measure theory.
\newblock {\em Amer. Math. Soc. Translation}, 1952(71):55, 1952.

\bibitem{Yang_exp}
J.~Yang.
\newblock Entropy along expanding foliations.
\newblock {\em Adv. Math.}, 389:Paper No. 107893, 39, 2021.

\bibitem{LSYoung}
L.S. Young.
\newblock What are {SRB} measures, and which dynamical systems have them?
\newblock {\em J. Statist. Phys.}, 108(5-6):733--754, 2002.

\end{thebibliography}
\begin{flushleft}
{\scshape S\'ebastien Alvarez}\\
	CMAT, Facultad de Ciencias, Universidad de la Rep\'ublica\\
	Igua 4225 esq. Mataojo. Montevideo, Uruguay.\\
	email: \texttt{salvarez@cmat.edu.uy}

	\vspace{0.2cm}

{\scshape Martin Leguil}\\
	Laboratoire Ami\'enois de Math\'ematique Fondamentale et Appliqu\'ee (LAMFA, UMR 7352)\\
	Université de Picardie Jules Verne, 33 rue Saint Leu, 80039 Amiens, France\\
	email: \texttt{martin.leguil@u-picardie.fr}

	\vspace{0.2cm}
	
{\scshape Davi Obata}\\
    Department of Mathematics, University of Chicago\\
	5734 S. University Avenue, Chicago, Illinois 60637. USA\\
	email:  \texttt{davi.obata@math.uchicago.edu}

	\vspace{0.2cm}
	
{\scshape Bruno Santiago}\\
	Instituto de Matem\'atica e Estat\'istica, Universidade Federal Fluminense\\
	Rua Prof. Marcos Waldemar de Freitas Reis, S/N -- Bloco H, 4o Andar\\
	Campus do Gragoatá, Niterói, Rio de Janeiro 24210-201, Brasil\\
	email:  \texttt{brunosantiago@id.uff.br}

\end{flushleft}

\end{document}